\documentclass[11pt,reqno]{amsart}
\usepackage{graphicx}
\usepackage{color}
\usepackage{amsmath,amssymb,amsthm,amsfonts,mathrsfs}
\usepackage{framed}
\usepackage[colorlinks=true]{hyperref}%
\usepackage{cases,indentfirst}
\usepackage{graphicx}
\usepackage{color}
\usepackage{multicol}
\makeatletter

\newcommand{\Rmnum}[1]{\expandafter\@slowromancap\romannumeral #1@}
\makeatother

\newtheorem{thm}{Theorem}[section]

\newtheorem{lemma}[thm]{Lemma}
\newtheorem{remark}{Remark}[section]
\newtheorem{theorem}[thm]{Theorem}
\newtheorem{proposition}[thm]{Proposition}

\usepackage[numbers,sort&compress]{natbib}
\numberwithin{equation}{section}

\allowdisplaybreaks[4]
\begin{document}
\author[Z.-F. Feng]{Zefu Feng}
\address{Zefu Feng\newline\indent School of Mathematics\newline\indent South China University of Technology\newline\indent  Guangzhou 510640, China}
\email{zefufeng@mails.ccnu.edu.cn}
\author[G.-Y. Hong]{Guangyi Hong}\address{Guangyi Hong\newline\indent School of Mathematics\newline\indent South China University of Technology\newline\indent Guangzhou 510641, China}
\email{magyhong@scut.edu.cn}
\author[Y.-H. Wang]{Yinghui Wang}
\address{Yinghui Wang\newline\indent MOE-LCSM  \newline\indent School of Mathematics and Statistics\newline\indent Hunan Normal University \newline\indent   Changsha 410081, China}
\email{yhwangmath@hunnu.edu.cn}
\author[J.H. Wu]{Jiahong Wu}
\address{Jiahong Wu\newline\indent Department of Mathematics\newline\indent University of Notre
Dame\newline\indent Notre Dame, IN 46556, USA }
\email{jwu29@nd.edu}

\title[Stability of compressible flows with eddy diffusion]{Global stability and anisotropic large-time behavior\\ of the three-dimensional compressible Navier--Stokes equations with eddy diffusion}

\begin{abstract}
	We study the Cauchy problem for the three-dimensional compressible
	Navier--Stokes equations with eddy diffusion, an anisotropic dissipative
	mechanism that arises naturally in geophysical fluid dynamics
	(cf.~\cite{Jabin-Bresch-2018,Temam-Ziane-2004}).
	In contrast to the classical compressible Navier--Stokes system, the
	momentum equation here carries no full vertical Laplacian: the velocity is
	diffused only in the horizontal directions, and the sole vertical
	regularization it receives is the partial one transmitted through the
	compressible mode $\operatorname{div}\mathbf{u}$.
	This degeneracy invalidates the standard parabolic energy framework as well
	as the classical high--low frequency Green-function bounds.
	We prove that the constant non-vacuum equilibrium $(\bar{\rho},0)$ is
	globally nonlinearly stable against small Sobolev perturbations: global
	classical solutions exist in $H^{N}(\mathbb{R}^{3})$ for every $N\ge 3$, and
	the density and velocity relax to equilibrium with explicit, genuinely
	anisotropic decay rates.
	The mechanism behind the result is a hidden dissipation produced by the
	pressure--divergence coupling between $\nabla\rho$ and
	$\operatorname{div}\mathbf{u}$, which compensates for the missing vertical
	smoothing of the density and the compressible part of the velocity; the
	solenoidal part of the velocity, by contrast, is governed by a purely
	horizontal heat flow and therefore decays only at the two-dimensional rate.
	The analysis rests on a refined anisotropic spectral decomposition of the
	Green matrix, a div--curl treatment of the velocity, and time-weighted
	nonlinear energy estimates tailored to the degenerate dissipation. To the
	best of our knowledge, this is the first global stability and large-time
	behavior result for the three-dimensional compressible Navier--Stokes
	equations with eddy diffusion in the whole space.
\end{abstract}



\maketitle

\numberwithin{equation}{section}
\section{Introduction}

A guiding question in the mathematical theory of viscous compressible fluids
is how much dissipation a constant equilibrium needs in order to be stable.
For the classical compressible Navier--Stokes equations with full viscosity
the answer is by now classical: a small perturbation of a non-vacuum constant
state generates a unique global smooth solution, and that solution relaxes to
equilibrium at the diffusion rates prescribed by the linearized Green
function. The picture changes markedly once the viscosity becomes anisotropic
and the vertical shear dissipation is lost. Such a degeneracy is not a
mathematical artifact: it appears naturally in geophysical fluid dynamics
(cf.~\cite{Temam-Ziane-2004,Jabin-Bresch-2018}), where turbulent transport is
routinely modeled by an eddy diffusion that acts preferentially in the
horizontal directions. The present paper is devoted to the global dynamics of
the three-dimensional compressible Navier--Stokes equations under precisely
this kind of degenerate, direction-dependent dissipation.

More precisely, we consider the system
 \begin{gather}\label{1-1}
\begin{cases}
 \displaystyle\rho_t+\mathop{\mathrm{div}}\nolimits(\rho \mathbf{u})=0,\\[2mm]
\displaystyle (\rho \mathbf{u})_t+ \mathop{\mathrm{div}}\nolimits(\rho \mathbf{u}\otimes \mathbf{u})+\nabla P(\rho)=\mu\Delta_h \mathbf{u}+(\mu+\zeta)\nabla \mathop{\mathrm{div}}\nolimits \mathbf{u},
\end{cases}
\end{gather}
posed for $ x \in \mathbb{R}^{3} $ and $ t >0  $. Here $\rho(t,x) \geq 0$ and
$ \mathbf{u}(t,x) $ denote the density and the velocity of the fluid, the
pressure $ P=P(\rho) $ is a smooth function with $ P'(\rho)>0 $, and
$\Delta_h:=\partial^2_{x_1}+\partial^2_{x_2}$ is the horizontal Laplacian. The
viscous operator appearing in \eqref{1-1},
\begin{align*}
\displaystyle \mathcal{D}= \mu\Delta_h \mathbf{u}+(\mu+\zeta)\nabla \mathop{\mathrm{div}}\nolimits \mathbf{u},
\end{align*}
is the eddy diffusion alluded to above
(cf.~\cite{Temam-Ziane-2004,Jabin-Bresch-2018}). Its defining feature, and the
source of every difficulty in this work, is that it supplies {no full
vertical smoothing} of the velocity: the horizontal Laplacian diffuses only
$\nabla_h\mathbf{u}$, and the lone trace of vertical regularization is the one
carried by the compressible mode through $\nabla\operatorname{div}\mathbf{u}$.
In particular, the shear (solenoidal) part of the velocity is diffused in the
horizontal variables alone.
The constants $\mu$ and $\zeta$ are the viscosity coefficients, subject to the physical constraints
\begin{equation*} \displaystyle
\mu>0, \ \ \ 3 \zeta+2\mu>0.
\end{equation*}
Let $ \bar{\rho}>0 $ be a given constant. We study the Cauchy problem for \eqref{1-1} subject to
\begin{align}
\displaystyle (\rho, \mathbf{u})(x,0)=(\rho _{0},\mathbf{u}_{0})(x),
\end{align}
and the far-field condition
\begin{gather}\label{far-field}
\displaystyle  (\rho, \mathbf{u})(t,x) \rightarrow (\bar{\rho}, 0) \quad \text { as } \quad \vert x\vert \rightarrow \infty.
\end{gather}
Our aim is twofold. First, we establish the global nonlinear stability of the
non-vacuum equilibrium $(\bar\rho,0)$ under small smooth perturbations.
Second, we identify the large-time profile dictated by the degenerate eddy
diffusion and prove explicit, anisotropic decay rates for the density and the
velocity.

\medskip
\noindent\textbf{Background and earlier work.}
The literature on the compressible Navier--Stokes equations with full
viscosity is vast, and we recall here only what is needed to frame our
results. Local well-posedness goes back to the seminal work of Nash
\cite{Nash-1962} and was developed further by Kanel \cite{Kanel-1968}, Itaya
\cite{Itaya-1971}, Vol'pert--Hudjaev \cite{Volpert-Hudjaev} and Tani
\cite{Tani-1974}, among others. In one space dimension, the first global
strong solutions were obtained by Kazhikhov and Shelukhin
\cite{kazhikhov-1977} for an initial--boundary value problem in $ H^{1} $; the
Cauchy problem was treated in \cite{Kazhikhov-1982-Cauchy}, and global weak
solutions were constructed in \cite{Chen-David-Trivisa-2000, Jiang-ZLIONIK}.

The large-time behavior in one dimension was studied by Jiang
\cite{JiangSong-1999} for initial--boundary value problems and by Li--Liang
\cite{LiangZhilei-Lijing-2016ARMA} for the Cauchy problem. In several space
dimensions, local well-posedness was first addressed by Serrin
\cite{Serrin-1959} and Nash \cite{Nash-1962}. The foundational result on
global existence and large-time asymptotics is due to Matsumura and Nishida
\cite{matsumura-nishida-1980}, who treated small perturbations of a constant
non-vacuum equilibrium in Sobolev spaces: for data close to the state
$(1,0)$ in $ H^{3} \cap L^{1}$, they proved the decay estimate
\begin{gather*}
\displaystyle  \|(\rho-1,\mathbf{u})\|_{H^2} \lesssim (1+t)^{-\frac{3}{4}},
\end{gather*}
where $ A\lesssim B $ means $ A\leq CB $ for some $ C>0 $ independent of $ t $.
This $ L^{2} $ rate is optimal, as it matches that of the linearized system.
Optimal $ L^{p} $ rates were subsequently obtained by Ponce
\cite{Ponce-1985-NS} for small data in $ H^{m} \cap W^{m,1} $ with $ m \geq 4 $,
\begin{align*}
\displaystyle \|\nabla ^{l}(\rho-1, \mathbf{u})\|_{L ^{p}} \lesssim (1+t)^{- \frac{3}{2}(1- \frac{1}{q})- \frac{l}{2}},
\end{align*}
for $ 2 \leq q \leq \infty $ and $ 0 \leq l \leq 2 $. Through a careful
analysis of the Green's function, optimal $ L^{q}\,(1 \leq q \leq \infty) $
rates were also derived by Hoff et al. \cite{Hoff-1997-ZAMP,Hoff-1995-INDIANA}
and Liu et al. \cite{liutaiping-1998}; see
\cite{K-1983, K-1987, SK-1985} and the references therein for further
developments. Problems involving vacuum form a substantial parallel direction;
see, among others,
\cite{Xin-1998-blowup,JiangSong-1999,Jiang-Zhang-2001-cmp,huang-LI-XIN-2012,sun-wang-zhang-jmpa,Li-Xin-2019-AnnPDE,Hong-Hou-Peng-ZHU-mathAnn}.

A common thread running through all of these results is that the decay
analysis in the fully dissipative theory hinges on the Green matrix
$G(t,x)$ of the linearized system obeying the standard parabolic pointwise
bounds
\begin{align}\label{usual-green}
|\widehat{G}(t,\xi)|\lesssim
\begin{cases}
 e^{-\vartheta_1 t}, & |\xi|\geq R,\\[1mm]
 e^{-\vartheta_2|\xi|^2t}, & |\xi|\leq R,
\end{cases}
\end{align}
for suitable positive constants $\vartheta_1$, $\vartheta_2$ and $R$. For
\eqref{1-1} these bounds simply fail: the horizontal eddy diffusion controls
only $\nabla_h\mathbf{u}$, and the vertical smoothing of the velocity cannot be
recovered from a full Laplacian. It is precisely this loss of \eqref{usual-green}
that places both the global energy estimates and the nonlinear decay analysis
outside the reach of the classical theory.

Beyond the Green-function approach, the large-time behavior of classical
solutions has also been obtained through pure energy methods; see, for
instance, \cite{Deckelnick-1992-mathZ, Deckelnick-1993-CPDE, guo-wang-12-cpde}.
For global weak solutions with discontinuous data of small energy we refer to
the work of Hoff \cite{david-hoff-1995-JDE, hoff-1997-arma}, whose long-time
behavior was later analyzed in \cite{Hu-Wu-2020-jde}.


By contrast, the eddy-diffusion system \eqref{1-1} has received far less
attention. To the best of our knowledge, neither the global existence nor the
large-time behavior of classical solutions to the three-dimensional
compressible Navier--Stokes equations with eddy diffusion in the whole space
has previously been established. The available results on partial or
anisotropic dissipation deal almost exclusively with the incompressible
Navier--Stokes equations or related incompressible models; see, for instance,
\cite{JWY-2021,LWZ-2023,LWZ-2022}. The compressible setting raises genuinely
new structural difficulties: the density equation is hyperbolic and carries no
dissipation of its own, the velocity is not divergence-free, and the
solenoidal dynamics enjoy only horizontal heat smoothing. The central message
of the present paper is that the missing vertical dissipation can nonetheless
be recovered \emph{from the compressible structure itself}: the pressure
coupling between $\nabla\rho$ and $\operatorname{div}\mathbf{u}$ generates a
hidden dissipation that closes the energy estimates, while the solenoidal
component is controlled by a refined anisotropic argument. The decay analysis
is correspondingly more delicate, owing to the anisotropy of the linear
semigroup and the strength of the convective nonlinearity.


{\bf Notations.} Before stating the main results, we explain the notations and conventions throughout this paper. We denote by $C$ a generic positive constant. For two quantities $A$ and $B$, we write $A\thicksim B$ if there exists a universal constant $ C>0 $ such that $C ^{-1} A\leq B\leq CA$. The notation $A\lesssim B$ means that $A\leq C B$ for a universal constant $C>0$ independent of time $t$. We denote by $ [A,B]:=AB -BA $ the standard commutator. For any integer $m\geq 0$, we use $H^m$ to denote the standard Sobolev space $H^m(\mathbb{R}^3)$ with norm $ \|\cdot\| _{H ^{m}}$ and $ L ^{p}~(1 \leq p \leq \infty) $ to denote the usual $ L ^{p}(\mathbb{R}^{3}) $ space with norm $ \|\cdot\| _{L ^{p}} $. Similarly, we denote
\begin{gather*}
\displaystyle L _{x _{k}}^{p}:=L _{x _{k}}^{p}(\mathbb{R}),\ \ L _{x _{k}x _{j}}^{p}:=L _{x _{k}x _{j}}^{p}(\mathbb{R}^{2}),\ \ L _{x _{3}}^{p}L _{x _{1}x _{2}}^{q}:=L _{x _{3}}^{p}L _{x _{1}x _{2}}^{q}(\mathbb{R}^{3}),\\[2mm]
\|\cdot\|_{L ^{p}L ^{r}L ^{q}}:= \Big\| \big \|\|\cdot\|_{L ^{q}}  \big\| _{L ^{r}}\Big\|_{L ^{p}},\ \ \|\cdot\|_{L _{x _{3}}^{p}L _{x _{1}x _{2}}^{q}}:= \left \|\|\cdot\|_{L _{x _{1}x _{2}}^{q}}\right\|_{L _{x _{3}}^{p}}.
\end{gather*}
We utilize $\langle\cdot,\cdot\rangle$ to denote the inner product over the Hilbert space $L^2(\mathbb{R}^3)$, i.e.,
$$\langle f,g\rangle=\int_{\mathbb{R}^{3}}f(x)g(x)\,\mathrm{d} x,~~f=f(x),~g=g(x)\in L^2(\mathbb{R}^3).$$
We shall denote $ \partial _{i}=\partial _{x _{i}} $. For a multi-index $\alpha=(\alpha_1,\alpha_2,\alpha_3),$ we denote the length of $ \alpha $ by $|\alpha|=\alpha_1+\alpha_2+\alpha_3$. For any multi-indexes $ \alpha $ and $ \beta $, we denote $ \alpha+\beta =(\alpha _{1}+ \beta _{1},\alpha _{2}+ \beta _{2},\alpha _{3}+ \beta _{3})$. We denote $\partial ^\alpha=\partial^{\alpha_1}_{1}\partial^{\alpha_2}_{2}\partial^{\alpha_3}_{3}$, and $ \partial _{h}=\partial _{1} $ or $ \partial _{2} $. For simplicity, we set
\begin{gather*}
\nabla^k\upsilon=\{\partial^\alpha\upsilon_i \vert\, |\alpha|=k;\,i=1,2,3\},\ \ \nabla _{h}\upsilon=\left\{ \partial _{i}\upsilon\vert\, i=1,2 \right\} ,
\end{gather*}
with
\begin{gather*}
\displaystyle \vert \nabla ^{k}\upsilon\vert =\left( \sum _{\vert \alpha\vert=k}\vert \partial ^{\alpha}\upsilon\vert ^{2} \right)^{\frac{1}{2}},\ \ \vert \nabla _{h}\upsilon\vert   =\left( \vert \partial _{1}\upsilon\vert ^{2}+ \vert \partial _{2}\upsilon\vert ^{2}\right) ^{\frac{1}{2}}
\end{gather*}
for $ \upsilon=(\upsilon_1,\upsilon_2,\upsilon_3) $. Also, for convenience, we write
\begin{gather*}
\displaystyle U(t):=(\rho,u),\ \ U _{0}:=(\rho _{0},u _{0}),\ \ \left\|(A,B)\right\|_{X}:=\|A\|_{X}+\|B\|_{X}.
\end{gather*}

\vspace{2mm}
We now present our main results. Our first result concerns the global existence of classical solutions to the problem \eqref{1-1}--\eqref{far-field}.
\begin{theorem}\label{thm1}
Let $N\geq 3$ be a given integer. There exists a constant $\delta_0>0$ such that if $$\|(\rho_0-\bar{\rho}, \mathbf{u}_0)\|_{H ^{N} }\leq \delta_0,$$
then the problem \eqref{1-1}--\eqref{far-field} admits a unique global classical solution $(\rho,\mathbf{u})$ such that
\begin{gather}
\displaystyle \frac{1}{2}\bar{\rho} \leq \rho \leq \frac{3}{2}\bar{\rho},\ \ \ \  (\rho-\bar{\rho}, \mathbf{u})\in C(0,\infty;H^N),\\[2mm]
  \displaystyle \nabla \rho\in L^2(0,\infty;H^{N-1}),\ \ \nabla_h \mathbf{u}, \mathop{\mathrm{div}}\nolimits \mathbf{u}\in L^2(0,\infty;H^N)
\end{gather}
and
\begin{align*}
\displaystyle \|(\rho- \bar{\rho},\mathbf{u})(\cdot,t)\|^2_{H^N}+ \int _{0}^{t}\left( \|\nabla \rho\|_{H ^{N-1}}^{2} + \|(\nabla _{h}\mathbf{u},  \mathop{\mathrm{div}}\nolimits \mathbf{u})\|_{H ^{N}}^{2}\right)  \mathrm{d}\tau \lesssim\|(\rho_0-\bar{\rho}, \mathbf{u}_0)\|^2_{H ^{N} },\ \ t \geq 0.
\end{align*}
\end{theorem}

\vspace{2mm}
The next theorem gives the nonlinear asymptotic stability of the constant state and quantifies the anisotropic decay induced by eddy diffusion.
\begin{theorem}\label{thm2}
 Assume that the initial data $ (\rho _{0}, \mathbf{u}_{0}) $ satisfy
\begin{equation*}
\displaystyle(\rho_0- \bar{\rho}, \mathbf{u}_0)\in H^{4}\cap L^1,\ (\mathbf{u}_0,\partial_3 \mathbf{u}_0,\partial^2_3 \mathbf{u}_0)\in L^2_{x_3}L^1_{x_1x_2}.
\end{equation*}
Then there exists a sufficiently small constant $\delta_1>0$ such that, if
\begin{align}
&\|(\rho_0- \bar{\rho},\mathbf{u}_0)\|_{H^4}+\|(\rho_0- \bar{\rho},\mathbf{u}_0)\|_{L^1}+\|(\mathbf{u}_0,\partial_3 \mathbf{u}_0,\partial^2_3 \mathbf{u}_0)\|_{L^2_{x_3}L^1_{x_1x_2}}\leq \delta_1,
\end{align}
then the problem \eqref{1-1}--\eqref{far-field} admits a unique global classical solution $(\rho,\mathbf{u})$ which obeys the following time decay estimates
\begin{gather*}
\displaystyle\|\rho-\bar{\rho}\|_{L^2}\lesssim(1+t)^{-\frac{3}{4}}\log(1+t), \ \ \|\mathbf{u}\|_{L^2}\lesssim(1+t)^{-\frac{1}{2}},\\[2mm]
\displaystyle\|(\nabla\rho,\nabla_h\mathbf{u},\mathrm{div} \mathbf{u})\|_{L^2}\lesssim(1+t)^{-1},\ \ \|\partial_3\mathbf{u}\|_{L^2}\lesssim(1+t)^{-\frac{1}{2}},\\[2mm]
\displaystyle\|(\nabla^2\rho,\nabla \mathrm{div} \mathbf{u})\|_{L^2}\lesssim(1+t)^{-1},\ \ \|\partial^2_3 \mathbf{u}\|_{L^2}\lesssim(1+t)^{-\frac{1}{2}}\log(1+t),\\[2mm]
\displaystyle\|\partial_i\partial_j \mathbf{u}\|_{L^2}\lesssim(1+t)^{-\frac{16}{15}},\ i=1,2,\ j=1,2, \ \ \ \|\partial_i\partial_j \mathbf{u}\|_{L^2}\lesssim(1+t)^{-1},\ i=1,2,\ j=3.
\end{gather*}
\end{theorem}

\vspace{3mm}
A few remarks help to situate these results. Theorem \ref{thm1} shows that,
even in the absence of a vertical Laplacian for the full velocity, the density
gradient, the horizontal velocity gradient and the velocity divergence
together form a closed dissipative structure. Theorem \ref{thm2} then exposes a
genuinely anisotropic large-time profile: the density and the compressible mode
retain an almost three-dimensional acoustic--diffusive decay, whereas the
solenoidal part of the velocity is governed by the horizontal heat semigroup
and decays only at the two-dimensional rate. Strikingly, higher vertical
derivatives of the velocity bring no improvement in the decay rate---a feature
that has no counterpart in the fully dissipative compressible Navier--Stokes
equations.

\medskip
\noindent\textbf{Strategy of the proof.}
Setting $\varrho=\rho-\bar\rho$ and dividing the momentum equation by $\rho$, we
recast the perturbation system as
\begin{gather}\label{perturbed-intro}
\begin{cases}
\displaystyle \varrho_t+\bar\rho\,\mathrm{div}\,\mathbf{u}=S_1,\\[1mm]
\displaystyle \mathbf{u}_t+\gamma\bar\rho\nabla\varrho-\bar\mu\Delta_h\mathbf{u}-(\bar\mu+\bar\zeta)\nabla\mathrm{div}\,\mathbf{u}=S_2,
\end{cases}
\end{gather}
where
\[
\bar\mu=\frac{\mu}{\bar\rho},\qquad
\bar\zeta=\frac{\zeta}{\bar\rho},\qquad
\gamma=\frac{P'(\bar\rho)}{\bar\rho^2},
\]
and $S_1, S_2$ represent the nonlinear terms. The first key observation is the hidden dissipative coupling
\begin{align}\label{coupling-divu}
\begin{cases}
\displaystyle (\nabla\varrho)_t+\bar\rho\nabla\mathrm{div}\,\mathbf{u}=\nabla S_1,\\[1mm]
\displaystyle (\mathrm{div}\,\mathbf{u})_t+\gamma\bar\rho\Delta\varrho
=\bar\mu\Delta_h\mathrm{div}\,\mathbf{u}+(\bar\mu+\bar\zeta)\Delta\mathrm{div}\,\mathbf{u}+\mathrm{div}\,S_2.
\end{cases}
\end{align}
This structure compensates for the lack of diffusion in the density equation and yields dissipation of $\nabla\varrho$ together with $\mathrm{div}\,\mathbf{u}$. It is the compressible analogue of an enhanced-dissipation mechanism: the missing smoothing is not inserted artificially, but is generated by the intrinsic pressure--divergence interaction.

For the global existence theorem, the main task is to prove a uniform-in-time energy inequality of the form
\[
\frac{d}{dt}\mathscr{E}(t)+\mathscr{D}(t)\leq0,
\]
where $\mathscr{E}(t)$ is equivalent to $\|U(t)\|_{H^N}^2$ and
\[
\mathscr{D}(t)\sim
\|\nabla\varrho(t)\|_{H^{N-1}}^2
+\|(\nabla_h\mathbf{u},\mathrm{div}\,\mathbf{u})(t)\|_{H^N}^2 .
\]
Compared with the standard compressible Navier--Stokes equations, the dissipation lacks $\|\partial_3\mathbf{u}\|_{H^N}^2$. The nonlinear convection terms therefore require anisotropic Sobolev inequalities and repeated use of the divergence dissipation. This is where the compressible structure is essential: estimates for $\mathrm{div}\,\mathbf{u}$ allow one to treat $\partial_3u_3$ as being controlled, in a suitable sense, by horizontal derivatives and lower-order quantities.

The decay analysis is more subtle. The Fourier representation of the linearized system shows that the Green matrix is not governed by the standard bound \eqref{usual-green}. Instead, the compressible variables enjoy acoustic--diffusive damping, while the solenoidal part of the velocity contains a purely horizontal heat factor $e^{-\bar\mu|\xi_h|^2t}$. To capture this behavior, we decompose the frequency space into anisotropic regions and establish pointwise estimates for the Green matrix adapted to each region. Schematically, the velocity satisfies bounds of the form
\begin{align*}
|\widehat{\mathbf{u}}(t,\xi)|\lesssim
\begin{cases}
 e^{-\vartheta|\xi|^2t}(|\widehat{\varrho}_0|+|\widehat{\mathbf{u}}_0|)
 +e^{-\bar\mu|\xi_h|^2t}|\widehat{\mathbf{u}}_0|, & \xi\in\mathcal{A}_1,\\[1mm]
 (e^{-\vartheta|\xi|^2t}+e^{-\eta t})(|\widehat{\varrho}_0|+|\widehat{\mathbf{u}}_0|)
 +e^{-\bar\mu|\xi_h|^2t}|\widehat{\mathbf{u}}_0|, & \xi\in\mathcal{A}_2,
\end{cases}
\end{align*}
where $\mathcal{A}_1$ and $\mathcal{A}_2$ are the frequency domains introduced in Proposition \ref{prop-Green}. These estimates explain the different decay rates for the density, the horizontal derivatives of the velocity, and the vertical derivatives of the velocity.

Finally, the nonlinear Duhamel terms are closed by a time-weighted energy method.  The weights are chosen according to the anisotropic linear decay, but they must be slightly modified to accommodate the convection term $\mathbf{u}\cdot\nabla\mathbf{u}$.  A representative part of the weighted functional is
\begin{align}\label{E-2-intr}
\mathcal{E}_2(t)=\sup_{0\leq\tau\leq t}\Big(&
(1+\tau)^{\frac32}(\log(1+t))^{-2}\|\varrho(\tau)\|_{L^2}^2
+(1+\tau)\|\mathbf{u}(\tau)\|_{L^2}^2
+(1+\tau)^{2}\|\nabla\varrho(\tau)\|_{L^2}^2
\nonumber\\
&+(1+\tau)^2\|\nabla_h\mathbf{u}(\tau)\|_{L^2}^2
+(1+\tau)\|\partial_3\mathbf{u}(\tau)\|_{L^2}^2+(1+\tau)^2\|\nabla^2\varrho(\tau)\|_{L^2}^2
\nonumber\\
&+\sum_{\substack{i=1,2\\ j=1,2}}(1+\tau)^{\frac{32}{15}}\|\partial_i\partial_j\mathbf{u}(\tau)\|_{L^2}^2
+(1+\tau)(\log(1+t))^{-2}\|\partial_3^2\mathbf{u}(\tau)\|_{L^2}^2\nonumber\\
&+\sum_{\substack{i=1,2\\ j=3}}(1+\tau)^{2}\|\partial_i\partial_j\mathbf{u}(\tau)\|_{L^2}^2\Big).
\end{align}
The estimates for $\mathrm{div}\,\mathbf{u}$ play a decisive role.  In particular, bounds such as
\begin{gather}\label{E-1-INTRO-1}
\sup_{0\leq\tau\leq t}
\Big((1+\tau)^{2}(\|\mathrm{div}\,\mathbf{u}(\tau)\|_{L^2}^2
+\|\nabla\mathrm{div}\,\mathbf{u}(\tau)\|_{L^2}^2)\Big)
\lesssim 1
\end{gather}
allow the nonlinear estimates to exploit the relation between $\partial_3u_3$ and the horizontal derivatives of $\mathbf{u}$. Combining the Green-function estimates, the weighted energy bounds and the dissipation generated by \eqref{coupling-divu}, we obtain an a priori inequality of the form
\[
\mathcal{E}(t)\lesssim F(U_0)+\mathcal{E}^{3/2}(t)+\mathcal{E}^2(t),
\]
where
\[
F(U_0)=\|U_0\|_{H^4}^2+\|U_0\|_{L^1}^2
+\|(\mathbf{u}_0,\partial_3\mathbf{u}_0,\partial_3^2\mathbf{u}_0)\|_{L^2_{x_3}L^1_{x_1x_2}}^2 .
\]
The desired decay estimates then follow from a standard bootstrap argument.
\vspace{2mm}


\medskip
\noindent\textbf{Organization.}
The rest of the paper is organized as follows. Section \ref{SEC-REFOR} collects
the preliminary material, in particular the anisotropic interpolation
inequalities used throughout. Section \ref{FIR-PROOF} is devoted to the global
existence stated in Theorem \ref{thm1}. Finally, Sections \ref{SEC-PROOF} and
\ref{sub:nonlinear_stability} carry out the decay analysis and complete the
proof of Theorem \ref{thm2}.
\medskip


\vspace{5mm}

\section{Preliminaries}\label{SEC-REFOR}

 In this section, we shall provide preparatory materials for the proofs of Theorem \ref{thm1} and Theorem \ref{thm2}.
 We provide several anisotropic upper bounds for the integral of products and triple products. It is a very powerful tool for studying problems with anisotropic dissipation. The proof of this lemma can be found in \cite{WZ-2021,LWZ-2022,LWZ-2023}.
\begin{lemma}\label{lemma2.1}
For $ i,j,k=1,2,3 $ and $ i\neq j\neq k $, the following estimates hold when the right-hand sides are all bounded.
\begin{subequations}
\begin{align}
\int_{\mathbb{R}^3}\vert fgh\vert \mathrm{d} x&\lesssim\|f\|^{\frac{1}{2}}_{L^2}\|\partial_1f\|^{\frac{1}{2}}_{L^2}\|g\|^{\frac{1}{2}}_{L^2}\|\partial_2g\|^{\frac{1}{2}}_{L^2}\|h\|^{\frac{1}{2}}_{L^2}\|\partial_3h\|^{\frac{1}{2}}_{L^2},\label{three-one-deri-ani}\\[2mm]
\int_{\mathbb{R}^3}\vert fgh\vert \mathrm{d} x&\lesssim\|f\|^{\frac{1}{4}}_{L^2}\|\partial_1f\|^{\frac{1}{4}}_{L^2}\|\partial_2f\|^{\frac{1}{4}}_{L^2}\|\partial_1\partial_2f\|^{\frac{1}{4}}_{L^2}\|g\|^{\frac{1}{2}}_{L^2}\|\partial_3g\|^{\frac{1}{2}}_{L^2}\|h\|_{L^2},\label{three-two-deri-ani}\\[2mm]
\|fg\|_{L^2}&\lesssim \|f\|^{\frac{1}{4}}_{L^2}\|\partial_if\|_{L^2}^{\frac{1}{4}}\|\partial_jf\|_{L^2}^{\frac{1}{4}}\|\partial_i\partial_jf\|_{L^2}^{\frac{1}{4}}\|g\|_{L^2}^{\frac{1}{2}}\|\partial_kg\|_{L^2}^{\frac{1}{2}},\label{two-two-deri-ani}\\[2mm]
\|fg\|_{L^2_{x_3}L^1_{x_1x_2}}&\lesssim \|f\|^{\frac{1}{2}}_{L^2}\|\partial_3f\|^{\frac{1}{2}}_{L^2}\|g\|_{L^2}.\label{two}
\end{align}
\end{subequations}

\end{lemma}
In what follows, we shall introduce some Sobolev inequalities for later use (cf. \cite{DRZ-2012,T-1996}). Let us begin with the following interpolation inequality.

 \begin{lemma}
 Let $ m $ be any positive integer. Then for any function $ u \in L ^{q}\cap H ^{l} $ with $ 1 \leq q \leq \infty $, it holds for any integer $ k \in [0,l) $ that
 \begin{align}\label{soboleve-ineq}
 \displaystyle  \|\nabla ^{k}u\|_{L ^{p}} \lesssim \|u\|_{L ^{q}}^{1- \theta}\|\nabla ^{l}u\|_{L ^{2}}^{\theta} ,
 \end{align}
 where
 \begin{align*}
 \displaystyle \frac{1}{p}= \frac{k}{3}+(1- \theta)\frac{1}{q} +\theta (\frac{1}{2}- \frac{l}{3})
 \end{align*}
 for all $ \theta \in [\frac{k}{l},1] $.

 \end{lemma}

 The second inequality is the $ L ^{p} $ estimate on any two product terms with the sum of the order of their derivatives
equal to a given integer.
 \begin{lemma}\label{lemma3.2}
 Let $n\geq 1.$ Let $\alpha^1=(\alpha^1_{1},\cdots,\alpha^1_{n})$ and $\alpha^2=(\alpha^2_{1},\cdots,\alpha^2_{n})$
 be two multi-indices with $|\alpha^1|=k_1$, $|\alpha^2|=k_2$ and set $k=k_1+k_2.$ Let $1\leq p,q,r\leq \infty$ with $\frac{1}{p}=\frac{1}{q}+\frac{1}{r}.$ Then, for
 $u_j: R^n\rightarrow R \ (j=1,2),$ one has
 \begin{equation}
 \|\partial^{\alpha^1}u_1\partial^{\alpha^2}u_2\|_{L^p} \lesssim\left(\|u_1\|_{L^q}\|\nabla^ku_2\|_{L^r}+\|u_2\|_{L^q}\|\nabla^ku_1\|_{L^r}\right)
 \end{equation}
 for some constant $C>0$ independent of $u_1$ and $u_2$.
 \end{lemma}
 As a generalization of Lemma \ref{lemma3.2}, we have also
 \begin{lemma}\label{lemma3.3}
 Let $n\geq 1,l>2$ be integers. Let $\alpha^j=(\alpha^j_{1},\cdots,\alpha^j_{n})$, $1\leq j\leq l$ be multi-indices with $|\alpha^j|=k_j$, $1\leq j\leq l$ and $k=k_1+k_2+\cdots+k_l$. Let $1\leq p,q,r\le \infty$ with $\frac{1}{p}=\frac{1}{q}+\frac{1}{r}$. Then, for $u=(u_1,\cdots,u_l):\mathbb{R}^n\rightarrow \mathbb{R}^{l}$, one has
 \begin{equation}\label{basic3}
 \left\|\prod^{l}_{j=1}\partial^{\alpha^{j}}u_j\right\|_{L^p}\lesssim \|u\|^{l-2}_{L^{\infty}}\|u\|_{L^q}\|\nabla^ku\|_{L^r}.
 \end{equation}
 \end{lemma}
As a consequence of Lemma \ref{lemma3.3}, we have the following lemma.

\begin{lemma}\label{lem-compose}
Let $ n \geq 1 $ and $ 1 \leq p \leq \infty $. Let $ \alpha =(\alpha _{1},\cdots, \alpha _{n}) $ be a multi-index with $ \vert \alpha\vert=k $. Assume that $ F(u) $ is a smooth function of $ \mathbf{u} $. Then, there is a constant $ C(\|u\| _{L ^{\infty}(\mathbb{R}^{n})}) $ depending only $ \|u\|_{L ^{\infty}(\mathbb{R}^{n})} $ with $ C(\varepsilon) \rightarrow 0 $ when $ \varepsilon \rightarrow 0 $ such that
\begin{align}
\displaystyle  \|\partial ^{\alpha}(F(u)u _{x _{i}})\| _{L ^{p}(\mathbb{R}^{n})} \leq C(\|u\| _{L ^{\infty}(\mathbb{R}^{n})})\|\nabla ^{k+1}u\|_{L ^{p}(\mathbb{R}^{n})} \label{nonlinear-con}
\end{align}
for all $ 1 \leq i \leq n $.
\end{lemma}

 Finally, the following lemma is frequently accessed in the proof of Theorem \ref{thm2}. Its proof is straightforward.
  \begin{lemma}[cf. \cite{Ruan-Zhu-2010-jde}]\label{decaylemma}
 For $ a,\, b>0 $, it holds that
 \begin{align*}
 \displaystyle  \int _{0}^{\frac{t}{2}}(1+t- \tau)^{- a}(1+ \tau)^{-b}\mathrm{d}\tau \lesssim \begin{cases}
   \displaystyle  (1+t)^{-a}& \mbox{for }b>1,\\[2mm]
  \displaystyle   (1+t)^{-a}\log(1+t) &  \mbox{for }b=1,\\[2mm]
 \displaystyle    (1+t)^{-(a+b-1)}&  \mbox{for }b<1,
 \end{cases}
 \end{align*}
 and
 \begin{align*}
 \displaystyle  \int _{\frac{t}{2}}^{t}(1+t- \tau)^{- a}(1+ \tau)^{-b}\mathrm{d}\tau \lesssim \begin{cases}
     (1+t)^{-b}& \mbox{for }a>1,\\[2mm]
     (1+t)^{-b}\log(1+t)  & \mbox{for }a=1,\\[2mm]
     (1+t)^{-(a+b-1)} & \mbox{for }a<1.
 \end{cases}
 \end{align*}
 \end{lemma}

\vspace{4mm}



  \section{Global existence of solutions to the nonlinear system} 
  \label{FIR-PROOF}

In this section, we shall devote ourselves to the proof of global existence of classical solutions to the problem \eqref{1-1}--\eqref{far-field}. To this attempt, we shall consider the linearized system around the constant equilibrium $ (\bar{\rho}, \mathbf{0}) $. Set $ \varrho= \rho- \bar{\rho} $. Then $ (\varrho,\mathbf{u}) $ satisfy
\begin{gather}\label{new}
 \begin{cases}
 \displaystyle\varrho_t+\bar{\rho} \mathop{\mathrm{div}}\nolimits \mathbf{u}=S_1,\\[2mm]
 \displaystyle\mathbf{ u}_t+\gamma\bar{\rho}\nabla \varrho - \bar{\mu}\Delta_h \mathbf{u}-(\bar{\mu}+ \bar{\zeta})\nabla \mathop{\mathrm{div}}\nolimits \mathbf{u}=S_2.
 \end{cases}
 \end{gather}
 The initial data are given by
 \begin{gather}\label{refor-Initial}
 \displaystyle (\varrho,\mathbf{u})(x, 0)=(\rho_0- \bar{\rho} , \mathbf{u}_0):=(\varrho _{0}, \mathbf{u} _{0}).
 \end{gather}
 Here, the nonhomogeneous source terms are
 \begin{gather}\label{nonlinear}
 \begin{cases}
 \displaystyle
 S_1:=-\varrho \mathop{\mathrm{div}}\nolimits\mathbf{u}-\mathbf{u}\cdot\nabla\varrho,\\[2mm]
 \displaystyle S_2:=-\mathbf{u}\cdot\nabla \mathbf{u}-g(\varrho)[\bar{\mu}\Delta_h \mathbf{u}+\bar{\zeta}\nabla \mathop{\mathrm{div}}\nolimits\mathbf{ u}]-f(\varrho)\nabla\varrho,
 \end{cases}
 \end{gather}
 where $\bar{\mu}=\frac{\mu}{\bar{\rho}}$, $\bar{\zeta}=\frac{\zeta}{\bar{\rho}}$, $\gamma=\frac{p'(\bar{\rho})}{\bar{\rho}^2}$, and the function $g$ and $f$ are defined as
 \begin{gather}\label{gf}
 g(\varrho):=\frac{\varrho}{\varrho+\bar{\rho}},\ \ \ f(\varrho):=\frac{p'(\varrho+\bar{\rho})}{\varrho+\bar{\rho}}-\frac{p'(\bar{\rho})}{\bar{\rho}}.
 \end{gather}

  \noindent{\bf Proof of Theorem \ref{thm1}.}
First of all, by adapting standard reasonings, such as regularization, Galerkin approximation and contraction mapping principle (see e.g.  \cite{K-1983,K-1987,SK-1985}), and the energy estimates derived in this paper, one can set up the local existence/uniqueness of solutions to the problem \eqref{new}, \eqref{refor-Initial}. We present the result without going through the technical details for brevity.
  \begin{lemma}\label{local}
Let $N\geq 3$ be an given integer and assume that $ (\varrho _{0},\mathbf{u}_{0}) \in H ^{N} $ such that
\begin{align*}
\displaystyle \inf _{x \in \mathbb{R}^{3}}\left\{ \varrho _{0}+ \bar{\rho} \right\} >0.
\end{align*}
Then there exists a constant $ T _{0}>0 $ such that the Cauchy problem \eqref{new}, \eqref{refor-Initial} admits a unique classical solution $ (\varrho,\mathbf{u}) $ on $ [0,T _{0}] $ which satisfies
\begin{gather*}
\displaystyle  (\varrho,\mathbf{u})\in C(0,T_0;H^N), \ \ \nabla\varrho\in L^2(0,T_0;H^{N-1}),\ \ \nabla_h \mathbf{u}\in L^2(0,T_0;H^N),
\end{gather*}
and
\begin{align*}
\displaystyle  \inf _{x \in \mathbb{R}^{3},\;
0 \leq t \leq T _{0}}\left\{ \varrho + \bar{\rho} \right\} >0.
\end{align*}

  \end{lemma}
We now proceed to state the results on the global existence and asymptotic behavior of solutions to the reformulated problem \eqref{new}, \eqref{refor-Initial}.
 \begin{proposition}\label{prop3.1}
Let $N\geq 3$ be an given integer. If $ \|(\varrho _{0},\mathbf{u}_{0})\|_{H ^{N}} $ is small enough, then the Cauchy problem \eqref{new}, \eqref{refor-Initial} admits a unique global solution $(\varrho,\mathbf{u})$ such that
 \begin{gather*}
 \displaystyle  (\varrho, \mathbf{u})\in C\left(0,\infty;H^{N}\right),\ \ \nabla \varrho\in L^2\left(0,\infty;H^{N-1}\right),\ \ (\nabla_h \mathbf{u},\mathop{\mathrm{div}}\nolimits \mathbf{u})\in L^2\left(0,\infty;H^N\right).
 \end{gather*}
 \end{proposition}
 With Proposition \ref{prop3.1}, one can immediately finish the proof of Theorem \ref{thm1}. \hfill $ \square $\\[2mm]

 \noindent{\bf Proof of Proposition \ref{prop3.1}.}
 With the local existence results in Lemma \ref{local}, in view of the standard continuation argument, it suffices to establish the desired \emph{a priori} estimates of the solution. In what follows, we assume that $ U:=(\varrho,u)(t,x) $ is a smooth solution to the problem \eqref{new}, \eqref{refor-Initial} on $ [0,T] $ for any $ T >0 $, and we will devote ourselves to deriving some \emph{a priori} estimates for $ U $. To this attempt, we make the following \emph{a priori} assumptions
   \begin{gather}\label{apriori-assumption}
   \displaystyle \mathcal{E}_{0}(T):=\sup_{0 \leq t \leq T}\|(\varrho,\mathbf{u})\|_{H ^{N}}^{2}+\int^T_0 \left( \|(\nabla_h\mathbf{u},\mathop{\mathrm{div}}\nolimits \mathbf{u}) \|^2_{H ^{N}}+\| \nabla \varrho\|^2_{H ^{N-1}} \right)  \mathrm{d}\tau\leq \delta ^{2},\ \ N \geq 3,
  \end{gather}
  where $\delta<1$ is a sufficiently small positive constant which will be determined later. Clearly, we have
  \begin{align*}
  \displaystyle \mathcal{E} _{0}(0)=\|(\varrho _{0},\mathbf{u} _{0})\| _{H ^{N}}^{2}.
  \end{align*}
  Furthermore, by \eqref{apriori-assumption} and the Sobolev inequality, we have
  \begin{gather}\label{apri-conlu}
  \displaystyle 0<\frac{\bar{\rho}}{2} \leq \varrho+ \bar{\rho} \leq \frac{3\bar{\rho}}{2},\ \ \  \|\nabla(\varrho,\mathbf{u})\|_{L ^{\infty}}\lesssim \delta.
  \end{gather}

To prove Proposition \ref{prop3.1}, the key of matter is to establish following estimates for $ (\varrho,\mathbf{u}) $.
\begin{lemma}\label{priori}
If $ \delta $ is sufficiently small, there exists a constant $ C>0 $ which is independent of $ T $ such that
 \begin{equation}\label{ine1}
 \mathcal{E}_0(T)\leq C\mathcal{E}_0(0).
 \end{equation}
 \end{lemma}
 \begin{proof}[Proof of Lemma \ref{priori}]
We split the proof into three steps.\\
\noindent\textbf{Step 1.}
 For suitably small $ \delta>0 $, it holds that
 \begin{align}\label{a}
 &\displaystyle\frac{1}{2}\frac{\mathrm{d}}{\mathrm{d}t}\Big(\gamma\|\varrho\| _{H ^{N} }^2+\|\mathbf{u}\|_{H ^{N}}^2 \Big)
 +\bar{\mu}\|\nabla_h \mathbf{u}\|_{H ^{N}}^2+(\bar{\mu}+\bar{\zeta})\|\mathop{\mathrm{div}}\nolimits \mathbf{u}\|_{H ^{N}}^2
  \nonumber \\[1mm]
 &~\displaystyle \lesssim\delta \Big(\|\nabla\varrho\|^2_{H ^{N-1}}+\|\nabla_h\mathbf{u}\|^2_{H ^{N}}+\|\mathop{\mathrm{div}}\nolimits \mathbf{u}\|^2_{H ^{N}}\Big).
 \end{align}
 To show \eqref{a}, let us begin with the zero-order estimates on $ (\varrho,\mathbf{u})$. Multiplying $ \eqref{new}_1 $ and $ \eqref{new}_2 $, respectively,  by $ \gamma\varrho $ and $\mathbf{u} $, and integrating the resulting equations over $\mathbb{R}^3$, we get
 \begin{align}\label{b}
 &\displaystyle\frac{1}{2}\frac{\mathrm{d}}{\mathrm{d}t}\Big(\gamma\|\varrho\| _{L ^{2}}^2+\|\mathbf{u}\|_{L ^{2}}^2\Big)+\bar{\mu}\|\nabla_h \mathbf{u}\|_{L ^{2}}^2+(\bar{\mu}+\bar{\zeta})\|\mathop{\mathrm{div}}\nolimits \mathbf{u}\|_{L ^{2}}^2
  \nonumber \\[2mm]
 &~\displaystyle=\gamma\langle S_1,\rho\rangle+\langle S_2, \mathbf{u}\rangle
  \nonumber \\[2mm]
 &~\displaystyle \lesssim \|\varrho\|_{L^3}\|\mathop{\mathrm{div}}\nolimits \mathbf{u}\|_{L ^{2}} \|\varrho\|_{L^6}+\int _{\mathbb{R}^{3}}\vert \mathop{\mathrm{div}}\nolimits \mathbf{u}\vert \vert \mathbf{u}\vert^2\mathrm{d}x+\int _{\mathbb{R}^{3}} (\vert g(\varrho) \vert \vert \nabla_h\mathbf{u}\vert^2+\vert \nabla_hg(\varrho)\vert \vert\nabla_h\mathbf{u} \vert \vert \mathbf{u}\vert)\mathrm{d}x\nonumber\\[2mm]
 &~\displaystyle \quad+\int _{\mathbb{R}^{3}}\vert g(\varrho)\vert \vert \mathop{\mathrm{div}}\nolimits \mathbf{u}\vert ^{2}\mathrm{d}x+\int _{\mathbb{R}^{3}}\vert \nabla g(\varrho)\vert \vert \mathop{\mathrm{div}}\nolimits \mathbf{u}\vert \vert \mathbf{u}\vert\mathrm{d}x+\|f(\varrho)\|_{L^6}\|\nabla\varrho\|_{L^2}\|\mathbf{u}\|_{L^3}
  \nonumber \\[2mm]
 &~\displaystyle \lesssim \|\varrho\|_{L^3}\Big(\|\mathop{\mathrm{div}}\nolimits \mathbf{u}\|_{L ^{2}}^2+\|\nabla\varrho\|_{L ^{2}}^2\Big)+\|\mathop{\mathrm{div}}\nolimits \mathbf{u}\|_{L ^{2}}^{\frac{1}{2}}\|\partial_3\mathop{\mathrm{div}}\nolimits \mathbf{u}\|_{L ^{2}}^{\frac{1}{2}}\|\mathbf{u}\|_{L ^{2}}^{\frac{1}{2}}\|\partial_1\mathbf{u}\|_{L ^{2}}^{\frac{1}{2}}\|\mathbf{u}\|_{L ^{2}}^{\frac{1}{2}}\|\partial_2\mathbf{u}\|_{L ^{2}}^{\frac{1}{2}}\nonumber\\
 &~\quad \displaystyle+\|\varrho\|_{L^{\infty}}\|\nabla_h \mathbf{u} \|^2_{L^2}+\|\mathbf{u}\|_{L^{\infty}}\Big(\|\nabla\varrho\|_{L ^{2}}^2+\|\nabla_h\mathbf{u}\|_{L ^{2}}^2\Big)+\|\varrho\|_{L^{\infty}}\|\mathop{\mathrm{div}}\nolimits \mathbf{u}\|_{L ^{2}}^2\nonumber\\[2mm]
 &~ \displaystyle \quad+\|\mathbf{u}\|_{L^{\infty}}\|\nabla \varrho\|_{L ^{2}}\|\mathop{\mathrm{div}}\nolimits \mathbf{u}\|_{L ^{2}}+\|\mathbf{u}\|_{H^1}\|\nabla\varrho\|^2_{L^2}\nonumber\\[2mm]
 &~\displaystyle \lesssim  \Big(\|\varrho\|_{H^2}+\| \mathbf{u}\|_{H^2}\Big)\Big(\|\nabla_h\mathbf{u}\|_{L ^{2}}^2+\|\nabla\varrho\|_{L ^{2}}^2+\|\mathop{\mathrm{div}}\nolimits \mathbf{u}\|^2_{H ^{1}}\Big),
  \end{align}
  where we have used \eqref{three-one-deri-ani}, \eqref{apriori-assumption} and \eqref{apri-conlu}. For $ 1 \leq k \leq N $, applying $\partial^{\alpha}$ to \eqref{new} with $ \vert \alpha\vert=k $, it follows that
 \begin{gather}\label{high}
 \begin{cases}
 \displaystyle\partial^{\alpha}\varrho_t+\partial^{\alpha} \mathop{\mathrm{div}}\nolimits \mathbf{u}=\partial^{\alpha}S_1,\\[2mm]
 \displaystyle\partial^{\alpha}\mathbf{u}_t+\gamma\partial^{\alpha}\nabla \varrho-\bar{\mu}\Delta_h \partial^{\alpha}u-(\bar{\mu}+\bar{\lambda})\nabla \mathop{\mathrm{div}}\nolimits \partial^{\alpha}\mathbf{u}=\partial^{\alpha}S_2.
 \end{cases}
 \end{gather}
 Testing $ \eqref{high}_1 $ and $ \eqref{high}_2 $ against $\gamma\partial^{\alpha}_{x}\varrho$ and $\partial^{\alpha}_{x}\mathbf{u}$, respectively, we get
 \begin{align}\label{a1}
 &\displaystyle\frac{1}{2}\frac{\mathrm{d}}{\mathrm{d}t}\Big(\gamma\|\partial^{\alpha} \varrho\|_{L ^{2}}^2
 +\|\partial^{\alpha}\mathbf{u}\|^2\Big)+\bar{\mu}\|\partial^{\alpha}\nabla_h \mathbf{u}\|_{L ^{2}}^2+(\bar{\mu}+\bar{\zeta})\|\partial^{\alpha}\mathop{\mathrm{div}}\nolimits \mathbf{u}\|_{L ^{2}}^2
  \nonumber\\[2mm]
 &\displaystyle~=-\gamma\int_{\mathbb{R}^3}\partial^{\alpha}(\varrho\mathop{\mathrm{div}}\nolimits\mathbf{u}) \partial ^{\alpha}\varrho \mathrm{d}x - \gamma\int _{\mathbb{R}^{3}}\partial ^{\alpha}(\mathbf{u}\cdot \nabla \varrho )\partial ^{\alpha}\varrho \mathrm{d}x - \int _{\mathbb{R}^{3}} \partial ^{\alpha}(\mathbf{u} \cdot \nabla \mathbf{u}) \cdot \partial  ^{\alpha}\mathbf{u}\mathrm{d}x
  \nonumber \\[2mm]
  &~\displaystyle \quad - \bar{\mu} \int _{\mathbb{R}^{3}}\partial ^{\alpha}(g(\varrho)\Delta _{h}\mathbf{u})\cdot \partial  ^{\alpha}\mathbf{u}\mathrm{d}x-(\bar{\mu}+\bar{\lambda})\int _{\mathbb{R}^{3}}\partial ^{\alpha} (g(\varrho)\nabla \mathop{\mathrm{div}}\nolimits \mathbf{u})\cdot \partial ^{\alpha}\mathbf{u}\mathrm{d}x
   \nonumber \\[2mm]
   &~\displaystyle \quad - \int _{\mathbb{R}^{3}} \partial ^{\alpha}(f(\varrho)\nabla \varrho)\cdot \partial ^{\alpha}\mathbf{u}\mathrm{d}x
   \nonumber \\[2mm]
    &~\displaystyle=\sum^{6}_{i=1}I _{i}.
     \end{align}
 We next estimate $ I _{i}~(1 \leq i \leq 6) $ term by term. First, for $I_1$, notice that
 \begin{align*}
 I_{1}\leq\sum _{\ell=0}^{k}C _{\ell}\int_{\mathbb{R}^3}\vert \nabla^{k-\ell}\varrho\vert \vert \nabla^{\ell}\mathop{\mathrm{div}}\nolimits\mathbf{u}\vert \vert \nabla^{k}\varrho \vert\mathrm{d}x
 \end{align*}
 for some generic constants $ C _{\ell}>0 $. Therefore we derive for $ k=1 $ that
 \begin{align*}
  \displaystyle I _{1} &\lesssim  \int _{\mathbb{R}^{3}} \vert \nabla \varrho\vert \vert \mathop{\mathrm{div}}\nolimits \mathbf{u}\vert \vert \nabla \varrho\vert \mathrm{d}x + \int _{\mathbb{R}^{3}}\vert \varrho\vert \vert \nabla \mathop{\mathrm{div}}\nolimits \mathbf{u}\vert \vert \nabla \varrho\vert \mathrm{d}x
   \nonumber \\[2mm]
   & \lesssim \|\mathop{\mathrm{div}}\nolimits \mathbf{u}\|_{L ^{\infty}}\|\nabla \rho\|_{L ^{2}}^{2}+ \|\varrho\|_{L ^{\infty}} \|\nabla \mathop{\mathrm{div}}\nolimits \mathbf{u}\|_{L ^{2}}\|\nabla \varrho\|_{L ^{2}}
    \nonumber \\[2mm]
    & \lesssim \Big(\|\mathop{\mathrm{div}}\nolimits\mathbf{u}\|^2_{H^N}+\|\nabla\varrho\|^2_{H^{N-1}}\Big)\|\varrho\|_{H^N}.
  \end{align*}
  While for $ k \geq2 $, we have
\begin{align*}
\displaystyle I _{1}& \lesssim  \int _{\mathbb{R}^{3}}\Big(\vert \nabla ^{k}\varrho\vert \vert \mathop{\mathrm{div}}\nolimits \mathbf{u}\vert \vert \nabla ^{k}\varrho\vert + \vert \nabla ^{k-1} \varrho\vert \vert \nabla \mathop{\mathrm{div}}\nolimits \mathbf{u} \vert \vert \nabla ^{k}\varrho\vert \Big)\mathrm{d}x
 \nonumber \\
 & \displaystyle \quad+ \sum _{\ell=2}^{k} \int_{\mathbb{R}^3}\vert \nabla^{k-\ell}\varrho\vert \vert \nabla^{\ell}\mathop{\mathrm{div}}\nolimits\mathbf{u}\vert \vert \nabla^{k}\varrho \vert\mathrm{d}x
 \nonumber \\[2mm]
 & \displaystyle \lesssim \|\mathop{\mathrm{div}}\nolimits \mathbf{u}\|_{L ^{\infty}} \|\nabla ^{k}\varrho\|_{L ^{2}}^{2}+\|\nabla\mathop{\mathrm{div}}\nolimits\mathbf{u}\|_{L^3}\|\nabla^{k-1}\varrho\|_{L^6}\|\nabla^k\varrho\| _{L ^{2}}
  \nonumber \\[2mm]
  & \displaystyle \quad+\sum _{\ell=2}^{k}\|\nabla ^{k- \ell}\varrho\|_{L ^{\infty}} \| \nabla ^{\ell}\mathop{\mathrm{div}}\nolimits \mathbf{u}\|_{L ^{2}} \|\nabla ^{k}\varrho\|_{L ^{2}}
  \nonumber \\[2mm]
  & \lesssim  \Big(\|\mathop{\mathrm{div}}\nolimits\mathbf{u}\|^2_{H^N}+\|\nabla\varrho\|^2_{H^{N-1}}\Big)\|\varrho\|_{H^N}.
\end{align*}
Therefore we get
\begin{gather*}
\displaystyle I _{1} \lesssim (\|\mathop{\mathrm{div}}\nolimits\mathbf{u}\|^2_{H^N}+\|\nabla\varrho\|^2_{H^{N-1}})\|\varrho\|_{H^N}.
\end{gather*}
Similarly, we have
 \begin{align*}
 \displaystyle I _{2}& \lesssim  \int _{\mathbb{R}^{3}} \vert \nabla \mathbf{u}\vert \vert \nabla \varrho\vert ^{2}\mathrm{d}x + \int _{\mathbb{R}^{3}}\vert \mathbf{u}\vert \vert \nabla ^{2}\varrho\vert \vert \nabla \varrho\vert \mathrm{d}x
  \nonumber \\[2mm]
  & \lesssim \|\nabla \mathbf{u}\|_{L ^{\infty}} \|\nabla \varrho\|_{L ^{2}}^{2}
  + \|\mathbf{u}\| _{L ^{\infty}} \|\nabla ^{2}\varrho\|_{L ^{2}}\|\nabla \rho\|_{L ^{2}}
   \nonumber \\[2mm]
   & \displaystyle \lesssim \|\nabla\varrho\|^2_{H^{N-1}}\|\mathbf{u}\|_{H^N}
 \end{align*}
 for $ k=1 $, and
 \begin{align*}
 \displaystyle I _{2}& \lesssim  \int _{\mathbb{R}^{3}}\vert \nabla ^{k}\mathbf{u}\vert \vert \nabla \varrho\vert \vert \nabla ^{k}\varrho\vert \mathrm{d}x + \int _{\mathbb{R}^{3}} \vert \nabla ^{k-1}\mathbf{u}\vert \vert \nabla ^{2}\varrho\vert \vert \nabla ^{k}\varrho\vert \mathrm{d}x
  \nonumber \\
  & \displaystyle \quad+\sum _{\ell=2}^{k-1} \int_{\mathbb{R}^3}\vert \nabla^{k-\ell}\mathbf{u}\vert \vert \nabla^{\ell+1}\varrho\vert \vert \nabla^{k}\varrho \vert\mathrm{d}x+ \int _{\mathbb{R}^{3}}\vert \mathop{\mathrm{div}}\nolimits \mathbf{u}\vert \vert \nabla ^{k}\varrho\vert \mathrm{d}x
  \nonumber \\[2mm]
  & \lesssim  \|\nabla\varrho\|_{L^\infty}\|\mathbf{u}\|_{H^N}\|\nabla\varrho\|_{H^{N-1}}+\|\nabla ^{k-1}\mathbf{u}\|_{L ^{6}}\|\nabla ^{2}\varrho\| _{L ^{3}}\|\nabla ^{k}\varrho\|_{L ^{2}}
   \nonumber \\[2mm]
   &~\displaystyle \quad+ \sum _{\ell=2}^{k-1}\|\nabla ^{k- \ell}\mathbf{u}\|_{L ^{\infty}} \|\nabla ^{\ell+1}\varrho\|_{L ^{2}} \|\nabla ^{k}\varrho\|_{L ^{2}}+\|\mathop{\mathrm{div}}\nolimits \mathbf{u}\|_{L ^{\infty}} \|\nabla ^{k}\varrho\|_{L ^{2}}^{2}
    \nonumber \\[2mm]
    &~\displaystyle \lesssim \|\nabla\varrho\|^2_{H^{N-1}}\|\mathbf{u}\|_{H^N}
 \end{align*}
 for $ k \geq 2 $. And we thus get
 \begin{gather*}
 \displaystyle  I _{2} \lesssim  \|\nabla\varrho\|^2_{H^{N-1}}\|\mathbf{u}\|_{H^N}.
 \end{gather*}
 Now let us turn to the estimate of $ I _{3} $. We first split it into two terms:
\begin{align}\label{I-3-split}
\displaystyle I _{3}&= \underbrace{\int _{\mathbb{R}^{3}}(\mathbf{u} \cdot \nabla \partial ^{\alpha}\mathbf{u})\cdot \partial ^{\alpha}\mathbf{u} \mathrm{d}x}_{I _{3,1}}  +  \underbrace{\int _{\mathbb{R}^{3}}[\partial ^{\alpha}, \mathbf{u} \cdot \nabla]\mathbf{u}\cdot \partial ^{\alpha}\mathbf{u}\mathrm{d}x}_{I _{3,2}},
\end{align}
where, thanks to \eqref{three-one-deri-ani} and integration by parts, $ I _{3,1} $ can be handled as
\begin{align}
 \displaystyle I _{3,1}&=- \frac{1}{2}\int _{\mathbb{R}^{3}}\mathop{\mathrm{div}}\nolimits \mathbf{u}\vert \partial ^{\alpha}\mathbf{u}\vert ^{2}\mathrm{d}x
  \nonumber \\[2mm]
  & \lesssim   \|\mathop{\mathrm{div}}\nolimits\mathbf{u}\|_{L ^{2}}^{\frac{1}{2}}\|\partial_3\mathop{\mathrm{div}}\nolimits\mathbf{u}\|_{L ^{2}}^{\frac{1}{2}}\|\nabla^k\mathbf{u}\|_{L ^{2}}^{\frac{1}{2}}\|\nabla^k\partial_1\mathbf{u}\|_{L ^{2}}^{\frac{1}{2}}\|\nabla^k\mathbf{u}\|_{L ^{2}}^{\frac{1}{2}}\|\nabla^k\partial_2\mathbf{u}\|_{L ^{2}}^{\frac{1}{2}}
  \nonumber \\[2mm]
   &\displaystyle \lesssim \Big(\|\nabla_h\mathbf{u}\|^2_{H^N}+\|\mathop{\mathrm{div}}\nolimits\mathbf{u}\|^2_{H^N}\Big)\|\mathbf{u}\|_{H^N}.
 \end{align}
 For $ I _{3,2} $, by direct computation, we get
 \begin{align*}
  \displaystyle I _{3,2}&=\underbrace{\int _{\mathbb{R}^{3}}\Big( [\partial ^{\alpha}, \mathbf{u} _{h} \cdot \partial _{h}]\mathbf{u}\cdot \partial ^{\alpha}\mathbf{u}+[\partial ^{\alpha},u _{3} \cdot \partial _{3}]u _{3}\cdot \partial ^{\alpha}u _{3} \Big) \mathrm{d}x}_{I _{3,21}}+\underbrace{\int _{\mathbb{R}^{3}}[\partial ^{\alpha},u _{3} \cdot \partial _{3}]\mathbf{u} _{h}\cdot \partial ^{\alpha}\mathbf{u} _{h}\mathrm{d}x}_{I _{3,22}},
  \end{align*}
  with
  \begin{align}\label{I321}
   \displaystyle I _{3,21}& \lesssim \sum _{\ell=1}^{k} \int _{\mathbb{R}^{3}} \left[ \vert \nabla ^{\ell}\mathbf{u}_{h}\vert \vert  \nabla ^{k- \ell}\partial _{h}\mathbf{u}\vert \vert \nabla ^{k}\mathbf{u}\vert+ \vert \nabla ^{\ell}u _{3}\vert \left( \vert \nabla ^{k- \ell}\mathop{\mathrm{div}}\nolimits \mathbf{u}\vert +\vert \nabla ^{k- \ell}\mathop{\mathrm{div}}\nolimits _{h}\mathbf{u}_{h}\vert\right) \vert \nabla ^{k}u _{3}\vert\right]  \mathrm{d}x
    \nonumber \\[2mm]
    & \lesssim \sum _{\ell=1}^{k}\left( \|\nabla ^{\ell}\mathbf{u}\|_{L ^{2}}^{\frac{1}{2}} \|\nabla ^{\ell}\partial _{1} \mathbf{u}\|_{L ^{2}}^{\frac{1}{2}} \|\nabla ^{k- \ell}\partial _{h}\mathbf{u}\|_{L ^{2}}^{\frac{1}{2}} \|\nabla ^{k- \ell}\partial _{h}\partial _{3}\mathbf{u}\|_{L ^{2}}^{\frac{1}{2}}\|\nabla^k\mathbf{u}\|_{L ^{2}}^{\frac{1}{2}}\|\nabla^k\partial_2\mathbf{u}\|_{L ^{2}}^{\frac{1}{2}}\right.
     \nonumber \\[2mm]
     &~~\displaystyle \left. \quad \qquad+ \|\nabla ^{\ell}\mathbf{u}\|_{L ^{2}}^{\frac{1}{2}} \|\nabla ^{\ell}\partial _{1} \mathbf{u}\|_{L ^{2}}^{\frac{1}{2}} \|\nabla ^{k- \ell}\mathop{\mathrm{div}}\nolimits\mathbf{u}\|_{L ^{2}}^{\frac{1}{2}} \|\nabla ^{k- \ell}\partial _{3}\mathop{\mathrm{div}}\nolimits\mathbf{u}\|_{L ^{2}}^{\frac{1}{2}}\|\nabla^k\mathbf{u}\|_{L ^{2}}^{\frac{1}{2}}\|\nabla^k\partial_2\mathbf{u}\|_{L ^{2}}^{\frac{1}{2}}\right)
      \nonumber \\[2mm]
      & \displaystyle \lesssim   \Big(\|\nabla_h\mathbf{u}\|^2_{H^N}+\|\mathop{\mathrm{div}}\nolimits\mathbf{u}\|^2_{H^N}\Big)\|\mathbf{u}\|_{H^N},
   \end{align}
   where we have used \eqref{three-one-deri-ani}. To derive the estimate of $ I _{3} $, it now remains to control $ I _{3,22} $. Indeed, on the one hand, if $ \partial ^{\alpha} $ involves at least one horizontal derivatve says $ \partial _{h} $, then we get, thanks to \eqref{three-one-deri-ani},
   \begin{align*}
   \displaystyle I _{3,22}& \lesssim \sum _{\ell=1}^{k}\int _{\mathbb{R}^{3}} \vert \nabla ^{\ell}u _{3}\vert \vert \nabla ^{k+1- \ell}\mathbf{u} _{h}\vert\vert \nabla ^{k-1}\partial _{h}\mathbf{u} _{h}\vert \mathrm{d}x
    \nonumber \\[2mm]
    & \lesssim  \sum _{\ell=1}^{k} \|\nabla ^{\ell}\mathbf{u}\|_{L ^{2}}^{\frac{1}{2}} \|\nabla ^{\ell}\partial _{1} \mathbf{u}\|_{L ^{2}}^{\frac{1}{2}} \|\nabla ^{k+1-\ell}\mathbf{u}\|_{L ^{2}}^{\frac{1}{2}} \|\nabla ^{k+1-\ell}\partial _{2}\mathbf{u}\|_{L ^{2}}^{\frac{1}{2}}\|\nabla^{k-1}\partial _{h}\mathbf{u}\|_{L ^{2}}^{\frac{1}{2}}\|\nabla^{k-1}\partial _{3}\partial _{h}\mathbf{u}\|_{L ^{2}}^{\frac{1}{2}}
     \nonumber \\[2mm]
     & \displaystyle \lesssim \|\nabla_h\mathbf{u}\|^2_{H^N}\|\mathbf{u}\|_{H^N}.
   \end{align*}
On the other hand, if $ \partial ^{\alpha}=\partial _{3}^{k} $, then we can control $ I _{3, 22} $ as follows
   \begin{align}\label{I-3-22}
   &\displaystyle   I _{3,22}
    \nonumber \\
    &~\displaystyle \lesssim \sum _{\ell=1}^{k}\int _{\mathbb{R}^{3}} \vert \partial _{3}^{\ell-1}(\mathop{\mathrm{div}}\nolimits \mathbf{u}- \mathop{\mathrm{div}}\nolimits _{h}\mathbf{u} _{h})\vert \vert \partial _{3} ^{k+1- \ell}\mathbf{u} _{h}\vert\vert \partial _{3}^{k} \mathbf{u} _{h}\vert \mathrm{d}x
    \nonumber \\[2mm]
     & \lesssim \sum _{\ell=1}^{k}\left( \|\nabla ^{ \ell-1}\partial _{h}\mathbf{u}\|_{L ^{2}}^{\frac{1}{2}} \|\nabla ^{\ell-1}\partial _{h}\partial _{3}\mathbf{u}\|_{L ^{2}}^{\frac{1}{2}} \|\nabla ^{k+1-\ell}\mathbf{u}\|_{L ^{2}}^{\frac{1}{2}} \|\nabla ^{k+1-\ell}\partial _{1} \mathbf{u}\|_{L ^{2}}^{\frac{1}{2}}\|\nabla^k\mathbf{u}\|_{L ^{2}}^{\frac{1}{2}}\|\nabla^k\partial_2\mathbf{u}\|_{L ^{2}}^{\frac{1}{2}}\right.
     \nonumber \\[2mm]
     &~~\displaystyle \left. \quad \qquad+  \|\nabla ^{ \ell-1}\mathop{\mathrm{div}}\nolimits \mathbf{u}\|_{L ^{2}}^{\frac{1}{2}} \|\nabla ^{\ell}\mathop{\mathrm{div}}\nolimits \mathbf{u}\|_{L ^{2}}^{\frac{1}{2}} \|\nabla ^{k+1-\ell}\mathbf{u}\|_{L ^{2}}^{\frac{1}{2}} \|\nabla ^{k+1-\ell}\partial _{1} \mathbf{u}\|_{L ^{2}}^{\frac{1}{2}}\|\nabla^k\mathbf{u}\|_{L ^{2}}^{\frac{1}{2}}\|\nabla^k\partial_2\mathbf{u}\|_{L ^{2}}^{\frac{1}{2}}\right)
      \nonumber \\[2mm]
      &\displaystyle \lesssim  \Big (\|\nabla_h\mathbf{u}\|^2_{H^N}+\|\mathop{\mathrm{div}}\nolimits\mathbf{u}\|^2_{H^N}\Big)\|\mathbf{u}\|_{H^N}.
   \end{align}
   Gathering \eqref{I-3-split}--\eqref{I-3-22}, we then arrive at
\begin{align*}
   \displaystyle I _{3}& \lesssim  \Big(\|\nabla_h\mathbf{u}\|^2_{H^N}+\|\mathop{\mathrm{div}}\nolimits\mathbf{u}\|^2_{H^N}\Big)\|\mathbf{u}\|_{H^N}.
   \end{align*}
 For the term $I_4$, if $ k=1 $, we get in virtue of \eqref{apri-conlu} and integration by parts that
 \begin{align}\label{I-4-1}
 \displaystyle I _{4}&  =\int _{\mathbb{R}^{3}} \partial ^{\alpha}(\nabla _{h}g(\rho) \cdot\nabla _{h}\mathbf{u})\cdot \partial ^{\alpha}\mathbf{u}\mathrm{d}x +\int _{\mathbb{R}^{3}} \partial ^{\alpha}(g(\rho) \nabla _{h}\mathbf{u}): \nabla _{h}\partial ^{\alpha}\mathbf{u}\mathrm{d}x
  \nonumber \\[2mm]
  & \lesssim  \int _{\mathbb{R}^{3}}\Big( \vert \nabla ^{2}\varrho\vert \vert \partial _{h}\mathbf{u}\vert + \vert \nabla \varrho\vert ^{2}\vert \partial _{h}\mathbf{u}\vert+ \vert \nabla \varrho\vert \vert \partial ^{\alpha}\partial _{h}\mathbf{u}\vert\Big) \vert \partial ^{\alpha}\mathbf{u}\vert \mathrm{d}x
   \nonumber \\[2mm]
   &\displaystyle \quad +\int _{\mathbb{R}^{3}}\Big( \vert \nabla \varrho\vert \vert \partial _{h}\mathbf{u}\vert+ \vert \varrho\vert \vert \partial ^{\alpha}\nabla _{h}\mathbf{u}\vert \Big) \vert \partial ^{\alpha}\partial _{h}\mathbf{u}\vert \mathrm{d}x
    \nonumber \\[2mm]
    &\displaystyle \lesssim \|\nabla ^{2}\varrho\| _{L ^{2}}\|\partial _{h}\mathbf{u}\|_{L ^{\infty}}\|\nabla \mathbf{u}\|_{L ^{2}}+\|\nabla \varrho\|_{L ^{\infty}}^{2}\|\nabla \mathbf{u}\|_{L ^{2}}\|\partial _{h}\mathbf{u}\|_{L ^{2}}+ \|\nabla \varrho\|_{L ^{\infty}}\|\nabla \partial _{h}\mathbf{u}\|_{L ^{2}}\|\nabla \mathbf{u}\|_{L ^{2}}
     \nonumber \\[2mm]
     &\displaystyle \quad+ \|\nabla \varrho\| _{L ^{\infty}} \|\partial _{h}\mathbf{u}\|_{L ^{2}}\|\nabla \partial _{h}\mathbf{u}\|_{L ^{2}}+\|\varrho\|_{L ^{\infty}}\|\nabla \partial _{h}\mathbf{u}\|_{L ^{2}}^{2}
      \nonumber \\[2mm]
      & \displaystyle \lesssim \Big( \|\nabla _{h}\mathbf{u}\|_{H ^{N}}^{2}+ \|\nabla \varrho\| _{H ^{N-1}}^{2} \Big) \left( \|\mathbf{u}\|_{H ^{N}}+ \|\varrho\|_{H ^{N}} \right)
 \end{align}
 for $ N \geq 3 $, where $ \nabla _{h}\mathbf{u}:\nabla _{h} \mathbf{u}=\sum _{i=1,2}\partial _{i}\mathbf{u} \cdot \partial _{i}\mathbf{u} $. If $ k \geq 2 $, we rewrite $ I _{4} $ as
 \begin{align*}
   \displaystyle I _{4}& =  \sum _{\vert \beta\vert=2}^{k-1} \int _{\mathbb{R}^{3}} C _{\beta}^{\alpha}\left[ (\nabla _{h}\partial ^{\alpha- \beta}g(\varrho) \cdot \nabla _{h}\partial ^{\beta} \mathbf{u})\cdot \partial ^{\alpha}\mathbf{u}+ \partial ^{\alpha- \beta}(g(\rho)) \nabla _{h}\partial ^{\beta}\mathbf{u}: \partial ^{\alpha}\nabla _{h}\mathbf{u}\right] \mathrm{d}x
    \nonumber \\[2mm]
     & \displaystyle \quad+\sum _{\vert \beta\vert=1} \int _{\mathbb{R}^{3}} C _{\beta}^{\alpha}\left[ (\nabla _{h}\partial ^{\alpha- \beta}g(\varrho) \cdot \nabla _{h}\partial ^{\beta} \mathbf{u})\cdot \partial ^{\alpha}\mathbf{u}+ \partial ^{\alpha- \beta}(g(\varrho) ) \nabla _{h}\partial ^{\beta}\mathbf{u}: \nabla _{h}\partial ^{\alpha}\mathbf{u}\right] \mathrm{d}x
    \nonumber \\[2mm]
    & \displaystyle \quad +\int _{\mathbb{R}^{3}}\partial ^{\alpha}( g(\varrho))\Delta _{h}\mathbf{u} \cdot \partial ^{\alpha}\mathbf{u}\mathrm{d}x + \int _{\mathbb{R}^{3}}\left[ (\nabla _{h}g(\varrho)\cdot \nabla _{h}\partial ^{\alpha}\mathbf{u})\cdot \partial ^{\alpha}\mathbf{u} +g(\varrho) \vert \nabla _{h}\partial ^{\alpha}\mathbf{u}\vert ^{2} \right] \mathrm{d}x
     \nonumber \\[2mm]
     &\displaystyle = \sum _{i=1}^{4} I _{4,i},
   \end{align*}
   where $ \nabla _{h}\partial ^{\beta}\mathbf{u}:\nabla _{h} \partial ^{\alpha}\mathbf{u}=\sum _{i=1,2}\partial _{i}\partial ^{\beta}\mathbf{u} \cdot \partial ^{\alpha}\partial _{i}\mathbf{u} $, $ C _{\beta}^{\alpha} $ is some constant depending on $ \alpha $ and $ \beta $. For $ I _{4,1} $ and $ I _{4,4} $, it holds for $ N \geq 3 $ that
   \begin{align}
   I _{4,1}&\lesssim \sum _{\vert \beta\vert=2}^{k-1}\Big(\|\nabla^{k-\vert \beta\vert}\nabla_h\varrho\|_{L^{3}}\|\nabla^{\vert \beta\vert}\nabla_h\mathbf{u}\|_{L^6}\|\nabla^k\mathbf{u}\| _{L ^{2}} +\|\nabla^{k-\vert \beta\vert}\varrho\|_{L^{3}}\|\nabla^{\vert \beta\vert}\nabla_h\mathbf{u}\|_{L^6}\|\nabla^k\nabla_h\mathbf{u}\| _{L ^{2}}\Big)
    \nonumber \\[2mm]
 &\lesssim \Big(\|\nabla\varrho\|^2_{H^{N-1}}+\|\nabla_h\mathbf{u}\|^2_{H^N}\Big)\Big(\|\mathbf{u}\|_{H^{N}}+\|\varrho\|_{H^N}\Big),
  \\[2mm]
  \displaystyle I _{4,4}& \lesssim  \|\nabla_h\varrho\|_{L^{\infty}}\|\nabla^k\nabla_h\mathbf{u}\| _{L ^{2}} \|\nabla^k\mathbf{u}\| _{L ^{2}} +\|\varrho\|_{L^{\infty}}\|\nabla^k\nabla_h\mathbf{u}\| _{L ^{2}} ^2
   \nonumber \\[2mm]
 &\lesssim \Big(\|\nabla\varrho\|^2_{H^{N-1}}+\|\nabla_h\mathbf{u}\|^2_{H^N}\Big)\Big(\|\mathbf{u}\|_{H^N}+\|\varrho\|_{H^N}\Big),
 \end{align}
 where we have used \eqref{nonlinear-con} and \eqref{apri-conlu}. By \eqref{two-two-deri-ani}, \eqref{nonlinear-con}, \eqref{apriori-assumption} and \eqref{apri-conlu}, we estimate $ I _{4,2} $ as follows.
\begin{align}\label{I-4-2}
 \displaystyle I _{4,2}& \lesssim \sum _{\vert \beta\vert=1}\|\nabla _{h}\partial ^{\alpha- \beta} g(\varrho) \|_{L ^{2}}\|\partial ^{\beta} \partial _{h}\mathbf{u}\cdot \partial ^{\alpha}\mathbf{u}\| _{L ^{2}}+\|\nabla\nabla_h\mathbf{u}\|_{L^{3}}\|\nabla^{k-1}\varrho\|_{L^6}\|\nabla^k\nabla_h\mathbf{u}\|_{L ^{2}}  \nonumber \\[2mm]
  &\lesssim \|\nabla\nabla_h\mathbf{u}\| _{L ^{2}}^{\frac{1}{4}}\|\nabla\nabla_h\partial_2\mathbf{u}\|_{L ^{2}}^{\frac{1}{4}}\|\nabla\nabla_h\partial_3\mathbf{u}\|_{L ^{2}}^{\frac{1}{4}}\|\nabla\nabla_h\partial_2\partial_3\mathbf{u}\|_{L ^{2}}^{\frac{1}{4}}\|\nabla^k\mathbf{u}\|_{L ^{2}}^{\frac{1}{2}}\|\nabla^k\partial_1\mathbf{u}\|_{L ^{2}}^{\frac{1}{2}}\|\nabla^{k}\varrho\|_{L ^{2}}\nonumber\\[2mm]
 &~\quad+\|\nabla\nabla_h\mathbf{u}\|_{L^{3}}\|\nabla^{k-1}\varrho\|_{L^6}\|\nabla^k\nabla_h\mathbf{u}\|_{L ^{2}}
  \nonumber \\[2mm]
 &\lesssim \|\nabla_h\mathbf{u}\|^{\frac{3}{2}}_{H^N}\|\nabla\varrho\|^{\frac{1}{2}}_{H^{N-1}}(\|\mathbf{u}\|^{\frac{1}{2}}_{H^N}\|\nabla\varrho\|_{H ^{N-1}}^{\frac{1}{2}})+\|\nabla_h\mathbf{u}\|_{H^2}\|\nabla^{k-1}\varrho\| _{H ^{1}} \|\nabla_h\mathbf{u}\| _{H ^{N}}
  \nonumber \\[2mm]
  &\lesssim \Big(\|\nabla_h\mathbf{u}\|^2_{H^N}+\|\nabla\varrho\|^2_{H^{N-1}}\Big)\Big(\|\mathbf{u}\|_{H^N}+\|\varrho\|_{H^N}\Big)
   \end{align}
   for $ N \geq 3 $. Similarly, we derive for  $ I _{4,3} $ that
 \begin{align}\label{i-4-3}
    I _{4,3}& \lesssim \|\partial ^{\alpha}(g(\varrho))\|_{L ^{2}} \|\Delta _{h}\mathbf{u}\cdot \partial^{\alpha}\mathbf{u}\| _{L ^{2}}
    \nonumber \\[1mm]
    & \displaystyle  \lesssim\|\Delta_h\mathbf{u}\| _{L ^{2}}^{\frac{1}{4}}\|\partial_2\Delta_h\mathbf{u}\| _{L ^{2}}^{\frac{1}{4}}\|\partial_3\Delta_h\mathbf{u}\| _{L ^{2}}^{\frac{1}{4}}\|\partial_2\partial_3\Delta_h\mathbf{u}\| _{L ^{2}}^{\frac{1}{4}}\|\nabla^k\mathbf{u}\| _{L ^{2}}^{\frac{1}{2}}\|\nabla^k\partial_1\mathbf{u}\| _{L ^{2}}^{\frac{1}{2}}\|\nabla^k\varrho\| _{L ^{2}}\nonumber\\[1mm]
 &\lesssim \|\nabla_h\mathbf{u}\|^{\frac{3}{2}}_{H^N}\|\nabla\varrho\|^{\frac{1}{2}}_{H^{N-1}}\|\mathbf{u}\|^{\frac{1}{2}}_{H^N}\|\nabla\varrho\|^{\frac{1}{2}}_{H^{N-1}}\nonumber\\[1mm]
 &\lesssim \displaystyle\Big(\|\nabla_h\mathbf{u}\|^2_{H^N}+\|\nabla\varrho\|^2_{H^{N-1}}\Big)\Big(\|\mathbf{u}\|_{H^N}+\|\varrho\|_{H^N}\Big).
 \end{align}
 Therefore we get for $ k \geq2 $ that
 \begin{align}\label{I-4}
 \displaystyle I _{4} \lesssim \Big( \|\nabla _{h}\mathbf{u}\|_{H ^{N}}^{2}+ \|\nabla \varrho\| _{H ^{N-1}}^{2} \Big) \Big( \|\mathbf{u}\|_{H ^{N}}+ \|\varrho\|_{H ^{N}} \Big).
 \end{align}
  By similar arguments as in deriving \eqref{I-4-1} and \eqref{I-4-2}, we can control $I_5$ as
  \begin{align*}
  \displaystyle I _{5}& \lesssim  \Big(\|\nabla\varrho\|^2_{H^{N-1}}+\|\mathop{\mathrm{div}}\nolimits\mathbf{u}\|^2_{H^N}\Big)\Big(\|\mathbf{u}\|_{H^N}+\|\varrho\|_{H^N}\Big)
  \end{align*}
  for $ 1 \leq k \leq N $ with $ N \geq 3 $. Finally, let us turn to the estimate of $I_6$. If $ k=1 $, we can directly derive that
  \begin{align}\label{I-6-esti1}
  \displaystyle I _{6}& \lesssim \int _{\mathbb{R}^{3}}\Big( \vert \varrho\vert \vert \nabla ^{2}\varrho\vert +\vert \nabla \varrho\vert ^{2}\Big)\vert \nabla \mathbf{u}\vert \mathrm{d}x
   \nonumber \\[2mm]
   & \displaystyle \lesssim \|\varrho\|_{L ^{6}}\|\nabla ^{2}\varrho\| _{L ^{2}}\|\nabla \mathbf{u}\| _{L ^{3}}+\|\nabla \mathbf{u}\|_{L ^{\infty}}\|\nabla ^{2}\varrho\| _{L ^{2}}^{2} \lesssim \|\nabla\varrho\|^2_{H^{N-1}}\|\mathbf{u}\|_{H^N}.
  \end{align}
  If $ k \geq 2 $, we have
  \begin{align}\label{I-6-split}
  \displaystyle I _{6}&=  \int _{\mathbb{R}^{3}} f(\varrho) \partial ^{\alpha}\varrho  \partial ^{\alpha}\mathop{\mathrm{div}}\nolimits\mathbf{u}\mathrm{d}x  +  \int _{\mathbb{R}^{3}} \partial ^{\alpha}\varrho\nabla f(\varrho) \cdot \partial ^{\alpha}\mathbf{u}\mathrm{d}x- \int _{\mathbb{R}^{3}}\partial ^{\alpha}(f(\varrho))\nabla \varrho \cdot \partial ^{\alpha}\mathbf{u}\mathrm{d}x
   \nonumber \\[2mm]
   & \displaystyle \quad- \sum _{\vert \beta\vert=1}^{k-2}C _{\beta}^{\alpha}\partial ^{\beta}(f(\varrho))\nabla \partial ^{\alpha- \beta}\varrho  \cdot \partial ^{\alpha}\mathbf{u}\mathrm{d}x - \sum _{\vert \beta\vert=k-1} \int _{\mathbb{R}^{3}} C _{\beta}^{\alpha}\partial ^{\beta}(f(\varrho))\nabla \partial ^{\alpha- \beta}\varrho \cdot \partial ^{\alpha}\mathbf{u}\mathrm{d}x
    \nonumber \\[2mm]
    & \displaystyle=I _{6,1}+I _{6,2}+I _{6,3}+I _{6,4}+I _{6,5},
  \end{align}
  where, in virtue of \eqref{apriori-assumption}, \eqref{apri-conlu} and Lemma \ref{lem-compose}, we derive for $ N \geq 3 $ that
  \begin{align}
  \displaystyle  \begin{cases}
    \displaystyle  I _{6,1} \lesssim \|\varrho\|_{L ^{\infty}}\|\partial ^{\alpha}\varrho\|_{L ^{2}} \|\partial ^{\alpha}\mathop{\mathrm{div}}\nolimits \mathbf{u}\|_{L ^{2}} \lesssim \|\varrho\|_{H ^{N}}\|\nabla \varrho\|_{H ^{N-1}}\|\mathop{\mathrm{div}}\nolimits \mathbf{u}\|_{H ^{N}},\\[3mm]
      \displaystyle  I _{6,2}\lesssim \|\partial ^{\alpha}\varrho\|_{L ^{2}}\|\nabla \varrho\|_{L ^{\infty}}\|\nabla ^{k}\mathbf{u}\|_{L ^{2}}\lesssim \|\nabla \varrho\|_{H ^{N-1}}^{2}\|\mathbf{u}\|_{H ^{N}},\\[3mm]
      \displaystyle I _{6,3} \lesssim  \|\partial ^{\alpha}(f(\varrho))\|_{L ^{2}}\|\nabla \varrho\| _{L ^{\infty}} \|\nabla ^{k}\mathbf{u}\| _{L ^{2}} \lesssim \|\varrho\|_{H ^{N}}\|\nabla \varrho\|_{H ^{N-1}}\|\mathbf{u}\|_{H ^{N}},\\[3mm]
      \displaystyle I _{6,4} \lesssim \sum _{\vert \beta\vert=1}^{k-2}\|\partial ^{\beta}(f(\varrho))\|_{L ^{\infty}}\|\nabla ^{k+1-|\beta|}\varrho\|_{L ^{2}}\|\nabla ^{k}\mathbf{u}\|_{L ^{2}}\lesssim \|\nabla \varrho\|_{H ^{N-1}}^{2}\|\mathbf{u}\|_{H ^{N}},
       \nonumber \\[4mm]
       \displaystyle I _{6,5}\lesssim \sum _{\vert \beta\vert=k-1}\|\partial ^{\beta}(f(\varrho))\|_{L ^{6}}\|\nabla ^{2}\varrho\|_{L ^{3}}\|\nabla ^{k}\mathbf{u}\|_{L ^{2}}\lesssim \|\nabla \varrho\|_{H ^{N-1}}^{2}\|\mathbf{u}\|_{H ^{N}}.
       \end{cases}
  \end{align}
  Therefore we get from \eqref{I-6-split} that
  \begin{align*}
  \displaystyle I _{6} \lesssim \Big(\|\nabla\varrho\|^2_{H^{N-1}}+\|\mathop{\mathrm{div}}\nolimits\mathbf{u}\|^2_{H^N}\Big)\Big(\|\mathbf{u}\|_{H^N}+\|\varrho\|_{H^N}\Big)
  \end{align*}
  for $ k \geq 2 $ and $ N \geq 3 $. This along with \eqref{I-6-esti1} implies that
   \begin{align*}
  \displaystyle I _{6} \lesssim \Big(\|\nabla\varrho\|^2_{H^{N-1}}+\|\mathop{\mathrm{div}}\nolimits\mathbf{u}\|^2_{H^N}\Big)\Big(\|\mathbf{u}\|_{H^N}+\|\varrho\|_{H^N}\Big)
  \end{align*}
  for $ k \geq 1 $. Substituting the estimates on $ I _{i}~(1 \leq i \leq 6) $ into \eqref{a1}, we then arrive at
  \begin{align}
  &\displaystyle  \frac{1}{2}\frac{\mathrm{d}}{\mathrm{d}t}\Big(\gamma\|\varrho\| _{L ^{2}}^2+\|\mathbf{u}\|_{L ^{2}}^2 \Big)+\bar{\mu}\|\nabla_h \mathbf{u}\|_{L ^{2}}^2+(\bar{\mu}+\bar{\zeta})\|\mathop{\mathrm{div}}\nolimits \mathbf{u}\|_{L ^{2}}^2
  \nonumber \\[2mm]
 &~\displaystyle \lesssim \Big(\|\nabla\varrho\|^2_{H^{N-1}}+\|\nabla_h \mathbf{u}\|^2_{H^N}+\|\mathop{\mathrm{div}}\nolimits \mathbf{u}\|^2_{H^N}\Big)\Big(\|\varrho\|_{H^N}+\|\mathbf{u}\|_{H^N}\Big) .
  \end{align}

 \vspace{3mm}

\noindent  \textbf{Step 2.} Dissipation of $\|\nabla\varrho\|^2_{H ^{N-1}}$. Precisely, we shall show that
 \begin{align}\label{jiaocha1}
 &\displaystyle \sum _{\vert \alpha\vert=0}^{N-1}\frac{\mathrm{d}}{\mathrm{d}t}\Big\langle\partial^{\alpha}\mathbf{u},\partial^{\alpha}\nabla\varrho \Big\rangle+\frac{\gamma \bar{\rho}}{2}\|\nabla \varrho\|_{H ^{N-1}} ^2
\lesssim\Big(\|\nabla_h\mathbf{u}\|^2_{N}+\|\mathop{\mathrm{div}}\nolimits\mathbf{u}\|^2_{H^N}\Big).
 \end{align}
Testing $\eqref{high}_2$ against $\partial ^{\alpha} \nabla\varrho$ with $ 0 \leq \vert \alpha\vert=k \leq N-1 $, we have
 \begin{align}\label{jiaocha0}
 &\frac{\mathrm{d}}{\mathrm{d}t}\int_{\mathbb{R}^3}\partial ^{\alpha}\mathbf{u}\cdot\nabla \partial ^{\alpha}\varrho\mathrm{d}x+\gamma\bar{\rho}\int_{\mathbb{R}^3} \vert \nabla \partial ^{\alpha}\varrho\vert^2\mathrm{d}x\nonumber\\
 &~\displaystyle= \bar{\mu} \int _{\mathbb{R}^{3}}\Delta _{h}\partial ^{\alpha}\mathbf{u}\cdot \nabla \partial ^{\alpha}\varrho
\mathrm{d}x +(\bar{\mu}+\bar{\zeta}) \int _{\mathbb{R}^{3}}\nabla \mathop{\mathrm{div}}\nolimits \partial ^{\alpha}\mathbf{u} \cdot \nabla \partial ^{\alpha}\varrho \mathrm{d}x
 \nonumber \\
 & \displaystyle \quad+ \int _{\mathbb{R}^{3}}\partial ^{\alpha}S _{2}\cdot \nabla \partial ^{\alpha}\varrho \mathrm{d}x + \int _{\mathbb{R}^{3}}\partial ^{\alpha}\mathbf{u}\cdot \nabla \partial ^{\alpha}\partial _{t}\varrho \mathrm{d}x,
 \end{align}
 where thanks to the Cauchy-Schwarz inequality, we get
\begin{align}
&\displaystyle \bar{\mu} \int _{\mathbb{R}^{3}}\Delta _{h}\partial ^{\alpha}\mathbf{u}\cdot \nabla \partial ^{\alpha}\varrho
\mathrm{d}x +(\bar{\mu}+\bar{\zeta}) \int _{\mathbb{R}^{3}}\nabla \mathop{\mathrm{div}}\nolimits \partial ^{\alpha}\mathbf{u} \cdot \nabla \partial ^{\alpha}\varrho \mathrm{d}x
 \nonumber \\
 &~\displaystyle \leq\frac{\gamma}{8}\bar{\rho} \int _{\mathbb{R}^{3}}\vert \nabla \partial ^{\alpha}\varrho\vert ^{2}\mathrm{d}x +C\Big( \| \nabla _{h}\mathbf{u} \| _{H ^{N}}^{2}+\|\mathop{\mathrm{div}}\nolimits \mathbf{u}\|_{H ^{N}}^{2} \Big) ,
\end{align}
where $ C>0 $ is a positive constant that is independent of $ T $. For the third term on the right hand side of \eqref{jiaocha0}, recalling the definition of $ S _{2} $ in \eqref{nonlinear}, we get
\begin{align*}
&\displaystyle  \int _{\mathbb{R}^{3}}\partial ^{\alpha}S _{2}\cdot \nabla \partial ^{\alpha}\varrho \mathrm{d}x
 \nonumber \\
&~ \displaystyle \leq \frac{\gamma \bar{\rho}}{8}\int _{\mathbb{R}^{3}}\vert \nabla \partial ^{\alpha}\varrho\vert ^{2}\mathrm{d}x +C\|\partial ^{\alpha}S _{2}\|_{L ^{2}}^{2}
  \nonumber \\[2mm]
  & ~\displaystyle \leq  \frac{\gamma \bar{\rho}}{8}\int _{\mathbb{R}^{3}}\vert \nabla \partial ^{\alpha}\varrho\vert ^{2}\mathrm{d}x+ C\| \partial ^{\alpha}(g(\varrho)\Delta _{h}\mathbf{u})\|_{L ^{2}}^{2}+ C\|\partial ^{\alpha}(g(\varrho)\nabla \mathop{\mathrm{div}}\nolimits \mathbf{u})\|_{L ^{2}}^{2}
   \nonumber \\[2mm]
   & ~\displaystyle \quad+ C\|\partial ^{\alpha}(f(\varrho)\nabla \varrho)\|_{L ^{2}}^{2} + C\|\partial ^{\alpha}(\mathbf{u}\cdot \nabla \mathbf{u})\|_{L ^{2}} ^{2}
    \nonumber \\[2mm]
    &~ \displaystyle \leq  \frac{\gamma \bar{\rho}}{8}\int _{\mathbb{R}^{3}}\vert \nabla \partial ^{\alpha}\varrho\vert ^{2}\mathrm{d}x+ C\|g(\varrho)\|_{L ^{\infty}}^{2}\Big( \|\nabla ^{k}\Delta _{h}\mathbf{u}\|_{L ^{2}}^{2}+\|\nabla ^{k+1}\mathop{\mathrm{div}}\nolimits \mathbf{u}\|_{L ^{2}}^{2} \Big)
     \nonumber \\[2mm]
     &~ \displaystyle \quad+ C\sum^{k}_{|\beta|=1}\|\partial^{\beta}g(\varrho)\|^2_{L^3}\big(\|\partial^{\alpha-\beta}\Delta_h\mathbf{u}\|^2_{L^6}+\|\partial^{\alpha-\beta}\nabla\mathop{\mathrm{div}}\nolimits \mathbf{u}\|^2_{L^6}\big)\nonumber\\
      &~\displaystyle\quad+C\|f(\varrho)\|^2_{L ^{\infty}}\|\nabla ^{k+1} \varrho\|_{L ^{2}}^{2}+C\sum^{k}_{|\beta|=1}\|\partial^{\beta}f(\varrho)\|^2_{L^3}\|\partial^{\alpha-\beta}\nabla\varrho\|^2_{L^6}+C\|\partial ^{\alpha}(\mathbf{u}\cdot \nabla \mathbf{u})\|_{L ^{2}} ^{2}
      \nonumber \\[2mm]
      &~ \displaystyle \leq \frac{\gamma \bar{\rho}}{8}\int _{\mathbb{R}^{3}}\vert \nabla \partial ^{\alpha}\varrho\vert ^{2}\mathrm{d}x+C\|\varrho\|^2_{H ^{N}}\|\nabla \varrho\|_{H ^{N-1}}^{2} +C\|\partial ^{\alpha}(\mathbf{u}\cdot \nabla \mathbf{u})\|_{L ^{2}} ^{2}
       \nonumber \\[2mm]
       &~\displaystyle \quad+C\Big( \|\nabla _{h}\mathbf{u}\|_{H^{N}}^{2}+\|\mathop{\mathrm{div}}\nolimits \mathbf{u}\|_{H ^{N} } ^{2} \Big)\|\varrho\|^2_{H ^{N}},
\end{align*}
with
\begin{align}\label{pa-3-u3-trans}
&\displaystyle \|\partial ^{\alpha}(\mathbf{u}\cdot \nabla \mathbf{u})\|_{L ^{2}} ^{2}\nonumber\\[2mm]
& \lesssim \|\partial ^{\alpha}(\mathbf{u}_{h}\cdot \nabla _{h}\mathbf{u})\|_{L ^{2}}^{2} +\|\partial ^{\alpha}(u _{3}\partial _{3}\mathbf{u} _{h})\| _{L ^{2}}^{2}+\|\partial ^{\alpha}(u _{3}\partial _{3}u _{3})\|_{L ^{2}}^{2}
 \nonumber \\[2mm]
 & \lesssim \|\mathbf{u}\|^2_{L ^{\infty}} \|\nabla _{h}\mathbf{u}\|_{H^{N}}^{2}+\sum^k_{|\beta|=1}\|\partial^{\beta}\mathbf{u}_h\|^2_{L^3}\| \nabla_h\partial^{\alpha-\beta}\mathbf{u}\|^2_{L^6}+\|u_3\partial^{\alpha}\partial_3\mathbf{u}_h\|^2_{L^2}\nonumber\\
 &\quad+\sum^{k}_{|\beta|=1}\|\partial^{\beta}u_3\partial^{\alpha-\beta}\partial_3\mathbf{u}_h\|^2_{L^2} +\sum^{k}_{|\beta|=0}\|\partial^{\beta}u_3\|^2_{L^6}\|\partial^{\alpha-\beta}\partial_3u_3\|^2_{L^3}\nonumber\\[2mm]
 &\lesssim \|\mathbf{u}\|^2_{H^N}\|\nabla_h\mathbf{u}\|^2_{H^N}+\|u_3\|^{\frac{1}{2}}_{L^2}\|\partial_1u_3\|^{\frac{1}{2}}_{L^2}\|\partial_3u_3\|^{\frac{1}{2}}_{L^2}\|\partial_1\partial_3u_3\|^{\frac{1}{2}}_{L^2}\|\partial^{\alpha}\partial_3\mathbf{u}_h\|_{L^2}
 \|\partial^{\alpha}\partial_3\partial_2\mathbf{u}_h\|_{L^2}\nonumber\\[2mm]
 &\quad+\sum^k_{|\beta|=1}\|\partial^{\alpha-\beta}\partial_3\mathbf{u}_h\|^{\frac{1}{2}}_{L^2}\|\partial^{\alpha-\beta}\partial_3\partial_1\mathbf{u}_h\|^{\frac{1}{2}}_{L^2}\|\partial^{\alpha-\beta}\partial^2_3\mathbf{u}_h\|^{\frac{1}{2}}_{L^2}\|\partial^{\alpha-\beta}\partial^2_3\partial_1\mathbf{u}_h\|^{\frac{1}{2}}_{L^2}
 \|\partial^{\beta}u_3\|_{L^2}\|\partial^{\beta}\partial_2u_3\|_{L^2}\nonumber\\[2mm]
 &\quad+\|u _{3}\|^2 _{H^N}\Big( \|\mathop{\mathrm{div}}\nolimits \mathbf{u}\|_{H ^{N} } ^{2}+ \|\nabla _{h}\mathbf{u}\| _{H ^{N}}^{2} \Big)
  \nonumber \\[2mm]
  &\lesssim \Big( \|\nabla _{h}\mathbf{u}\|_{H^{N}}^{2}+\|\mathop{\mathrm{div}}\nolimits \mathbf{u}\|_{H ^{N} } ^{2} \Big)\|\mathbf{u}\|^2 _{H ^{N}}
\end{align}
for $ N \geq 3 $, where we have used Lemma \ref{lemma2.1}, Lemma \ref{lemma3.2} and the following inequality
\begin{align*}
\displaystyle  \|\nabla ^{k}\nabla u _{3}\| _{L ^{2}} \lesssim \|\nabla ^{k}\nabla _{h}\mathbf{u}\|_{L ^{2}}+\|\nabla ^{k}\mathop{\mathrm{div}}\nolimits \mathbf{u}\|_{L ^{2}}.
\end{align*}
Therefore we get
\begin{align}
\displaystyle  \int _{\mathbb{R}^{3}}\partial ^{\alpha}S _{2}\cdot \nabla \partial ^{\alpha}\varrho \mathrm{d}x &\leq C \Big( \|\nabla _{h}\mathbf{u}\|_{H^{N}}^{2}+\|\mathop{\mathrm{div}}\nolimits \mathbf{u}\|_{H ^{N} } ^{2} \Big)\Big(\|\mathbf{u}\| _{H ^{N}}+\|\varrho\|_{H ^{N}}\Big)
 \nonumber \\[2mm]
 &\displaystyle \quad+\frac{\gamma \bar{\rho}}{8}\int _{\mathbb{R}^{3}}\vert \nabla \partial ^{\alpha}\varrho\vert ^{2}\mathrm{d}x+\|\varrho\|_{H ^{N}}\|\nabla \varrho\|_{H ^{N-1}}^{2}.
\end{align}
For the last term on the right hand side of \eqref{jiaocha0}, from \eqref{apriori-assumption},  $ \eqref{high}_{1} $ and Lemma \ref{lemma3.2}, we get for $ N \geq 3 $ that
\begin{align}\label{na-vrho-final}
&\displaystyle  \int _{\mathbb{R}^{3}}\partial ^{\alpha}\mathbf{u}\cdot \nabla \partial ^{\alpha}\partial _{t}\varrho \mathrm{d}x=- \int _{\mathbb{R}^{3}}\partial ^{\alpha}\mathop{\mathrm{div}}\nolimits \mathbf{u}\cdot \partial ^{\alpha}\partial _{t}\varrho \mathrm{d}x
 \nonumber \\
 &~\displaystyle = \bar{\rho}\int _{\mathbb{R}^{3}}\vert \partial ^{\alpha}\mathop{\mathrm{div}}\nolimits\mathbf{u}\vert ^{2}\mathrm{d}x+ \int _{\mathbb{R}^{3}}\partial ^{\alpha}\mathop{\mathrm{div}}\nolimits \mathbf{u} \cdot \partial ^{\alpha} (\varrho \mathop{\mathrm{div}}\nolimits \mathbf{u}+\mathbf{u}\cdot \nabla \varrho)\mathrm{d}x
  \nonumber \\[2mm]
  & ~\displaystyle\lesssim \|\mathop{\mathrm{div}}\nolimits \mathbf{u}\|_{H ^{N}}^{2}+ \|\mathop{\mathrm{div}}\nolimits \mathbf{u}\|_{H ^{N}}^{2}\|\varrho\|_{H ^{N}}+\|\mathop{\mathrm{div}}\nolimits \mathbf{u}\|_{H ^{N}}\|\mathbf{u}\|_{H ^{N}}\|\nabla \varrho\|_{H ^{N-1}}
   \nonumber \\[2mm]
   &~\displaystyle \lesssim \|\mathop{\mathrm{div}}\nolimits \mathbf{u}\|_{H ^{N}}^{2}+ \|\mathbf{u}\|_{H ^{N}}\|\nabla \varrho\|_{H ^{N-1}}^{2},
\end{align}
where we have used the fact $ \|\varrho\|_{H ^{N}} \lesssim 1$ from \eqref{apriori-assumption}. Gathering \eqref{jiaocha0}--\eqref{na-vrho-final}, we have
\begin{align}
&\displaystyle \frac{\mathrm{d}}{\mathrm{d}t}\int_{\mathbb{R}^3}\partial ^{\alpha}\mathbf{u}\cdot\nabla \partial ^{\alpha}\varrho\mathrm{d}x+ \frac{3}{4}\gamma\bar{\rho}\int_{\mathbb{R}^3} \vert \nabla \partial ^{\alpha}\varrho\vert^2\mathrm{d}x
 \nonumber \\
 &~\displaystyle \lesssim \|\nabla _{h}\mathbf{u}\|_{H^{N}}^{2}+\|\mathop{\mathrm{div}}\nolimits \mathbf{u}\|_{H ^{N} } ^{2} + \Big(  \|\mathbf{u}\|_{H ^{N}}+\|\varrho\|_{H ^{N}}\Big) \|\nabla \varrho\|_{H ^{N-1}}^{2},
\end{align}
which followed by a summation with respect to $ \alpha $ with $ 0 \leq\vert \alpha\vert \leq N-1$ further yields that
\begin{align*}
\displaystyle  &\displaystyle \sum _{\vert \alpha\vert=0}^{N-1}\frac{\mathrm{d}}{\mathrm{d}t}\Big\langle\partial^{\alpha}\mathbf{u},\partial^{\alpha}\nabla\varrho \Big\rangle+\frac{\gamma \bar{\rho}}{2}\|\nabla \varrho\|_{H ^{N-1}} ^2
\lesssim \Big(\|\nabla_h\mathbf{u}\|^2_{N}+\|\mathop{\mathrm{div}}\nolimits\mathbf{u}\|^2_{H^N}\Big),
\end{align*}
 where we have used \eqref{apriori-assumption} with $ \delta $ being suitably small. This gives \eqref{jiaocha1}.

  \textbf{Step 3.} Close of the estimates. Notice that there exists a suitably small constant $ \kappa _{0}>0 $ such that
 \begin{gather}\label{U-1}
 \displaystyle  \gamma\|\varrho\| _{H ^{N} }^2+\|\mathbf{u}\|_{H ^{N}}^2
 +\kappa _{0}\sum\limits_{\vert \alpha\vert\leq N-1}\Big\langle  \partial^{\alpha}_{x}\mathbf{u},\partial^{\alpha}_{x}\nabla\rho \Big\rangle\sim\|U(t)\| _{H ^{N}}^2.
 \end{gather}
 Then, we get after integrating $ \eqref{a}+\eqref{na-vrho-final}\times \kappa _{0}$ over $ [0,t]\subset [0,T] $ that
 \begin{align}\label{esti-exist-final}
 \displaystyle \|U(t)\|_{H ^{N}}^{2}+ \int _{0}^{t}\left(\|\nabla_h\mathbf{u}\|^2_{N}+\|\mathop{\mathrm{div}}\nolimits\mathbf{u}\|^2_{H^N}+\|\nabla \varrho\|_{H ^{N-1}}^{2}\right)\mathrm{d}\tau \lesssim  \|U(0)\|_{H ^{N}}^{2}=\|(\varrho _{0},\mathbf{u}_{0})\|_{H ^{N}}^{2}
 \end{align}
 for suitably small $ \delta $. This immediately gives \eqref{ine1}, and thus ends the proof of Lemma \ref{priori}.
\end{proof}
With \eqref{ine1}, we can readily verify the \emph{a priori} assumption \eqref{apriori-assumption} by setting $\|(\varrho _{0},\mathbf{u}_{0})\|_{H ^{N}}\leq \delta/2  $. The proof of Proposition \ref{prop3.1} is complete. \hfill $ \square $

\vspace{5mm}

\section{Decay rates of the linearized homogeneous system}\label{SEC-PROOF}

In this section, we shall first study the linearized homogeneous dynamics, and then proceed to the study on the nonlinear stability of the constant equilibrium.

\subsection{Green's function of the linearized system} 
\label{sub:the_linear_homogeneous_dynamics}

The linearized homogeneous system corresponding to \eqref{new} reads
\begin{gather}\label{new-linear}
 \begin{cases}
 \displaystyle\varrho_t+\bar{\rho}\mathop{\mathrm{div}}\nolimits\mathbf{u}=0,\ \ t \geq 0,\ \ x \in \mathbb{R}^{3},\\[2mm]
 \displaystyle\mathbf{ u}_t+\gamma\bar{\rho}\nabla \varrho -\bar{\mu}\Delta_h \mathbf{u}-(\bar{\mu}+\bar{\zeta})\nabla \mathop{\mathrm{div}}\nolimits \mathbf{u}=\mathbf{0}.
 \end{cases}
 \end{gather}
Here and in the sequel, for simplicity, we use the same notations $ (\varrho,\mathbf{u}) $ as in \eqref{new} to denote the unknown functions. Denote $ U=(\varrho,\mathbf{u}) $. We shall explicitly solve \eqref{new-linear} suplemented with the initial data
\begin{align}\label{initial-linear}
\displaystyle U \vert _{t=0}=U _{0}=(\varrho _{0},\mathbf{u}_{0}),\ \ x  \in \mathbb{R}^{3}.
\end{align}
In terms of the Fourier transform in $x$, the system \eqref{new-linear} is equivalent to
\begin{gather}\label{f-new-linear}
\begin{cases}
\displaystyle \hat{\varrho}_t+\bar{\rho}i\xi\cdot\hat{\mathbf{u}}=0,\\[2mm]
\displaystyle\hat{\mathbf{u}}_t+\gamma\bar{\rho}i\xi\hat{\varrho}+\bar{\mu}|\xi_{h}|^2\hat{\mathbf{u}}+(\bar{\mu}+\bar{\zeta})\xi\xi^T\hat{\mathbf{u}}=\mathbf{0}.
\end{cases}
\end{gather}
Furthermore, by taking the inner product of $\eqref{f-new-linear}_2$ with $\tilde{\xi}:=\frac{\xi}{|\xi|}$ for $ \xi \neq \mathbf{0} $, we have
\begin{equation*}
\tilde{\xi}\cdot\hat{\mathbf{u}}_t+\gamma\bar{\rho}i\vert \xi\vert\hat{\varrho}+\bar{\mu}\vert \xi _{h}\vert ^2\tilde{\xi}\cdot\hat{\mathbf{u}}+(\bar{\mu}+\bar{\zeta})\vert \xi\vert^2\tilde{\xi}\cdot\hat{\mathbf{u}}=0,
\end{equation*}
which alongside $ \eqref{new-linear}_{1} $ forms the following system
\begin{gather}\label{f-new-linear1}
\begin{cases}
\hat{\varrho}_t+\bar{\rho}i\xi\cdot\hat{\mathbf{u}}=0,\\[2mm]
\tilde{\xi}\cdot\hat{\mathbf{u}}_t+\gamma\bar{\rho}i \vert \xi\vert\hat{\varrho}+\bar{\mu}\vert \xi_h\vert^2\tilde{\xi}\cdot\hat{\mathbf{u}}+(\bar{\mu}+\bar{\zeta})\vert \xi\vert^2\tilde{\xi}\cdot\hat{\mathbf{u}}=0.
\end{cases}
\end{gather}
We remark that the system \eqref{f-new-linear1} is a variant of the Fourier system corresponding to the linearized system of the equations for $ (\varrho,\mathop{\mathrm{div}}\nolimits \mathbf{u}) $. To proceed, denote $ \hat{V}=(\hat{\varrho},\tilde{\xi}\cdot \hat{\mathbf{u}}) $. Then, we can rewrite \eqref{f-new-linear1} as
\begin{align*}
\displaystyle \partial _{t}\hat{V}=A(\xi)\hat{V},
\end{align*}
with
\begin{align*}
A(\xi)=\left(
\begin{matrix}
\displaystyle
0 & -\bar{\rho}i\vert \xi\vert\\[2mm]
\displaystyle-\gamma\bar{\rho}i \vert \xi\vert & -\bar{\mu} \vert \xi_h\vert^2-(\bar{\mu}+\bar{\zeta})\vert \xi\vert^2
\end{matrix}
\right).
\end{align*}
By direct computations, we know that the maxtrix $ A(\xi) $ has two eigenvalues

\begin{align}\label{eig-values}
\begin{cases}
    \displaystyle\lambda_1=\frac{-\bar{\mu}\vert \xi_h\vert^2-(\bar{\mu}+\bar{\zeta})\vert \xi\vert^2-\sqrt{(\bar{\mu}\vert \xi_h\vert^2+(\bar{\mu}+\bar{\zeta})\vert \xi\vert^2)^2-4\gamma\bar{\rho}^2 \vert \xi\vert^2}}{2},
    \\[2mm]
\displaystyle\lambda_2=\frac{-\bar{\mu}\vert \xi_h\vert^2-(\bar{\mu}+\bar{\zeta})\vert \xi\vert^2+\sqrt{(\bar{\mu}\vert \xi_h\vert^2+(\bar{\mu}+\bar{\zeta})\vert \xi\vert^2)^2-4\gamma\bar{\rho}^2 \vert \xi\vert^2}}{2}.
\end{cases}
\end{align}
The corresponding eigenvectors of $ \lambda _{1} $ and $ \lambda _{2} $, respectively, read
\begin{align*}
v_1=\left(\begin{matrix}
\bar{\rho}i \vert \xi\vert\\[2mm]
-\lambda_1
\end{matrix}\right)\ \  \mbox{and}\ \ v_2=\left(\begin{matrix}
\bar{\rho}i|\xi|\\[2mm]
-\lambda_2
\end{matrix}\right),
\end{align*}
from which we can set the representation of the solution $ \hat{V} $ as
\begin{align*}
\displaystyle \hat{V}= \left(\begin{matrix}
\bar{\rho}i\vert \xi\vert e^{\lambda_1t} & \bar{\rho}i\vert \xi\vert e^{\lambda_2t}\\[2mm]
-\lambda_1e^{\lambda_1t} & -\lambda_2e^{\lambda_2t}
\end{matrix}\right)\left(
\begin{matrix}
d_1 \\[2mm] d_2
\end{matrix}
\right),
\end{align*}
where $d_1$ and $d_2$ are determined by
\begin{equation}\label{general solution1}
\left(\begin{matrix}
\hat{\varrho} _{0}\\[2mm]
\tilde{\xi}\cdot\hat{\mathbf{u}}_{0}
\end{matrix}
\right)=\left(\begin{matrix}
\bar{\rho}i \vert \xi\vert & \bar{\rho}i\vert \xi\vert\\[2mm]
-\lambda_1 & -\lambda_2
\end{matrix}\right)
\left(
\begin{matrix}
d_1 \\[2mm] d_2
\end{matrix}
\right).
\end{equation}
That is,
\begin{equation*}
\left(\begin{matrix}
d_1\\[2mm]
d_2
\end{matrix}
\right)=\frac{1}{\bar{\rho}i \vert \xi\vert (\lambda_1-\lambda_2)}\left(\begin{matrix}
-\lambda_2 & -\bar{\rho}i \vert \xi\vert\\[2mm]
\lambda_1 & \bar{\rho}i \vert \xi\vert
\end{matrix}\right)\left(\begin{matrix}
\hat{\rho}_0\\
\tilde{\xi}\cdot \hat{\mathbf{u}}_0
\end{matrix}\right).
\end{equation*}
Therefore we get from \eqref{general solution1} that
\begin{align}\label{rhohat}
\hat{\varrho}=\frac{\lambda_1e^{\lambda_2t}-\lambda_2e^{\lambda_1t}}{\lambda_1-\lambda_2}\hat{\varrho}_0-i\bar{\rho}
\frac{e^{\lambda_1t}-e^{\lambda_2t}}{\lambda_1-\lambda_2}\xi\cdot\hat{\mathbf{u}}_0,
\end{align}
and
\begin{align}\label{uhat}
\tilde{\xi}\cdot\hat{\mathbf{u}}=\frac{\lambda_1\lambda_2}{\bar{\rho}i|\xi|}\frac{e^{\lambda_1t}-e^{\lambda_2t}}{\lambda_1-\lambda_2}\hat{\varrho}_0+\frac{\lambda_1e^{\lambda_1t}-\lambda_2e^{\lambda_2t}}{\lambda_1-\lambda_2}\tilde{\xi}\cdot\hat{\mathbf{u}}_0\ \ \mbox{for }\xi \neq \mathbf{0}.
\end{align}
Next, applying $ \nabla \times $ to $ \eqref{new-linear}_{2} $, we get
\begin{equation}\label{xuanduu}
\partial_t(\nabla\times \mathbf{u})-\bar{\mu}\Delta_h\nabla\times \mathbf{u}=\mathbf{0}.
\end{equation}
Taking the Fourier transform of the equations in \eqref{xuanduu} with resepect to the $x$ variable, we have
\begin{equation*}
\displaystyle\partial_t(\tilde{\xi}\times \hat{\mathbf{u}})+\bar{\mu}\vert \xi _{h}\vert ^{2} (\tilde{\xi}\times\hat{\mathbf{u}})=\mathbf{0}.
\end{equation*}
This along with the initial data
\begin{gather*}
(\tilde{\xi}\times\hat{\mathbf{u}})|_{t=0}=\tilde{\xi}\times \hat{\mathbf{u}}_0,
\end{gather*}
further implies that
\begin{equation}\label{curl-fourier}
(\tilde{\xi}\times\hat{\mathbf{u}})=e^{-\bar{\mu}\vert \xi_h\vert ^{2} t}\tilde{\xi}\times  \hat{\mathbf{u}}_0.
\end{equation}
Notice that
\begin{align*}
\displaystyle \nabla \times(\nabla \times \mathbf{u})=\nabla \mathop{\mathrm{div}}\nolimits \mathbf{u}- \Delta \mathbf{u}.
\end{align*}
Then we get
\begin{align*}
\displaystyle  \hat{\mathbf{u}}=\tilde{\xi}\tilde{\xi}\cdot\hat{\mathbf{u}}-\tilde{\xi}\times(\tilde{\xi}\times\hat{\mathbf{u}})
\end{align*}
for any $ \xi \in \mathbb{R}^{3} $ with $  \xi \neq \mathbf{0} $. This along with \eqref{uhat} and \eqref{curl-fourier} yields that
\begin{align}\label{u-formula-fourie}
\displaystyle \displaystyle\hat{\mathbf{u}}=\frac{\lambda_1\lambda_2\xi}{\bar{\rho}i \vert \xi\vert^2}\left(\frac{e^{\lambda_1t}-e^{\lambda_2t}}{\lambda_1-\lambda_2}\right)\hat{\varrho}_0+\left(\frac{\lambda_1e^{\lambda_1t}-\lambda_2e^{\lambda_2t}}{\lambda_1-\lambda_2}\right)\tilde{\xi}\tilde{\xi}\cdot\hat{\mathbf{u}}_0-e^{-\bar{\mu}\vert \xi_h\vert^2t}\tilde{\xi}\times(\tilde{\xi}\times \hat{\mathbf{u}}_0)
\end{align}
for any $ \xi \in \mathbb{R}^{3} $ with $ \xi \neq \mathbf{0} $.

To sum up, from \eqref{rhohat}, \eqref{uhat} and \eqref{u-formula-fourie}, we have
\begin{proposition}\label{prop-Green}
Let $ U=(\varrho, \mathbf{u}) $ satisfy the system \eqref{new-linear} supplemented with the initial data in \eqref{initial-linear}. Assume that the solution is written in the form of
\begin{gather}\label{GREEN-REPRE}
 \displaystyle U=G \ast U   _{0}
 \end{gather}
 for the Green function $ G=G(t,x) $. Then $ G $ is defined by
\begin{equation}\label{Gree-defi}
\hat{G}(t,\xi)=\left(\begin{matrix}
\hat{G}_{11}&\hat{G}_{12}\\[2mm]
\hat{G}_{21}&\hat{G}_{22}
\end{matrix}
\right),
\end{equation}
with
\begin{gather}
\displaystyle  \hat{G}_{11}=\frac{\lambda_1e^{\lambda_2t}-\lambda_2e^{\lambda_1t}}{\lambda_1-\lambda_2}, \ \ \hat{G}_{12}=-\bar{\rho}i\frac{e^{\lambda_1t}-e^{\lambda_2t}}{\lambda_1-\lambda_2}\xi, \label{G-11-12}\\[2mm]
\displaystyle \hat{G}_{21}=\frac{\lambda_1\lambda_2\xi}{\bar{\rho}i|\xi|^2}\frac{e^{\lambda_1t}-e^{\lambda_2t}}{\lambda_1-\lambda_2},\ \
\hat{G}_{22}=\frac{\lambda_1e^{\lambda_1t}-\lambda_2e^{\lambda_2t}}{\lambda_1-\lambda_2}\frac{\xi\otimes\xi}{|\xi|^2}+e^{-\bar{\mu}|\xi_h|^2t}\left(I_3-\frac{\xi\otimes\xi}{|\xi|^2}\right)\label{G-12-22}
\end{gather}
for $ \xi \neq \mathbf{0} $. Furthermore, we have
\begin{align}\label{divu-repres}
 \displaystyle \widehat{\mathop{\mathrm{div}}\nolimits \mathbf{u}}&=
 \frac{\lambda_1\lambda_2}{\bar{\rho}}\frac{e^{\lambda_1t}-e^{\lambda_2t}}{\lambda_1-\lambda_2}\hat{\varrho}_0+\frac{\lambda_1e^{\lambda_1t}-\lambda_2e^{\lambda_2t}}{\lambda_1-\lambda_2}\widehat{\mathop{\mathrm{div}}\nolimits \mathbf{u}_{0}}
  \nonumber \\
  & \displaystyle =:\hat{G}_{1}^{\ast}\hat{\varrho}_{0}+\hat{G}_{2}^{\ast}\widehat{\mathop{\mathrm{div}}\nolimits\mathbf{u }_{0}}=(\hat{G}_{1}^{\ast}\ \ \hat{G}_{2}^{\ast} )\left(\begin{matrix}
\hat{\varrho} _{0}\\[2mm]
\widehat{\mathop{\mathrm{div}}\nolimits\mathbf{u}_{0}}
\end{matrix}
\right)=:\hat{G}^{\ast}\left(\begin{matrix}
\hat{\varrho} _{0}\\[2mm]
\widehat{\mathop{\mathrm{div}}\nolimits\mathbf{u}_{0}}
\end{matrix}
\right)
 \end{align}
 for any $ \xi \neq \mathbf{0} $.
\end{proposition}
\begin{remark}
Clearly, it follows from \eqref{GREEN-REPRE} that
\begin{align}\label{ysf}
\displaystyle \left(\begin{matrix}
\hat{\varrho}(t,\xi)\\[2mm]
\hat{\mathbf{u}}(t,\xi)
\end{matrix}\right)=\hat{G}(t,\xi)\left(\begin{matrix}
\hat{\varrho}(0,\xi)\\[2mm]
\hat{\mathbf{u}}(0,\xi)
\end{matrix}
\right)=:\hat{G}(t,\xi)\left(\begin{matrix}
\hat{\varrho} _{0}\\[2mm]
\hat{\mathbf{u}}_{0}
\end{matrix}
\right),
\end{align}
where $ \hat{G}(t,\xi) $ is as in \eqref{Gree-defi}.
\end{remark}

\vspace{3mm}

\subsection {Time-decay property} In this subsection, we shall utilize \eqref{ysf} to
obtain some the $L^2$ time-decay properties for $U =(\varrho,\mathbf{u})$. To this attempt, we first derive some pointwise estimates on the Green function $ G(t,x) $. Notice that the decay properties of the solutions can be derived from the ones of their Fourier transforms. So we here mainly focus on the estimates on the Fourier transform of $ G $. Motivated by \cite{LWZ-2022}, we shall carry out the analysis based on a delicate decomposition of the frequency space $ \mathbb{R}^{3} $ into two subdomains. Specifically we have
\begin{proposition}\label{pro1}
Let $\mathbb{R}^3=\mathcal{A}_1\cup \mathcal{A}_2$ with
\begin{gather*}
\mathcal{A}_1=\left\{\xi\in \mathbb{R}^3:\Gamma\leq \frac{(\bar{\mu}|\xi_h|^2+(\bar{\mu}+\bar{\zeta})|\xi|^2)^2}{4}~or~ 3(\bar{\mu}\vert \xi_h\vert^2+(\bar{\mu}+\bar{\zeta})\vert \xi\vert^2)^2\leq 16\gamma\bar{\rho}^2 \vert \xi\vert^2\right\},\\[2mm]
\mathcal{A}_2=\left\{\xi\in \mathbb{R}^3:\Gamma>\frac{(\bar{\mu}\vert \xi_h\vert^2+(\bar{\mu}+\bar{\zeta})\vert \xi\vert^2)^2}{4}~or~3(\bar{\mu}\vert \xi_h\vert^2+(\bar{\mu}+\bar{\zeta}) \vert \xi\vert^2)^2> 16\gamma\bar{\rho}^2 \vert \xi\vert^2 \right\},
\end{gather*}
where $\Gamma=(\bar{\mu}|\xi_h|^2+(\bar{\mu}+\bar{\zeta})|\xi|^2)^2-4\gamma\bar{\rho}^2|\xi|^2$.
Then it holds that
\begin{gather}
\mathop{\mathrm{Re}}\lambda_1\leq -\frac{\bar{\mu}|\xi_h|^2+(\bar{\mu}+\bar{\zeta})\vert \xi\vert^2}{2},\ \ \  \mathop{\mathrm{Re}}\lambda_2\leq -\frac{\bar{\mu}\vert \xi_h\vert^2+(\bar{\mu}+\bar{\zeta})\vert \xi\vert^2}{4},
 \nonumber \\[2mm]
\displaystyle\vert \hat{G}_{11}\vert \lesssim  e^{-\vartheta \vert \xi\vert^2t},\ \ \  \vert \hat{G}_{12}\vert \lesssim e^{-\vartheta \vert \xi\vert^2t},\label{G-11-12-esti}\\[2mm]
\vert \hat{G}_{21}\vert\lesssim e^{-\vartheta \vert \xi\vert^2t}, \ \  \ \vert \hat{G}_{22}\vert \lesssim e^{-\vartheta \vert \xi\vert^2t}+Ce^{-\bar{\mu}\vert \xi_h\vert^2t}, \label{G-21-22-esti}\\[2mm]
\displaystyle\vert\hat{G}^*_1\vert \lesssim  e^{-\vartheta \vert \xi\vert^2t},\ \ \  \vert \hat{G}^*_2\vert \lesssim e^{-\vartheta \vert \xi\vert^2t}
 \nonumber
\end{gather}
 for any $\xi\in \mathcal{A}_1$, and that
\begin{gather*}
\lambda_1<-\frac{3(\bar{\mu}\vert \xi_h\vert^2+(\bar{\mu}+\bar{\zeta})\vert \xi\vert^2)}{4}, \ \ \ \lambda_2\leq -\frac{\gamma\bar{\rho}^2}{2\bar{\mu}+\bar{\zeta}},\\[2mm]
\vert \hat{G}_{11}\vert \lesssim e^{-\vartheta \vert \xi\vert^2t}+e^{- \eta t},\ \ \  \vert \hat{G}_{12}\vert \lesssim e^{-\vartheta \vert \xi\vert^2t}+e^{- \eta t},\\[2mm]
\vert \hat{G}_{21}\vert\lesssim e^{-\vartheta \vert \xi\vert^2t}+ e^{- \eta t}, \ \ \
 \vert \hat{G}_{22}\vert \lesssim e^{-\vartheta|\xi|^2t}+e^{-\bar{\mu}\vert \xi_h\vert^2t}+e^{- \eta t},\\[2mm]
 \vert \hat{G}_1^{\ast}\vert\lesssim e^{-\vartheta \vert \xi\vert^2t}+ e^{- \eta t}, \ \ \
 \vert \hat{G}^{\ast}_2\vert \lesssim e^{-\vartheta|\xi|^2t}+e^{- \eta t}
\end{gather*}
for any $\xi\in \mathcal{A}_2$. Here $ \vartheta = \frac{\bar{\mu}+\bar{\zeta}}{8} $, $\eta=\frac{\gamma\bar{\rho}^2}{2\bar{\mu}+\bar{\zeta}}$.
\end{proposition}
\begin{proof}
For $\xi\in \mathcal{A}_1$, we have $\Gamma\leq \frac{(\bar{\mu}\vert \xi_h\vert^2+(\bar{\mu}+\bar{\zeta})\vert \xi\vert^2)^2}{4}$. Recalling the representations of $ \lambda _{1} $ and $ \lambda _{2} $ in \eqref{eig-values}, we get, thanks to the mean-value theorem, that
\begin{gather}
-\frac{3(\bar{\mu}|\xi_h|^2+(\bar{\mu}+\bar{\zeta})|\xi|^2)}{4}\leq \mathop{\mathrm{Re}}\lambda_1\leq -\frac{\bar{\mu}|\xi_h|^2+(\bar{\mu}+\bar{\zeta})|\xi|^2}{2},\label{Re-lam1}\\[2mm]
\mathop{\mathrm{Re}}\lambda_2\leq -\frac{\bar{\mu}|\xi_h|^2+(\bar{\mu}+\bar{\zeta})|\xi|^2}{4}.\label{Re-lam2}
\end{gather}
Then we have
\begin{align}\label{teshu}
\left|\frac{e^{\lambda_1t}-e^{\lambda_2t}}{\lambda_1-\lambda_2}\right|&\leq t\sup_{0\leq s\leq 1}e^{t \mathop{\mathrm{Re}}(s\lambda_2+(1-s)\lambda_1)}\leq te^{-\frac{\bar{\mu}\vert \xi_h\vert^2+(\bar{\mu}+\bar{\zeta})\vert \xi\vert^2}{4}t}.
\end{align}
Let us proceed to the estimates of $\hat{G}_{11}$ and $\hat{G}_{22}$. We will divide the proof into two subcases: $\mathop{\mathrm{Im}}\lambda_1=0$ and $\mathop{\mathrm{Im}}\lambda_1\neq 0$. For the case $\mathop{\mathrm{Im}}\lambda_1=0$, we derive for $\hat{G}_{11}$ that
\begin{align*}
\displaystyle\vert \hat{G}_{11}(\xi)\vert& \leq \left \vert \frac{\lambda_1e^{\lambda_2t}-\lambda_2e^{\lambda_1t}}{\lambda_1-\lambda_2}\right\vert\leq \left \vert \lambda_1\frac{e^{\lambda_1t}-e^{\lambda_2t}}{\lambda_1-\lambda_2}\right\vert+  e^{\lambda_1t}\\[2mm]
\displaystyle&\leq\frac{3(\bar{\mu}\vert \xi_h\vert^2+(\bar{\mu}+\bar{\zeta})\vert \xi\vert^2)}{4}te^{-\frac{(\bar{\mu}\vert \xi_h\vert^2+(\bar{\mu}+\bar{\zeta})\vert \xi\vert^2)}{4}t}+e^{-\frac{\bar{\mu}|\xi_h|^2+(\bar{\mu}+\bar{\zeta})\vert \xi\vert^2}{2}t}\\[2mm]
\displaystyle& \lesssim e^{- \frac{\bar{\mu}+\bar{\zeta}}{8} |\xi|^2t},
\end{align*}
where we have used \eqref{Re-lam1}, \eqref{teshu} and the fact $xe^{-x}\leq C$ for $x\geq 0$.
If $\mathop{\mathrm{Im}}\lambda_1\neq 0$, namely $\Gamma<0$, then it holds that
\begin{equation*}
|\lambda_1|^2=\gamma\bar{\rho}^2|\xi|^2, \ \ \Gamma=(\bar{\mu}|\xi_h|^2+(\bar{\mu}+\bar{\zeta})|\xi|^2)^2-4|\lambda_1|^2.
\end{equation*}
To proceed, we consider two subcases: $\vert \lambda _{1}\vert \leq |\sqrt{\Gamma}|$ and $|\lambda_1|\geq|\sqrt{\Gamma}|$. When $|\lambda_1|\leq |\sqrt{\Gamma}|$, recalling the definition of $\hat{G}_{11}$ in \eqref{G-11-12}, we obtain
\begin{align*}
\vert \hat{G}_{11}(t)\vert&\leq \left \vert \lambda_1\frac{e^{\lambda_1t}-e^{\lambda_2t}}{\lambda_1-\lambda_2}\right\vert+e^{t\mathop{\mathrm{Re}}\lambda_1}=\frac{|\lambda_1|}{|\sqrt{\Gamma}|}|e^{\lambda_1t}-e^{\lambda_2t}|+e ^{t\mathop{\mathrm{Re}}\lambda _{1}}
  \nonumber \\[2mm]
  & \displaystyle \leq \vert e ^{\lambda _{1}t}\vert+ \vert e ^{\lambda _{2}t}\vert+e ^{t\mathop{\mathrm{Re}}\lambda _{1}}\leq 2 e ^{t\mathop{\mathrm{Re}}\lambda _{1}}+e ^{t\mathop{\mathrm{Re}}\lambda _{2}}
   \nonumber \\[2mm]
   &\lesssim e^{- \frac{\bar{\mu}+\bar{\zeta}}{4}|\xi|^2t},
\end{align*}
where we have used \eqref{Re-lam1} and \eqref{Re-lam2}. When $|\lambda_1|\geq |\sqrt{\Gamma}|=\sqrt{-\Gamma}$, we have
$$|\lambda_1|^2\geq 4|\lambda_1|^2-(\bar{\mu}|\xi_h|^2+(\bar{\mu}+\bar{\zeta})|\xi|^2)^2.$$
That is,
$$\sqrt{3}|\lambda_1|\leq\bar{\mu}|\xi_h|^2+(\bar{\mu}+\bar{\zeta})|\xi|^2.$$
This along with \eqref{Re-lam1} yields that
\begin{align*}
|\hat{G}_{11}|&\leq \left  \vert\lambda_1\frac{e^{\lambda_1t}-e^{\lambda_2t}}{\lambda_1-\lambda_2}\right\vert +e^{t\mathop{\mathrm{Re}}\lambda_1}\\[1mm]
&\leq \frac{1}{\sqrt{3}}(\bar{\mu}|\xi_h|^2+(\bar{\mu}+\bar{\zeta})|\xi|^2)\left \vert \frac{e^{\lambda_1t}-e^{\lambda_2t}}{\lambda_1-\lambda_2}\right\vert++e^{-\frac{(\bar{\mu}|\xi_h|^2+(\bar{\mu}+\bar{\zeta})|\xi|^2)}{2}t} \\[1mm]
&\displaystyle \lesssim e^{- \frac{\bar{\mu}+\bar{\zeta}}{8}|\xi|^2t}.
\end{align*}
The estimate of $ \hat{G}_{11} $ in \eqref{G-11-12-esti} is proved. Similarly, we derive for $\hat{G}_{22}$ and $ \hat{G}_{2}^{\ast} $ that
\begin{equation*}
|\hat{G}_{22}| \lesssim e^{ t \mathop{\mathrm{Re}}\lambda_2}+ \left \vert \lambda_1\frac{e^{\lambda_2t}-e^{\lambda_1t}}{\lambda_2-\lambda_1}\right \vert+e^{-\bar{\mu}|\xi_h|^2t} \lesssim e^{- \frac{\bar{\mu}+\bar{\zeta}}{8}|\xi|^2t}+e^{-\bar{\mu}|\xi_h|^2t}
\end{equation*}
and
\begin{equation*}
\hat{G}_2^{\ast}\lesssim e^{- \frac{\bar{\mu}+\bar{\zeta}}{8}|\xi|^2t}
\end{equation*}
for any $ \xi \in \mathcal{A}_{1} $. Next we shall bound $\hat{G}_{12}$ and $\hat{G}_{21}$. We divide the consideration into two subcases: $\bar{\rho}|\xi|\leq |\sqrt{\Gamma}|$ and $\bar{\rho}|\xi|\geq|\sqrt{\Gamma}|$. In the case when $\bar{\rho}|\xi|\leq |\sqrt{\Gamma}|$, by the definition of $\hat{G}_{12}$ in \eqref{G-11-12}, we obtain
\begin{equation*}
|\hat{G}_{12}|=\bar{\rho}\frac{|\xi|}{\sqrt{\Gamma}}|e^{\lambda_1t}-e^{\lambda_2t}| \lesssim        {\mathop{\mathrm{e}}}^{t \mathop{\mathrm{Re}}\lambda _{1}}+e ^{t\mathop{\mathrm{Re}}\lambda _{2}}  \lesssim e^{- \frac{\bar{\mu}+\bar{\zeta}}{4} \vert \xi\vert^2t}.
\end{equation*}
In the case $\bar{\rho}|\xi|\geq |\sqrt{\Gamma}|$, it holds that
\begin{gather*}
\displaystyle\left\vert (\bar{\mu}|\xi_h|^2+(\bar{\mu}+\bar{\zeta})|\xi|^2)^2-4\gamma\bar{\rho}^2 \vert \xi\vert^2 \right\vert\leq \bar{\rho}^2 \vert \xi\vert^2,
\end{gather*}
which leads to
\begin{align*}
\displaystyle  - \bar{\rho} ^{2}\vert \xi\vert ^{2} \leq \vert (\bar{\mu}|\xi_h|^2+(\bar{\mu}+\bar{\zeta})|\xi|^2)^2-4\gamma\bar{\rho}^2 \vert \xi\vert^2 \leq \bar{\rho} ^{2}\vert \xi\vert ^{2}.
\end{align*}
Then it follows that
\begin{equation*}
\sqrt{4\gamma-1}\bar{\rho}|\xi|\leq \bar{\mu}|\xi_h|^2+(\bar{\mu}+\bar{\zeta})|\xi|^2.
\end{equation*}
This along with \eqref{G-11-12} and \eqref{teshu} gives
\begin{align*}
\displaystyle\vert \hat{G}_{12}\vert &\leq \bar{\rho}\vert \xi\vert \left \vert \frac{e^{\lambda_1t}-e^{\lambda_2t}}{\lambda_1-\lambda_2}\right\vert
 \nonumber \\[1mm]
 & \displaystyle\leq \frac{1}{\sqrt{4\gamma-1}}(\bar{\mu}|\xi_h|^2+(\bar{\mu}+\bar{\zeta})|\xi|^2)te^{-\frac{\bar{\mu}\vert \xi_h\vert^2+(\bar{\mu}+\bar{\zeta})\vert \xi\vert^2}{4}t}
  \nonumber \\[1mm]
  & \displaystyle \lesssim e^{- \frac{\bar{\mu}+\bar{\zeta}}{8} |\xi|^2t}
\end{align*}
for any $ \xi \in \mathcal{A}_{1}. $
For $\hat{G}_{21}$, notice that $ \lambda_1\lambda_2=\gamma\bar{\rho}^2|\xi|^2 $. Then we can rewrite $\hat{G}_{21}$ as
\begin{equation*}
\displaystyle\vert \hat{G}_{21}\vert=\left \vert \frac{\lambda_1\lambda_2\xi}{\bar{\rho}|\xi|^2}\frac{e^{\lambda_1t}-e^{\lambda_2t}}{\lambda_1-\lambda_2}\right\vert=\gamma|\hat{G}_{12}|.
\end{equation*}
Therefore it holds that
\begin{align*}
\displaystyle \vert \hat{G}_{21}\vert \lesssim e^{- \frac{\bar{\mu}+\bar{\zeta}}{8} |\xi|^2t}.
\end{align*}
Similarly, we derive for $\hat{G}^{\ast}_{1}$ that
\begin{align*}
\hat{G}^{\ast}_1&=\frac{\lambda_1\lambda_2}{\bar{\rho}}\frac{e^{\lambda_1t}-e^{\lambda_2t}}{\lambda_1-\lambda_2}\leq\gamma\bar{\rho}|\xi|^2te^{-\frac{\bar{\mu}\vert \xi_h\vert^2+(\bar{\mu}+\bar{\zeta})\vert \xi\vert^2}{4}t}\\
&\lesssim e^{- \frac{\bar{\mu}+\bar{\zeta}}{8} |\xi|^2t}
\end{align*}
for any $ \xi \in \mathcal{A}_{1}.$

For any $\xi\in \mathcal{A}_2$, we have $\frac{\bar{\mu}|\xi_h|^2+(\bar{\mu}+\bar{\zeta})|\xi|^2}{2}<\sqrt{\Gamma}\leq \bar{\mu}|\xi_h|^2+(\bar{\mu}+\bar{\zeta})|\xi|^2$. It then follows that
\begin{gather}\label{lam-1-A2}
-(\bar{\mu}|\xi_h|^2+(\bar{\mu}+\bar{\zeta})|\xi|^2)\leq \lambda_1<-\frac{3}{4}(\bar{\mu}|\xi_h|^2+(\bar{\mu}+\bar{\zeta})|\xi|^2)
\end{gather}
and
\begin{gather}
\displaystyle \lambda_2=\frac{\Gamma-(\bar{\mu}|\xi_h|^2+(\bar{\mu}+\bar{\zeta})|\xi|^2)^2}{2(\bar{\mu}|\xi_h|^2+(\bar{\mu}+\bar{\zeta})|\xi|^2+\sqrt{\Gamma})}\leq -\frac{\gamma\bar{\rho}^2|\xi|^2}{\bar{\mu}|\xi_h|^2+(\bar{\mu}+\bar{\zeta})|\xi|^2}\leq -\frac{\gamma\bar{\rho}^2}{2\bar{\mu}+\bar{\zeta}}. \label{lam2-A2}
\end{gather}
Therefore we get
\begin{align}\label{A-2-struc}
\left \vert \frac{e^{\lambda_1t}-e^{\lambda_2t}}{\lambda_1-\lambda_2}\right\vert&\leq \frac{1}{|\lambda_1-\lambda_2|}(e^{\lambda_1t}+e^{\lambda_2t}) = \frac{1}{\sqrt{\Gamma}}(e^{\lambda_1t}+e^{\lambda_2t})
 \nonumber \\[2mm]
  & \displaystyle \leq\frac{2}{\bar{\mu}|\xi_h|^2+(\bar{\mu}+\bar{\zeta})|\xi|^2}(e^{-\frac{3}{4}(\bar{\mu}|\xi_h|^2+(\bar{\mu}+\bar{\zeta})|\xi|^2)t}+e^{-\frac{\gamma\bar{\rho}^2}{2\bar{\mu}+\bar{\zeta}}t}).
\end{align}
This combined with \eqref{G-11-12}, \eqref{G-12-22}, \eqref{lam-1-A2} and \eqref{lam2-A2} further implies that
\begin{align*}
|\hat{G}_{11}|&\leq \left \vert \lambda_1\frac{e^{\lambda_1t}-e^{\lambda_2t}}{\lambda_1-\lambda_2}\right\vert+e^{\lambda_1t}\\[1mm]
& \lesssim(\bar{\mu}|\xi_h|^2+(\bar{\mu}+\bar{\zeta})|\xi|^2) \left \vert \frac{e^{\lambda_1t}-e^{\lambda_2t}}{\lambda_1-\lambda_2}\right\vert+e^{-\frac{3}{4}(\bar{\mu}|\xi_h|^2+(\bar{\mu}+\bar{\zeta})|\xi|^2)t}\\[1mm]
&\displaystyle \lesssim(e^{-\frac{3(\bar{\mu}+\bar{\zeta})}{4}\vert \xi\vert ^{2} t}+e^{-\frac{\gamma\bar{\rho}^2}{2\bar{\mu}+\bar{\zeta}}t}),
\end{align*}
and that
\begin{align*}
|\hat{G}_{22}|& \lesssim e^{\lambda_2t}+ \left \vert \lambda_1\frac{e^{\lambda_2t}-e^{\lambda_1t}}{\lambda_2-\lambda_1}\right\vert+e^{-\bar{\mu}|\xi_h|^2t}\\[2mm]
&\lesssim  e^{-\frac{3(\bar{\mu}+\bar{\zeta})}{4}\vert \xi\vert ^{2} t}+e^{-\frac{\gamma\bar{\rho}^2}{2\bar{\mu}+\bar{\zeta}}t}+e^{-\bar{\mu}|\xi_h|^2t},\\[2mm]
|\hat{G}^{\ast}_2|&\lesssim e^{-\frac{3(\bar{\mu}+\bar{\zeta})}{4}\vert \xi\vert ^{2} t}+e^{-\frac{\gamma\bar{\rho}^2}{2\bar{\mu}+\bar{\zeta}}t}.
\end{align*}
Due to $\sqrt{\Gamma}>\frac{\bar{\mu}|\xi_h|^2+(\bar{\mu}+\bar{\zeta})|\xi|^2}{2}$, we find that
\begin{equation*}
\frac{3}{4}(\bar{\mu}|\xi_h|^2+(\bar{\mu}+\bar{\zeta})|\xi|^2)^2>4\gamma\bar{\rho}^2 \vert \xi\vert^2.
\end{equation*}
Hence, we get from \eqref{G-11-12}
and \eqref{A-2-struc} that
\begin{align*}
|\hat{G}_{12}|&\lesssim (\bar{\mu}|\xi_h|^2+(\bar{\mu}+\bar{\zeta})|\xi|^2) \left \vert \frac{e^{\lambda_1t}-e^{\lambda_2t}}{\lambda_1-\lambda_2}\right\vert\\[2mm]
&\displaystyle\lesssim e^{-\frac{3(\bar{\mu}+\bar{\zeta})}{4}\vert \xi\vert ^{2} t}+e^{-\frac{\gamma\bar{\rho}^2}{2\bar{\mu}+\bar{\zeta}}t}.
\end{align*}
For $\hat{G}_{21}$ in \eqref{G-12-22}, we get
\begin{align*}
|\hat{G}_{21}|&=\left \vert \frac{\lambda_1\lambda_2\xi}{\bar{\rho}|\xi|^2}\frac{e^{\lambda_1t}-e^{\lambda_2t}}{\lambda_1-\lambda_2}\right\vert=\gamma|\hat{G}_{12}|\\[1mm]
&\displaystyle \lesssim e^{-\frac{3(\bar{\mu}+\bar{\zeta})}{4}\vert \xi\vert ^{2} t}+e^{-\frac{\gamma\bar{\rho}^2}{2\bar{\mu}+\bar{\zeta}}t}.
\end{align*}
Finally, for $\hat{G}^{\ast}_1$, similarly, we get
\begin{align*}
\hat{G}^{\ast}_1\lesssim e^{-\frac{3(\bar{\mu}+\bar{\zeta})}{4}\vert \xi\vert ^{2} t}+e^{-\frac{\gamma\bar{\rho}^2}{2\bar{\mu}+\bar{\zeta}}t}.
\end{align*}
This completes the proof of Proposition \ref{pro1}.
\end{proof}

Combining Propositions \ref{prop-Green} and \ref{pro1}, we have the following pointwise estimates of $ (\hat{\varrho},\hat{\mathbf{u}}) $.
\begin{lemma}\label{lemma4.3}
Let $ U=(\varrho, \mathbf{u}) $ satisfy the the system \eqref{new-linear} supplemented with the initial data in \eqref{initial-linear}. Then it holds for any $t \geq 0$ and $ \xi\in \mathcal{A}_1$ that
\begin{gather}
|\hat{\varrho}(t, \xi)| \lesssim e^{-\vartheta|\xi|^2t}|(|\hat{\varrho}_0|+|\hat{\mathbf{u}}_0|),\label{rho}\\[2mm]
|\hat{\mathbf{u}}(t, \xi)|\lesssim e^{-\vartheta|\xi|^2t}(|\hat{\varrho}_0|+|\hat{\mathbf{u}}_0|)+e^{-\bar{\mu}|\xi_h|^2t}|\hat{\mathbf{u}}_0|,
 \nonumber
 \\[2mm]
|\widehat{\mathop{\mathrm{div}}\nolimits \mathbf{u}}(t, \xi)|\lesssim e^{-\vartheta|\xi|^2t} (\vert \xi\vert|\hat{\varrho}_0|+ \vert \widehat{\mathop{\mathrm{div}}\nolimits \mathbf{u}_{0}}\vert),\label{divu0}
\end{gather}
and for any $t \geq 0$ and $ \xi\in \mathcal{A}_2$ that
\begin{gather}
|\hat{\varrho}(t, \xi)|\lesssim (e^{-\vartheta|\xi|^2t}+e^{- \eta t})(|\hat{\varrho}_0|+|\hat{\mathbf{u}}_0|),\label{rho1}\\[2mm]
|\hat{\mathbf{u}}(t, \xi)|\lesssim(e^{-\vartheta|\xi|^2t}+e^{- \eta t})(|\hat{\varrho}_0|+|\hat{\mathbf{u}}_0|)+Ce^{-\bar{\mu}|\xi_h|^2t}|\hat{\mathbf{u}}_0|,
 \nonumber \\[2mm]
|\widehat{\mathop{\mathrm{div}}\nolimits \mathbf{u}}(t, \xi)|\lesssim (e^{-\vartheta|\xi|^2t}+e^{-\eta t}) (\vert \xi\vert|\hat{\varrho}_0|+ \vert \widehat{ \mathop{\mathrm{div}}\nolimits \mathbf{u}_{0}}\vert).\label{divu1}
\end{gather}
 Here $ \vartheta = \frac{\bar{\mu}+\bar{\zeta}}{8} $, $\eta=\frac{\gamma\bar{\rho}^2}{2\bar{\mu}+\bar{\zeta}}$.

\end{lemma}
\begin{proof}
We will only prove the results for the case $ \xi \in \mathcal{A}_{1} $, the ones for the case $ \xi \in \mathcal{A}_{2} $ can be proved by similar arguments. Let us begin with the estimates of $ \hat{\varrho} $. Indeed, for any $ \xi \in \mathcal{A}_{1} $, it follows directly from \eqref{GREEN-REPRE} and \eqref{G-11-12-esti} that
\begin{align*}
|\hat{\varrho}(t,\xi)|&=|\hat{G}_{11}\hat{\rho}_0+\hat{G}_{12}\hat{\mathbf{u}}_0| \leq |\hat{G}_{11}||\hat{\rho}_0|+|\hat{G}_{12}||\hat{\mathbf{u}}_0|\\[2mm]
& \lesssim e^{-\vartheta|\xi|^2t}(|\hat{\rho}_0|+|\hat{\mathbf{u}}_0|).
\end{align*}
Similarly, we get from \eqref{GREEN-REPRE} and \eqref{G-21-22-esti} that
\begin{align*}
|\hat{\mathbf{u}}(t,\xi)|&=|\hat{G}_{21}\hat{\rho}_0+\hat{G}_{22}\hat{\mathbf{u}}_0|\leq |\hat{G}_{21}||\hat{\rho}_0|+|\hat{G}_{22}||\hat{\mathbf{u}}_0|\\[2mm]
&\lesssim e^{-\vartheta|\xi|^2t}(|\hat{\rho}_0|+|\hat{\mathbf{u}}_0|)+e^{-\bar{\mu}|\xi_h|^2t}|\hat{\mathbf{u}}_0|
\end{align*}
for any $\xi\in \mathcal{A}_1$. As for the estimate of $ \widehat{\mathop{\mathrm{div}}\nolimits \mathbf{u}} $, we first utilize \eqref{G-11-12}, \eqref{divu-repres} and the fact $ \lambda _{1}\lambda _{2}=\gamma \bar{\rho}^{2}\vert \xi\vert ^{2} $ to get
\begin{align*}
\displaystyle  \vert \widehat{\mathop{\mathrm{div}}\nolimits \mathbf{u} }\vert \lesssim \vert \hat{G}_{12}\vert \vert \xi\hat{\varrho}_{0}\vert+\vert \hat{G}_{2}^{\ast}\vert \vert \widehat{\mathop{\mathrm{div}}\nolimits \mathbf{u}_{0}}\vert
\end{align*}
for any $ \xi\neq \mathbf{0} $. This along with \eqref{G-11-12-esti} leads to
\begin{align*}
\displaystyle \vert \widehat{\mathop{\mathrm{div}}\nolimits \mathbf{u}}\vert \lesssim e^{-\vartheta|\xi|^2t}(|\xi\hat{\varrho}_0|+\vert\widehat{\mathop{\mathrm{div}}\nolimits \mathbf{u}_{0}}\vert).
\end{align*}
This gives \eqref{divu0}, and ends the proof of Lemma \ref{lemma4.3}.
\end{proof}

For later use, we denote
\begin{align}\label{wt-U}
\displaystyle \widetilde{\mathbf{u}}_{0}=\frac{1}{\sqrt{2 \pi}} \int _{\mathbb{R}}e ^{- ix _{3}\xi _{3}} \mathbf{u}_{0} (x _{1},x _{2},x _{3})\mathrm{d}x _{3},
\end{align}
and define for any integer $ \ell \geq0 $,
\begin{align}\label{power-defi}
\displaystyle  \left[ \ell(\frac{1}{2}- \frac{1}{q}) \right] _{+}=\begin{cases}
    0 ,& q=2,\\[2mm]
    \left[ \ell(\frac{1}{2}- \frac{1}{q}) \right]+1,& \mbox{otherwise},
\end{cases}
\end{align}
where $ \left[ \;\cdot\; \right]  $ denotes the integer part of the nonnegative argument. Now we are ready to establish the decay estimates of the solution to the linearized homogeneous problem \eqref{new-linear}, \eqref{initial-linear}.

\begin{lemma}\label{xxlemma}
Let $2\leq q\leq \infty$, and let $m, n\geq 0$ be integers. Let $ U=(\varrho, \mathbf{u}) $ be a smooth solution to the problem \eqref{new-linear}, \eqref{initial-linear} satisfying \eqref{GREEN-REPRE} and \eqref{divu-repres}. Then we have the following decay estimates:
\begin{align}
\|\nabla^m\varrho\|_{L^q}&\lesssim (1+t)^{- \frac{m}{2}- \frac{3}{2}(1- \frac{1}{q}) }\|U_0\|_{L^1}+e^{- \eta t}\|\nabla^{m+[3(\frac{1}{2}-\frac{1}{q})]_+}U_0\|_{L^2},\label{vrho-con-line-lem}\\[2mm]
\displaystyle\|\nabla^m \mathop{\mathrm{div}}\nolimits \mathbf{u}\|_{L^q}&\lesssim (1+t)^{- \frac{m+1}{2}- \frac{3}{2}(1- \frac{1}{q}) }\|U_0\|_{L^1}+e^{- \eta t}\|\nabla^{m+1+[3(\frac{1}{2}-\frac{1}{q})]_+}U_0\|_{L^2},
\label{divu-line-lem} \\[2mm]
\displaystyle\|\nabla^m \mathop{\mathrm{div}}\nolimits \mathbf{u}\|_{L^q}&\lesssim (1+t)^{- \frac{m+1}{2}- \frac{3}{2}(1- \frac{1}{q}) }\|U_0\|_{L^1}+e^{- \eta t}\|\nabla^{m+1+[3(\frac{1}{2}-\frac{1}{q})]_+}\rho_0\|_{L^2}
 \nonumber \\[2mm]
& \displaystyle \quad+e^{- \eta t}\|\nabla^{m+[3(\frac{1}{2}-\frac{1}{q})]_+}\mathop{\mathrm{div}}\nolimits\mathbf{u}_0\|_{L^2},
\label{divu-line-lem2}
  \\[2mm]
\displaystyle\|\nabla^m_h \partial^n_3 \mathbf{u}\|_{L^q}&\lesssim (1+t)^{- \frac{m+n}{2}- \frac{3}{2} (1- \frac{1}{q})}\|\varrho_{0}\|_{L ^{1}}+e ^{-\eta t}\|\nabla^{m+n+[3(\frac{1}{2}-\frac{1}{q})]_+}\varrho_0\|_{L ^2}
 \nonumber \\[2mm]
 & \displaystyle \quad+ (1+t)^{-\frac{m+2}{2}+\frac{1}{q}}\|\xi _{3}^{n}\widetilde{\mathbf{u}}_{0}\|_{L _{\xi _{3}}^{q'}L _{x _{1}x _{2}}^{1}}+e ^{-\eta t}\|\nabla^{m+[2(\frac{1}{2}-\frac{1}{q})]_+}\xi _{3}^{n}\widetilde{\mathbf{u}}_0\|_{L _{\xi _{3}}^{q'} L _{x _{1}x _{2}}^2},\label{anisou-lem-con}\\[2mm]
\|\nabla^m_h \mathbf{u}\|_{L^q}&\lesssim (1+t)^{- \frac{m}{2}- \frac{3}{2} (1- \frac{1}{q})}\|\varrho_{0}\|_{L ^{1}}+e ^{-\eta t}\|\nabla^{m+[3(\frac{1}{2}-\frac{1}{q})]_+}\varrho_0\|_{L ^2}
 \nonumber \\
 & \displaystyle \quad+ (1+t)^{-\frac{m+2}{2}+\frac{1}{q}}\|\widetilde{\mathbf{u}}_{0}\|_{L^{q'}_{\xi_3}L^1_{x_1x_2}}+e ^{-\eta t}\|\nabla^{m+[2(\frac{1}{2}-\frac{1}{q})]_+}\widetilde{\mathbf{u}}_0\|_{L _{\xi _{3}}^{q'} L _{x _{1}x _{2}}^2},\label{hori-u-lem-con}\\[2mm]
\|\partial^m_3 \mathbf{u}\|_{L^q}&\lesssim (1+t)^{- \frac{m}{2}- \frac{3}{2} (1- \frac{1}{q})}\|\varrho_{0}\|_{L ^{1}}+e ^{-\eta t}\|\nabla^{m+[3(\frac{1}{2}-\frac{1}{q})]_+}\varrho_0\|_{L ^2}
 \nonumber \\
 & \displaystyle \quad+ (1+t)^{- \frac{1}{q'}}\|\xi_3^m \widetilde{\mathbf{u}}_{0}\|_{L^{q'}_{\xi_3}L^{1}_{x_1x_2}}+e ^{-\eta t}\|\nabla^{[2(\frac{1}{2}-\frac{1}{q})]_+}\xi _{3}^{m}\widetilde{\mathbf{u}}_0\|_{L _{\xi _{3}}^{q'} L _{x _{1}x _{2}}^2}\label{verti-u-conl-lem}
\end{align}
for any $t\geq 0$, where $\frac{1}{q}+\frac{1}{q'}=1$, and $ \eta $ is as in Lemma \ref{lemma4.3}.

\end{lemma}
\begin{proof}
Take $2\leq q\leq \infty$ and let $m\geq0$ be an integer. Using \eqref{rho}, \eqref{rho1} and the
Hausdorff-Young inequality, we obtain, with $1/q'=1- 1/q $,
\begin{align}\label{vrho-esti-line}
\|\nabla^m\varrho(t)\|_{L^q}&\lesssim \||\xi|^m\hat{\varrho}\|_{L^{q'}}
 \nonumber \\[1mm]
 & \lesssim\||\xi|^m\hat{\varrho}\|_{L^{q'}(\xi\in \mathcal{A}_1)}+\||\xi|^m\hat{\varrho}\|_{L^{q'}(\xi\in \mathcal{A}_2)}
  \nonumber \\[1mm]
  & \lesssim \||\xi|^me^{-\vartheta|\xi|^2t}(\hat{\varrho}_0,\hat{\mathbf{u}}_0)\|_{L^{q'}}+e^{- \eta t}\||\xi|^m(\hat{\varrho_0},\hat{\mathbf{u}}_0)\|_{L^{q'}}
 \nonumber \\[1mm]
 & \displaystyle =:J _{1}+J _{2}.
\end{align}
For $ J _{1} $, take a constant $ K>0 $ such that $ \vartheta K ^{2}> \eta$, then we have
\begin{align}\label{J-1-split}
\displaystyle J _{1}& \leq \underbrace{\||\xi|^me^{-\vartheta|\xi|^2t}(\hat{\varrho}_0,\hat{\mathbf{u}}_0)\|_{L^{q'}(\vert \xi\vert \leq K)}} _{J _{1,1}}+ \underbrace{\||\xi|^me^{-\vartheta|\xi|^2t}(\hat{\varrho}_0,\hat{\mathbf{u}}_0)\|_{L^{q'}(\vert \xi\vert \geq K)}}_{J _{1,2}},
\end{align}
where $ J _{1,1} $ can be estimated as
\begin{align*}
\displaystyle J _{1,1}\lesssim \||\xi|^me^{-\vartheta|\xi|^2t}\| _{L ^{q'}(\vert \xi\vert \leq K)}\|(\hat{\varrho}_0,\hat{\mathbf{u}}_0)\|_{L ^{\infty}(\vert \xi\vert \leq K)} .
\end{align*}
Here the first factor can be estimated in a standard way \cite{K-1983} as
\begin{gather*}
\displaystyle \||\xi|^me^{-\vartheta|\xi|^2t}\| _{L ^{q'}(\vert \xi\vert \leq K)} \lesssim \||\xi|^me^{-\vartheta|\xi|^2(t+1)}\| _{L ^{q'}(\vert \xi\vert \leq K)} \lesssim (1+t)^{- \frac{m}{2}- \frac{3}{2}(1- \frac{1}{q}) }
\end{gather*}
by the change of variable $\xi (1+t)^{1/2}$ to $\xi$, and the second factor can be estimated as $ \|(\hat{\varrho}_0,\hat{\mathbf{u}}_0)\|_{L ^{\infty}(\vert \xi\vert \leq K)} \lesssim \|U _{0}\|_{L ^{1}} $. Therefore we get, for $ J _{1,1} $,
\begin{gather*}
\displaystyle J _{1,1}\lesssim (1+t)^{- \frac{m}{2}- \frac{3}{2}(1- \frac{1}{q}) } \|U _{0}\|_{L ^{1}}.
\end{gather*}
For $ J _{1,2} $, fix a suitably small constant $ \epsilon >0 $, we get, in virtue of the H\"older inequality for $ \frac{1}{q'}=\frac{2-q'}{2q'}+\frac{1}{2} $ with $ \frac{1}{q}+\frac{1}{q'}=1 $,
\begin{align}\label{J-12-esti}
 \displaystyle J _{1,2} &\lesssim e ^{- \vartheta K ^{2}t}\|\vert \xi\vert ^{-3(1+\epsilon)\frac{2-q'}{2q'}}\|_{L ^{\frac{2q'}{2-q'}}(\vert \xi\vert \geq K)} \|\vert \xi\vert ^{m+3(1+\epsilon)\frac{2-q'}{2q'}}(\hat{\varrho}_{0},\hat{\mathbf{u}}_{0})\|_{L ^{2}(\vert \xi\vert \geq K)}
  \nonumber \\[2mm]
  & \displaystyle \leq C _{\epsilon} e ^{- \eta t}\|\vert \xi\vert ^{m+[3(\frac{1}{2}-\frac{1}{q})]_+}\hat{U}_0\|_{L ^{2}(\vert \xi\vert \geq K)} \leq C _{\epsilon}e ^{-\eta t} \|\nabla^{m+[3(\frac{1}{2}-\frac{1}{q})]_+}U_0\|_{L^2},
 \end{align}
 where we have used the fact $ \vartheta K ^{2}>\eta $, and $ [3(\frac{1}{2}-\frac{1}{q})]_+ $ is as in \eqref{power-defi}. Therefore we have from \eqref{J-1-split} that
 \begin{align}\label{J1-esti-con}
 \displaystyle  J _{1} \lesssim (1+t)^{- \frac{m}{2}- \frac{3}{2}(1- \frac{1}{q}) } \|U _{0}\|_{L ^{1}}+e ^{-\eta t} \|\nabla^{m+[3(\frac{1}{2}-\frac{1}{q})]_+}U_0\|_{L^2}.
 \end{align}
 By similar arguments as in deriving \eqref{J-12-esti}, we further deduce for $ J _{2} $ that
 \begin{align*}
 \displaystyle J _{2} &\lesssim      e^{- \eta t}\||\xi|^m(\hat{\varrho_0},\hat{\mathbf{u}}_0)\|_{L^{q'}(\vert \xi\vert \leq K)}+e^{- \eta t}\||\xi|^m(\hat{\varrho_0},\hat{\mathbf{u}}_0)\|_{L^{q'}(\vert \xi\vert \geq K)}
   \nonumber \\
   & \displaystyle \lesssim e ^{-\eta t}\|(\hat{\varrho}_0,\hat{\mathbf{u}}_0)\|_{L ^{\infty}}+e ^{-\eta t} \|\vert \xi\vert ^{-3(1+\epsilon)\frac{2-q'}{2q'}}\|_{L ^{\frac{2q'}{2-q'}}(\vert \xi\vert \geq K)} \|\vert \xi\vert ^{m+3(1+\epsilon)\frac{2-q'}{2q'}}(\hat{\varrho}_{0},\hat{\mathbf{u}}_{0})\|_{L ^{2}(\vert \xi\vert \geq K)}
    \nonumber \\
     & \displaystyle \lesssim e ^{-\eta t}\|U_0\|_{L^1}+e ^{-\eta t}\|\nabla^{m+[3(\frac{1}{2}-\frac{1}{q})]_+}U_0\|_{L^2},
 \end{align*}
 where $ K $ is as in \eqref{J-1-split}. Substituting the estimates of $ J  _{1} $ and $ J _{2} $ into \eqref{vrho-esti-line}, we thus have
 \begin{align*}
 \displaystyle  \|\nabla^m\varrho(t)\|_{L^q}&\lesssim \||\xi|^m\hat{\varrho}\|_{L^{q'}} \lesssim (1+t)^{- \frac{m}{2}- \frac{3}{2}(1- \frac{1}{q}) } \|U _{0}\|_{L ^{1}}+e ^{-\eta t} \|\nabla^{m+[3(\frac{1}{2}-\frac{1}{q})]_+}U_0\|_{L^2}.
 \end{align*}
 This gives \eqref{vrho-con-line-lem}. Similarly, in virtue of \eqref{divu0}, \eqref{divu1}, the Hausdorff-Young inequality and the H\"older inequality, we derive for $ \|\nabla ^{m}\mathop{\mathrm{div}}\nolimits \mathbf{u}\|_{L ^{q}} $ that

 \begin{align*}
 \displaystyle \|\nabla ^{m}\mathop{\mathrm{div}}\nolimits \mathbf{u}\|_{L ^{q}}& \lesssim \| \vert \xi\vert ^{m}\widehat{\mathop{\mathrm{div}}\nolimits \mathbf{u}}\| _{L ^{q'}}\lesssim \| \vert \xi\vert ^{m}\widehat{\mathop{\mathrm{div}}\nolimits \mathbf{u}}\| _{L ^{q'}(\xi \in \mathcal{A}_{1})}+\| \vert \xi\vert ^{m}\widehat{\mathop{\mathrm{div}}\nolimits \mathbf{u}}\| _{L ^{q'}(\xi \in \mathcal{A}_{2})}
  \nonumber \\
  & \displaystyle \lesssim  \||\xi|^{m+1}e^{-\vartheta|\xi|^2t}(\hat{\varrho}_0,\hat{\mathbf{u}}_0)\|_{L^{q'}}+e^{- \eta t}\||\xi|^{m+1}(\hat{\varrho_0},\hat{\mathbf{u}}_0)\|_{L^{q'}}
   \nonumber \\
   &\lesssim (1+t)^{- \frac{m+1}{2}- \frac{3}{2}(1- \frac{1}{q}) }\|U_0\|_{L^1}+e^{- \eta t}\|\nabla^{m+1+[3(\frac{1}{2}-\frac{1}{q})]_+}U_0\|_{L^2},
 \end{align*}
 and that
 \begin{align*}
 \displaystyle \|\nabla ^{m}\mathop{\mathrm{div}}\nolimits \mathbf{u}\|_{L ^{q}}& \lesssim \| \vert \xi\vert ^{m}\widehat{\mathop{\mathrm{div}}\nolimits \mathbf{u}}\| _{L ^{q'}}\lesssim \| \vert \xi\vert ^{m}\widehat{\mathop{\mathrm{div}}\nolimits \mathbf{u}}\| _{L ^{q'}(\xi \in \mathcal{A}_{1})}+\| \vert \xi\vert ^{m}\widehat{\mathop{\mathrm{div}}\nolimits \mathbf{u}}\| _{L ^{q'}(\xi \in \mathcal{A}_{2})}
  \nonumber \\
  & \displaystyle \lesssim  \||\xi|^{m+1}e^{-\vartheta|\xi|^2t}\hat{\varrho}_0\|_{L^{q'}}+\||\xi|^{m}e^{-\vartheta|\xi|^2t}\widehat{\mathop{\mathrm{div}}\nolimits \mathbf{u}_{0}}\|_{L^{q'}}+e^{- \eta t}\||\xi|^{m+1}\hat{\varrho_0}\|_{L^{q'}}
   \nonumber \\
   & \displaystyle \quad+e^{- \eta t}\||\xi|^{m}\widehat{\mathop{\mathrm{div}}\nolimits \mathbf{u}_{0}}\|_{L^{q'}}
   \nonumber \\
   &\lesssim  (1+t)^{- \frac{m+1}{2}- \frac{3}{2}(1- \frac{1}{q}) }\|\varrho_0\|_{L^1}+e^{- \eta t}\|\nabla^{m+1+[3(\frac{1}{2}-\frac{1}{q})]_+}\varrho_0\|_{L^2}
 \nonumber \\[2mm]
& \displaystyle \quad +(1+t)^{- \frac{m+1}{2}- \frac{3}{2}(1- \frac{1}{q}) }\| \mathbf{u}_0\|_{L^1}+e^{- \eta t}\|\nabla^{m+[3(\frac{1}{2}-\frac{1}{q})]_+}\mathop{\mathrm{div}}\nolimits\mathbf{u}_0\|_{L^2}.
 \end{align*}

 As for $ \|\nabla _{h}^{m}\partial _{3}^{n}\mathbf{u}\| _{L ^{q}} $, we get
 \begin{align}\label{aniso-na-u}
\|\nabla^m_h\nabla^n_3 \mathbf{u}(t)\|_{L^q}& \lesssim\||\xi_h|^m\widehat{\partial^n_3\mathbf{u}}\|_{L^{q'}} \lesssim \||\xi_h|^m\xi_3^n\widehat{\mathbf{u}}\|_{L^{q'}(\xi\in A_1)}+\||\xi_h|^m\xi_3^n\widehat{\mathbf{u}}\|_{L^{q'}(\xi\in A_2)}
 \nonumber \\[2mm]
 & \displaystyle \lesssim \|e ^{- \vartheta \vert \xi \vert ^{2}t} |\xi_h|^m\xi_3^n \hat{\varrho}_{0}\|_{L^{q'}}+\|e ^{- \vartheta \vert \xi _{h} \vert ^{2}t} |\xi_h|^m\xi_3^n \hat{\mathbf{u}}_{0}\|_{L^{q'}}
  \nonumber \\[2mm]
  & \displaystyle \quad+e ^{- \eta t}\| \vert \xi_h\vert^m\xi_3^n \hat{\varrho}_{0}\|_{L^{q'}}+e ^{- \eta t}\| \vert \xi_h\vert^m\xi_3^n \hat{\mathbf{u}}_{0}\|_{L^{q'}}
   \nonumber \\[2mm]
    & \displaystyle =L _{1}+L _{2}+L _{3}+L _{4},
\end{align}
where, in virtue of similar arguments as in deriving \eqref{J1-esti-con}, $ L _{1} $ and $ L _{2} $ can be estimated as
\begin{align*}
\displaystyle  L _{1}& \lesssim \|e ^{- \vartheta \vert \xi \vert ^{2}t} |\xi_h|^m\xi_3^n \hat{\varrho}_{0}\|_{L _{\xi }^{q'}(\vert \xi\vert \leq K)}+\|e ^{- \vartheta \vert \xi  \vert ^{2}t} |\xi_h|^m\xi_3^n \hat{\varrho}_{0}\|_{L _{\xi }^{q'}(\vert \xi\vert \geq K)}
 \nonumber \\[2mm]
 & \displaystyle \lesssim (1+t)^{- \frac{m+n}{2}- \frac{3}{2} (1- \frac{1}{q})}\|\varrho_{0}\|_{L ^{1}}+e ^{-\eta t}\|\nabla^{m+n+[3(\frac{1}{2}-\frac{1}{q})]_+}\varrho_0\|_{L ^2},
  \nonumber \\[2mm]
  \displaystyle  L _{2}& \lesssim \|e ^{- \vartheta \vert \xi _{h} \vert ^{2}t} |\xi_h|^m\xi_3^n \hat{\mathbf{u}}_{0}\|_{L _{\xi _{3}}^{q'}(\mathbb{R})L _{\xi _{1}\xi _{2}}^{q'}(\vert \xi _{h}\vert \leq K)}+\|e ^{- \vartheta \vert \xi _{h} \vert ^{2}t} |\xi_h|^m\xi_3^n \hat{\mathbf{u}}_{0}\|_{L _{\xi _{3}}^{q'}(\mathbb{R})L _{\xi _{1}\xi _{2}}^{q'}(\vert \xi _{h}\vert \geq K)}
 \nonumber \\[2mm]
 & \displaystyle \lesssim (1+t)^{- \frac{m}{2}-(1- \frac{1}{q})}\|\xi _{3}^{n}\widetilde{\mathbf{u}}_{0}\|_{L _{\xi _{3}}^{q'}L _{x _{1}x _{2}}^{1}}+e ^{-\eta t}\|\nabla^{m+[2(\frac{1}{2}-\frac{1}{q})]_+}\xi _{3}^{n}\widetilde{\mathbf{u}}_0\|_{L _{\xi _{3}}^{q'} L _{x _{1}x _{2}}^2},
\end{align*}
where $ \widetilde{\mathbf{u}}_{0}  $ is defined by \eqref{wt-U}, and $ K $ is as in \eqref{J-1-split}. While by similar arguments as in deriving \eqref{J-12-esti}, we can derive for $ L _{3} $ and $ L _{4} $ that
\begin{align}
\displaystyle L  _{3}&\lesssim  e^{- \eta t}\||\xi _{h}|^m \xi _{3}^{n}\hat{\varrho}_{0}\|_{L _{\xi }^{q'}(\vert \xi\vert \leq K)}+e^{- \eta t}\||\xi _{h} |^m \xi _{3}^{n}\hat{\varrho}_{0}\|_{L _{\xi }^{q'}(\vert \xi\vert \geq K)}
   \nonumber \\
   & \displaystyle \lesssim e ^{-\eta t}\|\hat{\varrho}_{0}\|_{L _{\xi}^{\infty}}+e ^{-\eta t} \|\vert \xi\vert ^{-3(1+\epsilon)\frac{2-q'}{2q'}}\|_{L ^{\frac{2q'}{2-q'}}(\vert \xi \vert \geq K)} \|\vert \xi\vert ^{m+n+3(1+\epsilon)\frac{2-q'}{2q'}}\hat{\varrho} _{0}\|_{L ^{2}(\vert \xi\vert \geq K)}
    \nonumber \\
     & \displaystyle \lesssim e ^{-\eta t}\|\varrho _{0}\|_{ L ^1}+e ^{-\eta t}\|\nabla^{m+n+[3(\frac{1}{2}-\frac{1}{q})]_+}\varrho_0\|_{L^2},
      \nonumber \\[3mm]
      \displaystyle L  _{4}&\lesssim  e^{- \eta t}\||\xi _{h}|^m \xi _{3}^{n}\hat{\mathbf{u}}_{0}\|_{ L _{\xi _{3}}^{q'}(\mathbb{R}) L _{\xi _{1}\xi _{2}}^{q'}(\vert \xi  _{h}\vert \leq K)}+e^{- \eta t}\||\xi _{h} |^m \xi _{3}^{n}\hat{\mathbf{u}}_{0}\|_{ L _{\xi _{3}}^{q'}(\mathbb{R}) L _{\xi _{1}\xi _{2}}^{q'}(\vert \xi  _{h}\vert \geq K)}
   \nonumber \\
   & \displaystyle \lesssim e ^{-\eta t}\|\xi _{3}^{n}\hat{\mathbf{u}}_{0}\|_{L _{\xi _{3}}^{q'}L _{\xi _{1}\xi _{2}}^{\infty}}+e ^{-\eta t} \|\vert \xi _{h}\vert ^{-2(1+\epsilon)\frac{2-q'}{2q'}}\|_{L _{\xi _{1}\xi _{2}}^{\frac{2q'}{2-q'}}(\vert \xi _{h}\vert \geq K)} \|\vert \xi _{h}\vert ^{m+2(1+\epsilon)\frac{2-q'}{2q'}}\xi _{3}^{n}\hat{\mathbf{u}} _{0}\|_{L _{\xi _{3}}^{q'} L _{\xi _{1}\xi _{2}}^{2}}
    \nonumber \\
     & \displaystyle \lesssim e ^{-\eta t}\|\xi _{3}^{n}\widetilde{\mathbf{u}}_0\|_{L _{\xi _{3}} ^{q'} L _{x _{1}x _{2}}^1}+e ^{-\eta t}\|\nabla^{m+[2(\frac{1}{2}-\frac{1}{q})]_+}\xi _{3}^{n}\widetilde{\mathbf{u}}_0\|_{L_{\xi _{3}}^{q'} L_{x _{1}x_{2}}^2},
      \nonumber
      \end{align}
where $\widetilde{\mathbf{u}}_{0}  $ is defined by \eqref{wt-U}. Therefore we have from \eqref{aniso-na-u} that
\begin{align*}
\displaystyle  \|\nabla^m_h\nabla^n_3 \mathbf{u}(t)\|_{L^q} &\lesssim  (1+t)^{- \frac{m+n}{2}- \frac{3}{2} (1- \frac{1}{q})}\|\varrho_{0}\|_{L ^{1}}+e ^{-\eta t}\|\nabla^{m+n+[3(\frac{1}{2}-\frac{1}{q})]_+}\varrho_0\|_{L ^2}
 \nonumber \\[2mm]
 &\displaystyle \quad+ (1+t)^{-\frac{m+2}{2}+\frac{1}{q}}\|\xi _{3}^{n}\widetilde{\mathbf{u}}_0\|_{L _{\xi _{3}} ^{q'} L _{x _{1}x _{2}}^1}+e ^{-\eta t}\|\nabla^{m+[2(\frac{1}{2}-\frac{1}{q})]_+}\xi _{3}^{n}\widetilde{\mathbf{u}}_0\|_{L _{\xi _{3}}^{q'} L _{x _{1}x _{2}}^2}.
\end{align*}
Then, we get \eqref{anisou-lem-con}.
Similarly, one can immediately prove \eqref{hori-u-lem-con} and \eqref{verti-u-conl-lem}. The proof of Lemma \ref{xxlemma} is complete.
\end{proof}


\vspace{4mm}

\section{Nonlinear stability of the constant state} 
\label{sub:nonlinear_stability}
In this section, we are devoted to establishing the decay estimates of solutions to \eqref{1-1}, and then finish the proof of Theorem \ref{thm2}. Now, let us define some energy functionals that will enable us to achieve our desired estimate.
\begin{align}
\displaystyle \mathcal{E}_0(t)&=\sup_{0\leq \tau\leq t}\|(\varrho,\mathbf{u})(\tau,\cdot)\|^2_{H^4}+\int^t_0(\|\nabla\varrho(\tau,\cdot)\|^2_{H^3}+\|(\nabla_h\mathbf{u},\mathop{\mathrm{div}}\nolimits\mathbf{u})(\tau,\cdot)\|^2_{H^4})\mathrm{d}\tau,
 \nonumber \\[2mm]
\displaystyle  \mathcal{E}_1(t)&=\sup _{0 \leq \tau \leq t}(1+\tau)^{2}\Big(\|\mathop{\mathrm{div}}\nolimits \mathbf{u}(\tau,\cdot)\| _{L ^{2}}^2+\|\nabla\mathop{\mathrm{div}}\nolimits \mathbf{u}(\tau,\cdot)\| _{L ^{2}}^2 \Big)
 \nonumber \\
 & \displaystyle \quad+\int^t_0(1+\tau)^{\frac{2}{15}}(\|\nabla\mathop{\mathrm{div}}\nolimits\mathbf{u},\Delta_h\mathbf{u})(\tau,\cdot)\|^2_{H^2}\mathrm{d}\tau,
 \nonumber \\[2mm]
 \displaystyle \mathcal{E}_2(t)&=\sup_{0\leq\tau\leq t}\Big((1+\tau)^{\frac{3}{2}}(\log(1+\tau))^{-2}\|\varrho(\tau,\cdot)\|^2_{L^2}+(1+\tau)(\|\mathbf{u}(\tau,\cdot)\|^2_{L^2}+\|\partial_3 \mathbf{u}(\tau,\cdot)\|^2_{L^2})\Big.\nonumber\\[2mm]
&\displaystyle \quad+\Big.(1+\tau)^{2}(\|\nabla\varrho(\tau,\cdot)\|^2_{L^2}+\|\nabla_h \mathbf{u}(\tau,\cdot)\|^2_{L^2})\Big.\nonumber\\[2mm]
&\displaystyle \quad+\Big.\sum_{i=1,2;j=1,2}(1+\tau)^{\frac{32}{15}}\|\partial_i\partial_j \mathbf{u}(\tau,\cdot)\|^2_{L^2}
+(1+\tau)^{2}\sum_{i=1,2}\|\partial_i\partial_3u(\tau,\cdot)\|^2_{L^2}\Big.
 \nonumber \\[2mm]
 & \displaystyle \quad+\Big.(1+\tau)(\log(1+\tau))^{-2}\|\partial^2_3 \mathbf{u}(\tau,\cdot)\|^2_{L^2}\Big).\label{E-2-defi}
 \end{align}
In what follows, we denote
\begin{gather}\label{E-defi}
\displaystyle  \mathcal{E}(t)=\mathcal{E}_{0}(t)+\mathcal{E}_{1}(t)+\mathcal{E}_{2}(t),
\end{gather}
and assume that
\begin{align}\label{appri-assum-Nonlin}
\displaystyle  \mathcal{E}(t) \lesssim \Lambda,
\end{align}
for some $ \Lambda $ to be determined later. Clearly, similar to \eqref{apri-conlu}, we have
\begin{gather}\label{sec4-appri-conlu}
  \displaystyle 0<\frac{\bar{\rho}}{2} \leq \varrho+ \bar{\rho} \leq \frac{3\bar{\rho}}{2},\ \ \  \|\nabla(\varrho,\mathbf{u})\|_{L ^{\infty}}\lesssim \Lambda,
  \end{gather}
  provided $ \Lambda $ is small enough. The goal of the present subsection is to show that under the \emph{a priori }assumption \eqref{appri-assum-Nonlin} with suitably small $ \Lambda $,
\begin{align}\label{goal-esti}
\displaystyle \mathcal{E}(t) \lesssim \mbox{initial norms}+ \mathcal{E}^{\frac{3}{2}}(t)+\mathcal{E}^{2}(t).
\end{align}
Taking $ N=4 $ in \eqref{esti-exist-final}, we get
\begin{align}\label{dissip-1}
\displaystyle  \|(\varrho ,\mathbf{u})(t, \cdot)\|_{H ^{4}}^{2}+ \int _{0}^{t}\left(\|\nabla_h\mathbf{u}\|^2_{H ^{4}}+\|\mathop{\mathrm{div}}\nolimits\mathbf{u}\|^2_{H^4}+\|\nabla \varrho\|_{H ^{3}}^{2}\right)\mathrm{d}\tau \leq \|(\varrho _{0},\mathbf{u}_{0})\|_{H ^{4}}^{2},
\end{align}
that is,
\begin{align}\label{E-0-esti-con}
\displaystyle \mathcal{E}_{0}(t) \lesssim \mathcal{E}_{0}(0).
\end{align}
Therefore, in order to get \eqref{goal-esti}, the key of matter is to establish estimates for $ \mathcal{E}_{1}(t) $ and $ \mathcal{E}_{2}(t) $. To this attempt, we need to derive the representations of solutions to the nonlinear system \eqref{new} first. In view of \eqref{ysf} and the Duhamel principle, we have the following representation of $ (\varrho,\mathbf{u}) $ to nonlinear system \eqref{new} in the frequency space:
\begin{align}\label{rhobiaoshi}
\displaystyle
    \left(\begin{matrix}
\hat{\varrho}(t,\xi)\\[2mm]
\hat{\mathbf{u}}(t,\xi)
\end{matrix}\right)
=\hat{G}(t,\xi)\left(\begin{matrix}
\hat{\varrho} _{0}(\xi)\\[2mm]
\hat{\mathbf{u}}_{0}(\xi)
\end{matrix}
\right)+ \int _{0}^{t}\hat{G}(t- \tau,\xi)\left(\begin{matrix}
\hat{S} _{1}(\tau,\xi)\\[2mm]
\hat{S}_{2}(\tau,\xi)
\end{matrix}
\right)\mathrm{d}\tau.
\end{align}
Similarly, we get for $ \mathop{\mathrm{div}}\nolimits \mathbf{u} $ that
\begin{align}\label{divu-nonlinear-rep}
\displaystyle \widehat{\mathop{\mathrm{div}}\nolimits \mathbf{u}}(t,\xi)= \hat{G}^{\ast}(t, \xi)\left(\begin{matrix}
\hat{\varrho} _{0}\\[2mm]
\widehat{\mathop{\mathrm{div}}\nolimits\mathbf{u}_{0}}
\end{matrix}
\right)+ \int _{0}^{t}\hat{G}^{\ast}(t- \tau,\xi)\left(\begin{matrix}
\hat{S} _{1}(\tau,\xi)\\[2mm]
\widehat{\mathop{\mathrm{div}}\nolimits S _{2}}(\tau,\xi)
\end{matrix}
\right)\mathrm{d}\tau,
\end{align}
where we have used \eqref{new}, \eqref{divu-repres} and the Duhamel principle. Here $ \hat{G}^{\ast} $ and $ \hat{G} $ are as in \eqref{divu-repres} and \eqref{ysf}, respectively.

\subsection{Estimates of $ \mathop{\mathrm{div}}\nolimits \mathbf{u} $} 
\label{sub:estimates_of_the_divergence_fo_gs}


Now we are ready to establish the estimates of the solution. To begin with, let us derive some decay estimates and dissipation estimates for $\mathop{\mathrm{div}}\nolimits\mathbf{u}$.
\begin{proposition}\label{prop4.1}
Assume that $(\varrho,u)$ is a smooth solution to \eqref{new} satisfying \eqref{appri-assum-Nonlin} and \eqref{sec4-appri-conlu}. Then it holds that
\begin{equation*}
\mathcal{E}_1(t)\lesssim\|(\varrho_0,\mathbf{u}_0)\|^2_{L^1}+\|(\nabla\varrho_0,\nabla\mathbf{u}_0)\|^2_{H^1}+\mathcal{E}^{\frac{3}{2}}(t)+\mathcal{E}^2(t).
\end{equation*}
\end{proposition}
\begin{proof}
 The proof of Proposition \ref{prop4.1} consists of Lemmas \ref{divulemma}--\ref{haosan} below.
\end{proof}
\begin{lemma}\label{divulemma}
 Under the assumptions of Theroem \ref{thm2},
it holds that
\begin{equation*}
(1+t)\|\mathop{\mathrm{div}}\nolimits\mathbf{u}\|_{L^2}\lesssim \|(\varrho_0,\mathbf{u}_0)\|_{L^1}+\|(\nabla\varrho_0,\nabla\mathbf{u}_0)\|_{L^2}+\mathcal{E}(t).
\end{equation*}
\end{lemma}
\begin{proof}
Recalling the representation of $ \mathop{\mathrm{div}}\nolimits \mathbf{u} $ in \eqref{divu-nonlinear-rep},  using Plancherel's theorem, \eqref{divu-line-lem}, \eqref{divu-line-lem2} and the fact $e^{-\eta t}(1+t)^m\leq C(\eta,m)$ for any $m\geq 0$, we have
\begin{align}\label{divuL20}
&\|\mathop{\mathrm{div}}\nolimits\mathbf{u}\|_{L^2}=\|\widehat{\mathop{\mathrm{div}}\nolimits\mathbf{u}}\|_{L^2}\nonumber\\[2mm]
& ~\displaystyle\lesssim (1+t)^{-\frac{5}{4}}(\|(\varrho_0, \mathbf{u} _{0})\|_{L^1}+\|(\nabla \varrho _{0},\nabla\mathbf{u}_0)\|_{L^2})+\int^t_0(1+t-\tau)^{-\frac{5}{4}}\|S_1\|_{L^1}\mathrm{d}\tau\nonumber\\[2mm]
& ~\quad+\int^t_0(1+t-\tau)^{-m}\|\nabla S_1\|_{L^2}\mathrm{d}\tau+\int^t_0(1+t-\tau)^{-\frac{5}{4}}\|S_2\|_{L^1}\mathrm{d}\tau
 \nonumber \\[2mm]
 & ~\displaystyle \quad+\int^t_0(1+t-\tau)^{-m}\|\mathop{\mathrm{div}}\nolimits{S}_2\|_{L^2}\mathrm{d}\tau\nonumber\\[2mm]
&~\displaystyle \lesssim (1+t)^{-\frac{5}{4}}(\|(\varrho_0, \mathbf{u} _{0})\|_{L^1}+\|(\nabla \varrho _{0},\nabla\mathbf{u}_0)\|_{L^2})+\sum^{4}_{i=1}N_i.
\end{align}
For the integral $N_1$,  using the H\"older inequality, Lemma \ref{decaylemma} and $S_1=\varrho\mathop{\mathrm{div}}\nolimits\mathbf{u}+\mathbf{u}\cdot\nabla\varrho$, we get
\begin{align}\label{N1}
N_1&=\int^t_0(1+t-\tau)^{-\frac{5}{4}}\|S_1\|_{L^1}\mathrm{d}\tau\nonumber\\[2mm]
&\lesssim \sup_{0\leq \tau\leq t}\Big[(1+\tau)^{\frac{3}{4}}(\log(1+\tau))^{-1}\|\varrho\|_{L^2}(1+\tau)\|\mathop{\mathrm{div}}\nolimits\mathbf{u}\|_{L^2}\Big]\int^t_0(1+t-\tau)^{-\frac{5}{4}}(1+\tau)^{-\frac{7}{4}}\log(1+\tau)\mathrm{d}\tau\nonumber\\[2mm]
& \quad+ \sup_{0\leq \tau\leq t}\Big[(1+\tau)^{\frac{1}{2}}\|\mathbf{u}\|_{L^2}(1+\tau)\|\nabla\varrho\|_{L^2}\Big]\int^t_0(1+t-\tau)^{-\frac{5}{4}}(1+\tau)^{-\frac{3}{2}}\mathrm{d}\tau\nonumber\\[2mm]
& \lesssim\mathcal{E}^{\frac{1}{2}}_2(t)\mathcal{E}^{\frac{1}{2}}_1(t)(1+t)^{-\frac{5}{4}}+\mathcal{E}_2(t)(1+t)^{-\frac{5}{4}}\nonumber\\[2mm]
& \lesssim\mathcal{E}(t)(1+t)^{-\frac{5}{4}}.
\end{align}
For the term $N_2$,  using the definition of $S_1$ in \eqref{nonlinear}, we split it into two terms.
\begin{align}\label{N2*}
N_2&=\int^t_0(1+t-\tau)^{-m}\|\nabla(\varrho\mathop{\mathrm{div}}\nolimits\mathbf{u})\|_{L^2}\mathrm{d}\tau+\int^t_0(1+t-\tau)^{-m}\|\nabla(\mathbf{u}\cdot\nabla\varrho)\|_{L^2}\mathrm{d}\tau\nonumber\\
&=N_{21}+N_{22},
\end{align}
where, in view of Lemma \ref{decaylemma} and the Sobolev embedding theorem, it holds that
\begin{align}\label{N21}
N_{21}&\leq\int^t_0(1+t-\tau)^{-m}(\|\nabla\varrho\|_{L^{\infty}}\|\mathop{\mathrm{div}}\nolimits\mathbf{u}\|_{L^2}+\|\varrho\|_{L^{\infty}}\|\nabla\mathop{\mathrm{div}}\nolimits\mathbf{u}\|_{L^2})\mathrm{d}\tau\nonumber\\[2mm]
&\displaystyle \lesssim\sup_{0\leq\tau\leq t}\Big[(1+\tau)^{\frac{3}{4}}\|\nabla^2\varrho\|^{\frac{3}{4}}_{L^2}\|\nabla^4\varrho\|^{\frac{1}{4}}_{L^2}(1+\tau)\|\mathop{\mathrm{div}}\nolimits\mathbf{u}\|_{L^2}\Big]
 \nonumber \\[2mm]
 & \displaystyle ~\qquad \cdot\int^t_0(1+t-\tau)^{-m}(1+\tau)^{-\frac{7}{4}}\mathrm{d}\tau\nonumber\\[2mm]
&\quad+ \sup_{0\leq\tau\leq t}\Big[(1+\tau)^{\frac{1}{2}}\|\nabla\varrho\|^{\frac{1}{2}}_{L^2}(1+\tau)^{\frac{1}{2}}\|\nabla^2\varrho\|^{\frac{1}{2}}_{L^2}(1+\tau)\|\nabla\mathop{\mathrm{div}}\nolimits\mathbf{u}\|_{L^2}\Big]\nonumber\\[2mm]
&~\quad\qquad\cdot\int^t_0(1+t-\tau)^{-m}(1+\tau)^{-2}\mathrm{d}\tau\nonumber\\[2mm]
&\displaystyle \lesssim\mathcal{E}^{\frac{3}{8}}_2(t)\mathcal{E}^{\frac{1}{8}}_0(t)\mathcal{E}^{\frac{1}{2}}_1(t)(1+t)^{-\frac{7}{4}}+\mathcal{E}^{\frac{1}{2}}_2(t)\mathcal{E}^{\frac{1}{2}}_1(t)(1+t)^{-2}\nonumber\\[2mm]
& \displaystyle \lesssim\mathcal{E}(t)(1+t)^{-\frac{7}{4}}.
\end{align}
For the term $N_{22}$, by \eqref{two-two-deri-ani} and Lemma \ref{decaylemma}, we have
\begin{align}\label{N22}
N_{22}&\leq \int^t_0(1+t-\tau)^{-m}(\|\nabla\mathbf{u}\cdot\nabla\varrho\|_{L^2}+\|(\mathbf{u}\cdot\nabla)\nabla\varrho\|_{L^2})\mathrm{d}\tau\nonumber\\[2mm]
&\displaystyle \lesssim \sup_{0\leq\tau\leq t}\Big[(1+\tau)^{\frac{1}{8}}\|\nabla\mathbf{u}\|^{\frac{1}{4}}_{L^2}(1+\tau)^{\frac{1}{4}}\|\nabla\partial_1\mathbf{u}\|^{\frac{1}{4}}_{L^2}(1+\tau)^{\frac{3}{8}}\|\nabla\partial_2\mathbf{u}\|^{\frac{3}{8}}_{L^2}\|\nabla\partial^2_1\partial_2\mathbf{u}\|^{\frac{1}{8}}_{L^2}\Big.\nonumber\\[2mm]
& \displaystyle \qquad \quad\Big.\cdot(1+\tau)^{\frac{1}{2}}\|\nabla\varrho\|^{\frac{1}{2}}_{L^2}(1+\tau)^{\frac{1}{2}}\|\nabla\partial_3\varrho\|^{\frac{1}{2}}_{L^2}\Big]\int^t_0(1+t-\tau)^{-m}(1+\tau)^{-\frac{7}{4}}\mathrm{d}\tau\nonumber\\[2mm]
&\quad+\sup_{0\leq \tau\leq t}\Big[(1+\tau)^{\frac{1}{8}}\|\mathbf{u}\|^{\frac{1}{4}}_{L^2}(1+\tau)^{\frac{1}{4}}\|\partial_1\mathbf{u}\|^{\frac{1}{4}}_{L^2}(1+\tau)^{\frac{1}{4}}\|\partial_2\mathbf{u}\|^{\frac{1}{4}}_{L^2}(1+\tau)^{\frac{4}{15}}\|\partial_1\partial_2\mathbf{u}\|^{\frac{1}{4}}_{L^2}\Big.\nonumber\\[2mm]
& \qquad \qquad \quad\Big.\cdot(1+\tau)^{\frac{3}{4}}\|\nabla ^{2} \varrho\|^{\frac{3}{4}}_{L^2}\|\nabla ^{2}\partial^2_3\varrho\|^{\frac{1}{4}}_{L^2}\Big]\int^t_0(1+t-\tau)^{-m}(1+\tau)^{-\frac{197}{120}}\mathrm{d}\tau\nonumber\\[2mm]
&\displaystyle\lesssim \mathcal{E}^{\frac{15}{16}}_2(t)\mathcal{E}^{\frac{1}{16}}_0(t)(1+t)^{-\frac{7}{4}}+\mathcal{E}^{\frac{7}{8}}_2(t)\mathcal{E}^{\frac{1}{8}}_0(t)(1+t)^{-\frac{197}{120}}\nonumber\\[2mm]
& \displaystyle \lesssim\mathcal{E}(t)(1+t)^{-\frac{197}{120}}.
\end{align}
Substituting \eqref{N21} and \eqref{N22} into \eqref{N2*}, we get
\begin{align}\label{N2}
N_2 \lesssim\mathcal{E}(t)(1+t)^{-\frac{197}{120}}.
\end{align}
For $N_3$, using the H\"older inequality, Lemma \ref{decaylemma} and the definition of $S_2$ in \eqref{nonlinear}, we obtain that
\begin{align}\label{N3}
N_3&\displaystyle\lesssim \sup_{0\leq \tau\leq t}\Big[(1+\tau)^{\frac{1}{2}}\|\mathbf{u}\|_{L^2}(1+\tau)^{\frac{1}{2}}\|\nabla\mathbf{u}\|_{L^2}\Big]\int^t_0(1+t-\tau)^{-\frac{5}{4}}(1+\tau)^{-1}\mathrm{d}\tau\nonumber\\[2mm]
&\displaystyle~ \quad+\sup_{0\leq\tau\leq t}\Big[(1+\tau)^{\frac{3}{4}}(\log(1+\tau))^{-1}\|\varrho\|_{L^2}(1+\tau)^{\frac{16}{15}}\|\Delta_h\mathbf{u}\|_{L^2}\Big]\nonumber\\[2mm]
&\displaystyle~ \qquad\cdot\int^t_0(1+t-\tau)^{-\frac{5}{4}}(1+\tau)^{-\frac{109}{60}}\log(1+\tau)\mathrm{d}\tau\nonumber\\[2mm]
&\displaystyle~ \quad+\sup_{0\leq\tau\leq t}\Big[(1+\tau)^{\frac{3}{4}}(\log(1+\tau))^{-1}\|\varrho\|_{L^2}(1+\tau)\|\nabla\mathop{\mathrm{div}}\nolimits\mathbf{u}\|_{L^2}\Big]\nonumber\\[2mm]
&\displaystyle~ \qquad\cdot\int^t_0(1+t-\tau)^{-\frac{5}{4}}(1+\tau)^{-\frac{7}{4}}\log(1+\tau)\mathrm{d}\tau\nonumber\\[2mm]
&~\displaystyle \quad+\sup_{0\leq \tau\leq t}\Big[(1+\tau)^{\frac{3}{4}}(\log(1+\tau))^{-1}\|\varrho\|_{L^2}(1+\tau)\|\nabla\varrho\|_{L^2}\Big]\nonumber\\[2mm]
&\displaystyle~ \qquad\cdot\int^t_0(1+t-\tau)^{-\frac{5}{4}}(1+\tau)^{-\frac{7}{4}}\log(1+\tau)\mathrm{d}\tau\nonumber\\[2mm]
&\displaystyle \lesssim \mathcal{E}_2(t)(1+t)^{-1}+\mathcal{E}_2(t)(1+t)^{- \frac{5}{4}} +\mathcal{E}^{\frac{1}{2}}_2(t)\mathcal{E}^{\frac{1}{2}}_1(t)(1+t)^{-\frac{5}{4}}\nonumber\\[2mm]
&\displaystyle \lesssim\mathcal{E}(t)(1+t)^{-1}.
\end{align}
For the last term $N_4$, we split it into four terms.
\begin{align}\label{N40}
N_4&=\int^t_0(1+t-\tau)^{-m}\|\mathop{\mathrm{div}}\nolimits S_2\|_{L^2}\mathrm{d}\tau\nonumber\\[2mm]
&\lesssim \int^t_0 (1+t-\tau)^{-m}\Big (\|\mathop{\mathrm{div}}(\mathbf{u}\cdot\nabla\mathbf{u})\|_{L^2}+\|\mathop{\mathrm{div}}(g(\varrho)\Delta_h\mathbf{u})\|_{L^2}\Big.
 \nonumber \\[2mm]
& \displaystyle~\Big. \qquad \quad+\|\mathop{\mathrm{div}}(g(\varrho)\nabla\mathop{\mathrm{div}}\nolimits\mathbf{u})\|_{L^2} +\|\mathop{\mathrm{div}}(f(\varrho)\nabla\varrho)\|_{L^2}\Big)\mathrm{d}\tau\nonumber\\[2mm]
&=:\sum^{4}_{j=1}N_{4j}.
\end{align}
Next we shall handle the terms on the right hand side of \eqref{N40}. First, it holds that
\begin{align}\label{N410}
\displaystyle N_{41}&=\int^t_0(1+t-\tau)^{-m}\|\mathop{\mathrm{div}}(\mathbf{u}\cdot\nabla \mathbf{u})\|_{L^2}\mathrm{d}\tau\nonumber\\[2mm]
&\displaystyle \lesssim\int^t_0(1+t-\tau)^{-m}(\|\mathbf{u}\cdot\nabla\mathop{\mathrm{div}}\mathbf{u}\|_{L^2}+\|\nabla\mathbf{u}_h\cdot\nabla_h\mathbf{u}\|_{L^{2}}+\|\partial_3\mathbf{u}\cdot\nabla u_3\|_{L^2})\mathrm{d}\tau\nonumber\\[2mm]
&\displaystyle \lesssim N_{411}+N_{412}+N_{413}.
\end{align}
By \eqref{two-two-deri-ani} and Lemma \ref{decaylemma}, we can bound $N_{411}$-$N_{413}$ as follows:
\begin{align}\label{N411}
N_{411}&\lesssim
\int^t_0(1+t-\tau)^{-m}\Big(\|\mathbf{u}\|^{\frac{1}{4}}_{L^2}\|\partial_1\mathbf{u}\|^{\frac{1}{4}}_{L^2}\|\partial_2\mathbf{u}\|^{\frac{1}{4}}_{L^2}\|\partial_1\partial_2\mathbf{u}\|^{\frac{1}{4}}_{L^2}\|\nabla\mathop{\mathrm{div}}\mathbf{u}\|^{\frac{1}{2}}_{L^2}\|\nabla\partial_3\mathop{\mathrm{div}}\mathbf{u}\|^{\frac{1}{2}}_{L^2}\Big)\mathrm{d}\tau\nonumber\\[3mm]
&\displaystyle \lesssim\sup_{0\leq\tau\leq t}\Big[(1+\tau)^{\frac{1}{8}}\|\mathbf{u}\|^{\frac{1}{4}}_{L^2}(1+\tau)^{\frac{1}{4}}\|\partial_1\mathbf{u}\|^{\frac{1}{4}}_{L^2}(1+\tau)^{\frac{1}{4}}\|\partial_2\mathbf{u}\|^{\frac{1}{4}}_{L^2}(1+\tau)^{\frac{4}{15}}\|\partial_1\partial_2\mathbf{u}\|^{\frac{1}{4}}_{L^2}\Big.\nonumber\\[3mm]
&~\qquad \qquad \Big.\cdot(1+\tau)^{\frac{3}{4}}\|\nabla\mathop{\mathrm{div}}\nolimits\mathbf{u}\|^{\frac{3}{4}}_{L^2}\|\nabla^3\mathop{\mathrm{div}}\nolimits\mathbf{u}\|^{\frac{1}{4}}_{L^2}\Big]\int^t_0(1+t-\tau)^{-m}(1+\tau)^{-\frac{197}{120}}\mathrm{d}\tau\nonumber\\[3mm]
& \displaystyle \lesssim\mathcal{E}^{\frac{1}{2}}_2(t)\mathcal{E}^{\frac{3}{8}}_1(t)\mathcal{E}^{\frac{1}{8}}_0(t)\int^t_0(1+t-\tau)^{-m}(1+\tau)^{-\frac{197}{120}}\mathrm{d}\tau\nonumber\\[3mm]
& \displaystyle \lesssim\mathcal{E}(t)(1+t)^{-\frac{197}{120}},
\end{align}
\begin{align}
&\displaystyle N_{412}
 \nonumber \\
 &~\displaystyle\lesssim \int^t_0(1+t-\tau)^{-m}(\|\nabla\mathbf{u}_h\|^{\frac{1}{4}}_{L^2}\|\nabla\partial_1\mathbf{u}_h\|^{\frac{1}{4}}_{L^2}\|\nabla\partial_2\mathbf{u}_h\|^{\frac{1}{4}}_{L^2}\|\nabla\partial_1\partial_2\mathbf{u}_h\|^{\frac{1}{4}}_{L^2}\|\nabla_h\mathbf{u}\|^{\frac{1}{2}}_{L^2}\|\nabla_h\partial_3\mathbf{u}\|^{\frac{1}{2}}_{L^2})\mathrm{d}\tau\nonumber\\
&~\displaystyle \lesssim \sup_{0\leq \tau\leq t}\Big[(1+\tau)^{\frac{1}{8}}\|\nabla\mathbf{u}_h\|_{L^2}^{\frac{1}{4}}(1+\tau)^{\frac{1}{4}}\|\nabla\partial_1\mathbf{u}_h\|_{L^2}^{\frac{1}{4}}(1+\tau)^{\frac{3}{8}}\|\nabla\partial_2\mathbf{u}_h\|_{L^2}^{\frac{3}{8}}\|\nabla\partial^2_1\partial_2\mathbf{u}_h\|_{L^2}^{\frac{1}{8}}\Big.\nonumber\\
&~~\qquad \qquad\Big.\cdot(1+\tau)^{\frac{1}{2}}\|\nabla_h\mathbf{u}\|^{\frac{1}{2}}_{L^2}(1+\tau)^{\frac{1}{2}}\|\nabla_h\partial_3\mathbf{u}\|^{\frac{1}{2}}_{L^2}\Big]\int^t_0(1+t-\tau)^{-m}(1+\tau)^{-\frac{7}{4}}\mathrm{d}\tau\nonumber\\
&~\displaystyle \lesssim \mathcal{E}^{\frac{15}{16}}_2(t) \mathcal{E}^{\frac{1}{16}}_0(t)\int^t_0(1+t-\tau)^{-m}(1+\tau)^{-\frac{7}{4}}\mathrm{d}\tau\nonumber\\[2mm]
&~\displaystyle \lesssim \mathcal{E}(t)(1+t)^{-\frac{7}{4}},
\end{align}
and
\begin{align}\label{N413}
N_{413}&= \int^t_0(1+t-\tau)^{-m}(\|\nabla u_3\|^{\frac{1}{4}}_{L^2}\|\nabla\partial_1 u_3\|^{\frac{1}{4}}_{L^2}\|\nabla\partial_2 u_3\|^{\frac{1}{4}}_{L^2}\|\nabla\partial_1\partial_2 u_3\|^{\frac{1}{4}}_{L^2}\|\partial_3\mathbf{u}\|^{\frac{1}{2}}_{L^2}\|\partial^2_3\mathbf{u}\|^{\frac{1}{2}}_{L^2})\mathrm{d}\tau\nonumber\\[2mm]
& \displaystyle \lesssim\sup_{0\leq \tau\leq t}\Big[ (1+\tau)^{\frac{1}{4}}\|\nabla u_3\|^{\frac{1}{4}}_{L^2}(1+\tau)^{\frac{1}{4}}\|\nabla\partial_1 u_3\|^{\frac{1}{4}}_{L^2}(1+\tau)^{\frac{3}{8}}\|\nabla\partial_2 u_3\|^{\frac{3}{8}}_{L^2}\|\nabla\partial^2_1\partial_2 u_3\|^{\frac{1}{8}}_{L^2}\Big.\nonumber\\[2mm]
&~~ \Big.\cdot(1+\tau)^{\frac{1}{4}}\|\partial_3\mathbf{u}\|^{\frac{1}{2}}_{L^2}(1+\tau)^{\frac{1}{4}}(\log(1+\tau))^{-\frac{1}{2}}\|\partial^2_3\mathbf{u}\|^{\frac{1}{2}}_{L^2}\Big]\nonumber\\[2mm]
&~~\cdot\int^t_0(1+t-\tau)^{-m}(1+\tau)^{-\frac{11}{8}}(\log(1+\tau))^{\frac{1}{2}}\mathrm{d}\tau \nonumber\\[2mm]
&\displaystyle \lesssim \mathcal{E}^{\frac{7}{16}}_1(t)\mathcal{E}^{\frac{1}{16}}_0(t)\mathcal{E}^{\frac{1}{2}}_2(t)\int^t_0(1+t-\tau)^{-m}(1+\tau)^{-\frac{11}{8}}(\log(1+\tau))^{\frac{1}{2}}\mathrm{d}\tau\nonumber\\[2mm]
&\displaystyle \lesssim\mathcal{E}(t)(1+t)^{-\frac{11}{8}}(\log(1+\tau))^{\frac{1}{2}}.
\end{align}
Substituting \eqref{N411}--\eqref{N413} into \eqref{N410}, we then get
\begin{align}\label{N41}
N_{41}\lesssim\mathcal{E}(t)(1+t)^{-\frac{11}{8}}(\log(1+\tau))^{\frac{1}{2}}.
\end{align}
For $N_{42}$, by using \eqref{soboleve-ineq}, \eqref{sec4-appri-conlu} and Lemma \ref{decaylemma}, we get
\begin{align}
N_{42}&=\int^t_0(1+t-\tau)^{-m}\|\mathop{\mathrm{div}}\nolimits(g(\varrho)\Delta_h\mathbf{u})\|_{L^2}\mathrm{d}\tau\nonumber\\[2mm]
&\displaystyle \lesssim \sup_{0\leq\tau\leq t}\Big[(1+\tau)^{\frac{3}{4}}\|\nabla^2\varrho\|^{\frac{3}{4}}_{L^2}\|\nabla^4\varrho\|^{\frac{1}{4}}_{L^2}(1+\tau)^{\frac{16}{15}}\|\Delta_h\mathbf{u}\|_{L^2}\Big]\int^t_0(1+t-\tau)^{-m}(1+\tau)^{-\frac{109}{60}}\mathrm{d}\tau\nonumber\\[2mm]
&\displaystyle \quad+\sup_{0\leq\tau\leq t}\Big[(1+\tau)^{\frac{1}{2}}\|\nabla\varrho\|^{\frac{1}{2}}_{L^2}(1+\tau)^{\frac{1}{2}}\|\nabla^2\varrho\|^{\frac{1}{2}}_{L^2}(1+\tau)^{\frac{1}{2}}\|\nabla_h\mathop{\mathrm{div}}\nolimits\mathbf{u}\|^{\frac{1}{2}}_{L^2}\|\nabla_h\Delta_h\mathop{\mathrm{div}}\nolimits\mathbf{u}\|^{\frac{1}{2}}_{L^2}\Big]\nonumber\\[2mm]
&~~~\qquad \qquad\cdot\int^t_0(1+t-\tau)^{-m}(1+\tau)^{-\frac{3}{2}}\mathrm{d}\tau\nonumber\\[2mm]
&\lesssim \mathcal{E}^{\frac{3}{8}}_2(t)\mathcal{E}^{\frac{1}{8}}_0(t)(1+t)^{-\frac{109}{60}}+\mathcal{E}^{\frac{1}{2}}_2(t)\mathcal{E}^{\frac{1}{4}}_1(t)\mathcal{E}^{\frac{1}{4}}_0(t)(1+t)^{-\frac{3}{2}}\nonumber\\[2mm]
&\lesssim \mathcal{E}(t)(1+t)^{-\frac{3}{2}}.
\end{align}
Similarly, we have
\begin{align}
N_{43}&=\int^t_0(1+t-\tau)^{-m}\|\mathop{\mathrm{div}}\nolimits(g(\varrho)\nabla \mathop{\mathrm{div}}\nolimits\mathbf{u})\|_{L^2}\mathrm{d}\tau\nonumber\\[2mm]
&\displaystyle \lesssim \sup_{0\leq \tau\leq t}\Big[
(1+\tau)^{\frac{3}{4}}\|\nabla^2\varrho\|^{\frac{3}{4}}_{L^2}\|\nabla^4\varrho\|^{\frac{1}{4}}_{L^2}(1+\tau)\|\nabla \mathop{\mathrm{div}}\nolimits\mathbf{u}\|_{L^2}\Big]\int^t_0(1+t-\tau)^{-m}(1+\tau)^{-\frac{7}{4}}\mathrm{d}\tau\nonumber\\[2mm]
&\displaystyle \quad+\sup_{0\leq\tau \leq t}\Big[(1+\tau)^{\frac{1}{2}}\|\nabla\varrho\|^{\frac{1}{2}}_{L^2}(1+\tau)^{\frac{1}{2}}\|\nabla^2\varrho\|^{\frac{1}{2}}_{L^2}(1+\tau)^{\frac{1}{2}}\|\nabla\mathop{\mathrm{div}}\nolimits\mathbf{u}\|^{\frac{1}{2}}_{L^2}\|\nabla\Delta \mathop{\mathrm{div}}\nolimits\mathbf{u}\|^{\frac{1}{2}}_{L^2}\Big]\nonumber\\[2mm]
&~~\qquad\cdot\int^t_0(1+t-\tau)^{-m}(1+\tau)^{-\frac{3}{2}}\mathrm{d}\tau\nonumber\\[2mm]
&\displaystyle \lesssim\mathcal{E}^{\frac{3}{8}}_2(t)\mathcal{E}^{\frac{1}{8}}_0(t)\mathcal{E}^{\frac{1}{2}}_1(t)(1+t)^{- \frac{7}{4}}+\mathcal{E}^{\frac{1}{2}}_2(t)\mathcal{E}^{\frac{1}{4}}_1(t)\mathcal{E}^{\frac{1}{4}}_0(1+t)^{-\frac{3}{2}}\nonumber\\[2mm]
& \lesssim \displaystyle \mathcal{E}(t)(1+t)^{-\frac{3}{2}}
\end{align}
and
\begin{align}\label{N44}
N_{44}& \displaystyle \lesssim\int^t_0(1+t-\tau)^{-m}(\|\nabla\varrho \cdot\nabla\varrho\|_{L^2}+\|\varrho\Delta\varrho\|_{L^2})\mathrm{d}\tau\nonumber\\[2mm]
&\displaystyle \lesssim\sup_{0\leq\tau\leq t}\Big[(1+\tau)^{\frac{3}{4}}\|\nabla^2\varrho\|^{\frac{3}{4}}_{L^2}\|\nabla^4\varrho\|^{\frac{1}{4}}_{L^2}(1+\tau)\|\nabla\varrho\|_{L^2}\Big]
 \nonumber \\[2mm]
 & \displaystyle ~\displaystyle  \qquad\cdot\int^t_0(1+t-\tau)^{-m}(1+\tau)^{-\frac{7}{4}}\mathrm{d}\tau\nonumber\\[2mm]
&\displaystyle \quad+\sup_{0\leq \tau\leq t} \Big[(1+\tau)^{\frac{1}{2}}\|\nabla\varrho\|^{\frac{1}{2}}_{L^2}(1+\tau)^{\frac{1}{2}}\|\nabla^2\varrho\|^{\frac{1}{2}}_{L^2}(1+\tau)\|\nabla^2\varrho\|_{L^2}\Big]\nonumber\\[2mm]
&\displaystyle ~\qquad\cdot\int^t_0(1+t-\tau)^{-m}(1+\tau)^{-2}\mathrm{d}\tau\nonumber\\[2mm]
&\displaystyle \lesssim \mathcal{E}^{\frac{7}{8}}_2(t)\mathcal{E}^{\frac{1}{8}}_0(t)(1+t)^{-\frac{7}{4}}+\mathcal{E}_2(t)(1+t)^{-2}\nonumber\\[2mm]
&\displaystyle \lesssim \mathcal{E}(t)(1+t)^{-\frac{7}{4}}.
\end{align}
Inserting \eqref{N41}--\eqref{N44} into \eqref{N40}, we get that
\begin{align}\label{N4}
N_4=\int^t_0(1+t-\tau)^{-m}\|\mathop{\mathrm{div}}S_2\|_{L^2}\mathrm{d}\tau \lesssim\mathcal{E}(t)(1+t)^{-\frac{11}{8}}(\log(1+\tau))^{\frac{1}{2}}.
\end{align}
Hence, substituting the estimates \eqref{N1}, \eqref{N2}, \eqref{N3} and \eqref{N4} into \eqref{divuL20}, we derive that
\begin{align*}
\displaystyle (1+t)\|\mathop{\mathrm{div}}\nolimits\mathbf{u}\|_{L^2} \lesssim \|(\varrho_0,\mathbf{u}_0)\|_{L^1}+\|(\nabla\varrho_0,\nabla\mathbf{u}_0)\|_{L^2}+\mathcal{E}(t).
\end{align*}
The proof of Lemma \ref{divulemma} is complete.
\end{proof}
\begin{lemma}\label{divulemma1}
Assume that $(\varrho,u)$ is a smooth solution to the problem \eqref{new}. Then it holds for any $ t \in [0,\infty) $ that
\begin{equation*}
(1+t)\|\nabla\mathop{\mathrm{div}}\nolimits\mathbf{u}\|_{L^2}\leq \mathcal{E}(t)+\|(\varrho_0, \mathbf{u}_0)\|_{L^1}+\|(\nabla^2\varrho_0,\nabla^2\mathbf{u}_0)\|_{L^2}.
\end{equation*}
\end{lemma}
\begin{proof}
From \eqref{divu-nonlinear-rep}, we get
\begin{align*}
\displaystyle \widehat{\mathop{\mathrm{div}}\nolimits \partial _{k}\mathbf{u}}(t,\xi)= \hat{G}^{\ast}(t, \xi)\left(\begin{matrix}
\widehat{\partial _{k} \varrho} _{0}\\[2mm]
\widehat{\mathop{\mathrm{div}}\nolimits \partial _{k}\mathbf{u}_{0}}
\end{matrix}
\right)+ \int _{0}^{t}\hat{G}^{\ast}(t- \tau,\xi)\left(\begin{matrix}
\widehat{\partial _{k}S} _{1}(\tau,\xi)\\[2mm]
\widehat{\mathop{\mathrm{div}}\nolimits \partial _{k} S _{2}}(\tau,\xi)
\end{matrix}
\right)\mathrm{d}\tau
\end{align*}
for $ k=1,2,3 $. This along with \eqref{divu-line-lem}, \eqref{divu-line-lem2} and the fact $e^{-\eta t}(1+t)^m\leq C(\eta,m)$ for any $m\geq 0$ further implies that
\begin{align}\label{divH1}
&\displaystyle \|\nabla \mathop{\mathrm{div}}\nolimits \mathbf{u}\| _{L ^{2}}= \|\widehat{\nabla \mathop{\mathrm{div}}\nolimits \mathbf{u}}\|_{L ^{2}}
 \nonumber \\[2mm]
 &~\displaystyle \lesssim (1+t)^{-\frac{7}{4}}(\|(\varrho_0,\mathbf{u}_{0})\|_{L^1}+\|(\nabla^2\varrho_0,\nabla^2\mathbf{u}_0)\|_{L^2})+\int^t_0(1+t-\tau)^{-\frac{7}{4}}\|S_1\|_{L^1}\mathrm{d}\tau\nonumber\\[2mm]
&~\displaystyle\quad+\int^t_0(1+t-\tau)^{-m}\|\nabla^2 S_1\|_{L^2}\mathrm{d}\tau+\int^t_0(1+t-\tau)^{-\frac{7}{4}}\|S_2\|_{L^1}\mathrm{d}\tau\nonumber\\[2mm]
&~ \displaystyle \quad+\int^t_0(1+t-\tau)^{-m}\|\nabla\mathop{\mathrm{div}}\nolimits{S}_2\|_{L^2}\mathrm{d}\tau\nonumber\\[2mm]
&~\displaystyle=:(1+t)^{-\frac{7}{4}}(\|(\varrho_0,\mathbf{u}_{0})\|_{L^1}+\|(\nabla^2\varrho_0,\nabla^2\mathbf{u}_0)\|_{L^2})+\sum^{4}_{i=1}T_i.
\end{align}
In the same manner as \eqref{N1}, by choosing $\varepsilon<\frac{1}{6}$, we get
\begin{align}\label{T1}
\displaystyle T_{1}
&\displaystyle \lesssim \sup_{0\leq \tau\leq t}(1+\tau)^{\frac{3}{4}}(\log(1+\tau))^{-1}\|\varrho\|_{L^2}(1+\tau)\|\mathop{\mathrm{div}}\nolimits\mathbf{u}\|_{L^2}\int^t_0(1+t-\tau)^{-\frac{7}{4}}(1+\tau)^{-\frac{7}{4}}\log(1+\tau)\mathrm{d}\tau\nonumber\\[2mm]
&~\quad+ \sup_{0\leq \tau\leq t}(1+\tau)\|\nabla\varrho\|_{L^2}(1+\tau)^{\frac{1}{2}}\| \mathbf{u}\|_{L^2}\int^t_0(1+t-\tau)^{-\frac{7}{4}}(1+\tau)^{-\frac{3}{2}}\mathrm{d}\tau\nonumber\\[2mm]
&\displaystyle \lesssim\mathcal{E}^{\frac{1}{2}}_2(t)\mathcal{E}^{\frac{1}{2}}_1(t)(1+t)^{-\frac{7}{4}}\log(1+t)+
  \mathcal{E}_{2}(t)(1+t)^{-\frac{3}{2}}\nonumber\\[2mm]
&\displaystyle \lesssim\mathcal{E}(t)(1+t)^{-\frac{3}{2}},
\end{align}
where we have used Lemma \ref{decaylemma} and the H\"older inequality. For the term $T_2$, we split it into two terms.
\begin{align}\label{T20}
T_2
& \lesssim\int^t_0(1+t-\tau)^{-m}\|\nabla ^{2}(\varrho \mathop{\mathrm{div}}\nolimits\mathbf{u})\|_{L^2}\mathrm{d}\tau+\int^t_0(1+t-\tau)^{-m}\|\nabla ^{2}(\mathbf{u}\cdot\nabla\varrho)\|_{L^2}\mathrm{d}\tau\nonumber\\[2mm]
&=T_{21}+T_{22}.
\end{align}
It follows from \eqref{soboleve-ineq} that
\begin{align*}
\displaystyle\|\nabla^2(\varrho \mathop{\mathrm{div}}\nolimits\mathbf{u})\|_{L^2}& \lesssim \|\nabla^2\varrho \mathop{\mathrm{div}}\nolimits\mathbf{u}\|_{L^2}+2\|\nabla\varrho\nabla\mathop{\mathrm{div}}\nolimits\mathbf{u}\|_{L^2}+\|\varrho\nabla^2\mathop{\mathrm{div}}\nolimits\mathbf{u}\|_{L^2}\nonumber\\[2mm]
\displaystyle&\lesssim\|\mathop{\mathrm{div}}\nolimits\mathbf{u}\|^{\frac{1}{4}}_{L^2}\|\nabla^2\mathop{\mathrm{div}}\nolimits\mathbf{u}\|^{\frac{3}{4}}_{L^2}\|\nabla^2\varrho\|_{L^2}+\|\nabla\varrho\|^{\frac{1}{4}}_{L^2}\|\nabla^3\varrho\|^{\frac{3}{4}}_{L^2}\|\nabla\mathop{\mathrm{div}}\nolimits\mathbf{u}\|_{L^2}\nonumber\\[2mm]
&\displaystyle \quad+\|\nabla\varrho\|^{\frac{1}{2}}_{L^2}\|\nabla^2\varrho\|^{\frac{1}{2}}_{L^2}\|\nabla^2\mathop{\mathrm{div}}\nolimits\mathbf{u}\|_{L^2}.
\end{align*}
Therefore we have for $ T _{21} $ that
\begin{align}\label{T21}
T_{21}&\lesssim\sup_{0\leq\tau\leq t}\Big[(1+\tau)^{\frac{1}{4}}\|\mathop{\mathrm{div}}\nolimits\mathbf{u}\|^{\frac{1}{4}}_{L^2}(1+\tau)^{\frac{3}{8}}\|\nabla \mathop{\mathrm{div}}\nolimits\mathbf{u}\|^{\frac{3}{8}}_{L^2}\|\nabla^3 \mathop{\mathrm{div}}\nolimits\mathbf{u}\|^{\frac{3}{8}}_{L^2}(1+\tau)\|\nabla^2\varrho\|_{L^2}\Big]\nonumber\\[2mm]
&~~\displaystyle\qquad\cdot\int^t_0(1+t-\tau)^{-m}(1+\tau)^{-\frac{13}{8}}\mathrm{d}\tau\nonumber\\[2mm]
&\displaystyle~\quad+\sup_{0\leq\tau\leq t}\Big[(1+\tau)^{\frac{1}{4}}\|\nabla\varrho\|^{\frac{1}{4}}_{L^2}(1+\tau)^{\frac{3}{8}}\|\nabla^2\varrho\|^{\frac{3}{8}}_{L^2}\|\nabla^4\varrho\|^{\frac{3}{8}}_{L^2}(1+\tau)\|\nabla\mathop{\mathrm{div}}\nolimits\mathbf{u}\|_{L^2}\Big]
 \nonumber \\[2mm]
 & ~\displaystyle ~\qquad \cdot\int^t_0(1+t-\tau)^{-m}(1+\tau)^{-\frac{13}{8}}\mathrm{d}\tau\nonumber\\
&\displaystyle~ \quad+\sup_{0\leq\tau\leq t}\Big[(1+\tau)^{\frac{1}{2}}\|\nabla\varrho\|^{\frac{1}{2}}_{L^2}(1+\tau)^{\frac{1}{2}}\|\nabla^2\varrho\|^{\frac{1}{2}}_{L^2}(1+\tau)^{\frac{1}{2}}\|\nabla\mathop{\mathrm{div}}\nolimits\mathbf{u}\|^{\frac{1}{2}}_{L^2}\|\nabla^3\mathop{\mathrm{div}}\nolimits\mathbf{u}\|^{\frac{1}{2}}_{L^2}\Big]
 \nonumber \\[2mm]
 & ~\displaystyle ~\qquad  \cdot\int^t_0(1+t-\tau)^{-m}(1+\tau)^{-\frac{3}{2}}\mathrm{d}\tau\nonumber\\
&\displaystyle \lesssim \mathcal{E}^{\frac{5}{16}}_1(t)\mathcal{E}^{\frac{3}{16}}_{0}(t)\mathcal{E}^{\frac{1}{2}}_2(t)(1+t)^{-\frac{13}{8}}+\mathcal{E}_{2}^{\frac{5}{16}}(t)\mathcal{E}_{0}^{\frac{3}{16}}(t) \mathcal{E}_{1}^{\frac{1}{2}}(t)(1+t)^{-\frac{13}{8}}
 \nonumber \\[2mm]
 &\displaystyle \quad
+\mathcal{E}_{2}^{\frac{1}{2}}(t)\mathcal{E}_{1}^{\frac{1}{4}}(t)\mathcal{E}_{0}^{\frac{1}{4}}(t) (1+t)^{-\frac{3}{2}}
 \nonumber \\[2mm]
 & \displaystyle \lesssim \mathcal{E}(t)(1+t)^{-\frac{3}{2}},
\end{align}
where we have used Lemma \ref{decaylemma}.  Similarly, for the term $T_{22}$, we utilize \eqref{two-two-deri-ani}  in Lemma \ref{lemma2.1} and \eqref{soboleve-ineq} to derive that
\begin{align*}
\displaystyle\|\nabla^2(\mathbf{u}\cdot\nabla\varrho)\|_{L^2}&\lesssim \|\nabla^2\mathbf{u}\cdot\nabla\varrho\|_{L^2}+\|\nabla\mathbf{u}\cdot\nabla^2\varrho\|_{L^2}+\|\mathbf{u}\cdot\nabla^3\varrho\|_{L^2}\nonumber\\[2mm]
&\displaystyle \lesssim\|\nabla^2\mathbf{u}\|^{\frac{1}{4}}_{L^2}\|\nabla^3\mathbf{u}\|^{\frac{3}{4}}_{L^2}\|\nabla\varrho\|^{\frac{1}{4}}_{L^2}\|\nabla^2\varrho\|^{\frac{3}{4}}_{L^2}+\|\nabla\mathbf{u}\|^{\frac{1}{4}}_{L^2}\|\nabla^3\mathbf{u}\|^{\frac{3}{4}}_{L^2}\|\nabla^2\varrho\|_{L^2}\nonumber\\[2mm]
&\displaystyle \quad+\|\mathbf{u}\|^{\frac{1}{4}}_{L^2}\|\partial_1\mathbf{u}\|^{\frac{1}{4}}_{L^2}\|\partial_2\mathbf{u}\|^{\frac{1}{4}}_{L^2}\|\partial_1\partial_2\mathbf{u}\|^{\frac{1}{4}}_{L^2}\|\nabla^3\varrho\|^{\frac{1}{2}}_{L^2}\|\nabla^3\partial_3\varrho\|^{\frac{1}{2}}_{L^2},
\end{align*}
which along with Lemma \ref{decaylemma} further yields that
\begin{align}\label{T22}
T_{22}&\lesssim\sup_{0\leq\tau\leq t}\Big [ (1+\tau)^{\frac{5}{16}}(\log(1+\tau))^{-\frac{5}{8}}\|\nabla^2\mathbf{u}\|^{\frac{5}{8}}_{L^2}\|\nabla^4\mathbf{u}\|^{\frac{3}{8}}_{L^2}(1+\tau)^{\frac{1}{4}}\|\nabla\varrho\|^{\frac{1}{4}}_{L^2}(1+\tau)^{\frac{3}{4}}\|\nabla^2\varrho\|^{\frac{3}{4}}_{L^2}\Big]\nonumber\\[2mm]
&~\qquad \quad\cdot\int^t_0(1+t-\tau)^{-m}(1+\tau)^{-\frac{21}{16}}\log(1+\tau)^{\frac{5}{8}}\mathrm{d}\tau\nonumber\\[2mm]
&\displaystyle \quad+\sup_{0\leq\tau\leq t}\Big[(1+\tau)^{\frac{1}{8}}\|\nabla\mathbf{u}\|^{\frac{1}{4}}_{L^2}(1+\tau)^{\frac{3}{16}}(\log(1+\tau))^{-\frac{3}{8}}\|\nabla^2\mathbf{u}\|^{\frac{3}{8}}_{L^2}\|\nabla^4\mathbf{u}\|^{\frac{3}{8}}_{L^2}(1+\tau)\|\nabla^2\varrho\|_{L^2}\Big]
 \nonumber \\[2mm]
 &~\displaystyle \qquad \cdot\int^t_0(1+t-\tau)^{-m}(1+\tau)^{-\frac{21}{16}}(\log(1+\tau))^{\frac{3}{8}}\mathrm{d}\tau\nonumber\\[2mm]
&\displaystyle \quad+\sup_{0\leq\tau\leq t}\Big[(1+\tau)^{\frac{1}{8}}\|\mathbf{u}\|^{\frac{1}{4}}_{L^2}(1+\tau)^{\frac{1}{4}}\|\partial_1\mathbf{u}\|^{\frac{1}{4}}_{L^2}(1+\tau)^{\frac{1}{4}}\|\partial_2\mathbf{u}\|^{\frac{1}{4}}_{L^2}(1+\tau)^{\frac{4}{15}}\|\partial_1\partial_2\mathbf{u}\|^{\frac{1}{4}}_{L^2}\Big.\nonumber\\[2mm]
&\displaystyle~ \quad\qquad \qquad\Big. \cdot(1+\tau)^{\frac{1}{4}}\|\nabla^2\varrho\|^{\frac{1}{4}}_{L^2}\|\nabla^4\varrho\|^{\frac{1}{4}}_{L^2}\|\nabla^3\partial_3\varrho\|^{\frac{1}{2}}_{L^2}\Big]\int^t_0(1+t-\tau)^{-m}(1+\tau)^{-\frac{137}{120}}\mathrm{d}\tau\nonumber\\[2mm]
& \displaystyle \lesssim  \mathcal{E}_{2}^{\frac{13}{16}}(t)\mathcal{E}_{0}^{\frac{3}{16}}(t) (1+t)^{-\frac{21}{16}}\log(1+t)^{\frac{3}{8}}
+\mathcal{E}_{2}^{\frac{5}{8}}(t)\mathcal{E}_{0}^{\frac{3}{8}}(t) (1+t)^{-\frac{137}{120}} \notag\\[2mm]
&\displaystyle \lesssim \mathcal{E}(t)(1+t)^{-\frac{137}{120}}.
\end{align}
Substituting \eqref{T21}--\eqref{T22} into \eqref{T20}, we have
\begin{align}\label{T2}
T_2 \lesssim \mathcal{E}(t)(1+t)^{-\frac{3}{2}}+\mathcal{E}(t)(1+t)^{-\frac{137}{120}} \lesssim \mathcal{E}(t)(1+t)^{-\frac{137}{120}}.
\end{align}
We now proceed to the estimate of $T_3$. Using the H\"older inequality again, we deduce that
\begin{align}\label{T3}
T_3
&\displaystyle \lesssim\sup_{0\leq \tau\leq t}\Big[(1+\tau)^{\frac{1}{2}}\|\mathbf{u}\|_{L^2}(1+\tau)^{\frac{1}{2}}\|\nabla\mathbf{u}\|_{L^2}\Big]\int^t_0(1+t-\tau)^{-\frac{7}{4}}(1+\tau)^{-1}\mathrm{d}\tau\nonumber\\[2mm]
&\displaystyle \quad+\sup_{0\leq \tau\leq t}\Big[(1+\tau)^{\frac{3}{4}}(\log(1+\tau))^{-1}\|\varrho\|_{L^2}(1+\tau)^{\frac{16}{15}}\|\Delta_h\mathbf{u}\|_{L^2}\Big]\nonumber\\[2mm]
&\displaystyle \quad\cdot\int^t_0(1+t-\tau)^{-\frac{7}{4}}(1+\tau)^{-\frac{109}{60}}\log(1+\tau)\mathrm{d}\tau\nonumber\\[2mm]
&\displaystyle \quad+\sup_{0\leq \tau\leq t}\Big[(1+\tau)^{\frac{3}{4}}(\log(1+\tau))^{-1}\|\varrho\|_{L^2}(1+\tau)\|\nabla\mathop{\mathrm{div}}\nolimits\mathbf{u}\|_{L^2}\Big]\nonumber\\[2mm]
&\displaystyle \quad\cdot\int^t_0(1+t-\tau)^{-\frac{7}{4}}(1+\tau)^{-\frac{7}{4}}\log(1+\tau)\mathrm{d}\tau\nonumber\\[2mm]
&\displaystyle \quad+\sup_{0\leq \tau\leq t}\Big[(1+\tau)^{\frac{3}{4}}(\log(1+\tau))^{-1}\|\varrho\|_{L^2}(1+\tau)\|\nabla\varrho\|_{L^2}\Big]\nonumber\\[2mm]
&\displaystyle \quad\cdot\int^t_0(1+t-\tau)^{-\frac{7}{4}}(1+\tau)^{-\frac{7}{4}}\log(1+\tau)\mathrm{d}\tau\nonumber\\[2mm]
& \lesssim\mathcal{E}_2(t)(1+t)^{-1}+\mathcal{E}_2(t)((1+t)^{-\frac{7}{4}}+(1+t)^{-\frac{7}{4}}\log(1+t))+\mathcal{E}^{\frac{1}{2}}_2(t)\mathcal{E}^{\frac{1}{2}}_1(t)(1+t)^{-\frac{7}{4}}\log(1+t)\nonumber\\[2mm]
&\displaystyle \lesssim\mathcal{E}(t)(1+t)^{-1}.
\end{align}
For the last term $T_4$, notice that
\begin{align*}
\displaystyle\|\nabla\mathop{\mathrm{div}}S_2\|_{L^2}\lesssim\|\nabla\mathop{\mathrm{div}}(\mathbf{u}\cdot\nabla \mathbf{u})\|_{L^2}+\|\nabla^2(g(\varrho)\Delta_h\mathbf{u})\|_{L^2}+\|\nabla^2(g(\varrho)\nabla \mathop{\mathrm{div}}\nolimits\mathbf{u})\|_{L^2}+\|\nabla^2(f(\varrho)\nabla\varrho)\|_{L^2}.
\end{align*}
Then the term $T_4$ can be split into four terms:
\begin{align*}
  T_4
&\displaystyle \lesssim\int^t_0(1+t-\tau)^{-m}\|\nabla\mathop{\mathrm{div}}(\mathbf{u}\cdot\nabla \mathbf{u})\|_{L^2}\mathrm{d}\tau
+ \int^t_0(1+t-\tau)^{-m}\|\nabla^2(g(\varrho)\Delta_h\mathbf{u})\|_{L^2}\mathrm{d}\tau \nonumber\\[2mm]
&\displaystyle \quad+ \int^t_0(1+t-\tau)^{-m}\|\nabla^2(g(\varrho)\nabla \mathop{\mathrm{div}}\nolimits\mathbf{u})\|_{L^2}\mathrm{d}\tau
+ \int^t_0(1+t-\tau)^{-m}\|\nabla^2(f(\varrho)\nabla\varrho)\|_{L^2}\mathrm{d}\tau\nonumber\\[2mm]
&=: T_{41}+T_{42}+T_{43}+T_{44}.
\end{align*}
By \eqref{two-two-deri-ani} and the Sobolev inequality \eqref{soboleve-ineq}, we obtain that
\begin{align*}
\|\nabla\mathop{\mathrm{div}}(\mathbf{u}\cdot\nabla \mathbf{u})\|_{L^2}
&\lesssim \|\nabla^2\mathbf{u}_h\|_{L^{\infty}}\|\nabla_h\mathbf{u}\|_{L^2}+\|\nabla^2u_3\|_{L^2}\|\partial_3\mathbf{u}\|_{L^{\infty}}+\|\nabla\mathbf{u}_h\|_{L^{\infty}}\| \nabla \nabla_h\mathbf{u}\|_{L^2}\nonumber\\[2mm]
&\quad+\|\nabla\partial_3\mathbf{u}\|_{L^{\infty}}\|\nabla u_3\|_{L^2}+\|\mathbf{u}\cdot\nabla^2\mathop{\mathrm{div}}\mathbf{u}\|_{L^2}\\[2mm]
&\displaystyle \lesssim\|\nabla^2\mathbf{u}_h\|^{\frac{1}{4}}_{L^2}\|\nabla^4\mathbf{u}_h\|^{\frac{3}{4}}_{L^2}\|\nabla_h\mathbf{u}\|_{L^2}+\|\partial_3\mathbf{u}\|^{\frac{1}{4}}_{L^2}\|\nabla^2\partial_3 \mathbf{u}\|^{\frac{3}{4}}_{L^2}\|\nabla^2u_3\|_{L^2}\\[2mm]
&\displaystyle \quad+\|\nabla\mathbf{u}_h\|^{\frac{1}{4}}_{L^2}\|\nabla^3\mathbf{u}_h\|^{\frac{3}{4}}_{L^2}\|\nabla_h\nabla\mathbf{u}\|_{L^2}+\|\nabla\partial_3\mathbf{u}\|^{\frac{1}{4}}_{L^2}\|\nabla^3\partial_3\mathbf{u}\|^{\frac{3}{4}}_{L^2}\|\nabla u_3\|_{L^2}\\[2mm]
&\displaystyle \quad+\|\mathbf{u}\|^{\frac{1}{4}}_{L^2}\|\partial_1\mathbf{u}\|^{\frac{1}{4}}_{L^2}\|\partial_2\mathbf{u}\|^{\frac{1}{4}}_{L^2}\|\partial_1\partial_2\mathbf{u}\|^{\frac{1}{4}}_{L^2}\|\nabla^2\mathop{\mathrm{div}}\mathbf{u}\|^{\frac{1}{2}}_{L^2}\|\nabla^3\partial_3\mathbf{u}\|^{\frac{1}{2}}_{L^2}.
\end{align*}
Thus, we have
\begin{align}\label{T41}
T_{41}&\lesssim \sup_{0\leq\tau\leq t} \Big[(1+\tau)^{\frac{1}{8}}(\log(1+\tau))^{-\frac{1}{4}}\|\nabla^2\mathbf{u}_h\|^{\frac{1}{4}}_{L^2}\|\nabla^4\mathbf{u}_h\|^{\frac{3}{4}}_{L^2}(1+\tau)\|\nabla_h\mathbf{u}\|_{L^2}\Big]\nonumber\\ &\displaystyle\qquad\cdot\int^t_0(1+t-\tau)^{-m}(1+\tau)^{-\frac{9}{8}}(\log(1+\tau))^{\frac{1}{4}}\mathrm{d}\tau\nonumber\\[2mm]
&\displaystyle \quad+ \sup_{0\leq\tau\leq t} \Big[(1+\tau)^{\frac{1}{8}}\|\nabla\mathbf{u}\|^{\frac{1}{4}}_{L^2}(1+\tau)^{\frac{3}{16}}(\log(1+\tau))^{-\frac{3}{8}}\|\nabla^2\mathbf{u}\|^{\frac{3}{8}}_{L^2}\|\nabla^4\mathbf{u}\|^{\frac{3}{8}}_{L^2}(1+\tau)\|\nabla^2u_3\|_{L^2}\Big]\nonumber\\
&\displaystyle\qquad\cdot\int^t_0(1+t-\tau)^{-m}(1+\tau)^{-\frac{21}{16}}(\log(1+\tau))^{\frac{3}{8}}\mathrm{d}\tau\nonumber\\[2mm]
&\displaystyle \quad+\sup_{0\leq\tau\leq t}\Big[ (1+\tau)^{\frac{1}{8}}\|\nabla\mathbf{u}_h\|^{\frac{1}{4}}_{L^2}(1+\tau)^{\frac{3}{16}}(\log(1+\tau))^{-\frac{3}{8}}\|\nabla^2\mathbf{u}_h\|^{\frac{3}{8}}_{L^2}\|\nabla^4\mathbf{u}_h\|^{\frac{3}{8}}_{L^2}(1+\tau)\|\nabla_h\nabla\mathbf{u}\|_{L^2}\Big]
 \nonumber \\[2mm]
 & ~\displaystyle \qquad \cdot\int^t_0(1+t-\tau)^{-m}(1+\tau)^{-\frac{21}{16}}(\log(1+\tau))^{\frac{3}{8}}\mathrm{d}\tau\nonumber\\
&\displaystyle \quad+\sup_{0\leq\tau\leq t} \Big [(1+\tau)^{\frac{1}{8}}(\log(1+\tau))^{-\frac{1}{4}}\|\nabla^2\mathbf{u}\|^{\frac{1}{4}}_{L^2}\|\nabla^4\mathbf{u}\|^{\frac{3}{4}}_{L^2}(1+\tau)\|\nabla u_3\|_{L^2}\Big]
 \nonumber \\[2mm]
 & ~\displaystyle \qquad \cdot\int^t_0(1+t-\tau)^{-m}(1+\tau)^{-\frac{9}{8}}(\log(1+\tau))^{\frac{1}{4}}\mathrm{d}\tau\nonumber\\[2mm]
&\displaystyle  \quad+\sup_{0\leq\tau\leq t}\Big[(1+\tau)^{\frac{1}{8}}\|\mathbf{u}\|^{\frac{1}{4}}_{L^2}(1+\tau)^{\frac{1}{4}}\|\partial_1\mathbf{u}\|^{\frac{1}{4}}_{L^2}(1+\tau)^{\frac{1}{4}}\|\partial_2\mathbf{u}\|^{\frac{1}{4}}_{L^2}(1+\tau)^{\frac{4}{15}}\|\partial_1\partial_2\mathbf{u}\|^{\frac{1}{4}}_{L^2}\Big.\nonumber\\[2mm]
&~~~\displaystyle \qquad\; \Big. \cdot(1+\tau)^{\frac{1}{4}}\|\nabla\mathop{\mathrm{div}}\mathbf{u}\|^{\frac{1}{4}}_{L^2}\|\nabla^4\mathbf{u}\|^{\frac{1}{4}}_{L^2}\|\nabla^3\partial_3\mathbf{u}\|^{\frac{1}{2}}_{L^2} \Big]\cdot\int^t_0(1+t-\tau)^{-m}(1+\tau)^{-\frac{137}{120}}\mathrm{d}\tau\nonumber\\[2mm]
&\lesssim\Big(\mathcal{E}^{\frac{5}{8}}_2(t)\mathcal{E}^{\frac{3}{8}}_0(t)+\mathcal{E}_{2}^{\frac{1}{8}}(t)\mathcal{E}_{0}^{\frac{3}{8}}(t)\mathcal{E}^{\frac{1}{2}}_1(t)\Big)(1+t)^{-\frac{9}{8}}(\log(1+t))^{\frac{1}{4}}+\mathcal{E}_{2}^{\frac{1}{2}}(t)\mathcal{E}_{1}^{\frac{1}{8}}(t)\mathcal{E}_{0}^{\frac{3}{8}}(t)(1+t)^{-\frac{137}{120}}\nonumber\\[2mm]
&\displaystyle \quad+ \Big(\mathcal{E}^{\frac{5}{16}}_2(t)\mathcal{E}^{\frac{3}{16}}_0(t)\mathcal{E}^{\frac{1}{2}}_1(t)+\mathcal{E}^{\frac{13}{16}}_2(t)\mathcal{E}^{\frac{3}{16}}_0(t) \Big)(1+t)^{-\frac{21}{16}}(\log(1+t))^{\frac{3}{8}}\nonumber\\
&\lesssim\mathcal{E}(t)(1+t)^{-\frac{9}{8}}(\log(1+t))^{\frac{1}{4}}.
\end{align}
For the term $T_{42}$, by the Sobolev inequality \eqref{soboleve-ineq}, we have
\begin{align*}
\displaystyle\|\nabla^2(\varrho\Delta_h\mathbf{u})\|_{L^2}& \lesssim\|\nabla^2\varrho\Delta_h\mathbf{u}\|_{L^2}+\|\nabla\varrho\nabla\Delta_h\mathbf{u}\|_{L^2}
+\|\varrho\nabla^2\Delta_h\mathbf{u}\|_{L^2}\nonumber\\[2mm]
& \displaystyle \lesssim \|\Delta _{h}\mathbf{u}\|_{L ^{\infty}}\|\nabla ^{2}\varrho\| _{L ^{2}}+ \|\nabla \varrho\| _{L ^{4}}\|\nabla \Delta _{h}\mathbf{u}\|_{L ^{4}}+ \|\varrho\|_{L ^{\infty}} \|\nabla ^{2}\Delta _{h}\mathbf{u}\|_{L ^{2}}
 \nonumber \\[2mm]
 &\displaystyle \lesssim\|\Delta_h\mathbf{u}\|^{\frac{1}{4}}_{L^2}\|\nabla^2\Delta_h\mathbf{u}\|^{\frac{3}{4}}_{L^2}\|\nabla^2\varrho\|_{L^2}+\|\nabla\varrho\|^{\frac{1}{4}}_{L^2}\|\nabla^2\varrho\|^{\frac{3}{4}}_{L^2}\|\nabla\Delta_h\mathbf{u}\|^{\frac{1}{4}}_{L^2}\|\nabla^2\Delta_h\mathbf{u}\|^{\frac{3}{4}}_{L^2}\nonumber\\[2mm]
&\displaystyle \quad+\|\nabla\varrho\|^{\frac{1}{2}}_{L^2}\|\nabla^2\varrho\|^{\frac{1}{2}}_{L^2}\|\nabla^2\Delta_h\mathbf{u}\|_{L^2}.
\end{align*}
Hence, we have
\begin{align}\label{T42}
T_{42}&\lesssim \sup_{0\leq\tau\leq t}\Big[(1+\tau)^{\frac{4}{15}}\|\Delta_h\mathbf{u}\|^{\frac{1}{4}}_{L^2}\|\nabla^2\Delta_h\mathbf{u}\|^{\frac{3}{4}}_{L^2}(1+\tau)\|\nabla^2\varrho\|_{L^2}\Big]\int^t_0(1+t-\tau)^{-m}(1+\tau)^{-\frac{19}{15}}\mathrm{d}\tau\nonumber\\[2mm]
&\displaystyle \quad+\sup_{0\leq \tau\leq t}\Big[(1+\tau)^{\frac{1}{4}}\|\nabla\varrho\|^{\frac{1}{4}}_{L^2}(1+\tau)^{\frac{3}{4}}\|\nabla^2\varrho\|^{\frac{3}{4}}_{L^2}(1+\tau)^{\frac{1}{8}}\|\nabla_h\nabla\mathbf{u}\|^{\frac{1}{8}}_{L^2}\|\nabla\nabla_h\Delta_h\mathbf{u}\|^{\frac{1}{8}}_{L^2}\|\nabla^2\Delta_h\mathbf{u}\|^{\frac{3}{4}}_{L^2}\Big]\nonumber\\[2mm]
&~\displaystyle \qquad\cdot\int^t_0(1+t-\tau)^{-m}(1+\tau)^{-\frac{9}{8}}\mathrm{d}\tau\nonumber\\[2mm]
&\displaystyle \quad+\sup_{0\leq \tau\leq t}\Big[(1+\tau)^{\frac{1}{2}}\|\nabla\varrho\|^{\frac{1}{2}}_{L^2}(1+\tau)^{\frac{1}{2}}\|\nabla^2\varrho\|^{\frac{1}{2}}_{L^2}\Big]
 \nonumber \\[2mm]
 & \displaystyle ~\qquad \cdot\Big(\int^t_0[(1+t-\tau)^{-m}(1+\tau)^{-\frac{16}{15}}]^2\mathrm{d}\tau\Big)^{\frac{1}{2}}\int^t_0(1+\tau)^{\frac{2}{15}}\|\nabla^2\Delta_h\mathbf{u}\|^2_{L^2}\mathrm{d}\tau\nonumber\\
&\displaystyle \lesssim\mathcal{E}_{2}^{\frac{5}{8}}(t)\mathcal{E}_{0}^{\frac{3}{8}}(t) (1+t)^{-\frac{19}{15}}
+\mathcal{E}_{2}^{\frac{9}{16}}(t)\mathcal{E}_{0}^{\frac{7}{16}}(t)(1+t)^{-\frac{9}{8}}
+\mathcal{E}_{2}^{\frac{1}{2}}(t)\mathcal{E}_{1}^{\frac{1}{2}}(t)(1+t)^{-\frac{16}{15}}
 \notag\\[2mm]
&\displaystyle \lesssim\mathcal{E}(t)(1+t)^{-\frac{16}{15}}.
\end{align}
Similarly, we have
\begin{align}\label{T43}
T_{43}&=\int^t_0(1+t-\tau)^{-m}\|\nabla^2(g(\varrho)\nabla \mathop{\mathrm{div}}\nolimits\mathbf{u})\|_{L^2}\mathrm{d}\tau\nonumber\\[2mm]
& \lesssim\mathcal{E}(t)(1+t)^{-\frac{16}{15}}.
\end{align}
For the term $T_{44}$,  using the Sobolev inequality \eqref{soboleve-ineq} again, we get
\begin{align*}
\displaystyle&\|\nabla^2(\varrho\nabla\varrho)\|_{L^2}\\[2mm]
&~\displaystyle \lesssim \|\nabla^2\varrho\nabla\varrho\|_{L^2}
+\|\varrho\nabla^3\varrho\|_{L^2}\\[2mm]
&~\displaystyle \lesssim \|\nabla\varrho\|^{\frac{1}{4}}_{L^2}\|\nabla^3\varrho\|^{\frac{3}{4}}_{L^2}\|\nabla^2\varrho\|_{L^2}
+\|\nabla\varrho\|_{L^2}^{\frac{1}{2}}\|\nabla^2\varrho\|_{L^2}^{\frac{1}{2}}\|\nabla^3\varrho\|_{L^2}.
\end{align*}
Then, we have
\begin{align}\label{T44}
\displaystyle T_{44}
& \lesssim\sup_{0\leq\tau\leq t}\Big[(1+\tau)^{\frac{1}{4}}\|\nabla\varrho\|^{\frac{1}{4}}_{L^2}(1+\tau)^{\frac{11}{8}}\|\nabla^2\varrho\|^{\frac{11}{8}}_{L^2}\|\nabla^4\varrho\|^{\frac{3}{8}}_{L^2}\Big]
 \nonumber \\
 & \displaystyle ~\qquad \cdot\int^t_0(1+t-\tau)^{-m}(1+\tau)^{-\frac{13}{8}}\mathrm{d}\tau\nonumber\\
&\displaystyle \quad+\sup_{0\leq\tau\leq t}\Big[(1+\tau)^{\frac{1}{2}}\|\nabla\varrho\|^{\frac{1}{2}}_{L^2}(1+\tau)\|\nabla^2\varrho\|_{L^2}\|\nabla^4\varrho\|^{\frac{1}{2}}_{L^2}\Big]
 \nonumber \\
 & \displaystyle ~\qquad \cdot\int^t_0(1+t-\tau)^{-m}(1+\tau)^{-\frac{3}{2}}\mathrm{d}\tau\nonumber\\
&\displaystyle \lesssim\mathcal{E}_{2}^{\frac{13}{16}}(t)\mathcal{E}_{0}^{\frac{3}{16}}(t) (1+t)^{-\frac{13}{8}}
+\mathcal{E}_{2}^{\frac{3}{4}}(t)\mathcal{E}_{0}^{\frac{1}{4}}(t)(1+t)^{-\frac{3}{2}}
 \notag\\[2mm]
&\displaystyle \lesssim \mathcal{E}(t)(1+t)^{-\frac{3}{2}},
\end{align}
where we have used \eqref{sec4-appri-conlu} and Lemma \ref{decaylemma}. Combining
\eqref{T41}--\eqref{T44}, we then arrive at
\begin{align}\label{T4}
\displaystyle T_4 &\lesssim \mathcal{E}(t)(1+t)^{-\frac{9}{8}}(\log(1+t))^{\frac{1}{4}}+\mathcal{E}(t)(1+t)^{-\frac{16}{15}}+\mathcal{E}(t)(1+t)^{-\frac{3}{2}}
 \nonumber \\[2mm]
 & \displaystyle \lesssim \mathcal{E}(t)(1+t)^{-\frac{16}{15}}.
\end{align}
Hence, substituting \eqref{T1}, \eqref{T2}, \eqref{T3} and \eqref{T4} into \eqref{divH1}, we derive that
\begin{align*}
\displaystyle (1+t)\|\nabla\mathop{\mathrm{div}}\nolimits\mathbf{u}\|_{L^2} \lesssim\mathcal{E}(t)+(\|(\rho_0,\mathbf{u}_0)\|_{L^1}+\|(\nabla^2\varrho_0,\nabla^2\mathbf{u}_0)\|_{L^2}).
\end{align*}
We complete the proof of Lemma \ref{divulemma1}.
%
%
%
%
\end{proof}

\subsection{Coupling estimates} 
\label{sub:coupling_estimates}


In the following lemma, we shall derive the dissipation estimates of the high-order derivatives of $ \mathbf{u} $ which are of importance in deriving the estimate of $ \mathcal{E}_{2}(t) $.
\begin{lemma}\label{haosan}
Let $(\varrho,u)$ be a smooth solution to the problem \eqref{new}. Then, we have
\begin{gather}\label{con-dissip-lem}
\displaystyle\int^t_0(1+\tau)^{\frac{2}{15}}\|(\nabla\mathop{\mathrm{div}}\nolimits\mathbf{u},\Delta_h\mathbf{u})\|^2_{H^2}\mathrm{d}\tau\lesssim\mathcal{E}_0(0)+\mathcal{E}^{\frac{3}{2}}(t).
\end{gather}
\end{lemma}
\begin{proof}
Applying $\nabla$ and $\mathop{\mathrm{div}}\nolimits $ to $ \eqref{new} _{1} $ and $ \eqref{new} _{2} $, respectively, and taking the $H^2$-inner product of the resulting equations with $ \nabla\varrho $ and $ \mathop{\mathrm{div}}\nolimits \mathbf{u} $, respectively, we get
\begin{align}\label{diff-dispa-divu}
&\displaystyle \frac{1}{2}\frac{\mathrm{d}}{\mathrm{d}t} (\|\mathop{\mathrm{div}}\nolimits \mathbf{u}\|^2_{H^2}+\|\nabla\varrho\|^2_{H^2})+  (\bar{\mu}\|\nabla_h\mathop{\mathrm{div}}\nolimits \mathbf{u}\|^2_{H^2}+(\bar{\mu}+\bar{\zeta})\|\nabla \mathop{\mathrm{div}}\nolimits \mathbf{u}\|^2_{H^2})
 \nonumber \\[2mm]
 &~\displaystyle = - (S_2,\nabla \mathop{\mathrm{div}}\nolimits \mathbf{u})_{H^2}+(\nabla S_1,\nabla\varrho)_{H^2}=:Q _{1}+Q  _{2}.
\end{align}
Multiplying \eqref{diff-dispa-divu} by $ (1+t)^{\frac{2}{15}} $, we get, after integrating the resulting identity over $ (0,t) $, that
\begin{align}\label{dissip-integral}
&\displaystyle \makebox[-4pt]{~}  (1+t)^{\frac{2}{15}}(\|\mathop{\mathrm{div}}\nolimits \mathbf{u}\|^2_{H^2}+\|\nabla\varrho\|^2_{H^2})+\int^t_0(1+\tau)^{\frac{2}{15}}(\|\nabla_h\mathop{\mathrm{div}}\nolimits \mathbf{u}\|^2_{H^2}+\|\nabla \mathop{\mathrm{div}}\nolimits \mathbf{u}\|^2_{H^2})\mathrm{d}\tau
 \nonumber \\[2mm]
 &~\displaystyle \makebox[-4pt]{~} \lesssim \|\mathop{\mathrm{div}}\nolimits \mathbf{u}_{0}\| _{H ^{2}}^{2}+ \|\nabla \varrho _{0}\|_{H ^{2}}^{2}+ \int _{0}^{t}(\|\mathop{\mathrm{div}}\nolimits \mathbf{u}\|^2_{H^2}+\|\nabla\varrho\|^2_{H^2})\mathrm{d}\tau + \int _{0}^{t}(1+\tau)^{\frac{2}{15}}(Q _{1}+Q _{2})\mathrm{d}\tau
  \nonumber \\[2mm]
  &~\displaystyle  \makebox[-4pt]{~}\lesssim\mathcal{E}_{0}(0)+ \int _{0}^{t}(1+\tau)^{\frac{2}{15}}Q _{1}\mathrm{d}\tau+ \int _{0}^{t}(1+\tau)^{\frac{2}{15}}Q _{2}\mathrm{d}\tau,
\end{align}
where we have used \eqref{E-0-esti-con}. Notice that
\begin{align*}
Q_1&=-\int_{\mathbb{R}^3}S_2\cdot\nabla \mathop{\mathrm{div}}\nolimits \mathbf{u}\mathrm{d}x-\int_{\mathbb{R}^3}\nabla S_2:\nabla^2\mathop{\mathrm{div}}\nolimits \mathbf{u}\mathrm{d}x-\int_{\mathbb{R}^3}\nabla^2 S_2:\nabla^3\mathop{\mathrm{div}}\nolimits \mathbf{u}\mathrm{d}x=Q_{11}+Q_{12}+Q_{13}.
\end{align*}
Then we have
\begin{align}\label{Q-1-split}
\displaystyle  \int _{0}^{t}(1+\tau)^{\frac{2}{15}} Q _{1}\mathrm{d}\tau= \int _{0}^{t}(1+\tau)^{\frac{2}{15}} Q _{11}\mathrm{d}\tau+\int _{0}^{t}(1+\tau)^{\frac{2}{15}} Q _{12}\mathrm{d}\tau+\int _{0}^{t}(1+\tau)^{\frac{2}{15}}Q_{13}\mathrm{d}\tau.
\end{align}
In virtue of \eqref{three-one-deri-ani}, \eqref{soboleve-ineq}, \eqref{sec4-appri-conlu} and the H\"older inequality, we get
\begin{align*}
Q_{11}&=-\int _{\mathbb{R}^{3}}\Big(\mathbf{u}\cdot\nabla \mathbf{u}+g(\varrho)\Delta_h \mathbf{u}+g(\varrho)\nabla \mathop{\mathrm{div}}\nolimits \mathbf{u}+f(\varrho)\nabla\varrho \Big)\cdot\nabla \mathop{\mathrm{div}}\nolimits \mathbf{u}\mathrm{d}x\nonumber\\[2mm]
&\lesssim \|\mathbf{u}\|^{\frac{1}{2}}_{L^2}\|\partial_1\mathbf{u}\|^{\frac{1}{2}}_{L^2}\|\nabla \mathbf{u}\|^{\frac{1}{2}}_{L^2}\|\nabla \partial_2\mathbf{u}\|^{\frac{1}{2}}_{L^2}\|\nabla \mathop{\mathrm{div}}\nolimits \mathbf{u}\|^{\frac{1}{2}}_{L^2}\|\nabla \partial_3\mathop{\mathrm{div}}\nolimits \mathbf{u}\|^{\frac{1}{2}}_{L^2}\nonumber\\[2mm]
&\displaystyle \quad+\|\varrho\|_{H^2}\|\Delta_h\mathbf{u}\|_{L^2}\|\nabla \mathop{\mathrm{div}}\nolimits \mathbf{u}\|_{L^2}+\|\varrho\|_{H^2}\|\nabla \mathop{\mathrm{div}}\nolimits \mathbf{u}\|^2_{L^2}+\|\varrho\|_{H^2}\|\nabla\varrho\|_{L^2}\|\nabla \mathop{\mathrm{div}}\nolimits \mathbf{u}\|_{L^2}.
\end{align*}
Therefore,
\begin{align}\label{Q11**}
&\int^t_0(1+\tau)^{\frac{2}{15}}Q_{11}\mathrm{d}\tau\nonumber\\[2mm]
&\displaystyle \lesssim\sup_{0\leq \tau\leq t}\Big[(1+\tau)^{\frac{1}{2}}\|\mathbf{u}\|_{H^1}\Big]\int^t_0(1+\tau)^{\frac{1}{15}}\|\nabla \mathop{\mathrm{div}}\nolimits \mathbf{u}\|_{H^1}(\|\partial_1\mathbf{u}\|^{\frac{1}{2}}_{L^2}\|\partial_2\nabla \mathbf{u}\|^{\frac{1}{2}}_{L^2})\mathrm{d}\tau\nonumber\\[2mm]
&\displaystyle \quad+\sup_{0\leq\tau\leq t}\Big[(1+\tau)^{\frac{3}{4}}\log^{-1}(1+\tau)\|\varrho\|_{H^2}\Big]\int^t_0(1+\tau)^{\frac{1}{15}}\|\nabla \mathop{\mathrm{div}}\nolimits \mathbf{u}\|_{L^2}\|\Delta_h\mathbf{u}\|_{L^2}\mathrm{d}\tau\nonumber\\
&\displaystyle \quad+\sup_{0\leq\tau\leq t}\Big[(1+\tau)^{\frac{3}{4}}\log^{-1}(1+\tau)\|\varrho\|_{H^2}\Big]\int^t_0(1+\tau)^{\frac{1}{15}}\|\nabla \mathop{\mathrm{div}}\nolimits \mathbf{u}\|^2_{L^2}\mathrm{d}\tau\nonumber\\[2mm]
&\displaystyle \quad+\sup_{0\leq\tau\leq t}\Big[(1+\tau)^{\frac{3}{4}}\log^{-1}(1+\tau)\|\varrho\|_{H^2}\Big]\int^t_0(1+\tau)^{\frac{1}{15}}\|\nabla \mathop{\mathrm{div}}\nolimits \mathbf{u}\|_{L^2}\|\nabla\varrho\|_{L^2}\mathrm{d}\tau\nonumber\\[2mm]
&\displaystyle \lesssim\mathcal{E}^{\frac{1}{2}}_2(t)\mathcal{E}^{\frac{1}{2}}_1(t)\mathcal{E}^{\frac{1}{2}}_0(t)+\mathcal{E}^{\frac{1}{2}}_2(t)\mathcal{E}_1(t)\nonumber\\[2mm]
&\displaystyle \lesssim\mathcal{E}^{\frac{3}{2}}(t).
\end{align}

Similarly, for the term $Q_{12}$, we get
\begin{align*}%
Q_{12}&=-\int _{\mathbb{R}^{3}}\nabla \Big (\mathbf{u}\cdot\nabla \mathbf{u}+g(\varrho)\Delta_h \mathbf{u}+g(\varrho)\nabla \mathop{\mathrm{div}}\nolimits \mathbf{u}+f(\varrho)\nabla\varrho \Big):\nabla^2\mathop{\mathrm{div}}\nolimits \mathbf{u}\mathrm{d}x\nonumber\\
&=Q_{{12}_1}+Q_{{12}_2}+Q_{{12}_3}+Q_{{12}_4}.
\end{align*}
Then we have
\begin{align}\label{Q12*}
 \displaystyle \int _{0}^{t}(1+\tau)^{\frac{2}{15}} Q _{12}\mathrm{d}\tau&= \int _{0}^{t}(1+\tau)^{\frac{2}{15}} Q _{12 _{1}}\mathrm{d}\tau+\int _{0}^{t}(1+\tau)^{\frac{2}{15}} Q _{12 _{2}}\mathrm{d}\tau+\int _{0}^{t}(1+\tau)^{\frac{2}{15}} Q _{12 _{3}}\mathrm{d}\tau\nonumber\\
 &~~~+\int^t_0(1+\tau)^{\frac{2}{15}}Q_{{12}_4}\mathrm{d}\tau.
 \end{align}
Using \eqref{three-two-deri-ani}, we get
\begin{align*}
Q_{{12}_1}&\lesssim\|\nabla \mathbf{u}\|^{\frac{1}{4}}_{L^2}\|\nabla \partial_1\mathbf{u}\|^{\frac{1}{4}}_{L^2}\|\nabla\partial_3 \mathbf{u}\|^{\frac{1}{4}}_{L^2}\|\nabla \partial_1\partial_3\mathbf{u}\|^{\frac{1}{4}}_{L^2}\|\nabla \mathbf{u}\|^{\frac{1}{2}}_{L^2}\|\nabla \partial_2\mathbf{u}\|^{\frac{1}{2}}_{L^2}\|\nabla^2\mathop{\mathrm{div}}\nolimits \mathbf{u}\|_{L^2}\nonumber\\[2mm]
&\displaystyle \quad+\|\mathbf{u}\|^{\frac{1}{4}}_{L^2}\|\partial_1\mathbf{u}\|^{\frac{1}{4}}_{L^2}\|\partial_3\mathbf{u}\|^{\frac{1}{4}}_{L^2}\|\partial_1\partial_3\mathbf{u}\|^{\frac{1}{4}}_{L^2}\|\nabla^2\mathbf{u}\|^{\frac{1}{2}}_{L^2}\|\nabla^2\partial_2\mathbf{u}\|^{\frac{1}{2}}_{L^2}\|\nabla^2\mathop{\mathrm{div}}\nolimits \mathbf{u}\|_{L^2}.
\end{align*}
It thus holds that
\begin{align}\label{Q121}
&\displaystyle\int^t_0(1+\tau)^{\frac{2}{15}}Q_{{12}_1}\mathrm{d}\tau\nonumber\\[2mm]
&~\displaystyle\lesssim\sup_{0\leq \tau\leq t}\Big[(1+\tau)^{\frac{3}{8}}\|\nabla  \mathbf{u}\|^{\frac{3}{4}}_{L^2}(1+\tau)^{\frac{1}{8}}\log^{-\frac{1}{4}}(1+\tau)\|\nabla \partial_3\mathbf{u}\|^{\frac{1}{4}}_{L^2}\Big]\nonumber\\[2mm]
&~\displaystyle ~\qquad\cdot\int^t_0(1+\tau)^{\frac{1}{15}}\|\nabla^2 \mathop{\mathrm{div}}\nolimits \mathbf{u}\|_{L^2}\|\nabla\partial_1\mathbf{u}\|^{\frac{1}{4}}_{L^2}\|\nabla\partial_2\mathbf{u}\|^{\frac{1}{2}}_{L^2}\|\nabla\partial_1\partial_3\mathbf{u}\|^{\frac{1}{4}}_{L^2}\mathrm{d}\tau\nonumber\\[2mm]
&~\displaystyle \quad+\sup_{0\leq \tau\leq t}\Big[(1+\tau)^{\frac{1}{8}}\|\mathbf{u}\|^{\frac{1}{4}}_{L^2}(1+\tau)^{\frac{1}{8}}\|\partial_3\mathbf{u}\|^{\frac{1}{4}}_{L^2}(1+\tau)^{\frac{1}{4}}\log^{-\frac{1}{2}}(1+\tau)\|\nabla^2\mathbf{u}\|^{\frac{1}{2}}_{L^2}\Big]\nonumber\\[2mm]
&~\displaystyle~\qquad\cdot\int^t_0(1+\tau)^{\frac{1}{15}}\|\nabla^2\mathop{\mathrm{div}}\nolimits \mathbf{u}\|_{L^2}\|\partial_1\mathbf{u}\|^{\frac{1}{4}}_{L^2}\|\partial_1\partial_3\mathbf{u}\|^{\frac{1}{4}}_{L^2}\|\nabla^2\partial_2\mathbf{u}\|^{\frac{1}{2}}_{L^2}\mathrm{d}\tau\nonumber\\[2mm]
&~\displaystyle\lesssim \mathcal{E}^{\frac{1}{2}}_2(t)\mathcal{E}^{\frac{1}{2}}_1(t)\mathcal{E}^{\frac{1}{2}}_0(t)\nonumber\\[2mm]
 &~\displaystyle\lesssim \mathcal{E}^{\frac{3}{2}}(t).
\end{align}
Similarly, we get in virtue of the Sobolev inequality \eqref{soboleve-ineq} that
\begin{align*}
\displaystyle\makebox[-8pt]{~} Q_{{12}_2}&\lesssim \|\nabla\varrho\|^{\frac{1}{2}}_{L^2}\|\nabla^2\varrho\|^{\frac{1}{2}}_{L^2}\|\nabla\Delta_h\mathbf{u}\|_{L^2}\|\nabla^2 \mathop{\mathrm{div}}\nolimits \mathbf{u}\|_{L^2}+\|\Delta_h\mathbf{u}\|_{L^2}\|\nabla^2\varrho\|^{\frac{1}{2}}_{L^2}\|\nabla^3\varrho\|^{\frac{1}{2}}_{L^2}\|\nabla^2 \mathop{\mathrm{div}}\nolimits \mathbf{u}\|_{L^2},\nonumber\\[2mm]
\displaystyle \makebox[-8pt]{~}Q_{{12}_3}& \lesssim\|\nabla\varrho\|^{\frac{1}{2}}_{L^2}\|\nabla^2\varrho\|^{\frac{1}{2}}_{L^2}\|\nabla^2\mathop{\mathrm{div}}\nolimits \mathbf{u}\|^2_{L^2}+\|\nabla^2\varrho\|^{\frac{1}{2}}_{L^2}\|\nabla^3\varrho\|^{\frac{1}{2}}_{L^2}\|\nabla \mathop{\mathrm{div}}\nolimits \mathbf{u}\|_{L^2}\|\nabla^2\mathop{\mathrm{div}}\nolimits \mathbf{u}\|_{L^2},
 \end{align*}
 and that
 \begin{align*}
 \displaystyle  Q_{{12}_4}& \lesssim\|\nabla^2\varrho\|^{\frac{1}{2}}_{L^2}\|\nabla^3\varrho\|^{\frac{1}{2}}_{L^2}\|\nabla\varrho\|_{L^2}\|\nabla^2 \mathop{\mathrm{div}}\nolimits \mathbf{u}\|_{L^2}+\|\nabla\varrho\|^{\frac{1}{2}}_{L^2}\|\nabla^2\varrho\|^{\frac{1}{2}}_{L^2}\|\nabla^2\varrho\|_{L^2}\|\nabla^2 \mathop{\mathrm{div}}\nolimits \mathbf{u}\|_{L^2}.
 \end{align*}
 Therefore we can further derive that
\begin{align}\label{Q122}
&\int^t_0(1+\tau)^{\frac{2}{15}}Q_{{12}_2}\mathrm{d}\tau\nonumber\\[2mm]
&~\displaystyle \lesssim\sup_{0\leq \tau\leq t}\Big[(1+\tau)^{\frac{1}{2}}\|\nabla\varrho\|^{\frac{1}{2}}_{L^2}(1+\tau)^{\frac{1}{2}}\|\nabla^2\varrho\|^{\frac{1}{2}}_{L^2}\Big]\int^t_0(1+\tau)^{\frac{1}{15}}\|\nabla^2 \mathop{\mathrm{div}}\nolimits \mathbf{u}\|_{L^2}\|\nabla\Delta_h\mathbf{u}\|_{L^2}\mathrm{d}\tau\nonumber\\[2mm]
&~\displaystyle \quad+\sup_{0\leq \tau\leq t}\Big[(1+\tau)^{\frac{16}{15}}\|\Delta_h\mathbf{u}\|_{L^2}\Big]\int^t_0(1+\tau)^{\frac{1}{15}}\|\nabla^2 \mathop{\mathrm{div}}\nolimits \mathbf{u}\|_{L^2}\|\nabla^2\varrho\|^{\frac{1}{2}}_{L^2}\|\nabla^3\varrho\|^{\frac{1}{2}}_{L^2}\mathrm{d}\tau\nonumber\\[2mm]
&~\displaystyle\lesssim\mathcal{E}^{\frac{1}{2}}_2(t)\mathcal{E}^{\frac{1}{2}}_1(t)\mathcal{E}^{\frac{1}{2}}_0(t)\nonumber\\[2mm]
 &~\lesssim \mathcal{E}^{\frac{3}{2}}(t),
\end{align}
\begin{align}\label{Q123}
&\int^t_0(1+\tau)^{\frac{2}{15}}Q_{{12}_3}\mathrm{d}\tau\nonumber\\[2mm]
&\leq\sup_{0\leq \tau\leq t}\Big[(1+\tau)^{\frac{1}{2}}\|\nabla\varrho\|^{\frac{1}{2}}_{L^2}(1+\tau)^{\frac{1}{2}}\|\nabla^2\varrho\|^{\frac{1}{2}}_{L^2}\Big]\int^t_0(1+\tau)^{\frac{2}{15}}\|\nabla^2 \mathop{\mathrm{div}}\nolimits \mathbf{u}\|^2_{L^2}\mathrm{d}\tau\nonumber\\[2mm]
&\displaystyle \quad+\sup_{0\leq\tau\leq t}\Big[(1+\tau)\|\nabla \mathop{\mathrm{div}}\nolimits \mathbf{u}\|_{L^2}\Big]\int^t_0(1+\tau)^{\frac{1}{15}}\|\nabla^2\mathop{\mathrm{div}}\nolimits \mathbf{u}\|_{L^2}\|\nabla^2\varrho\|^{\frac{1}{2}}_{L^2}\|\nabla^3\varrho\|^{\frac{1}{2}}_{L^2}\mathrm{d}\tau\nonumber\\[2mm]
 &\lesssim \mathcal{E}^{\frac{1}{2}}_2(t)\mathcal{E}_1(t)+\mathcal{E}_{1}(t)\mathcal{E}^{\frac{1}{2}}_0(t)\nonumber\\[2mm]
 &\lesssim \mathcal{E}^{\frac{3}{2}}(t),
\end{align}
and
\begin{align}\label{Q124}
&\int^t_0(1+\tau)^{\frac{2}{15}}Q_{{12}_4}\mathrm{d}\tau\nonumber\\[2mm]
&~\displaystyle \lesssim\sup_{0\leq \tau\leq t}\Big[(1+\tau)\|\nabla\varrho\|_{L^2}\Big]\int^t_0(1+\tau)^{\frac{1}{15}}\|\nabla^2 \mathop{\mathrm{div}}\nolimits \mathbf{u}\|_{L^2}\|\nabla^2\varrho\|^{\frac{1}{2}}_{L^2}\|\nabla^3\varrho\|^{\frac{1}{2}}_{L^2}\mathrm{d}\tau\nonumber\\[2mm]
&~\quad+\sup_{0\leq \tau\leq t}\Big[(1+\tau)^{\frac{1}{2}}\|\nabla\varrho\|^{\frac{1}{2}}_{L^2}(1+\tau)^{\frac{1}{2}}\|\nabla^2\varrho\|^{\frac{1}{2}}_{L^2}\Big]\int^t_0(1+\tau)^{\frac{1}{15}}\|\nabla^2 \mathop{\mathrm{div}}\nolimits \mathbf{u}\|_{L^2}\|\nabla^2\varrho\|_{L^2}\mathrm{d}\tau\nonumber\\[2mm]
&~\displaystyle\lesssim\mathcal{E}_2^{\frac{1}{2}}(t)\mathcal{E}_1^{\frac{1}{2}}(t)\mathcal{E}^{\frac{1}{2}}_0(t)\nonumber\\[2mm]
&~\displaystyle\lesssim \mathcal{E}^{\frac{3}{2}}(t).
\end{align}
Inserting \eqref{Q121}--\eqref{Q124} into \eqref{Q12*}, we then get
\begin{align}\label{Q12**}
\int^t_0(1+\tau)^{\frac{2}{15}}Q_{12}\mathrm{d}\tau \lesssim\mathcal{E}^{\frac{3}{2}}(t).
\end{align}
Similarly, for the term $Q_{13}$, by using the same method as deriving the estimates of $Q_{12}$, we also have
\begin{align*}
\int^t_0(1+\tau)^{\frac{2}{15}}Q_{13}\mathrm{d}\tau \lesssim\mathcal{E}^{\frac{3}{2}}(t).
\end{align*}
This along with \eqref{Q-1-split},\eqref{Q11**} and \eqref{Q12**} yields that
\begin{align}\label{Q-1-con}
\displaystyle \int _{0}^{t}(1+\tau)^{\frac{2}{15}}Q _{1}\mathrm{d}\tau   \lesssim\mathcal{E}^{\frac{3}{2}}(t).
\end{align}
To bound the last term on the right hand side of \eqref{dissip-integral}, we split $Q_2$ into three parts
\begin{align*}
Q_2&=\int  _{\mathbb{R}^{3}}\nabla S_1 \cdot\nabla\varrho \mathrm{d}x-\int _{\mathbb{R}^{3}}\nabla S_1 \cdot \nabla\Delta\varrho \mathrm{d}x-\int _{\mathbb{R}^{3}}\nabla^2 S_1 \cdot \nabla^2\Delta\varrho \mathrm{d}x =:Q_{21}+Q_{22}+Q_{23}.
\end{align*}
And thus,
\begin{align}\label{Q-2-SPLIT}
\displaystyle \int _{0}^{t}(1+\tau)^{\frac{2}{15}}Q _{2}\mathrm{d}\tau=  \int _{0}^{t}(1+\tau)^{\frac{2}{15}}Q _{21}\mathrm{d}\tau+\int _{0}^{t}(1+\tau)^{\frac{2}{15}}Q _{22}\mathrm{d}\tau+\int _{0}^{t}(1+\tau)^{\frac{2}{15}}Q _{23}\mathrm{d}\tau.
\end{align}
By \eqref{three-one-deri-ani}, \eqref{soboleve-ineq} and integration by parts, we have
\begin{align*}
\makebox[-4pt]{~}Q_{21}
&=-\int _{\mathbb{R}^{3}}\nabla(\varrho \mathop{\mathrm{div}}\nolimits \mathbf{u})\cdot\nabla\varrho \mathrm{d}x-\int _{\mathbb{R}^{3}}\nabla(\mathbf{u}\cdot\nabla\varrho)\cdot\nabla\varrho \mathrm{d}x\nonumber\\[2mm]
&\lesssim \int _{\mathbb{R}^{3}}\vert \varrho\vert \vert \nabla \mathop{\mathrm{div}}\nolimits \mathbf{u}\vert \vert \nabla\varrho \vert\mathrm{d}x+\int _{\mathbb{R}^{3}} \vert \nabla \mathbf{u}\vert \vert \nabla\varrho\vert ^{2}\mathrm{d}x+\int _{\mathbb{R}^{3}} |\nabla\varrho|^2 \vert \mathop{\mathrm{div}}\nolimits \mathbf{u}\vert \mathrm{d}x\nonumber\\[2mm]
&\lesssim \|\varrho\|^{\frac{1}{2}}_{L^2}\|\partial_1\varrho\|^{\frac{1}{2}}_{L^2}\|\nabla\varrho\|^{\frac{1}{2}}_{L^2}\|\nabla\partial_2\varrho\|^{\frac{1}{2}}_{L^2}\|\nabla \mathop{\mathrm{div}}\nolimits \mathbf{u}\|^{\frac{1}{2}}_{L^2}\nonumber\|\nabla \partial_3\mathop{\mathrm{div}}\nolimits \mathbf{u}\|^{\frac{1}{2}}_{L^2}
+\|\nabla\varrho\|^2_{L^4}\|\nabla \mathbf{u}\|_{L^2}\nonumber\\[2mm]
& \displaystyle \quad+\|\nabla\varrho\|^2_{L^2}\|\nabla\mathop{\mathrm{div}}\nolimits \mathbf{u}\|^{\frac{1}{2}}_{L^2}\|\nabla^2\mathop{\mathrm{div}}\nolimits \mathbf{u}\|^{\frac{1}{2}}_{L^2}
 \nonumber \\[2mm]
 &\lesssim\|\varrho\|^{\frac{1}{2}}_{L^2}\|\partial_1\varrho\|^{\frac{1}{2}}_{L^2}\|\nabla\varrho\|^{\frac{1}{2}}_{L^2}\|\nabla\partial_2\varrho\|^{\frac{1}{2}}_{L^2}\|\nabla \mathop{\mathrm{div}}\nolimits \mathbf{u}\|^{\frac{1}{2}}_{L^2}\|\nabla \partial_3\mathop{\mathrm{div}}\nolimits \mathbf{u}\|^{\frac{1}{2}}_{L^2}+\|\nabla\varrho\|^{\frac{1}{2}}_{L^2}\|\nabla^2\varrho\|^{\frac{3}{2}}_{L^2}\|\nabla \mathbf{u}\|_{L^2}
  \nonumber \\[2mm]
  & \displaystyle \quad+\|\nabla\varrho\|_{L^2}(\|\nabla\mathop{\mathrm{div}}\nolimits \mathbf{u}\|_{L^2}+\|\nabla^2\mathop{\mathrm{div}}\nolimits \mathbf{u}\|_{L^2})\|\nabla\varrho\|_{L^2}.
\end{align*}
This further implies that
\begin{align}\label{Q21**}
&\displaystyle\int^t_0(1+\tau)^{\frac{2}{15}}Q_{21}\mathrm{d}\tau\nonumber\\
&~\displaystyle\lesssim \sup_{0\leq \tau\leq t}\Big[(1+\tau)^{\frac{3}{8}}\log^{-\frac{1}{2}}(1+\tau)\|\varrho\|^{\frac{1}{2}}_{L^2}(1+\tau)^{\frac{1}{2}}\|\partial_1\varrho\|^{\frac{1}{2}}_{L^2}\Big]
 \nonumber \\[2mm]
 & ~\displaystyle \qquad \cdot\int^t_0(1+\tau)^{\frac{1}{15}}\|\nabla \mathop{\mathrm{div}}\nolimits \mathbf{u}\|^{\frac{1}{2}}_{L^2}\|\nabla \partial_3\mathop{\mathrm{div}}\nolimits \mathbf{u}\|^{\frac{1}{2}}_{L^2}\cdot\|\nabla\varrho\|^{\frac{1}{2}}_{L^2}\|\nabla\partial_2\varrho\|^{\frac{1}{2}}_{L^2}\mathrm{d}\tau \nonumber \\[2mm]
 &~ \displaystyle \quad
+\sup_{0\leq\tau\leq t}\Big[(1+\tau)^{\frac{1}{2}}\|\nabla\varrho\|^{\frac{1}{2}}_{L^2}(1+\tau)^{\frac{3}{2}}\|\nabla^2\varrho\|^{\frac{3}{2}}_{L^2}\|\nabla \mathbf{u}\|_{L^2}\Big]\int^t_0(1+\tau)^{-\frac{28}{15}}\mathrm{d}\tau\nonumber\\
&~ \displaystyle \quad+\sup_{0\leq \tau \leq t}\Big[(1+\tau)\|\nabla\varrho\|_{L^2}\Big]\int^t_0(1+\tau)^{\frac{1}{15}}(\|\nabla\mathop{\mathrm{div}}\nolimits \mathbf{u}\|_{L^2}+\|\nabla^2\mathop{\mathrm{div}}\nolimits \mathbf{u}\|_{L^2})\|\nabla\varrho\|_{L^2}\mathrm{d}\tau\nonumber\\[2mm]
&~\displaystyle\lesssim\mathcal{E}^{\frac{1}{2}}_2(t)\mathcal{E}^{\frac{1}{2}}_1(t)\mathcal{E}^{\frac{1}{2}}_0(t)+\mathcal{E}_2(t)\mathcal{E}^\frac{1}{2}_0(t)\nonumber\\[2mm]
&~\displaystyle\lesssim\mathcal{E}^{\frac{3}{2}}(t).
\end{align}
For the second term on the right hand side of \eqref{Q-2-SPLIT}, by \eqref{three-two-deri-ani}, \eqref{soboleve-ineq} and integration by parts, we obtain that
\begin{align*}
Q_{22}&=\int _{\mathbb{R}^{3}}\nabla(\mathbf{u}\cdot\nabla\varrho)\cdot\nabla\Delta\varrho \mathrm{d}x+\int _{\mathbb{R}^{3}}\nabla(\varrho \mathop{\mathrm{div}}\nolimits \mathbf{u})\cdot\nabla\Delta\varrho \mathrm{d}x\nonumber\\[2mm]
&=-\int _{\mathbb{R}^{3}}\Delta^2\varrho \mathbf{u}\cdot\nabla\varrho\mathrm{d}x+\int _{\mathbb{R}^{3}} \mathop{\mathrm{div}}\nolimits \mathbf{u}\nabla\varrho \cdot\nabla\Delta\varrho \mathrm{d}x+\int _{\mathbb{R}^{3}}(\varrho\nabla \mathop{\mathrm{div}}\nolimits \mathbf{u})\cdot\nabla\Delta\varrho \mathrm{d}x\nonumber\\[2mm]
&\lesssim \|\mathbf{u}\|^{\frac{1}{4}}_{L^2}\|\partial_1\mathbf{u}\|^{\frac{1}{4}}_{L^2}\|\partial_2\mathbf{u}\|^{\frac{1}{4}}_{L^2}\|\partial_1\partial_2\mathbf{u}\|^{\frac{1}{4}}_{L^2}\|\nabla\varrho\|^{\frac{1}{2}}_{L^2}\|\partial_3\nabla\varrho\|^{\frac{1}{2}}_{L^2}\|\nabla^4\varrho\|_{L^2}\nonumber\\[2mm]
&\displaystyle \quad+\|\nabla\varrho\|_{L^2}\|\nabla\mathop{\mathrm{div}}\nolimits \mathbf{u}\|^{\frac{1}{2}}_{L^2}\|\nabla^2\mathop{\mathrm{div}}\nolimits \mathbf{u}\|^{\frac{1}{2}}_{L^2}\|\nabla\Delta\varrho\|_{L^2}\nonumber\\[2mm]
&\displaystyle \quad+\|\nabla\varrho\|^{\frac{1}{2}}_{L^2}\|\nabla^2\varrho\|^{\frac{1}{2}}_{L^2}\|\nabla \mathop{\mathrm{div}}\nolimits \mathbf{u}\|_{L^2}\|\nabla\Delta\varrho\|_{L^2}.
\end{align*}
Then we have
\begin{align}\label{Q22**}
&\displaystyle\int^t_0(1+\tau)^{\frac{2}{15}}Q_{22}\mathrm{d}\tau\nonumber\\[2mm]
&~\displaystyle\lesssim \sup_{0\leq \tau\leq t}\Big[(1+\tau)^{\frac{1}{8}}\|\mathbf{u}\|^{\frac{1}{4}}_{L^2}(1+\tau)^{\frac{1}{4}}\|\partial_1\mathbf{u}\|^{\frac{1}{4}}_{L^2}(1+\tau)^{\frac{1}{4}}\|\partial_2\mathbf{u}\|^{\frac{1}{4}}_{L^2}(1+\tau)^{\frac{4}{15}}\|\partial_1\partial_2\mathbf{u}\|^{\frac{1}{4}}_{L^2}\Big.\nonumber\\[2mm]
&~\displaystyle\qquad\Big.\cdot(1+\tau)^{\frac{1}{2}}\|\nabla\varrho\|^{\frac{1}{2}}_{L^2}(1+\tau)^{\frac{1}{2}}\|\partial_3\nabla\varrho\|^{\frac{1}{2}}_{L^2}\|\nabla^4\varrho\|_{L^2}\Big]\int^t_0(1+\tau)^{-\frac{211}{120} }\mathrm{d}\tau\nonumber\\[2mm]
&~\displaystyle \quad+\sup_{0\leq \tau\leq t}\Big[(1+\tau)\|\nabla\varrho\|_{L^2}\Big]\int^t_0(1+\tau)^{\frac{1}{15}}(\|\nabla\mathop{\mathrm{div}}\nolimits \mathbf{u}\|_{L^2}+\|\nabla^2\mathop{\mathrm{div}}\nolimits \mathbf{u}\|_{L^2})\|\nabla\Delta\varrho\|_{L^2}\mathrm{d}\tau\nonumber\\[2mm]
&~\displaystyle \quad+\sup_{0\leq\tau\leq t}\Big[(1+\tau)^{\frac{1}{2}}\|\nabla\varrho\|^{\frac{1}{2}}_{L^2}(1+\tau)^{\frac{1}{2}}\|\nabla^2\varrho\|^{\frac{1}{2}}_{L^2}\Big]\int^t_0(1+\tau)^{\frac{1}{15}}\|\nabla \mathop{\mathrm{div}}\nolimits \mathbf{u}\|_{L^2}\|\nabla\Delta\varrho\|_{L^2}\mathrm{d}\tau\nonumber\\[2mm]
&~\displaystyle\lesssim \mathcal{E}_2(t)\mathcal{E}^{\frac{1}{2}}_0(t)+\mathcal{E}^{\frac{1}{2}}_2(t)\mathcal{E}^{\frac{1}{2}}_1(t)\mathcal{E}^{\frac{1}{2}}_0(t)+\mathcal{E}^{\frac{1}{2}}_2(t)\mathcal{E}^{\frac{1}{2}}_1(t)\mathcal{E}^{\frac{1}{2}}_0(t)\nonumber\\[2mm]
&~\displaystyle\lesssim \mathcal{E}^{\frac{3}{2}}(t).
\end{align}

For the third term on the right hand side of \eqref{Q-2-SPLIT}, by \eqref{three-two-deri-ani}, \eqref{soboleve-ineq} and integration by parts, we obtain that
\begin{align*}
Q_{23}&=\int _{\mathbb{R}^{3}}\nabla^2(\mathbf{u}\cdot\nabla\varrho)\cdot\nabla^2\Delta\varrho \mathrm{d}x+\int _{\mathbb{R}^{3}}\nabla^2(\varrho \mathop{\mathrm{div}}\nolimits \mathbf{u})\cdot\nabla^2\Delta\varrho \mathrm{d}x\nonumber\\[2mm]
&=\int _{\mathbb{R}^{3}} \mathbf{u}\cdot\nabla^3\varrho\nabla^2\Delta\varrho\mathrm{d}x+\int _{\mathbb{R}^{3}} \nabla\mathbf{u}\cdot\nabla^2\varrho\nabla^2\Delta\varrho\mathrm{d}x+\int _{\mathbb{R}^{3}} \nabla^2\mathbf{u}\cdot\nabla\varrho\nabla^2\Delta\varrho\mathrm{d}x\nonumber\\
&\quad+\int _{\mathbb{R}^{3}} \nabla^2\varrho \mathop{\mathrm{div}}\nolimits \mathbf{u}\cdot\nabla^2\Delta\varrho \mathrm{d}x+\int _{\mathbb{R}^{3}} \nabla\varrho\nabla\mathop{\mathrm{div}}\nolimits \mathbf{u} \cdot\nabla^2\Delta\varrho \mathrm{d}x+\int _{\mathbb{R}^{3}}(\varrho\nabla^2 \mathop{\mathrm{div}}\nolimits \mathbf{u})\cdot\nabla^2\Delta\varrho \mathrm{d}x\nonumber\\[2mm]
&\lesssim \|\mathbf{u}\|^{\frac{1}{4}}_{L^2}\|\partial_1\mathbf{u}\|^{\frac{1}{4}}_{L^2}\|\partial_2\mathbf{u}\|^{\frac{1}{4}}_{L^2}\|\partial_1\partial_2\mathbf{u}\|^{\frac{1}{4}}_{L^2}\|\nabla^2\varrho\|^{\frac{1}{4}}_{L^2}\|\nabla^4\varrho\|^{\frac{7}{4}}_{L^2}\nonumber\\[2mm]
&\displaystyle \quad+\|\nabla\mathbf{u}\|^{\frac{1}{4}}_{L^2}\|\partial_1\nabla\mathbf{u}\|^{\frac{1}{4}}_{L^2}\|\partial_2\nabla\mathbf{u}\|^{\frac{1}{4}}_{L^2}\|\partial_1\partial_2\mathbf{u}\|^{\frac{1}{8}}_{L^2}\|\partial_1\partial_2\nabla^2\mathbf{u}\|^{\frac{1}{8}}_{L^2}\|\nabla^2\varrho\|^{\frac{3}{4}}_{L^2}\|\nabla^4\varrho\|^{\frac{5}{4}}_{L^2}\nonumber\\[2mm]
&\displaystyle\quad+\|\nabla^2\varrho\|^{\frac{3}{4}}_{L^2}\|\nabla^4\varrho\|^{\frac{1}{4}}_{L^2}\|\nabla^2\mathbf{u}\|_{L^2}\|\nabla^2\Delta\varrho\|_{L^2}+\|\nabla^2\varrho\|_{L^2}\|\nabla\mathop{\mathrm{div}}\nolimits \mathbf{u}\|^{\frac{1}{2}}_{L^2}\|\nabla^2\mathop{\mathrm{div}}\nolimits \mathbf{u}\|^{\frac{1}{2}}_{L^2}\|\nabla^2\Delta\varrho\|_{L^2}\nonumber\\[2mm]
&\displaystyle \quad+\|\nabla\varrho\|_{L^2}\|\nabla^2\mathop{\mathrm{div}}\nolimits \mathbf{u}\|^{\frac{1}{2}}_{L^2}\|\nabla^3\mathop{\mathrm{div}}\nolimits \mathbf{u}\|^{\frac{1}{2}}_{L^2}\|\nabla^2\Delta\varrho\|_{L^2}+\|\nabla\varrho\|^{\frac{1}{2}}_{L^2}\|\nabla^2\varrho\|^{\frac{1}{2}}_{L^2}\|\nabla^2\mathop{\mathrm{div}}\nolimits \mathbf{u}\|_{L^2}\|\nabla^2\Delta\varrho\|_{L^2}.
\end{align*}
Then we have
\begin{align}\label{Q23**}
&\displaystyle\int^t_0(1+\tau)^{\frac{2}{15}}Q_{23}\mathrm{d}\tau\nonumber\\[2mm]
&~\displaystyle\lesssim \sup_{0\leq \tau\leq t}\Big[(1+\tau)^{\frac{1}{8}}\|\mathbf{u}\|^{\frac{1}{4}}_{L^2}(1+\tau)^{\frac{1}{4}}\|\partial_1\mathbf{u}\|^{\frac{1}{4}}_{L^2}(1+\tau)^{\frac{1}{4}}\|\partial_2\mathbf{u}\|^{\frac{1}{4}}_{L^2}(1+\tau)^{\frac{4}{15}}\|\partial_1\partial_2\mathbf{u}\|^{\frac{1}{4}}_{L^2}\Big.\nonumber\\[2mm]
&~\displaystyle\qquad\Big.\cdot(1+\tau)^{\frac{1}{4}}\|\nabla^2\varrho\|^{\frac{1}{4}}_{L^2}\|\nabla^4\varrho\|^{\frac{7}{4}}_{L^2}\Big]\int^t_0(1+\tau)^{-\frac{121}{120} }\mathrm{d}\tau\nonumber\\[2mm]
&~\displaystyle \quad+\sup_{0\leq\tau\leq t}\Big[(1+\tau)^{\frac{1}{8}}\|\nabla\mathbf{u}\|^{\frac{1}{4}}_{L^2}(1+\tau)^{\frac{1}{4}}\|\partial_1\nabla\mathbf{u}\|^{\frac{1}{4}}_{L^2}(1+\tau)^{\frac{1}{4}}\|\partial_2\nabla\mathbf{u}\|^{\frac{1}{4}}_{L^2}(1+\tau)^{\frac{2}{15}}\|\partial_1\partial_2\mathbf{u}\|^{\frac{1}{8}}_{L^2}\Big.\nonumber\\[2mm]
&~\displaystyle\qquad\Big.\cdot\|\partial_1\partial_2\nabla^2\mathbf{u}\|^{\frac{1}{8}}_{L^2}(1+\tau)^{\frac{3}{4}}\|\nabla^2\varrho\|^{\frac{3}{4}}_{L^2}\|\nabla^4\varrho\|^{\frac{5}{4}}_{L^2}\Big]\int^t_0(1+\tau)^{-\frac{11}{8}}\mathrm{d}\tau\nonumber\\[2mm]
&~\displaystyle \quad+\sup_{0\leq \tau\leq t} \Big[(1+\tau)^{\frac{3}{4}}\|\nabla^2\varrho\|^{\frac{3}{4}}_{L^2}\|\nabla^4\varrho\|^{\frac{1}{4}}_{L^2}(1+\tau)^{\frac{1}{2}}(\log(1+\tau))^{-1}\|\nabla^2\mathbf{u}\|_{L^2}\|\nabla^2\Delta\varrho\|_{L^2}\Big]\nonumber\\
&~\displaystyle\qquad\cdot\int^t_0(1+\tau)^{-\frac{67}{60}}\log(1+\tau)\mathrm{d}\tau\nonumber\\[2mm]
&~\displaystyle \quad+\sup_{0\leq \tau\leq t}\Big[(1+\tau)\|\nabla^2\varrho\|_{L^2}\Big]\int^t_0(1+\tau)^{\frac{1}{15}}(\|\nabla\mathop{\mathrm{div}}\nolimits \mathbf{u}\|_{L^2}+\|\nabla^2\mathop{\mathrm{div}}\nolimits \mathbf{u}\|_{L^2})\|\nabla^2\Delta\varrho\|_{L^2}\mathrm{d}\tau\nonumber\\[2mm]
&~\displaystyle \quad+\sup_{0\leq \tau\leq t}\Big[(1+\tau)\|\nabla\varrho\|_{L^2}\Big]\int^t_0(1+\tau)^{\frac{1}{15}}(\|\nabla^2\mathop{\mathrm{div}}\nolimits \mathbf{u}\|_{L^2}+\|\nabla^3\mathop{\mathrm{div}}\nolimits \mathbf{u}\|_{L^2})\|\nabla^2\Delta\varrho\|_{L^2}\mathrm{d}\tau\nonumber\\[2mm]
&~\displaystyle \quad+\sup_{0\leq\tau\leq t}\Big[(1+\tau)^{\frac{1}{2}}\|\nabla\varrho\|^{\frac{1}{2}}_{L^2}(1+\tau)^{\frac{1}{2}}\|\nabla^2\varrho\|^{\frac{1}{2}}_{L^2}\Big]\int^t_0(1+\tau)^{\frac{1}{15}}\|\nabla^2 \mathop{\mathrm{div}}\nolimits \mathbf{u}\|_{L^2}\|\nabla^2\Delta\varrho\|_{L^2}\mathrm{d}\tau\nonumber\\[2mm]
&~\displaystyle\lesssim \mathcal{E}^{\frac{5}{8}}_2(t)\mathcal{E}^{\frac{7}{8}}_0(t)+\mathcal{E}^{\frac{13}{16}}_2(t)\mathcal{E}^{\frac{11}{16}}_0(t)+\mathcal{E}_{2}(t)\mathcal{E}^{\frac{1}{2}}_0(t)+\mathcal{E}^{\frac{1}{2}}_2(t)\mathcal{E}^{\frac{1}{2}}_1(t)\mathcal{E}^{\frac{1}{2}}_0(t)\nonumber\\[2mm]
&~\displaystyle\lesssim \mathcal{E}^{\frac{3}{2}}(t).
\end{align}
Therefore, we have from \eqref{Q-2-SPLIT}, \eqref{Q21**}, \eqref{Q22**} and \eqref{Q23**} that
\begin{align*}
\displaystyle \int _{0}^{t}(1+ \tau)^{\frac{2}{15}}Q _{2}\mathrm{d}\tau \lesssim \mathcal{E}^{\frac{3}{2}}(t).
\end{align*}

This alongside \eqref{dissip-integral} as well as \eqref{Q-1-con} leads to
\begin{gather}\label{divujiaquan}
\displaystyle(1+t)^{\frac{2}{15}}\|(\mathop{\mathrm{div}}\nolimits \mathbf{u},\nabla\varrho)\|^2_{H^2}+\int^t_0(1+\tau)^{\frac{2}{15}}\|(\nabla_h\mathop{\mathrm{div}}\nolimits \mathbf{u},\nabla \mathop{\mathrm{div}}\nolimits \mathbf{u})\|^2_{H^2}\mathrm{d}\tau\lesssim\mathcal{E}_0(0)+\mathcal{E}^{\frac{3}{2}}(t).
\end{gather}
Taking the $H^2$ inner product of the equations in \eqref{new} with $\Delta_h\varrho$ and $\Delta_h \mathbf{u}$, respectively, and multiplying the resulting identity by $(1+t)^{\frac{2}{15}}$, we get by similar arguments as in the derivation of \eqref{divujiaquan} that
\begin{gather}\label{sptjiaquan}
\displaystyle(1+t)^{\frac{2}{15}}\|(\nabla_h\varrho,\nabla_h \mathbf{u})\|^{2}_{H^2}+\int^t_0(1+\tau)^{\frac{2}{15}}(\|(\Delta_h \mathbf{u},\nabla _{h} \mathop{\mathrm{div}}\nolimits \mathbf{u})\|^2_{H^2}\mathrm{d}\tau
\lesssim\mathcal{E}_0(0)+C\mathcal{E}^{\frac{3}{2}}(t).
\end{gather}
Gathering \eqref{divujiaquan} and \eqref{sptjiaquan}, we then get \eqref{con-dissip-lem}, and thus finish the proof of Lemma \ref{haosan}.
\end{proof}

\subsection{Estimates of $ (\varrho,\mathbf{u}) $} 
\label{sub:estimates_of_gs}

In what follows, let us turn to the estimates of $ \mathcal{E}_{2}(t) $. For later use, we split $\mathcal{E}_2(t)$ into three parts
\begin{equation*}
\mathcal{E}_2(t)=\mathcal{E}_{21}(t)+\mathcal{E}_{22}(t)+\mathcal{E}_{23}(t),
\end{equation*}
where
\begin{align*}
\mathcal{E}_{21}(t)&=\sup_{0\leq\tau\leq t}\Big((1+\tau)^{\frac{3}{2}}\log^{-2}(1+\tau)\|\varrho(\tau,\cdot)\|^2_{L^2}+(1+\tau)\|\mathbf{u}(\tau,\cdot)\|^2_{L^2}\Big),
\\[2mm]
\mathcal{E}_{22}(t)&=\sup_{0\leq\tau\leq t}\Big((1+\tau)^{2}\|\nabla\varrho(\tau,\cdot)\|^2_{L^2}+(1+\tau)^{2}\|\nabla_h\mathbf{u}(\tau,\cdot)\|^2_{L^2} +(1+\tau)\|\partial_3\mathbf{u}(\tau,\cdot)\|^2_{L^2}\Big),
\\[2mm]
\mathcal{E}_{23}(t)&=\sup_{0\leq\tau\leq t}\Big(\sum_{i=1,2;j=3}(1+\tau)^{2}\|(\partial_i\partial_j\mathbf{u})(\tau,\cdot)\|^2_{L^2}
+(1+\tau)^{2}\|\nabla^2\varrho(\tau,\cdot)\|^2_{L^2}\Big.\nonumber\\[2mm]
&\displaystyle \qquad+\Big.\sum_{i=1,2;j=1,2}(1+\tau)^{\frac{32}{15}}\|(\partial_i\partial_j\mathbf{u})(\tau,\cdot)\|^2_{L^2}+(1+\tau)\log^{-2}(1+\tau)\|\partial^2_3\mathbf{u}(\tau,\cdot)\|^2_{L^2}\Big).
\end{align*}
The estimates of $ \mathcal{E}_{2}(t) $ will be established via three steps based on the above split. \\[1mm]

{\bf Step One:} The estimate of $ \mathcal{E}_{21}(t) $. Clearly, we have
\begin{lemma}\label{lemma5.5}
 Assume that $(\varrho,u)$ is a smooth solution to the problem \eqref{new} satisfying \eqref{appri-assum-Nonlin} and \eqref{sec4-appri-conlu}. Then we have
\begin{equation*}
\mathcal{E}_{21}(t) \lesssim\mathcal{E}^2(t)+\|(\varrho_0,\mathbf{u}_0)\|^2_{L^1}+\|\mathbf{u}_0\|^2_{L^2_{x_3}L^1_{x_1x_2}}+\|(\varrho_0,\mathbf{u}_0)\|^2_{L^2}.
\end{equation*}
\end{lemma}
\begin{proof}
Recalling \eqref{rhobiaoshi}, by using the Plancherel theorem and \eqref{vrho-con-line-lem}, we get
\begin{align}\label{rhoL2}
\displaystyle\|\varrho\|_{L^2}&=\|\widehat{\varrho}\|_{L^2}
 \nonumber \\
 &\displaystyle\lesssim (1+t)^{-\frac{3}{4}}(\|(\varrho_0,\mathbf{u}_{0})\|_{L^1}+\|(\varrho_0,\mathbf{u}_{0})\|_{L^2})+\int^t_0(1+t-\tau)^{-\frac{3}{4}}\|S_1\|_{L^1}\mathrm{d}\tau\nonumber\\
&\displaystyle \quad+\int^t_0(1+t-\tau)^{-m}\|S_1\|_{L^2}\mathrm{d}\tau+\int^t_0(1+t-\tau)^{-\frac{3}{4}}\|S_2\|_{L^1}\mathrm{d}\tau
 \nonumber \\
 &\displaystyle \quad+\int^t_0(1+t-\tau)^{-m}\|S_2\|_{L^2}\mathrm{d}\tau\nonumber\\
&\displaystyle\lesssim (1+t)^{-\frac{3}{4}}(\|(\varrho_0,\mathbf{u}_{0})\|_{L^1}+\|(\varrho_0,\mathbf{u}_{0})\|_{L^2})+\sum^{4}_{i=1}M_{0i}.
\end{align}
For term $M_{01}$,  using Lemma \ref{decaylemma} and the H\"older inequality, we get
\begin{align}\label{M01}
\displaystyle \makebox[-8pt]{~}
M_{01}& \lesssim\sup_{0\leq \tau\leq t}\Big[(1+\tau)^{\frac{3}{4}}(\log(1+\tau))^{-1}\|\varrho\|_{L^2}(1+\tau)\|\mathop{\mathrm{div}}\nolimits \mathbf{u}\|_{L^2}\Big]\int^t_0(1+t-\tau)^{-\frac{3}{4}}(1+\tau)^{-\frac{7}{4}}\log(1+\tau)\mathrm{d}\tau\nonumber\\[2mm]
&\displaystyle \quad+\sup_{0\leq \tau\leq t}\Big[(1+\tau)\|\nabla\varrho\|_{L^2}(1+\tau)^{\frac{1}{2}}\|\mathbf{u}\|_{L^2}\Big]\int^t_0(1+t-\tau)^{-\frac{3}{4}}(1+\tau)^{-\frac{3}{2}}\mathrm{d}\tau\nonumber\\[2mm]
&\displaystyle \lesssim\mathcal{E}^{\frac{1}{2}}_2(t)\mathcal{E}^{\frac{1}{2}}_{1}(t)(1+t)^{-\frac{3}{4}}+\mathcal{E}_2(t)(1+t)^{-\frac{3}{4}}\nonumber\\[2mm]
&\displaystyle \lesssim\mathcal{E}(t)(1+t)^{-\frac{3}{4}}.
\end{align}
Similarly, we get
\begin{align}
\displaystyle \makebox[-10pt]{~} M_{03}&\lesssim\sup_{0\leq\tau\leq t}\Big[(1+\tau)^{\frac{1}{2}}\|\mathbf{u}\|_{L^2}(1+\tau)^{\frac{1}{2}}\|\nabla \mathbf{u}\|_{L^2}\Big]\int^t_0(1+t-\tau)^{-\frac{3}{4}}(1+\tau)^{-1}\mathrm{d}\tau\nonumber\\[2mm]
&\displaystyle \quad+\sup_{0\leq \tau\leq t }\Big[(1+\tau)^{\frac{3}{4}}(\log(1+\tau))^{-1}\|\varrho\|_{L^2}(1+\tau)^{\frac{16}{15}}\|\Delta_h\mathbf{u}\|_{L^2}\Big]\nonumber\\[2mm]
&\displaystyle \quad\cdot\int^t_0(1+t-\tau)^{-\frac{3}{4}}(1+\tau)^{-\frac{109}{60}}\log(1+\tau)\mathrm{d}\tau\nonumber\\[2mm]
&\displaystyle \quad+\sup_{0\leq \tau\leq t}\Big[(1+\tau)^{\frac{3}{4}}(\log(1+\tau))^{-1}\|\varrho\|_{L^2}(1+\tau)\|\nabla \mathop{\mathrm{div}}\nolimits \mathbf{u}\| _{L ^{2}}\Big]\nonumber\\[2mm]
&\displaystyle \quad\cdot\int^t_0(1+t-\tau)^{-\frac{3}{4}}(1+\tau)^{-\frac{7}{4}}\log(1+\tau)\mathrm{d}\tau\nonumber\\[2mm]
&\displaystyle \quad+\sup_{0\leq \tau\leq t}\Big[ (1+\tau)^{\frac{3}{4}}(\log(1+\tau))^{-1}\|\varrho\|_{L^2}(1+\tau)\|\nabla\varrho\|_{L ^{2}}\Big]\nonumber\\[2mm]
&\displaystyle \quad\cdot\int^t_0(1+t-\tau)^{-\frac{3}{4}}(1+\tau)^{-\frac{7}{4}}\log(1+\tau)\mathrm{d}\tau\nonumber\\[2mm]
&\displaystyle \lesssim\mathcal{E}_2(t)(1+t)^{-\frac{3}{4}}\log(1+t)+\mathcal{E}_2(t)(1+t)^{-\frac{3}{4}}+\mathcal{E}^{\frac{1}{2}}_2(t)\mathcal{E}^{\frac{1}{2}}_1(t)(1+t)^{-\frac{3}{4}}\nonumber\\[2mm]
&\displaystyle \lesssim\mathcal{E}(t)(1+t)^{-\frac{3}{4}}\log(1+t).
\end{align}
Using \eqref{soboleve-ineq} and Lemma \ref{decaylemma}, we can estimate $M_{02}$ as follows:
\begin{align}\label{M02}
M_{02}&\lesssim\sup_{0\leq\tau\leq t}\Big[(1+\tau)^{\frac{3}{16}}(\log(1+\tau))^{-\frac{1}{4}}\|\varrho\|^{\frac{1}{4}}_{L^2}(1+\tau)^{\frac{3}{4}}\|\nabla^2\varrho\|^{\frac{3}{4}}_{L^2}(1+\tau)\|\mathop{\mathrm{div}}\nolimits \mathbf{u}\|_{L^2}\Big]\nonumber\\[2mm]
&~\displaystyle \qquad\cdot\int^t_0(1+t-\tau)^{-m}(1+\tau)^{-\frac{31}{16}}(\log(1+\tau))^{\frac{1}{4}}\mathrm{d}\tau\nonumber\\
&\displaystyle \quad+\sup_{0\leq\tau\leq t} \Big[(1+\tau)^{\frac{1}{8}}\|\mathbf{u}\|^{\frac{1}{4}}_{L^2}(1+\tau)^{\frac{3}{8}}(\log(1+\tau))^{-\frac{3}{4}}\|\nabla^2\mathbf{u}\|^{\frac{3}{4}}_{L^2}(1+\tau)\|\nabla\varrho\|_{L^2}\Big]\nonumber\\[2mm]
&~\displaystyle \qquad\cdot\int^t_0(1+t-\tau)^{-m}(1+\tau)^{-\frac{3}{2}}(\log(1+\tau))^{\frac{3}{4}}\mathrm{d}\tau\nonumber\\
&\displaystyle \lesssim\mathcal{E}^{\frac{1}{2}}_2(t)\mathcal{E}^{\frac{1}{2}}_1(t)(1+t)^{-\frac{31}{16}}(\log(1+t))^{\frac{1}{4}}+\mathcal{E}_2(t)(1+t)^{-\frac{3}{2}}(\log(1+t))^{\frac{3}{4}}\nonumber\\[2mm]
&\displaystyle \lesssim \mathcal{E}(t)(1+t)^{-\frac{3}{2}}(\log(1+t))^{\frac{3}{4}}.
\end{align}
In the same manner as in the derivation of \eqref{M02}, we can estimate the  last term $M_{04}$. Precisely, we split it into four terms.
\begin{align}\label{M04*}
M_{04}&\lesssim \int^t_0(1+t-\tau)^{-m}(\|\mathbf{u}\cdot\nabla\mathbf{u}\|_{L^2}+\|g(\varrho)\Delta_h\mathbf{u}\|_{L^2}+\|g(\varrho)\nabla \mathop{\mathrm{div}}\nolimits \mathbf{u}\|_{L^2}+\|f(\varrho)\nabla\varrho\|_{L ^{2}})\mathrm{d}\tau\nonumber\\[2mm]
&=: M_{{04}_1}+M_{{04}_2}+M_{{04}_3}+M_{{04}_4}.
\end{align}
First, for the term $M_{{04}_1}$, using \eqref{soboleve-ineq}, Lemma \ref{decaylemma} and the H\"older inequality, we get
\begin{align}\label{M041}
M_{{04}_1}
&\lesssim\int^t_0(1+t-\tau)^{-m}\|\mathbf{u}\|_{L^4}\|\nabla \mathbf{u}\|_{L^4}\mathrm{d}\tau\nonumber\\
&\displaystyle\lesssim\int^t_0(1+t-\tau)^{-m}\|\mathbf{u}\|^{\frac{1}{4}}_{L^2}\|\nabla \mathbf{u}\|_{L^2}\|\nabla^2\mathbf{u}\|^{\frac{3}{4}}_{L^2}\mathrm{d}\tau\nonumber\\
&\lesssim\sup_{0\leq\tau\leq t}\Big[(1+\tau)^{\frac{1}{8}}\|\mathbf{u}\|^{\frac{1}{4}}_{L^2}(1+\tau)^{\frac{1}{2}}\|\nabla \mathbf{u}\|_{L^2}(1+\tau)^{\frac{3}{8}}(\log(1+\tau))^{-\frac{3}{4}}\|\nabla^2\mathbf{u}\|^{\frac{3}{4}}_{L^2}\Big]\nonumber\\
&\displaystyle ~\quad\cdot\int^t_0(1+t-\tau)^{-m}(1+\tau)^{-1}(\log(1+\tau))^{\frac{3}{4}}\mathrm{d}\tau\nonumber\\
&\lesssim\mathcal{E}_2(t)\int^t_0(1+t-\tau)^{-m}(1+\tau)^{-1}(\log(1+\tau))^{\frac{3}{4}}\mathrm{d}\tau\nonumber\\
&\lesssim\mathcal{E}_2(t)(1+t)^{-1}(\log(1+t))^{\frac{3}{4}},
\end{align}
where $m>1$. For the second term on the right hand side of \eqref{M04*}, in view of the Sobolev inequality \eqref{soboleve-ineq}, it is clear that
\begin{align}
M_{{04}_2}&\displaystyle \lesssim\int^t_0(1+t-\tau)^{-m}\|\varrho\|_{L^{\infty}}\|\Delta_h\mathbf{u}\|_{L^2}\mathrm{d}\tau\nonumber\\[2mm]
&\displaystyle \lesssim\int^t_0(1+t-\tau)^{-m}\|\varrho\|^{\frac{1}{4}}_{L^{2}}\|\nabla^2\varrho\|^{\frac{3}{4}}_{L^2}\|\Delta_h\mathbf{u}\|_{L^2}\mathrm{d}\tau\nonumber\\[2mm]
&\displaystyle \lesssim\sup_{0\leq \tau\leq t}\Big[(1+\tau)^{\frac{3}{16}}(\log(1+\tau))^{-\frac{1}{4}}\|\varrho\|^{\frac{1}{4}}_{L^2}(1+\tau)^{\frac{3}{4}}\|\nabla^2\varrho\|^{\frac{3}{4}}_{L^2}(1+\tau)^{\frac{16}{15}}\|\Delta_h\mathbf{u}\|_{L^2}\Big]\nonumber\\[2mm]
&\displaystyle \quad\cdot\int^t_0(1+t-\tau)^{-m}(1+\tau)^{-\frac{481}{240}}(\log(1+\tau))^{\frac{1}{4}}\mathrm{d}\tau\nonumber\\[2mm]
&\displaystyle \lesssim \mathcal{E}_2(t)(1+t)^{-\frac{481}{240}}(\log(1+t))^{\frac{1}{4}}.
\end{align}
Similarly, we get
\begin{align}
M_{{04}_3}\lesssim \mathcal{E}^{\frac{1}{2}}_2(t)\mathcal{E}^{\frac{1}{2}}_1(t)(1+t)^{-\frac{481}{240}}(\log(1+t))^{\frac{1}{4}},
\end{align}
and
\begin{align}\label{M044}
\displaystyle M_{{04}_4}
&\displaystyle \lesssim \int^t_0(1+t-\tau)^{-m}\|\varrho\|_{L^4}\|\nabla\varrho\|_{L^4}\mathrm{d}\tau\nonumber\\[2mm]
&\displaystyle \lesssim \int^t_0(1+t-\tau)^{-m}\|\varrho\|^{\frac{1}{4}}_{L^2}\|\nabla\varrho\|_{L^2}\|\nabla^2\varrho\|_{L^2}^{\frac{3}{4}}\mathrm{d}\tau\nonumber\\[2mm]
&\displaystyle \lesssim \sup_{0\leq \tau\leq t}\Big[(1+\tau)^{\frac{3}{16}}(\log(1+\tau))^{-\frac{1}{4}}\|\varrho\|^{\frac{1}{4}}_{L^2}(1+\tau)\|\nabla\varrho\|_{L^2}(1+\tau)^{\frac{3}{4}}\|\nabla^2\varrho\|_{L^2}^{\frac{3}{4}}\Big]\nonumber\\[2mm]
&\displaystyle ~\quad\cdot\int^t_0(1+t-\tau)^{-m}(1+\tau)^{-\frac{31}{16}}(\log(1+\tau))^{\frac{1}{4}}\mathrm{d}\tau\nonumber\\[2mm]
&\displaystyle \lesssim \mathcal{E}_2(t)(1+t)^{-\frac{31}{16}}(\log(1+t))^{\frac{1}{4}}.
\end{align}
As a consequence, we have from \eqref{M041}--\eqref{M044} that
\begin{align}\label{M04}
M_{04}=\int^t_0(1+t-\tau)^{-m}\|S_2\|_{L^2}\mathrm{d}\tau \lesssim \mathcal{E}(t)(1+t)^{-1}(\log(1+t))^{\frac{3}{4}}.
\end{align}
Substituting \eqref{M01}--\eqref{M02} and \eqref{M04} into \eqref{rhoL2}, we get that
\begin{align*}
\displaystyle (1+t)^{\frac{3}{4}}(\log(1+t))^{-1}\|\varrho\|_{L^2} \lesssim\mathcal{E}(t)+\|(\varrho_0,\mathbf{u}_0)\|_{L^1(\mathbb{R}^3)}+\|(\varrho_0,\mathbf{u}_0)\|_{L^2}.
\end{align*}
Next, let us turn to the estimate of $\|\mathbf{u}\|_{L^2}$. Using the Plancherel theorem again and \eqref{anisou-lem-con}, we get for any $ m \geq 0 $,
\begin{align}\label{uL2}
&\|\mathbf{u}\|_{L^2}=\|\widehat{\mathbf{u}}\|_{L^2}
 \nonumber \\[2mm]
 &~\displaystyle \lesssim (1+t)^{-\frac{3}{4}}(\|\varrho_0\|_{L^1}+\|\varrho_0\|_{L^2})+(1+t)^{-\frac{1}{2}}(\|\mathbf{u}_0\|_{L^2_{x_3}L^1_{x_1x_2}}+\|\mathbf{u}_0\|_{L^2})\nonumber\\[2mm]
&~\displaystyle \quad+\int^t_0(1+t-\tau)^{-\frac{3}{4}}\|S_1\|_{L^1}\mathrm{d}\tau+\int^t_0(1+t-\tau)^{-m}\|S_1\|_{L^2}\mathrm{d}\tau\nonumber\\[2mm]
&~\displaystyle \quad+\int^t_0(1+t-\tau)^{-\frac{1}{2}}\|S_2\|_{L^2_{x_3}L^1_{x_1x_2}}\mathrm{d}\tau+\int^t_0(1+t-\tau)^{-m}\|S_2\|_{L^2}\mathrm{d}\tau\nonumber\\[2mm]
&~\displaystyle \lesssim (1+t)^{-\frac{3}{4}}(\|\varrho_0\|_{L^1}+\|\varrho_0\|_{L^2})+(1+t)^{-\frac{1}{2}}(\|\mathbf{u}_0\|_{L^2_{x_3}L^1_{x_1x_2}}+\|\mathbf{u}_0\|_{L^2})+\sum^{4}_{i=1}H_{0i}.
\end{align}
By the same method as in the proof of estimates for $M_{01}$, $M_{02}$ and $M_{04}$, we can estimate $H_{01}$, $H_{02}$ and $H_{04}$ as follows:
\begin{align}\label{H0-1-2-4}
\displaystyle \begin{cases}
    \displaystyle H_{01} \lesssim\mathcal{E}(t)(1+t)^{-\frac{3}{4}},\ \ H_{02} \lesssim\mathcal{E}(t)(1+t)^{-\frac{3}{2}}(\log(1+t))^{\frac{3}{4}},\\[3mm]
\displaystyle H_{04} \lesssim\mathcal{E}(t)(1+t)^{-1}(\log(1+t))^{\frac{3}{4}}.
\end{cases}
\end{align}
For the term $H_{03}$, recalling the definition of $S_2$ in \eqref{nonlinear}, we get
\begin{align}\label{H030}
\displaystyle H_{03}& \lesssim \int^t_0(1+t-\tau)^{-\frac{1}{2}}\|\mathbf{u}\cdot\nabla\mathbf {u}\|_{L^2_{x_3}L^1_{x_1x_2}} \mathrm{d}\tau+ \int _{0}^{t}(1+t-\tau)^{-\frac{1}{2}}\|\varrho\Delta_h\mathbf {u}\|_{L^2_{x_3}L^1_{x_1x_2}}\mathrm{d}\tau
 \nonumber \\[2mm]
 & \displaystyle \quad +\int^t_0(1+t-\tau)^{-\frac{1}{2}}\|\varrho\nabla\mathop{\mathrm{div}}\nolimits \mathbf {u}\|_{L^2_{x_3}L^1_{x_1x_2}}\mathrm{d}\tau+\int^t_0(1+t-\tau)^{-\frac{1}{2}}\|\varrho\nabla\varrho\|_{L^2_{x_3}L^1_{x_1x_2}}\mathrm{d}\tau\nonumber\\[2mm]
&=: \sum^{4}_{i=1}H_{{03}_i}.
\end{align}
 We shall estimate $ H_{{03}_i}\;(1 \leq i \leq 4) $ term by term. For the term $H_{{03}_1}$, using \eqref{two} and Lemma \ref{decaylemma}, we derive that
\begin{align}\label{H031}
H_{{03}_1}& \lesssim\int^t_0(1+t-\tau)^{-\frac{1}{2}}\Big(\|\mathbf{u}_h\|^{\frac{1}{2}}_{L^2}\|\partial_3\mathbf{u}_h\|^{\frac{1}{2}}_{L^2}\|\nabla_h \mathbf{u}\|_{L^2}+\|u_3\|^{\frac{1}{2}}_{L^2}\|\partial_3u_3\|^{\frac{1}{2}}_{L^2}\|\partial_3 \mathbf{u}\|_{L^2}\Big)\mathrm{d}\tau\nonumber\\[2mm]
& \displaystyle \lesssim\sup_{0\leq \tau\leq t}\Big[(1+\tau)^{\frac{1}{4}}\|\mathbf{u}_h\|^{\frac{1}{2}}_{L^2}(1+\tau)^{\frac{1}{4}}\|\partial_3\mathbf{u}_h\|^{\frac{1}{2}}_{L^2}(1+\tau)\|\nabla_h\mathbf{u}\|_{L^2} \Big]\nonumber\\[2mm]
&\displaystyle ~\qquad \cdot \int^t_0(1+t-\tau)^{-\frac{1}{2}}(1+\tau)^{-\frac{3}{2}}\mathrm{d}\tau\nonumber\\[2mm]
&\displaystyle \quad+\sup_{0\leq \tau\leq t}\Big[(1+\tau)^{\frac{1}{4}}\|u_3\|^{\frac{1}{2}}_{L^2}(1+\tau)^{\frac{1}{2}}\|\partial_3u_3\|^{\frac{1}{2}}_{L^2}(1+\tau)^{\frac{1}{2}}\|\partial_3\mathbf{u}\|_{L^2}\Big]\nonumber\\[2mm]
&\displaystyle ~\qquad \cdot \int^t_0(1+t-\tau)^{-\frac{1}{2}}(1+\tau)^{-\frac{5}{4}}\mathrm{d}\tau\nonumber\\[2mm]
&\displaystyle \lesssim \mathcal{E}_2(t)(1+t)^{-\frac{1}{2}}+\mathcal{E}^{\frac{3}{4}}_2(t)\mathcal{E}^{\frac{1}{4}}_1(t)(1+t)^{-\frac{1}{2}}\nonumber\\[2mm]
&\displaystyle \lesssim\mathcal{E}(t)(1+t)^{-\frac{1}{2}}.
\end{align}
Similarly, using \eqref{two} and Lemma \ref{decaylemma} again, we deduce that
\begin{align}
\displaystyle H_{{03}_2}& \lesssim\sup_{0\leq \tau\leq t}\Big[(1+\tau)^{\frac{3}{8}}(\log(1+\tau))^{-\frac{1}{2}}\|\varrho\|^{\frac{1}{2}}_{L^2}(1+\tau)^{\frac{1}{2}}\|\partial_3\varrho\|^{\frac{1}{2}}_{L^2}(1+\tau)^{\frac{16}{15}}\|\Delta_h\mathbf{u}\|_{L^2}\Big]\nonumber\\[2mm]
&\displaystyle ~\qquad \cdot \int^t_0(1+t-\tau)^{-\frac{1}{2}}(1+\tau)^{-\frac{233}{120}}(\log(1+\tau))^{\frac{1}{2}}\mathrm{d}\tau\nonumber\\[2mm]
& \displaystyle \lesssim\mathcal{E}_2(t)(1+t)^{-\frac{1}{2}},
\end{align}
\begin{align}
\displaystyle H_{{03}_3}
& \lesssim\sup_{0\leq \tau\leq t}\Big[(1+\tau)^{\frac{3}{8}}(\log(1+\tau))^{-\frac{1}{2}}\|\varrho\|^{\frac{1}{2}}_{L^2}(1+\tau)^{\frac{1}{2}}\|\partial_3\varrho\|^{\frac{1}{2}}_{L^2}(1+\tau)\|\nabla \mathop{\mathrm{div}}\nolimits \mathbf{u}\|_{L^2}\Big]\nonumber\\[2mm]
&\displaystyle ~ \quad \cdot\int^t_0(1+t-\tau)^{-\frac{1}{2}}(1+\tau)^{-\frac{15}{8}}(\log(1+\tau))^{\frac{1}{2}}\mathrm{d}\tau\nonumber\\[2mm]
&\displaystyle \lesssim\mathcal{E}^{\frac{1}{2}}_2(t)\mathcal{E}^{\frac{1}{2}}_1(t)(1+t)^{-\frac{1}{2}}\nonumber\\[2mm]
&\displaystyle \lesssim\mathcal{E}(t)(1+t)^{-\frac{1}{2}}
\end{align}
and
\begin{align}\label{H034}
\displaystyle H_{{03}_4}&\lesssim\sup_{0\leq \tau\leq t}\Big[(1+\tau)^{\frac{3}{8}}(\log(1+\tau))^{-\frac{1}{2}}\|\varrho\|^{\frac{1}{2}}_{L^2}(1+\tau)^{\frac{1}{2}}\|\partial_3\varrho\|^{\frac{1}{2}}_{L^2}(1+\tau)\|\nabla\varrho\|_{L^2}\Big]\nonumber\\[2mm]
&\displaystyle ~\quad \cdot\int^t_0(1+t-\tau)^{-\frac{1}{2}}(1+\tau)^{-\frac{15}{8}}(\log(1+\tau))^{\frac{1}{2}}\mathrm{d}\tau\nonumber\\[2mm]
&\displaystyle \lesssim\mathcal{E}_2(t)(1+t)^{-\frac{1}{2}}.
\end{align}

Thereby, substituting \eqref{H031}--\eqref{H034} into \eqref{H030}, we infer that
\begin{equation}\label{H03}
H_{03} \lesssim\mathcal{E}(t)(1+t)^{-\frac{1}{2}}.
\end{equation}
Therefore, substituting \eqref{H0-1-2-4} and \eqref{H03} into \eqref{uL2}, we obtain that
\begin{align*}
(1+t)^{\frac{1}{2}}\|\mathbf{u}\|_{L^2} \lesssim \mathcal{E}(t)+\|\varrho_0\|_{L^1}+\|\mathbf{u}_0\|_{L^2_{x_3}L^1_{x_1,x_2}}+\|(\varrho_0,\mathbf{u}_0)\|_{L^2}.
\end{align*}
Hence, we complete the proof of Lemma \ref{lemma5.5}.
\end{proof}

{\bf Step two:} The estimate of $ \mathcal{E}_{22}(t) $. In what follows, we are devoted to proving the following lemma.
\begin{lemma}\label{lemma5.6}
Assume that the assumptions of Theroem \ref{thm2} hold. Then we have
\begin{equation*}
\mathcal{E}_{22}(t) \lesssim\mathcal{E}^2(t)+\|(\varrho_0,\mathbf{u}_0)\|^2_{L^1}+\|(\mathbf{u}_0,\partial_3\mathbf{u}_0)\|^2_{L^2_{x_3}L^1_{x_1x_2}}+\|(\nabla\varrho_0,\nabla \mathbf{u}_0)\|^2_{L^2}.
\end{equation*}
\end{lemma}
\begin{proof}
From \eqref{rhobiaoshi}, we get, for $i=1,2,3$,
\begin{align}\label{deri-rho-u-reprens}
\displaystyle
    \left(\begin{matrix}
\widehat{\partial _{i}\varrho}(t,\xi)\\[2mm]
\widehat{\partial _{i}\mathbf{u}}(t,\xi)
\end{matrix}\right)
=\hat{G}(t,\xi)\left(\begin{matrix}
\widehat{\partial _{i}\varrho} _{0}(\xi)\\[2mm]
\widehat{\partial _{i}\mathbf{u}}_{0}(\xi)
\end{matrix}
\right)+ \int _{0}^{t}\hat{G}(t- \tau,\xi)\left(\begin{matrix}
\widehat{\partial _{i}S} _{1}(\tau,\xi)\\[2mm]
\widehat{\partial _{i}S}_{2}(\tau,\xi)
\end{matrix}
\right)\mathrm{d}\tau.
\end{align}
We shall divide the analysis into two cases: (i) $ i=1~\mbox{or}~2 $ ; (ii) $ i=3 $. Let us begin with {\bf the case (i): $i=1$ or $i=2$}.
We now focus on the case $i=2$. Using \eqref{vrho-con-line-lem}, we have
\begin{align}\label{varrhobiaoshi1}
\displaystyle\|\partial_2\varrho(t)\|_{L^2}
&\lesssim (1+t)^{-\frac{5}{4}}(\|(\varrho_0,\mathbf{u}_0)\|_{L^1}+\|(\partial_2\varrho_0,\partial_2\mathbf{u}_0)\|_{L^2})+\int^t_0(1+t-\tau)^{-\frac{5}{4}}\|S_1\|_{L^1}\mathrm{d}\tau\nonumber\\[2mm]
&\displaystyle \quad+\int^t_0(1+t-\tau)^{-m}\|\partial_2S_1\|_{L^2}\mathrm{d}\tau+\int^t_0(1+t-\tau)^{-\frac{5}{4}}\|S_2\|_{L^1}\mathrm{d}\tau \nonumber\\[2mm]
&\displaystyle \quad+\int^t_0(1+t-\tau)^{-m}\|\partial_2S_2\|_{L^2}\mathrm{d}\tau\nonumber\\[2mm]
&\displaystyle=:(1+t)^{-\frac{5}{4}}(\|(\varrho_0,\mathbf{u}_0)\|_{L^1}+\|(\partial_2\varrho_0,\partial_2\mathbf{u}_0)\|_{L^2})+\sum^{4}_{i=1}M_{2i}.
\end{align}
Since the term $M_{21}$ and $M_{23}$ are exactly the same as $N_{1}$ and $N_{3}$ in Lemma \ref{divulemma}, respectively, we thus directly get
\begin{gather}\label{M21-23}
\displaystyle  M_{21} \lesssim \mathcal{E}(t)(1+t)^{- \frac{5}{4}},\ \ M_{23} \lesssim \mathcal{E}(t)(1+t)^{-1}.
\end{gather}
For the term $M_{22}$, as in the derivation of $N_2$ in Lemma \ref{divulemma}, we split it into two terms.
 \begin{align}\label{M220}
 M_{22}&=\int^t_0(1+t-\tau)^{-m}\|\partial_2(\varrho\mathop{\mathrm{div}}\nolimits \mathbf{u})\|_{L^2}\mathrm{d}\tau+\int^t_0(1+t-\tau)^{-m}\|\partial_2(u\cdot\nabla\varrho)\|_{L^2}\mathrm{d}\tau\nonumber\\[2mm]
 &=:M_{221}+M_{222},
 \end{align}
 where, in view of Lemma \ref{decaylemma} and the Sobolev inequality \eqref{soboleve-ineq}, $M_{221}$ can be controlled as
\begin{align}\label{M221}
\displaystyle M_{221}&\leq\int^t_0(1+t-\tau)^{-m}(\|\partial_2\varrho \mathop{\mathrm{div}}\nolimits \mathbf{u}\|_{L^2}+\|\varrho \partial_2\mathop{\mathrm{div}}\nolimits \mathbf{u}\|_{L^2})\mathrm{d}\tau\nonumber\\[2mm]
&\displaystyle \lesssim \sup_{0\leq\tau\leq t}\Big[(1+\tau)^{\frac{3}{4}}\|\nabla\partial_2\varrho\|^{\frac{3}{4}}_{L^2}\|\nabla^3\partial_2\varrho\|^{\frac{1}{4}}_{L^2}(1+\tau)\|\mathop{\mathrm{div}}\nolimits \mathbf{u}\|_{L^2}\Big]
 \nonumber \\[2mm]
 & \displaystyle ~\qquad \cdot\int^t_0(1+t-\tau)^{-m}(1+\tau)^{-\frac{7}{4}}\mathrm{d}\tau\nonumber\\[2mm]
&\displaystyle \quad+\sup_{0\leq \tau\leq t}\Big[(1+\tau)^{\frac{1}{2}}\|\nabla\varrho\|^{\frac{1}{2}}_{L^2}(1+\tau)^{\frac{1}{2}}\|\nabla^2\varrho\|^{\frac{1}{2}}_{L^2}(1+\tau)\|\partial_2\mathop{\mathrm{div}}\nolimits \mathbf{u}\|_{L^2}\Big]\nonumber\\[2mm]
&\displaystyle ~ \qquad\cdot\int^t_0(1+t-\tau)^{-m}(1+\tau)^{-2}\mathrm{d}\tau\nonumber\\[2mm]
&\lesssim \mathcal{E}^{\frac{3}{8}}_2(t)\mathcal{E}^{\frac{1}{8}}_0(t)\mathcal{E}^{\frac{1}{2}}_1(t)(1+t)^{-\frac{7}{4}}+\mathcal{E}^{\frac{1}{2}}_2(t)\mathcal{E}^{\frac{1}{2}}_1(t)(1+t)^{-2}\nonumber\\[2mm]
&\lesssim\mathcal{E}(t)(1+t)^{-\frac{7}{4}}.
\end{align}
For $M_{222}$, thanks to \eqref{two-two-deri-ani} and Lemma \ref{decaylemma}, we derive that
\begin{align}\label{M222}
\displaystyle \makebox[-8pt]{~} M_{222}&\leq\int^t_0(1+t-\tau)^{-m}(\|\partial_2\mathbf{u}\cdot\nabla\varrho\|_{L^2}+\|\mathbf{u}\cdot\nabla\partial_2\varrho\|_{L^2})\mathrm{d}\tau\nonumber\\[2mm]
&\displaystyle \lesssim\sup_{0\leq\tau\leq t}\Big[(1+\tau)^{\frac{3}{4}}\|\nabla^2\varrho\|^{\frac{3}{4}}_{L^2}\|\nabla^4\varrho\|^{\frac{1}{4}}_{L^2}(1+\tau)\|\partial_2\mathbf{u}\|_{L^2}\Big]\int^t_0(1+t-\tau)^{-m}(1+\tau)^{-\frac{7}{4}}\mathrm{d}\tau\nonumber\\[2mm]
&\displaystyle \quad+\sup_{0\leq \tau\leq t}\Big[(1+\tau)^{\frac{1}{8}}\|\mathbf{u}\|^{\frac{1}{4}}_{L^2}(1+\tau)^{\frac{1}{4}}\|\partial_1\mathbf{u}\|^{\frac{1}{4}}_{L^2}(1+\tau)^{\frac{1}{4}}\|\partial_2\mathbf{u}\|^{\frac{1}{4}}_{L^2}(1+\tau)^{\frac{4}{15}}\|\partial_1\partial_2\mathbf{u}\|^{\frac{1}{4}}_{L^2}\Big.
\nonumber\\[2mm]
&\displaystyle \quad\qquad \qquad~ \qquad\Big.\cdot(1+\tau)^{\frac{3}{4}}\|\nabla\partial_2\varrho\|^{\frac{3}{4}}_{L^2}\|\nabla\partial_2\partial^2_3\varrho\|^{\frac{1}{4}}_{L^2}\Big]\int^t_0(1+t-\tau)^{-m}(1+\tau)^{-\frac{197}{120}}\mathrm{d}\tau\nonumber\\[2mm]
&\lesssim \mathcal{E}^{\frac{7}{8}}_2(t)\mathcal{E}^{\frac{1}{8}}_0(t)(1+t)^{-\frac{7}{4}}+\mathcal{E}^{\frac{7}{8}}_2(t)\mathcal{E}^{\frac{1}{8}}_0(t)(1+t)^{-\frac{197}{120}}\nonumber\\[2mm]
&\lesssim\mathcal{E}(t)(1+t)^{-\frac{197}{120}}.
\end{align}
Substituting \eqref{M221} and \eqref{M222} into \eqref{M220}, we have
\begin{align}\label{M22}
M_{22} \lesssim\mathcal{E}(t)(1+t)^{-\frac{197}{120}}.
\end{align}
Now let us turn to the last term $M_{24}$. Recalling the definition of $S_2$ in \eqref{nonlinear}, we get
\begin{align}\label{M240}
\displaystyle M_{24}&\leq \int^t_0(1+t-\tau)^{-m}\Big(\|\partial_2(\mathbf{u}\cdot\nabla \mathbf{u})\|_{L^2}+\|\partial_2(g(\varrho)\Delta_h\mathbf{u})\|_{L^2}\mathrm{d}\tau \Big.\nonumber\\[2mm]
&\displaystyle\qquad \qquad+\Big.\|\partial_2(g(\varrho)\nabla\mathop{\mathrm{div}}\nolimits \mathbf{u})\|_{L^2}+\|\partial_2(f(\varrho)\nabla\varrho)\|_{L^2}\Big)\mathrm{d}\tau\nonumber\\[2mm]
&\displaystyle=:\sum^{4}_{i=1}M_{{24}_i}.
\end{align}
It follows from the Sobolev inequality \eqref{soboleve-ineq} that
for $m>2$ ,
\begin{align}\label{M241}
\displaystyle M_{{24}_1}&\lesssim \int^t_0(1+t-\tau)^{-m}(\|\partial_2\mathbf{u}\|_{L^2}\|\nabla\mathbf{u}\|_{L^{\infty}}+\|\mathbf{u}\|_{L^{\infty}}\|\nabla\partial_2\mathbf{u}\|_{L^2})\mathrm{d}\tau\nonumber\\[2mm]
&\displaystyle \lesssim\sup_{0\leq t\leq \tau}\Big[(1+\tau)\|\partial_2\mathbf{u}\|_{L^2}(1+\tau)^{\frac{3}{8}}(\log(1+\tau))^{-\frac{3}{4}}\|\nabla^2\mathbf{u}\|^{\frac{3}{4}}_{L^2}\|\nabla^4\mathbf{u}\|^{\frac{1}{4}}_{L^2}\Big]\nonumber\\[2mm]
&~\displaystyle \qquad\cdot\int^t_0(1+t-\tau)^{-m}(1+\tau)^{-\frac{11}{8}}(\log(1+\tau))^{\frac{3}{4}}\mathrm{d}\tau\nonumber\\[2mm]
&\displaystyle \quad+\sup_{0\leq t\leq \tau}\Big[(1+\tau)^{\frac{1}{4}}\|\nabla \mathbf{u}\|^{\frac{1}{2}}_{L^2}(1+\tau)^{\frac{1}{4}}(\log(1+\tau))^{-\frac{1}{2}}\|\nabla^2\mathbf{u}\|^{\frac{1}{2}}_{L^2}(1+\tau)\|\nabla\partial_2\mathbf{u}\|_{L^2}\Big]\nonumber\\
&~\displaystyle \qquad \cdot\int^t_0(1+t-\tau)^{-m}(1+\tau)^{-\frac{3}{2}}(\log(1+\tau))^{\frac{1}{2}}\mathrm{d}\tau\nonumber\\[2mm]
&\displaystyle \lesssim\mathcal{E}^{\frac{7}{8}}_2(t)\mathcal{E}^{\frac{1}{8}}_0(t)(1+t)^{-\frac{11}{8}}(\log(1+t))^{\frac{3}{4}}+\mathcal{E}_2(t)(1+t)^{-\frac{3}{2}}(\log(1+t))^{\frac{1}{2}}\nonumber\\[2mm]
& \displaystyle \lesssim\mathcal{E}(t)(1+t)^{-\frac{11}{8}}(\log(1+t))^{\frac{3}{4}}.
\end{align}
Similarly, we have
\begin{align}\label{M244}
\displaystyle M_{{24}_4}&\leq\int^t_0(1+t-\tau)^{-m}(\|\partial_2\varrho\|_{L^{\infty}}\|\nabla \varrho\|_{L^2}+\|\varrho\|_{L^{\infty}}\|\nabla\partial_2\varrho\|_{L^2})\mathrm{d}\tau\nonumber\\[2mm]
&\displaystyle\leq\sup_{0\leq\tau\leq t}\Big[(1+\tau)^{\frac{3}{4}}\|\nabla\partial_2\varrho\|^{\frac{3}{4}}_{L^2}\|\nabla^3\partial_2\varrho\|^{\frac{1}{4}}_{L^2}(1+\tau)\|\nabla\varrho\|_{L^2}\Big]
 \nonumber \\
 & ~\displaystyle \qquad \cdot\int^t_0(1+t-\tau)^{-m}(1+\tau)^{-\frac{7}{4}}\mathrm{d}\tau\nonumber\\[2mm]
&\displaystyle \quad+\sup_{0\leq \tau\leq t}\Big[(1+\tau)^{\frac{1}{2}}\|\nabla\varrho\|^{\frac{1}{2}}_{L^2}(1+\tau)^{\frac{1}{2}}\|\nabla^2\varrho\|^{\frac{1}{2}}_{L^2}(1+\tau)\|\nabla\partial_2\varrho\|_{L^2}\Big]\nonumber\\[2mm]
&~\displaystyle \qquad\cdot\int^t_0(1+t-\tau)^{-m}(1+\tau)^{-2}\mathrm{d}\tau\nonumber\\[2mm]
& \displaystyle \lesssim\mathcal{E}^{\frac{7}{8}}_2(t)\mathcal{E}^{\frac{1}{8}}_0(t)(1+t)^{-\frac{7}{4}}+\mathcal{E}_2(t)(1+t)^{-2}\nonumber\\[2mm]
& \displaystyle \lesssim\mathcal{E}(t)(1+t)^{-\frac{7}{4}}.
\end{align}
Using the Sobolev inequality \eqref{soboleve-ineq} again along with the H\"older inequality, we can bound $M_{{24}_2}$ and $M_{{24}_3}$ as follows:
\begin{align}
\displaystyle M_{{24}_2}& \lesssim\int^t_0(1+t-\tau)^{-m}(\|\partial_2\varrho\|_{L^{\infty}}\|\Delta_h\mathbf{u}\|_{L^2}+\|\varrho\|_{L^{\infty}}\|\partial_2\Delta_h\mathbf{u}\|_{L^2})\mathrm{d}\tau\nonumber\\[2mm]
&\displaystyle \lesssim \sup_{0\leq\tau\leq t}\Big[(1+\tau)^{\frac{3}{4}}\|\nabla\partial_2\varrho\|^{\frac{3}{4}}_{L^2}\|\nabla^3\partial_2\varrho\|^{\frac{1}{4}}_{L^2}(1+\tau)^{\frac{16}{15}}\|\Delta_h\mathbf{u}\|_{L^2}\Big]
 \nonumber \\
 & ~\displaystyle \qquad \cdot\int^t_0(1+t-\tau)^{-m}(1+\tau)^{-\frac{109}{60}}\mathrm{d}\tau\nonumber\\[2mm]
&\displaystyle \quad+\sup_{0\leq\tau\leq t}\Big[(1+\tau)^{\frac{1}{2}}\|\nabla\varrho\|^{\frac{1}{2}}_{L^2}(1+\tau)^{\frac{1}{2}}\|\nabla^2\varrho\|^{\frac{1}{2}}_{L^2}(1+\tau)^{\frac{8}{15}}\|\partial_2\nabla_h\mathbf{u}\|^{\frac{1}{2}}_{L^2}\|\partial_2\nabla_h\Delta_h\mathbf{u}\|^{\frac{1}{2}}_{L^2}\Big]\nonumber\\[2mm]
&~\displaystyle \qquad\cdot \int^t_0(1+t-\tau)^{-m}(1+\tau)^{-\frac{23}{15}}\mathrm{d}\tau\nonumber\\[2mm]
& \displaystyle \lesssim\mathcal{E}^{\frac{7}{8}}_2(t)\mathcal{E}^{\frac{1}{8}}_0(t)(1+t)^{-\frac{109}{60}}+\mathcal{E}^{\frac{3}{4}}_2(t)\mathcal{E}^{\frac{1}{4}}_0(t)(1+t)^{-\frac{23}{15}}\nonumber\\[2mm]
& \displaystyle \lesssim\mathcal{E}(t)(1+t)^{-\frac{23}{15}},
\end{align}
and
\begin{align}\label{M243}
\displaystyle M_{{24}_3}&\lesssim\int^t_0(1+t-\tau)^{-m}(\|\partial_2\varrho\|_{L^{\infty}}\|\nabla \mathop{\mathrm{div}}\nolimits\mathbf{u}\|_{L^2}+\|\varrho\|_{L^{\infty}}\|\nabla \partial_2\mathop{\mathrm{div}}\nolimits\mathbf{u}\|_{L^2})\mathrm{d}\tau\nonumber\\[2mm]
\displaystyle&\lesssim \sup_{0\leq \tau\leq t}\Big[
(1+\tau)^{\frac{3}{4}}\|\nabla\partial_2\varrho\|^{\frac{3}{4}}_{L^2}\|\nabla^3\partial_2\varrho\|^{\frac{1}{4}}_{L^2}(1+\tau)\|\nabla \mathop{\mathrm{div}}\nolimits\mathbf{u}\|_{L^2}\Big]
 \nonumber \\[2mm]
 &~\displaystyle \qquad \cdot\int^t_0(1+t-\tau)^{-m}(1+\tau)^{-\frac{7}{4}}\mathrm{d}\tau\nonumber\\[2mm]
&\displaystyle \quad+\sup_{0\leq \tau \leq t} \Big [(1+\tau)^{\frac{1}{2}}\|\nabla\varrho\|^{\frac{1}{2}}_{L^2}(1+\tau)^{\frac{1}{2}}\|\nabla^2\varrho\|^{\frac{1}{2}}_{L^2}(1+\tau)^{\frac{1}{2}}\|\nabla \mathop{\mathrm{div}}\nolimits\mathbf{u}\|^{\frac{1}{2}}_{L^2}\|\nabla \partial^2_2\mathop{\mathrm{div}}\nolimits\mathbf{u}\|^{\frac{1}{2}}_{L^2}\Big]\nonumber\\[2mm]
&~\displaystyle \qquad \cdot\int^t_0(1+t-\tau)^{-m}(1+\tau)^{-\frac{3}{2}}\mathrm{d}\tau\nonumber\\[2mm]
&\displaystyle \lesssim \mathcal{E}^{\frac{7}{8}}_2(t)\mathcal{E}^{\frac{1}{8}}_0(t)\mathcal{E}^{\frac{1}{2}}_1(t)\left(\int^t_0(1+t-\tau)^{-m}(1+\tau)^{-\frac{7}{4}}\mathrm{d}\tau\right)\nonumber\\[2mm]
&\displaystyle \quad+\mathcal{E}^{\frac{1}{2}}_2(t)\mathcal{E}^{\frac{1}{4}}_0(t)\mathcal{E}^{\frac{1}{4}}_0(t)\int^t_0(1+t-\tau)^{-m}(1+\tau)^{-\frac{3}{2}}\mathrm{d}\tau\nonumber\\[2mm]
& \displaystyle \lesssim\mathcal{E}^{\frac{1}{2}}_2(t)\mathcal{E}^{\frac{1}{4}}_0(t)\mathcal{E}^{\frac{1}{2}}_1(t)(1+t)^{-\frac{7}{4}}+\mathcal{E}^{\frac{1}{2}}_2(t)\mathcal{E}^{\frac{1}{2}}_1(t)(1+t)^{-\frac{3}{2}}\nonumber\\[2mm]
&\displaystyle \lesssim\mathcal{E}(t)(1+t)^{-\frac{3}{2}}.
\end{align}
Substituting \eqref{M241}--\eqref{M243} into \eqref{M240}, we get
\begin{gather}\label{M24}
M_{24} \lesssim\mathcal{E}(t)(1+t)^{-\frac{11}{8}}(\log(1+t))^{\frac{3}{4}}.
\end{gather}
This along with \eqref{varrhobiaoshi1}, \eqref{M21-23} and \eqref{M22} further implies that
\begin{gather*}
(1+t)\|\partial_2\varrho\|_{L^2} \lesssim\mathcal{E}(t)+\|(\varrho_0,\mathbf{u}_0)\|_{L^1}+\|(\partial_2\varrho_0,\partial_2\mathbf{u}_0)\|_{L^2}.
\end{gather*}
By similar arguments for the case $ i=2 $, we get for the case $ i=1 $ that
\begin{equation*}
(1+t)\|\partial_1\varrho\|_{L^2} \lesssim\mathcal{E}(t)+\|(\varrho_0,\mathbf{u}_0)\|_{L^1}+\|(\partial_1\varrho_0,\partial_1\mathbf{u}_0)\|_{L^2}.
\end{equation*}
As for $\|\partial_2\mathbf{u}\|_{L^2}$, we utilize \eqref{anisou-lem-con} and \eqref{deri-rho-u-reprens} to get
\begin{align}\label{ubiaoshi1}
\displaystyle\|\partial_2\mathbf{u}\|_{L^2}
&\lesssim (1+t)^{-\frac{5}{4}}(\|\varrho_0\|_{L^1}+\|\partial_2\varrho_0\|_{L^2})+(1+t)^{-1}(\|\mathbf{u}_0\|_{L^2_{x_3}L^1_{x_1x_2}}+\|\partial_2\mathbf{u}_0\|_{L^2})\nonumber\\[2mm]
&\displaystyle \quad+\int^t_0(1+t-\tau)^{-\frac{5}{4}}\|S_1\|_{L^1}\mathrm{d}\tau+\int^t_0(1+t-\tau)^{-m}\|\partial_2S_1\|_{L^2}\mathrm{d}\tau\nonumber\\[2mm]
&\displaystyle \quad+\int^t_0(1+t-\tau)^{-1}\|S_2\|_{L^2_{x_3}L^1_{x_1x_2}}\mathrm{d}\tau
+\int^t_0(1+t-\tau)^{-m}\|\partial_2S_2\|_{L^2}\mathrm{d}\tau\nonumber\\[2mm]
&\displaystyle \lesssim (1+t)^{-\frac{5}{4}}(\|\varrho_0\|_{L^1}+\|\partial_2\varrho_0\|_{L^2})+(1+t)^{-1}(\|\mathbf{u}_0\|_{L^2_{x_3}L^1_{x_1x_2}}+\|\partial_2\mathbf{u}_0\|_{L^2})
 \nonumber \\[2mm]
 & \displaystyle \quad
+\sum^4_{i=1}H_{2i}.
\end{align}
In the same method as in the derivation of \eqref{M21-23}, \eqref{M22} and \eqref{M24}, we have
\begin{gather*}
\displaystyle
H_{21}\lesssim\mathcal{E}(t)(1+t)^{-\frac{5}{4}},\ \ \ H_{22}\lesssim\mathcal{E}(t)(1+t)^{-\frac{197}{120}},\ \ \ H_{24} \lesssim\mathcal{E}(t)(1+t)^{-\frac{11}{8}}(\log(1+t))^{\frac{3}{4}}.
\end{gather*}
By similar arguments as in the derivation of \eqref{H03}, we can estimate $H_{23}$ as follows:
\begin{align*}\label{H23}
\displaystyle H_{23}& \lesssim \int^t_0(1+t-\tau)^{-1}\Big(\|\mathbf{u}\cdot\nabla \mathbf{u}\|_{L^2_{x_3}L^1_{x_1x_2}}+\|g(\varrho)\Delta_h\mathbf{u}\|_{L^2_{x_3}L^1_{x_1x_2}}\Big.\nonumber\\[2mm]
&\displaystyle \qquad \qquad \qquad \qquad\Big.+\|g(\varrho)\nabla \mathop{\mathrm{div}}\nolimits\mathbf{u}\|_{L^2_{x_3}L^1_{x_1x_2}}+\|f(\varrho)\nabla\varrho\|_{L^2_{x_3}L^1_{x_1x_2}}\Big)\mathrm{d}\tau\nonumber\\[2mm]
&\displaystyle \lesssim \mathcal{E}_2(t)\int^t_0(1+t-\tau)^{-1}(1+\tau)^{-\frac{3}{2}}\mathrm{d}\tau+\mathcal{E}^{\frac{3}{4}}_2(t)\mathcal{E}^{\frac{1}{4}}_1(t)\int^t_0(1+t-\tau)^{-1}(1+\tau)^{-\frac{5}{4}}\mathrm{d}\tau\nonumber\\[2mm]
&\displaystyle \quad+ (\mathcal{E}_2(t)+\mathcal{E}^{\frac{1}{2}}_2(t)\mathcal{E}^{\frac{1}{2}}_1(t))\int^t_0(1+t-\tau)^{-1}(1+\tau)^{-\frac{15}{8}}(\log(1+\tau))^{\frac{1}{2}}\mathrm{d}\tau\nonumber\\[2mm]
&\displaystyle \lesssim\mathcal{E}(t)(1+t)^{-1},
\end{align*}
where we have used \eqref{two} and Lemma \ref{decaylemma}. Therefore we have from \eqref{ubiaoshi1} that
\begin{gather*}
(1+t)\|\partial_2\mathbf{u}\|_{L^2}\lesssim\mathcal{E}(t)+\|\varrho_0\|_{L^1}+\|\mathbf{u}_0\|_{L^2_{x_3}L^1_{x_1x_2}}+\|(\partial_2\varrho_0,\partial_2\mathbf{u}_0)\|_{L^2}.
\end{gather*}
Similarly, we can derive that
\begin{gather*}
(1+t)\|\partial_1\mathbf{u}\|_{L^2}\lesssim\mathcal{E}(t)+\|\varrho_0\|_{L^1}+\|\mathbf{u}_0\|_{L^2_{x_3}L^1_{x_1x_2}}+\|(\partial_1\varrho_0,\partial_1\mathbf{u}_0)\|_{L^2}.
\end{gather*}

Next let us turn to {\bf the case (ii)}: $ i=3 $. By \eqref{vrho-con-line-lem} and \eqref{deri-rho-u-reprens}, we get
\begin{align}\label{varrhobiaoshi13}
\displaystyle\|\partial_3\varrho\|_{L^2}& \lesssim (1+t)^{-\frac{5}{4}}(\|(\varrho_0,\mathbf{u}_0)\|_{L^1}+\|(\partial_3\varrho_0,\partial_3\mathbf{u}_0)\|_{L^2})\nonumber\\[2mm]
&\displaystyle \quad+\int^t_0(1+t-\tau)^{-\frac{5}{4}}\|S_1\|_{L^1}\mathrm{d}\tau+\int^t_0(1+t-\tau)^{-m}\|\partial_3S_1\|_{L^2}\mathrm{d}\tau\nonumber\\[2mm]
&\displaystyle \quad+\int^t_0(1+t-\tau)^{-\frac{5}{4}}\|S_2\|_{L^1}\mathrm{d}\tau+\int^t_0(1+t-\tau)^{-m}\|\partial_3S_2\|_{L^2}\mathrm{d}\tau\nonumber\\[2mm]
&\displaystyle=:(1+t)^{-\frac{5}{4}}(\|(\varrho_0,\mathbf{u}_0)\|_{L^1}+\|(\partial_3\varrho_0,\partial_3\mathbf{u}_0)\|_{L^2})+\sum^{4}_{i=1}M_{3i},
\end{align}
where, in view of similar arguments in the derivation of $N_1$, $N_2$ and $N_3$ in Lemma \ref{divulemma}, we readily obtain that
\begin{gather}\label{M31}
M_{31}\lesssim\mathcal{E}(t)(1+t)^{-\frac{5}{4}},\ \ \ M _{32}\lesssim\mathcal{E}(t)(1+t)^{-\frac{197}{120}},\ \ \ M _{33}\lesssim \mathcal{E}(t)(1+t)^{-1}.
\end{gather}
We now proceed to the estimate of $M_{34}$. It follows from direct computation that
\begin{align}\label{M340}
\displaystyle M_{34}& \lesssim \int^t_0(1+t-\tau)^{-m}\|\partial_3(\mathbf{u}\cdot\nabla \mathbf{u})\|_{L^2}\mathrm{d}\tau+\int^t_0(1+t-\tau)^{-m}\|\partial_3(g(\varrho)\Delta_h\mathbf{u})\|_{L^2}\mathrm{d}\tau\nonumber\\[2mm]
&\displaystyle\quad+\int^t_0(1+t-\tau)^{-m}\|\partial_3(g(\varrho)\nabla\mathop{\mathrm{div}}\nolimits \mathbf{u})\|_{L^2}\mathrm{d}\tau+\int^t_0(1+t-\tau)^{-m}\|\partial_3(f(\varrho)\nabla\varrho)\|_{L^2}\mathrm{d}\tau\nonumber\\[2mm]
&=\sum^{4}_{i=1}M_{{34}_i}.
\end{align}
For the term $M_{34 _{1}}$, we have
\begin{align}\label{M3410}
\displaystyle
&\makebox[-6pt]{~}M_{{34}_1}
 \nonumber \\
 &\makebox[-6pt]{~}~\displaystyle \lesssim \int^t_0(1+t-\tau)^{-m}\Big(\|\partial_3\mathbf{u}_h\cdot\nabla_h\mathbf{u}\|_{L^2}+\|\partial_3u_3\partial_3\mathbf{u}\|_{L^{2}}+\|u_h\cdot\nabla_h\partial_3\mathbf{u}\|_{L^2}+\|u_3\partial^2_3\mathbf{u}\|_{L^2}\Big)\mathrm{d}\tau\nonumber\\[2mm]
&\makebox[-6pt]{~}~\displaystyle=: M_{{34}_{11}}+M_{{34}_{12}}+M_{{34}_{13}}+M_{{34}_{14}}.
\end{align}
In the following, we shall handle the terms on the right hand side of \eqref{M3410}. Thanks to \eqref{two-two-deri-ani}, we get
\begin{align*}
\displaystyle \|\partial_3\mathbf{u}_h\cdot\nabla_h\mathbf{u}\|_{L^2} \lesssim \|\partial_3\mathbf{u}_h\|^{\frac{1}{4}}_{L^2}\|\partial_1\partial_3\mathbf{u}_h\|^{\frac{1}{4}}_{L^2}\|\partial_2\partial_3\mathbf{u}_h\|^{\frac{1}{4}}_{L^2}\|\partial_1\partial_2\partial_3\mathbf{u}_h\|^{\frac{1}{4}}_{L^2}\|\nabla_h\mathbf{u}\|^{\frac{1}{2}}_{L^2}\|\nabla_h\partial_3\mathbf{u}\|^{\frac{1}{2}}_{L^2}.
\end{align*}
Therefore we derive that
\begin{align}\label{M3411}
M_{{34}_{11}}&=\int^t_0(1+t-\tau)^{-m}\|\partial_3\mathbf{u}_h\cdot\nabla_h\mathbf{u}\|_{L^2}\mathrm{d}\tau
 \nonumber \\[2mm]
 &\displaystyle \lesssim\sup_{0\leq \tau\leq t}\Big[(1+\tau)^{\frac{1}{8}}\|\partial_3\mathbf{u}_h\|_{L^2}^{\frac{1}{4}}(1+\tau)^{\frac{1}{4}}\|\partial_1\partial_3\mathbf{u}_h\|_{L^2}^{\frac{1}{4}}(1+\tau)^{\frac{3}{8}}\|\partial_2\partial_3\mathbf{u}_h\|_{L^2}^{\frac{3}{8}}\|\partial^2_1\partial_2\partial_3\mathbf{u}_h\|_{L^2}^{\frac{1}{8}}\Big.\nonumber\\[2mm]
&~\displaystyle \qquad\quad \Big. \cdot(1+\tau)^{\frac{1}{2}}\|\nabla_h\mathbf{u}\|^{\frac{1}{2}}_{L^2}(1+\tau)^{\frac{1}{2}}\|\nabla_h\partial_3\mathbf{u}\|^{\frac{1}{2}}_{L^2} \Big]\int^t_0(1+t-\tau)^{-m}(1+\tau)^{-\frac{7}{4}}\mathrm{d}\tau\nonumber\\
&\displaystyle \lesssim \mathcal{E}^{\frac{15}{16}}_2(t) \mathcal{E}^{\frac{1}{16}}_0(t)\int^t_0(1+t-\tau)^{-m}(1+\tau)^{-\frac{7}{4}}\mathrm{d}\tau\nonumber\\
&\displaystyle \lesssim\mathcal{E}(t)(1+t)^{-\frac{7}{4}},
\end{align}
where we have used Lemma \ref{decaylemma}. Similarly, we can bound $ M _{34 _{12}} $--$ M _{34  _{14}} $ as follows:
\begin{align}
\displaystyle \makebox[-8pt]{~}
M_{{34}_{12}}&\lesssim\int^t_0(1+t-\tau)^{-m}\|\partial_3u_3\|^{\frac{1}{4}}_{L^2}\|\partial_1\partial_3u_3\|^{\frac{1}{4}}_{L^2}\|\partial_2\partial_3u_3\|^{\frac{1}{4}}_{L^2}\|\partial_1\partial_2\partial_3u_3\|^{\frac{1}{4}}_{L^2}\|\partial_3\mathbf{u}\|^{\frac{1}{2}}_{L^2}\|\partial^2_3\mathbf{u}\|^{\frac{1}{2}}_{L^2}\mathrm{d}\tau\nonumber\\[2mm]
&\displaystyle \lesssim\sup_{0\leq \tau\leq t}\Big[(1+\tau)^{\frac{1}{4}}\|\partial_3u_3\|^{\frac{1}{4}}_{L^2}(1+\tau)^{\frac{1}{4}}\|\partial_1\partial_3u_3\|^{\frac{1}{4}}_{L^2}(1+\tau)^{\frac{3}{8}}\|\partial_2\partial_3u_3\|^{\frac{3}{8}}_{L^2}\|\partial_1\partial^2_2\partial_3u_3\|^{\frac{1}{8}}_{L^2} \Big. \nonumber\\[2mm]
&~~\displaystyle \Big. \cdot(1+\tau)^{\frac{1}{4}}\|\partial_3\mathbf{u}\|^{\frac{1}{2}}_{L^2}(1+\tau)^{\frac{1}{4}}(\log(1+\tau))^{-\frac{1}{2}}\|\partial^2_3\mathbf{u}\|^{\frac{1}{2}}_{L^2}\Big]
\int^t_0(1+t-\tau)^{-m}(1+\tau)^{-\frac{11}{8}}(\log(1+\tau))^{\frac{1}{2}}\mathrm{d}\tau\nonumber\\[2mm]
&\displaystyle \lesssim \mathcal{E}^{\frac{1}{16}}_0(t)\mathcal{E}^{\frac{15}{16}}_2(t)\int^t_0(1+t-\tau)^{-m}(1+\tau)^{-\frac{11}{8}}(\log(1+\tau))^{\frac{1}{2}}\mathrm{d}\tau\nonumber\\[2mm]
& \displaystyle \lesssim\mathcal{E}(t)(1+t)^{-\frac{11}{8}}\log(1+t),
\end{align}
\begin{align}
\displaystyle  \makebox[-8pt]{~}M_{{34}_{13}}& \lesssim
\int^t_0(1+t-\tau)^{-m}\|\mathbf{u}_h\|^{\frac{1}{4}}_{L^2}\|\partial_1\mathbf{u}_h\|^{\frac{1}{4}}_{L^2}\|\partial_2\mathbf{u}_h\|^{\frac{1}{4}}_{L^2}\|\partial_1\partial_2\mathbf{u}_h\|^{\frac{1}{4}}_{L^2}\|\nabla_h\partial_3\mathbf{u}\|^{\frac{1}{2}}_{L^2}\|\nabla_h\partial^2_3\mathbf{u}\|^{\frac{1}{2}}_{L^2}\mathrm{d}\tau\nonumber\\[2mm]
&\displaystyle \lesssim\sup_{0\leq \tau\leq t}\Big[(1+\tau)^{\frac{1}{8}}\|\mathbf{u}_h\|^{\frac{1}{4}}_{L^2}(1+\tau)^{\frac{1}{4}}\|\partial_1\mathbf{u}_h\|^{\frac{1}{4}}_{L^2}(1+\tau)^{\frac{1}{4}}\|\partial_2\mathbf{u}_h\|^{\frac{1}{4}}_{L^2}(1+\tau)^{\frac{4}{15}}\|\partial_1\partial_2\mathbf{u}_h\|^{\frac{1}{4}}_{L^2} \Big.\nonumber\\[2mm]
&~\displaystyle \qquad \qquad \Big.\cdot(1+\tau)^{\frac{3}{4}}\|\nabla_h\partial_3\mathbf{u}\|^{\frac{3}{4}}_{L^2}\|\nabla_h\partial^3_3\mathbf{u}\|^{\frac{1}{4}}_{L^2}\Big] \int^t_0(1+t-\tau)^{-m}(1+\tau)^{-\frac{197}{120}}\mathrm{d}\tau\nonumber\\[2mm]
&\leq \mathcal{E}^{\frac{7}{8}}_2(t)\mathcal{E}^{\frac{1}{8}}_0(t)\int^t_0(1+t-\tau)^{-m}(1+\tau)^{-\frac{197}{120}}\mathrm{d}\tau\nonumber\\[2mm]
&\leq \mathcal{E}(t)(1+t)^{-\frac{197}{120}},
\end{align}
and
\begin{align}\label{M3414}
\displaystyle  \makebox[-8pt]{~}M_{{34}_{14}}& \lesssim
\int^t_0(1+t-\tau)^{-m}(\|u_3\|^{\frac{1}{4}}_{L^2}\|\partial_1u_3\|^{\frac{1}{4}}_{L^2}\|\partial_2u_3\|^{\frac{1}{4}}_{L^2}\|\partial_1\partial_2u_3\|^{\frac{1}{4}}_{L^2}\|\partial^2_3\mathbf{u}\|^{\frac{1}{2}}_{L^2}\|\partial^3_3\mathbf{u}\|^{\frac{1}{2}}_{L^2})\mathrm{d}\tau\nonumber\\[2mm]
&\displaystyle \lesssim\sup_{0\leq\tau\leq t} \Big [(1+\tau)^{\frac{1}{8}}\|u_3\|^{\frac{1}{4}}_{L^2}(1+\tau)^{\frac{1}{4}}\|\partial_1u_3\|^{\frac{1}{4}}_{L^2}(1+\tau)^{\frac{1}{4}}\|\partial_2u_3\|^{\frac{1}{4}}_{L^2}(1+\tau)^{\frac{4}{15}}\|\partial_1\partial_2u_3\|^{\frac{1}{4}}_{L^2}\Big.\nonumber\\[2mm]
&~\displaystyle \qquad \Big.\cdot(1+\tau)^{\frac{3}{8}}(\log(1+\tau))^{-\frac{3}{4}}\|\partial^2_3\mathbf{u}\|^{\frac{3}{4}}_{L^2}\|\partial^4_3\mathbf{u}\|^{\frac{1}{4}}_{L^2}\Big]\int^t_0(1+t-\tau)^{-m}(1+\tau)^{-\frac{19}{15}}(\log(1+\tau))^{\frac{3}{4}}\mathrm{d}\tau\nonumber\\[2mm]
&\displaystyle \lesssim\mathcal{E}^{\frac{7}{8}}_2(t)\mathcal{E}^{\frac{1}{8}}_0(t)\int^t_0(1+t-\tau)^{-m}(1+\tau)^{-\frac{19}{15}}(\log(1+\tau))^{\frac{3}{4}}\mathrm{d}\tau\nonumber\\[2mm]
&\displaystyle \lesssim\mathcal{E}(t)(1+t)^{-\frac{19}{15}}(\log(1+t))^{\frac{3}{4}}.
\end{align}
Here we have used the fact
\begin{align*}
\displaystyle \partial _{3}u _{3}=\mathop{\mathrm{div}}\nolimits \mathbf{u}- \partial _{1}u _{1}- \partial _{2}u _{2}.
\end{align*}
Substituting \eqref{M3411}--\eqref{M3414} into \eqref{M3410}, we get
\begin{align}\label{M341}
M_{{34}_1}=\int^t_0(1+t-\tau)^{-m}\|\partial_3(\mathbf{u}\cdot\nabla \mathbf{u})\|_{L^2}\mathrm{d}\tau \lesssim\mathcal{E}(t)(1+t)^{-\frac{19}{15}}(\log(1+t))^{\frac{3}{4}}.
\end{align}
 The terms $M_{{34}_2}$--$M_{{34}_4}$ can be estimated by the same method as in $M_{{24}_2}$--$M_{{24}_4}$~(see precisely \eqref{M244}--\eqref{M343}), as follows:
\begin{align}\label{M342}
\displaystyle M_{{34}_2}&\lesssim \sup_{0\leq\tau\leq t}\Big[(1+\tau)^{\frac{3}{4}}\|\nabla\partial_3\varrho\|^{\frac{3}{4}}_{L^2}\|\nabla^3\partial_3\varrho\|^{\frac{1}{4}}_{L^2}(1+\tau)^{\frac{16}{15}}\|\Delta_h\mathbf{u}\|_{L^2}\Big]
 \nonumber \\[2mm]
 & ~\displaystyle \qquad \cdot\int^t_0(1+t-\tau)^{-m}(1+\tau)^{-\frac{109}{60}}\mathrm{d}\tau\nonumber\\[2mm]
&\displaystyle \quad+\sup_{0\leq\tau\leq t}\Big[(1+\tau)^{\frac{1}{2}}\|\nabla\varrho\|^{\frac{1}{2}}_{L^2}(1+\tau)^{\frac{1}{2}}\|\nabla^2\varrho\|^{\frac{1}{2}}_{L^2}(1+\tau)^{\frac{1}{2}}\|\partial_3\nabla_h\mathbf{u}\|^{\frac{1}{2}}_{L^2}\|\partial_3\nabla_h\Delta_h\mathbf{u}\|^{\frac{1}{2}}_{L^2}\Big]\nonumber\\[2mm]
&~\displaystyle \qquad\cdot\int^t_0(1+t-\tau)^{-m}(1+\tau)^{-\frac{3}{2}}\mathrm{d}\tau\nonumber\\[2mm]
&\lesssim \mathcal{E}^{\frac{7}{8}}_2(t)\mathcal{E}^{\frac{1}{8}}_0(t)(1+t)^{-\frac{109}{60}}+\mathcal{E}^{\frac{3}{4}}_2(t)\mathcal{E}^{\frac{1}{4}}_0(t)(1+t)^{-\frac{3}{2}}\nonumber\\[2mm]
&\lesssim \mathcal{E}(t)(1+t)^{-\frac{3}{2}},
\end{align}
\begin{align}\label{M343}
\displaystyle M_{{34}_3}&\leq \sup_{0\leq \tau\leq t}
\Big[(1+\tau)^{\frac{3}{4}}\|\nabla\partial_3\varrho\|^{\frac{3}{4}}_{L^2}\|\nabla^3\partial_3\varrho\|^{\frac{1}{4}}_{L^2}(1+\tau)\|\nabla \mathop{\mathrm{div}}\nolimits\mathbf{u}\|_{L^2}\Big]
 \nonumber \\[2mm]
 & ~\displaystyle \qquad \cdot\int^t_0(1+t-\tau)^{-m}(1+\tau)^{-\frac{7}{4}}\mathrm{d}\tau\nonumber\\[2mm]
&\displaystyle \quad+\sup_{0\leq\tau}\Big[(1+\tau)^{\frac{1}{2}}\|\nabla\varrho\|^{\frac{1}{2}}_{L^2}(1+\tau)^{\frac{1}{2}}\|\nabla^2\varrho\|^{\frac{1}{2}}_{L^2}(1+\tau)^{\frac{1}{2}}\|\nabla\mathop{\mathrm{div}}\nolimits\mathbf{u}\|^{\frac{1}{2}}_{L^2}\|\nabla\partial^2_3 \mathop{\mathrm{div}}\nolimits\mathbf{u}\|^{\frac{1}{2}}_{L^2}\Big]\nonumber\\[2mm]
&~\displaystyle \qquad \cdot\int^t_0(1+t-\tau)^{-m}(1+\tau)^{-\frac{3}{2}}\mathrm{d}\tau\nonumber\\[2mm]
&\displaystyle \lesssim\mathcal{E}^{\frac{3}{8}}_2(t)\mathcal{E}^{\frac{1}{8}}_0(t)\mathcal{E}^{\frac{1}{2}}_1(t)(1+t)^{-\frac{7}{4}}+\mathcal{E}^{\frac{1}{2}}_2(t)\mathcal{E}^{\frac{1}{4}}_1(t)\mathcal{E}^{\frac{1}{4}}_0(t)(1+t)^{-\frac{3}{2}}\nonumber\\[2mm]
&\displaystyle \lesssim\mathcal{E}(t)(1+t)^{-\frac{3}{2}},
\end{align}
and
\begin{align}\label{M344}
\displaystyle M_{{34}_4}
& \lesssim\sup_{0\leq\tau\leq t}\Big[(1+\tau)^{\frac{3}{4}}\|\nabla\partial_3\varrho\|^{\frac{3}{4}}_{L^2}\|\nabla^3\partial_3\varrho\|^{\frac{1}{4}}_{L^2}(1+\tau)\|\nabla\varrho\|_{L^2}\Big]
 \nonumber \\
 &~\displaystyle \qquad \cdot\int^t_0(1+t-\tau)^{-m}(1+\tau)^{-\frac{7}{4}}\mathrm{d}\tau\nonumber\\[2mm]
&\displaystyle \quad+\sup_{0\leq \tau\leq t}\Big[(1+\tau)^{\frac{1}{2}}\|\nabla\varrho\|^{\frac{1}{2}}_{L^2}(1+\tau)^{\frac{1}{2}}\|\nabla^2\varrho\|^{\frac{1}{2}}_{L^2}(1+\tau)\|\nabla\partial_3\varrho\|_{L^2}\Big]\nonumber\\[2mm]
&~\displaystyle \qquad\cdot\int^t_0(1+t-\tau)^{-m}(1+\tau)^{-\frac{3}{2}}\mathrm{d}\tau\nonumber\\[2mm]
&\displaystyle \lesssim\mathcal{E}^{\frac{7}{8}}_2(t)\mathcal{E}^{\frac{1}{8}}_0(t)(1+t)^{-\frac{7}{4}}+\mathcal{E}_2(t)(1+t)^{-\frac{3}{2}}\nonumber\\[2mm]
&\displaystyle \lesssim\mathcal{E}(t)(1+t)^{-\frac{3}{2}}.
\end{align}
Gathering \eqref{M341}--\eqref{M344}, we get from \eqref{M340} that
\begin{align*}
M_{34}=\int^t_0(1+t-\tau)^{-m}\|\partial_3S_2\|_{L^2}\mathrm{d}\tau \lesssim\mathcal{E}(t)(1+t)^{-\frac{19}{15}}(\log(1+t))^{\frac{3}{4}}.
\end{align*}
This along with \eqref{varrhobiaoshi13} and \eqref{M31} implies that
\begin{gather*}
(1+t)\|\partial_3\varrho\| _{L ^{2}} \lesssim\mathcal{E}(t)+\|(\varrho_0,\mathbf{u}_0)\|_{L^1}+\|(\partial_3\varrho_0,\partial_3\mathbf{u}_0)\|_{L^2}.
\end{gather*}
Now let us turn to the estimate of $\|\partial_3\mathbf{u}(t)\|_{L^2}$. From \eqref{verti-u-conl-lem} and \eqref{deri-rho-u-reprens}, we have
\begin{align}\label{ubiaoshi13}\displaystyle
\|\partial_3\mathbf{u}\|_{L^2}
&\displaystyle \lesssim (1+t)^{-\frac{5}{4}}(\|\varrho_0\|_{L^1}+\|\partial_3\varrho_0\|_{L^2})+(1+t)^{-\frac{1}{2}}(\|\partial_3\mathbf{u}_0\|_{L^2_{x_3}L^1_{x_1x_2}}+\|\partial_3\mathbf{u}_0\|_{L^2})\nonumber\\[2mm]
&\displaystyle \quad+\int^t_0(1+t-\tau)^{-\frac{5}{4}}\|S_1\|_{L^1}\mathrm{d}\tau+\int^t_0(1+t-\tau)^{-m}\|\partial_3S_1\|_{L^2}\mathrm{d}\tau\nonumber\\[2mm]
&\displaystyle \quad+\int^t_0(1+t-\tau)^{-\frac{1}{2}}\|\partial_3S_2\|_{L^2_{x_3}L^1_{x_1x_2}}\mathrm{d}\tau+\int^t_0(1+t-\tau)^{-m}\|\partial_3S_2\|_{L^2}\mathrm{d}\tau\nonumber\\[2mm]
&\displaystyle \lesssim (1+t)^{-\frac{5}{4}}(\|\varrho_0\|_{L^1}+\|\partial_2\varrho_0\|_{L^2})+(1+t)^{-\frac{1}{2}}(\|\partial_3\mathbf{u}_0\|_{L^2_{x_3}L^1_{x_1x_2}}+\|\partial_3\mathbf{u}_0\|_{L^2})\nonumber\\[2mm]
&~\displaystyle \quad+H_{31}+H_{32}+H_{33}+H_{34},
\end{align}
where $H_{31}$, $H_{32}$ and $H_{34}$ can be estimated by similar arguments as in the derivation of the estimates of $M_{31}$, $M_{32}$ and $M_{34}$. Specifically, we get
\begin{gather}\label{H-3-124}
 H_{31}\lesssim \mathcal{E}(t)(1+t)^{-\frac{5}{4}},\ \  H _{32} \lesssim \mathcal{E}(t)(1+t)^{-\frac{197}{120}},\ \mbox{and \ } H _{34} \lesssim \mathcal{E}(t)(1+t)^{-\frac{19}{15}}(\log(1+t))^{\frac{3}{4}}.
\end{gather}
For the term $H_{33}$, we have
\begin{align}\label{H330}
&\makebox[-8pt]{~}\displaystyle H_{33}
 \nonumber \\
 &\makebox[-8pt]{~}~\displaystyle \lesssim \int^t_0(1+t-\tau)^{-\frac{1}{2}}\|\partial_3(\mathbf{u}\cdot\nabla \mathbf{u})\|_{L^2_{x_3}L^1_{x_1x_2}} \mathrm{d}\tau +\int^t_0(1+t-\tau)^{-\frac{1}{2}}\|\partial_3(g(\varrho)\Delta_h\mathbf{u})\|_{L^2_{x_3}L^1_{x_1x_2}}\mathrm{d}\tau\nonumber\\
&\makebox[-8pt]{~}~~\displaystyle  \quad+\int^t_0(1+t-\tau)^{-\frac{1}{2}}\|\partial_3(g(\varrho)\nabla\mathop{\mathrm{div}}\nolimits\mathbf{u})\|_{L^2_{x_3}L^1_{x_1x_2}}\mathrm{d}\tau
 \nonumber \\
 &\makebox[-8pt]{~}~~\displaystyle  \quad
+\int^t_0(1+t-\tau)^{-\frac{1}{2}}\|\partial_3(f(\varrho)\nabla\varrho)\|_{L^2_{x_3}L^1_{x_1x_2}}\mathrm{d}\tau=:\sum^{4}_{i=1}H_{{33}_i}.
\end{align}
Thanks to \eqref{two}, we derive that
\begin{align}\label{p3d}
\displaystyle
\|\partial_3(\mathbf{u}\cdot\nabla \mathbf{u})\|_{L^2_{x_3}L^1_{x_1x_2}}& \lesssim \|\partial_3\mathbf{u}_h\|^{\frac{1}{2}}_{L^2}\|\partial^2_3\mathbf{u}_h\|^{\frac{1}{2}}_{L^2}\|\nabla_h\mathbf{u}\|_{L^2}+\|\partial_3\mathbf{u}\|^{\frac{1}{2}}_{L^2}\|\partial^2_3\mathbf{u}\|^{\frac{1}{2}}_{L^2}\|\partial_3u_3\|_{L^2}\nonumber\\[2mm]
&\displaystyle \quad+\|\mathbf{u}_h\|^{\frac{1}{2}}_{L^2}\|\partial_3\mathbf{u}_h\|^{\frac{1}{2}}_{L^2}\|\nabla_h\partial_3\mathbf{u}\|_{L^2}+\|u_3\|^{\frac{1}{2}}_{L^2}\|\partial_3u_3\|^{\frac{1}{2}}_{L^2}\|\partial^2_3\mathbf{u}\|_{L^2}.
\end{align}
This along with Lemma \ref{decaylemma} yields that
\begin{align}\label{H33_1}
\displaystyle H_{{33}_1}&\lesssim \sup_{0\leq \tau\leq t}\Big[(1+\tau)^{\frac{1}{4}}\|\partial_3\mathbf{u}_h\|^{\frac{1}{2}}_{L^2}(1+\tau)^{\frac{1}{4}}(\log(1+\tau))^{-\frac{1}{2}}\|\partial^2_3\mathbf{u}_h\|^{\frac{1}{2}}_{L^2}(1+\tau)\|\nabla_h\mathbf{u}\|_{L^2}\Big]\nonumber\\[2mm]
&~\displaystyle \qquad\cdot\int^t_0(1+t-\tau)^{-\frac{1}{2}}(1+\tau)^{-\frac{3}{2}}(\log(1+\tau))^{\frac{1}{2}}\mathrm{d}\tau\nonumber\\[2mm]
&\displaystyle \quad +\sup_{0\leq\tau\leq t}\Big[(1+\tau)^{\frac{1}{4}}\|\partial_3\mathbf{u}\|^{\frac{1}{2}}_{L^2}(1+\tau)^{\frac{1}{4}}(\log(1+\tau))^{-\frac{1}{2}}\|\partial^2_3\mathbf{u}\|^{\frac{1}{2}}_{L^2}(1+\tau)\|\partial_3u_3\|_{L^2}\Big]\nonumber\\[2mm]
&~\displaystyle \quad\qquad\cdot\int^t_0(1+t-\tau)^{-\frac{1}{2}}(1+\tau)^{-\frac{3}{2}}(\log(1+\tau))^{\frac{1}{2}}\mathrm{d}\tau\nonumber\\[2mm]
&\displaystyle \quad+\sup_{0\leq\tau\leq t}\Big[(1+\tau)^{\frac{1}{4}}\|\mathbf{u}_h\|^{\frac{1}{2}}_{L^2}(1+\tau)^{\frac{1}{4}}\|\partial_3\mathbf{u}_h\|^{\frac{1}{2}}_{L^2}(1+\tau)\|\nabla_h\partial_3\mathbf{u}\|_{L^2}\Big]\nonumber\\[2mm]
&~\displaystyle \quad\qquad\cdot\int^t_0(1+t-\tau)^{-\frac{1}{2}}(1+\tau)^{-\frac{3}{2}}\mathrm{d}\tau\nonumber\\[2mm]
&\displaystyle \quad+\sup_{0\leq\tau\leq t}\Big[(1+\tau)^{\frac{1}{4}}\|u_3\|^{\frac{1}{2}}_{L^2}(1+\tau)^{\frac{1}{2}}\|\partial_3u_3\|^{\frac{1}{2}}_{L^2}(1+\tau)^{\frac{1}{2}}(\log(1+\tau))^{-1}\|\partial^2_3\mathbf{u}\|_{L^2})\Big]\nonumber\\[2mm]
&~\displaystyle \qquad\cdot\int^t_0(1+t-\tau)^{-\frac{1}{2}}(1+\tau)^{-\frac{5}{4}}\log(1+\tau)\mathrm{d}\tau\nonumber\\[2mm]
&\displaystyle \lesssim\mathcal{E}_2(t)(1+t)^{-\frac{1}{2}}+\mathcal{E}^{\frac{1}{2}}_2(t)\mathcal{E}^{\frac{1}{2}}_1(t) (1+t)^{-\frac{1}{2}}+\mathcal{E}^{\frac{3}{4}}_2(t)\mathcal{E}^{\frac{1}{4}}_1(t)(1+t)^{-\frac{1}{2}}\nonumber\\[2mm]
&\displaystyle \lesssim\mathcal{E}(t)(1+t)^{-\frac{1}{2}}.
\end{align}
Similarly, using \eqref{two} and Lemma \ref{decaylemma} again, we have
\begin{align}\label{H33_2}
\displaystyle H_{{33}_2}&\lesssim\sup_{0\leq\tau\leq t}\Big[(1+\tau)^{\frac{1}{2}}\|\partial_3\varrho\|^{\frac{1}{2}}_{L^2}(1+\tau)^{\frac{1}{2}}\|\partial^2_3\varrho\|^{\frac{1}{2}}_{L^2}(1+\tau)^{\frac{16}{15}}\|\Delta_h\mathbf{u}\|_{L^2}\Big]\nonumber\\
&\displaystyle \qquad\cdot\int^t_0(1+t-\tau)^{-\frac{1}{2}}(1+\tau)^{-\frac{31}{15}}\mathrm{d}\tau\nonumber\\
&\displaystyle \quad+\sup_{0\leq\tau\leq t}\Big[(1+\tau)^{\frac{3}{8}}(\log(1+\tau))^{-\frac{1}{2}}\|\varrho\|^{\frac{1}{2}}_{L^2}(1+\tau)^{\frac{1}{2}}\|\partial_3\varrho\|^{\frac{1}{2}}_{L^2}(1+\tau)^{\frac{1}{2}}\|\partial_3\nabla_h\mathbf{u}\|^{\frac{1}{2}}_{L^2}\|\partial_3\nabla_h\Delta_h\mathbf{u}\|^{\frac{1}{2}}_{L^2}\Big]\nonumber\\
&\displaystyle \quad \qquad\cdot\int^t_0(1+t-\tau)^{-\frac{1}{2}}(1+\tau)^{-\frac{11}{8}}(\log(1+\tau))^{\frac{1}{2}}\mathrm{d}\tau\nonumber\\
&\displaystyle \lesssim\mathcal{E}(t)(1+t)^{-\frac{1}{2}},
\end{align}
\begin{align}\label{H33_3}
\displaystyle H_{{33}_3}&\lesssim\sup_{0\leq\tau\leq t}\Big[(1+\tau)^{\frac{1}{2}}\|\partial_3\varrho\|^{\frac{1}{2}}_{L^2}(1+\tau)^{\frac{1}{2}}\|\partial^2_3\varrho\|^{\frac{1}{2}}_{L^2}(1+\tau)\|\nabla\mathop{\mathrm{div}}\nolimits\mathbf{u}\|_{L^2}\Big]\nonumber\\
&\displaystyle \qquad\cdot\int^t_0(1+t-\tau)^{-\frac{1}{2}}(1+\tau)^{-\frac{3}{2}}\mathrm{d}\tau\nonumber\\
&\displaystyle \quad+\sup_{0\leq\tau\leq t}\Big[(1+\tau)^{\frac{3}{8}}(\log(1+\tau))^{-\frac{1}{2}}\|\varrho\|^{\frac{1}{2}}_{L^2}(1+\tau)^{\frac{1}{2}}\|\partial_3\varrho\|^{\frac{1}{2}}_{L^2}(1+\tau)^{\frac{1}{2}}\|\nabla\mathop{\mathrm{div}}\nolimits\mathbf{u}\|^{\frac{1}{2}}_{L^2}\|\nabla\partial^2_3\mathop{\mathrm{div}}\nolimits\mathbf{u}\|^{\frac{1}{2}}_{L^2}\Big]\nonumber\\
&\displaystyle \quad \qquad\cdot\int^t_0(1+t-\tau)^{-\frac{1}{2}}(1+\tau)^{-\frac{11}{8}}(\log(1+\tau))^{\frac{1}{2}}\mathrm{d}\tau\nonumber\\
&\displaystyle \lesssim\mathcal{E}(t)(1+t)^{-\frac{1}{2}},
\end{align}
and
\begin{align}\label{H33_4}
\displaystyle H_{{33}_4}
&\lesssim \sup_{0\leq \tau\leq t}\Big[(1+\tau)^{\frac{1}{2}}\|\partial_3\varrho\|^{\frac{1}{2}}_{L^2}(1+\tau)^{\frac{1}{2}}\|\partial^2_3\varrho\|^{\frac{1}{2}}_{L^2}(1+\tau)\|\nabla\varrho\|_{L^2}\Big]\nonumber\\
&\displaystyle \quad \qquad\cdot\int^t_0(1+t-\tau)^{-\frac{1}{2}}(1+\tau)^{-2}\mathrm{d}\tau\nonumber\\
&\displaystyle \quad+\sup_{0\leq \tau\leq t}(1+\tau)^{\frac{3}{8}}(\log(1+\tau))^{-\frac{1}{2}}\|\varrho\|^{\frac{1}{2}}_{L^2}(1+\tau)^{\frac{1}{2}}\|\partial_3\varrho\|^{\frac{1}{2}}_{L^2}(1+\tau)\|\nabla\partial_3\varrho\|_{L^2}\nonumber\\
&\displaystyle \quad \qquad\cdot\int^t_0(1+t-\tau)^{-\frac{1}{2}}(1+\tau)^{-\frac{15}{8}}(\log(1+\tau))^{\frac{1}{2}}\mathrm{d}\tau\nonumber\\
&\displaystyle \lesssim\mathcal{E}(t)(1+t)^{-\frac{1}{2}}.
\end{align}
Then substituting \eqref{H33_1}--\eqref{H33_4} into \eqref{H330}, we obtain that
\begin{align*}
H_{33} \lesssim\mathcal{E}(t)(1+t)^{-\frac{1}{2}}.
\end{align*}
This combined with \eqref{ubiaoshi13} and \eqref{H-3-124} leads to
\begin{equation*}
\displaystyle(1+t)^{\frac{1}{2}}\|\partial_3\mathbf{u}\|_{L^2}\lesssim\mathcal{E}(t)+\|\varrho_0\|_{L^1}+\|\partial_3\mathbf{u}_0\|_{L^2_{x_3}L^1_{x_1x_2}}+\|(\partial_3\varrho_0,\partial_3\mathbf{u}_0)\|_{L^2}.
\end{equation*}
The proof of Lemma \eqref{lemma5.6} is complete.
\end{proof}

\vspace{4mm}
{\bf Step Three:} The estimate of $ \mathcal{E}_{23}(t) $. Precisely, we shall establish some desired estimates of the second-order derivatives of $(\varrho, \mathbf{u})$.
\begin{lemma}\label{lemma4.8}
Assume that $(\varrho,u)$ is a smooth solution to the problem \eqref{new}. Then it holds that
\begin{align*}
\mathcal{E}_{23}(t) \lesssim \mathcal{E}^2(t)+\|(\varrho_0,\mathbf{u}_0)\|^2_{L^1}
+\|(\mathbf{u}_0,\partial_3\mathbf{u}_0,\partial^2_3\mathbf{u}_0)\|^2_{L^2_{x_3}L^1_{x_1x_2}}+\|(\nabla^2 \mathbf{u}_0,\nabla^2\varrho_0)\|^2_{L^2}.
\end{align*}
\end{lemma}
\begin{proof}
The proof is similar to that of Lemma \ref{lemma5.6}. First of all, we have, for $i,j=1,2,3$,
\begin{align}\label{sec-deri-rho-u-reprens}
\displaystyle
    \left(\begin{matrix}
\widehat{\partial _{i}\partial _{j}\varrho}(t,\xi)\\[2mm]
\widehat{\partial _{i}\partial _{j}\mathbf{u}}(t,\xi)
\end{matrix}\right)
=\hat{G}(t,\xi)\left(\begin{matrix}
\widehat{\partial _{i}\partial _{j}\varrho} _{0}(\xi)\\[2mm]
\widehat{\partial _{i}\partial _{j}\mathbf{u}}_{0}(\xi)
\end{matrix}
\right)+ \int _{0}^{t}\hat{G}(t- \tau,\xi)\left(\begin{matrix}
\widehat{\partial _{i}\partial _{j}S} _{1}(\tau,\xi)\\[2mm]
\widehat{\partial _{i}\partial _{j}S}_{2}(\tau,\xi)
\end{matrix}
\right)\mathrm{d}\tau.
\end{align}
As in the proof of Lemma \ref{lemma5.6}, we shall divide the subsequent analysis into three case: $i=1,2,j=1,2 $; $i=1,2,j=3$; and $i=j=3$. Let us begin with the first case for which we have four subcases: $ (i,j)=(1,1) $, $ (i,j)=(1,2) $, $ (i,j)=(2,1) $ and $ (i,j)=(2,2) $. Now let us focus on the case $ (i,j)=(1,2) $. By \eqref{vrho-con-line-lem}, \eqref{sec-deri-rho-u-reprens} and the Duhamel principle, we get
\begin{align}\label{varrhobiaoshi212}
\displaystyle\|\partial_1\partial_2\varrho\|_{L^{2}}&\lesssim (1+t)^{-\frac{7}{4}}(\|(\varrho_0,\mathbf{u}_0)\|_{L^{1}}+\|(\partial_1\partial_2\varrho_0,\partial_1\partial_2\mathbf{u}_0)\|_{L^2})
 \nonumber \\[2mm]
&\displaystyle \quad +  \int _{0}^{t}(1+t-\tau)^{-\frac{7}{4}}
\| S_1\|_{L^1}\mathrm{d}\tau+\int _{0}^{t} (1+t-\tau)^{-m}\|\partial_1\partial_2 S_1\|_{L^2}\mathrm{d}\tau\nonumber\\[2mm]
&\displaystyle \quad+\int _{0}^{t} (1+t-\tau)^{-\frac{7}{4}}\| S_2\|_{L^1}\mathrm{d}\tau+\int _{0}^{t} (1+t-\tau)^{-m}\|\partial_1\partial_2 S_2\|_{L^2}\mathrm{d}\tau\nonumber\\[2mm]
&\displaystyle \lesssim  (1+t)^{-\frac{7}{4}}(\|(\varrho_0,\mathbf{u}_0)\|_{L^{1}}+\|(\partial_1\partial_2\varrho_0,\partial_1\partial_2\mathbf{u}_0)\|_{L^2})+ \sum _{i=1}^{4}M _{12i}.
\end{align}
Next we shall estimate $ M _{12i}~(1 \leq i \leq 4) $ on the right hand side of \eqref{varrhobiaoshi212}. Thanks to the H\"older inequality, we get
\begin{align}\label{M121}
\displaystyle M_{121}
&\lesssim \sup_{0\leq \tau\leq t}\Big[(1+\tau)^{\frac{3}{4}}(\log(1+\tau))^{-1}\|\varrho\|_{L^2}(1+\tau)\|\mathop{\mathrm{div}}\nolimits\mathbf{u}\|_{L^2}\Big]\nonumber\\[2mm]
&\displaystyle \quad\cdot\int^t_0(1+t-\tau)^{-\frac{7}{4}}(1+\tau)^{-\frac{7}{4}}\log(1+\tau)\mathrm{d}\tau\nonumber\\[2mm]
&\displaystyle \quad+\sup_{0\leq \tau\leq t}\Big[(1+\tau)\|\nabla\varrho\|_{L^2}(1+\tau)^{\frac{1}{2}}\| \mathbf{u}\|_{L^2}\Big]\int^t_0(1+t-\tau)^{-\frac{7}{4}}(1+\tau)^{-\frac{3}{2}}\mathrm{d}\tau\nonumber\\[2mm]
&\displaystyle \lesssim \mathcal{E}^{\frac{1}{2}}_2(t)\mathcal{E}^{\frac{1}{2}}_1(t)(1+t)^{-\frac{7}{4}}\log(1+t)+\mathcal{E}_2(t)(1+t)^{-\frac{3}{2}}\nonumber\\[2mm]
&\displaystyle\lesssim\mathcal{E}(t)(1+t)^{-\frac{3}{2}}.
\end{align}
For the term $M_{122}$, we divide it into two terms.
\begin{align}\label{M1220}
M_{122}&\leq\int^t_0(1+t-\tau)^{-m}\|\partial_1\partial_2(\varrho \mathop{\mathrm{div}}\nolimits\mathbf{u})+\partial_1\partial_2(\mathbf{u}\cdot\nabla\varrho)\|_{L^2}\mathrm{d}\tau\nonumber\\
&:=M_{1221}+M_{1222}.
\end{align}
By \eqref{soboleve-ineq} and the H\"older inequality, we get
\begin{align*}
&\displaystyle \|\partial_1\partial_2(\varrho \mathop{\mathrm{div}}\nolimits\mathbf{u})\|_{L^2}
 \nonumber \\[2mm]
 &~\displaystyle \lesssim \|\partial_1\partial_2\varrho \mathop{\mathrm{div}}\nolimits\mathbf{u}\|_{L^2}+\|\partial_1\varrho\partial_2\mathop{\mathrm{div}}\nolimits\mathbf{u}\|_{L^2}+\|\partial_2\varrho\partial_1\mathop{\mathrm{div}}\nolimits\mathbf{u}\|_{L^2}+\|\varrho\partial_1\partial_2\mathop{\mathrm{div}}\nolimits\mathbf{u}\|_{L^2}\nonumber\\[2mm]
&~\displaystyle \lesssim\|\partial_1\partial_2\varrho\|_{L^{\infty}}\|\mathop{\mathrm{div}}\nolimits\mathbf{u}\|_{L^2}+\|\partial_2\mathop{\mathrm{div}}\nolimits\mathbf{u}\|_{L^{\infty}}\|\partial_1\varrho\|_{L^2}+\|\partial_1\mathop{\mathrm{div}}\nolimits\mathbf{u}\|_{L^{\infty}}\|\partial_2\varrho\|_{L^2}
 \nonumber \\[2mm]
 &~\displaystyle \quad+\|\varrho\|_{L^{\infty}}\|\partial_1\partial_2\mathop{\mathrm{div}}\nolimits\mathbf{u}\|_{L^2}\nonumber\\[2mm]
&~\displaystyle \lesssim\|\partial_1\partial_2\varrho\|^{\frac{1}{4}}_{L^2}\|\nabla^2\partial_1\partial_2\varrho\|^{\frac{3}{4}}_{L^2}\|\mathop{\mathrm{div}}\nolimits\mathbf{u}\|_{L^2}+\|\partial_2\mathop{\mathrm{div}}\nolimits\mathbf{u}\|^{\frac{1}{4}}_{L^2}\|\nabla^2\partial_2\mathop{\mathrm{div}}\nolimits\mathbf{u}\|^{\frac{3}{4}}_{L^2}\|\partial_1\varrho\|_{L^2}\nonumber\\[2mm]
&~\displaystyle \quad+\|\partial_1\mathop{\mathrm{div}}\nolimits\mathbf{u}\|^{\frac{1}{4}}_{L^2}\|\nabla^2\partial_1\mathop{\mathrm{div}}\nolimits\mathbf{u}\|^{\frac{3}{4}}_{L^2}\|\partial_2\varrho\|_{L^2}
+\|\nabla\varrho\|^{\frac{1}{2}}_{L^2}\|\nabla^2\varrho\|^{\frac{1}{2}}_{L^2}\|\partial_1\partial_2\mathop{\mathrm{div}}\nolimits\mathbf{u}\|_{L^2}.
\end{align*}
Thus, we can derive that
\begin{align}\label{M1221}
\displaystyle \makebox[-8pt]{~} M_{1221}&=\int^t_0(1+t-\tau)^{-m}\|\partial_1\partial_2(\varrho \mathop{\mathrm{div}}\nolimits\mathbf{u})\|_{L^2}\mathrm{d}\tau\nonumber\\[2mm]
&\displaystyle \lesssim\sup_{0\leq\tau\leq t}\Big[(1+\tau)^{\frac{1}{4}}\|\partial_1\partial_2\varrho\|^{\frac{1}{4}}_{L^2}\|\nabla^2\partial_1\partial_2\varrho\|^{\frac{3}{4}}_{L^2}(1+\tau)\|\mathop{\mathrm{div}}\nolimits\mathbf{u}\|_{L^2}\Big]\nonumber\\
&~\displaystyle \qquad\cdot\int^t_0(1+t-\tau)^{-m}(1+\tau)^{-\frac{5}{4}}\mathrm{d}\tau\nonumber\\[2mm]
&\displaystyle \quad+\sup_{0\leq\tau\leq t}\Big[(1+\tau)^{\frac{1}{4}}\|\partial_2\mathop{\mathrm{div}}\nolimits\mathbf{u}\|^{\frac{1}{4}}_{L^2}\|\nabla^2\partial_2\mathop{\mathrm{div}}\nolimits\mathbf{u}\|^{\frac{3}{4}}_{L^2}(1+\tau)\|\partial_1\varrho\|_{L^2}\Big]\nonumber\\[2mm]
&~\displaystyle \qquad\cdot\int^t_0(1+t-\tau)^{-m}(1+\tau)^{-\frac{5}{4}}\mathrm{d}\tau\nonumber\\[2mm]
&\displaystyle \quad+\sup_{0\leq\tau\leq t}\Big[(1+\tau)^{\frac{1}{4}}\|\partial_1\mathop{\mathrm{div}}\nolimits\mathbf{u}\|^{\frac{1}{4}}_{L^2}\|\nabla^2\partial_1\mathop{\mathrm{div}}\nolimits\mathbf{u}\|^{\frac{3}{4}}_{L^2}(1+\tau)\|\partial_2\varrho\|_{L^2}\Big]\nonumber\\[2mm]
&~\displaystyle \qquad\cdot\int^t_0(1+t-\tau)^{-m}(1+\tau)^{-\frac{5}{4}}\mathrm{d}\tau\nonumber\\[2mm]
&\displaystyle \quad+\sup_{0\leq \tau\leq t}\Big[(1+\tau)^{\frac{1}{2}}\|\nabla\varrho\|^{\frac{1}{2}}_{L^2}(1+\tau)^{\frac{1}{2}}\|\nabla^2\varrho\|^{\frac{1}{2}}_{L^2}(1+\tau)^{\frac{1}{2}}\|\partial_1\mathop{\mathrm{div}}\nolimits\mathbf{u}\|^{\frac{1}{2}}_{L^2}\|\partial_1\partial^2_2\mathop{\mathrm{div}}\nolimits\mathbf{u}\|^{\frac{1}{2}}_{L^2}\Big]\nonumber\\
&~\displaystyle \qquad\cdot\int^t_0(1+t-\tau)^{-m}(1+\tau)^{-\frac{3}{2}}\mathrm{d}\tau\nonumber\\[2mm]
&\displaystyle \lesssim \left(\mathcal{E}_{2}^{\frac{1}{8}}(t)\mathcal{E}_{0}^{\frac{3}{8}}(t)\mathcal{E}_{1}^{\frac{1}{2}}(t)+\mathcal{E}^{\frac{1}{8}}_1(t)\mathcal{E}^{\frac{1}{2}}_2(t)\mathcal{E}^{\frac{3}{8}}_0(t)\right)(1+t)^{-\frac{5}{4}}  + \mathcal{E}_{2}^{\frac{1}{2}}(t)\mathcal{E}_{1}^{\frac{1}{4}}(t)\mathcal{E}_{0}^{\frac{1}{4}}(t)(1+t)^{-\frac{3}{2}} \notag \\[2mm]
&\displaystyle \lesssim\mathcal{E}(t)(1+t)^{-\frac{5}{4}},
\end{align}
where we have used Lemma \ref{decaylemma}. Similarly, we have
\begin{align*}
&\makebox[-6pt]{~}\|\partial_1\partial_2(\mathbf{u}\cdot\nabla\varrho)\|_{L^2}\\[2mm]
&\makebox[-6pt]{~}~\displaystyle \lesssim \|\partial_1\partial_2 \mathbf{u}\cdot \nabla\varrho\|_{L^2}+\|\partial_1\mathbf{u}\cdot\nabla\partial_2\varrho\|_{L^2}+\|\partial_2\mathbf{u}\cdot\nabla\partial_1\varrho\|_{L^2}+\|\mathbf{u}\cdot\nabla\partial_1\partial_2\varrho\|_{L^2}\nonumber\\[2mm]
&\makebox[-6pt]{~}~\displaystyle \lesssim\|\nabla\varrho\|_{L^{\infty}}\|\partial_1\partial_2 \mathbf{u}\|_{L^2}+\|\partial_1\mathbf{u}\|_{L^{\infty}}\|\nabla\partial_2\varrho\|_{L^2}+\|\partial_2\mathbf{u}\|_{L^{\infty}}\|\nabla\partial_1\varrho\|_{L^2}+\|\mathbf{u}\cdot\nabla\partial_1\partial_2\varrho\|_{L^2}\nonumber\\[2mm]
&\makebox[-6pt]{~}~\displaystyle \lesssim\|\nabla^2\varrho\|^{\frac{1}{2}}_{L^2}\|\nabla^3\varrho\|^{\frac{1}{2}}_{L^2}\|\partial_1\partial_2\mathbf{u}\|_{L^2}+\|\nabla\partial_1\mathbf{u}\|^{\frac{1}{2}}_{L^2}\|\nabla^2\partial_1\mathbf{u}\|^{\frac{1}{2}}_{L^2}\|\nabla\partial_2\varrho\|_{L^2}\nonumber\\[2mm]
&\makebox[-6pt]{~}~\displaystyle \quad+\|\nabla\partial_2\mathbf{u}\|^{\frac{1}{2}}_{L^2}\|\nabla^2\partial_2\mathbf{u}\|^{\frac{1}{2}}_{L^2}\|\nabla\partial_1\varrho\|_{L^2}\nonumber\\[2mm]
&\makebox[-6pt]{~}~\displaystyle \quad+\|\mathbf{u}\|^{\frac{1}{4}}_{L^{2}}\|\partial_1\mathbf{u}\|^{\frac{1}{4}}_{L^2}\|\partial_2\mathbf{u}\|^{\frac{1}{4}}_{L^2}\|\partial_1\partial_2\mathbf{u}\|^{\frac{1}{4}}_{L^2}\|\nabla\partial_1\partial_2\varrho\|^{\frac{1}{2}}_{L^2}\|\nabla\partial_1\partial_2\partial_3\varrho\|^{\frac{1}{2}}_{L^2},
\end{align*}
which along with Lemma \ref{decaylemma} further implies that
\begin{align}\label{M1222}
\displaystyle M_{1222}&=\int^t_0(1+t-\tau)^{-m}\|\partial_1\partial_2(\mathbf{u}\cdot\nabla\varrho)\|_{L^2}\mathrm{d}\tau\nonumber\\[2mm]
&\displaystyle \lesssim\sup_{0\leq \tau\leq t}\Big[(1+\tau)^{\frac{3}{4}}\|\nabla^2\varrho\|^{\frac{3}{4}}_{L^2}\|\nabla^4\varrho\|^{\frac{1}{4}}_{L^2}(1+\tau)^{\frac{16}{15}}\|\partial_1\partial_2\mathbf{u}\|_{L^2}\Big]
 \nonumber \\
 &~\displaystyle \qquad \cdot\int^t_0(1+t-\tau)^{-m}(1+\tau)^{-\frac{109}{60}}\mathrm{d}\tau\nonumber\\[2mm]
&~+\sup_{0\leq \tau\leq t}\Big[(1+\tau)^{\frac{3}{4}}\|\nabla\partial_1\mathbf{u}\|^{\frac{3}{4}}_{L^2}\|\nabla^3\partial_1\mathbf{u}\|^{\frac{1}{4}}_{L^2}(1+\tau)\|\nabla\partial_2\varrho\|_{L^2}\Big]
 \nonumber \\
 &~\displaystyle \qquad \cdot\int^t_0(1+t-\tau)^{-m}(1+\tau)^{-\frac{7}{4}}\mathrm{d}\tau\nonumber\\[2mm]
&~+\sup_{0\leq \tau\leq t}\Big[(1+\tau)^{\frac{3}{4}}\|\nabla\partial_2\mathbf{u}\|^{\frac{3}{4}}_{L^2}\|\nabla^3\partial_2\mathbf{u}\|^{\frac{1}{4}}_{L^2}(1+\tau)\|\nabla\partial_1\varrho\|_{L^2}\Big]
 \nonumber \\
 & \displaystyle~ \qquad \cdot\int^t_0(1+t-\tau)^{-m}(1+\tau)^{-\frac{7}{4}}\mathrm{d}\tau\nonumber\\[2mm]
&~+\sup_{0\leq \tau\leq t} \Big[(1+\tau)^{\frac{1}{8}}\|\mathbf{u}\|^{\frac{1}{4}}_{L^{2}}(1+\tau)^{\frac{1}{4}}\|\partial_1\mathbf{u}\|^{\frac{1}{4}}_{L^2}(1+\tau)^{\frac{1}{4}}\|\partial_2\mathbf{u}\|^{\frac{1}{4}}_{L^2}(1+\tau)^{\frac{4}{15}}\|\partial_1\partial_2\mathbf{u}\|^{\frac{1}{4}}_{L^2} \Big.\nonumber\\[2mm]
&~\displaystyle \qquad \qquad\Big.\cdot(1+\tau)^{\frac{1}{4}}\|\nabla\partial_2\varrho\|^{\frac{1}{4}}_{L^2}\|\nabla\partial^2_1\partial_2\varrho\|^{\frac{1}{4}}_{L^2}\|\nabla\partial_1\partial_2\partial_3\varrho\|^{\frac{1}{2}}_{L^2}\Big]\int^t_0(1+t-\tau)^{-m}(1+\tau)^{-\frac{137}{120}}\mathrm{d}\tau\nonumber\\[2mm]
&\displaystyle \lesssim  \mathcal{E}_{2}^{\frac{7}{8}}(t)\mathcal{E}_{0}^{\frac{1}{8}}(t)((1+t)^{-\frac{109}{60}}+ (1+t)^{-\frac{7}{ 4}})
+ \mathcal{E}_{2}^{\frac{5}{8}}(t)\mathcal{E}_{0}^{\frac{3}{8}}(t)(1+t)^{-\frac{137}{120}}   \notag \\[2mm]
&\displaystyle \lesssim \mathcal{E}(t)(1+t)^{-\frac{137}{120}}.
\end{align}
Thus, substituting \eqref{M1221}--\eqref{M1222} into \eqref{M1220}, we find that
\begin{align}\label{M122}
\displaystyle
M_{122} \lesssim\mathcal{E}(t)(1+t)^{-\frac{137}{120}}.
\end{align}
For the term $M_{123}$, in the same method as in the derivation of \eqref{N3}, we have
\begin{align}\label{M123} M_{123}\lesssim\mathcal{E}(t)(1+t)^{-1}.
\end{align}
For the term $M_{124}$, we split it into four terms.
\begin{align}\label{M1240}
&\displaystyle\makebox[-6pt]{~} M_{124}
 \nonumber \\
 &\makebox[-6pt]{~}~\displaystyle\leq \int^t_0(1+t-\tau)^{-m}\|\partial_1\partial_2(\mathbf{u}\cdot\nabla \mathbf{u})\|_{L^2}\mathrm{d}\tau+ \int _{0}^{t}(1+t-\tau)^{-m}\|\partial_1\partial_2(g(\varrho)\Delta_h \mathbf{u})\|_{L^2}\mathrm{d}\tau\Big.\nonumber\\[2mm]
&\makebox[-6pt]{~}\displaystyle \quad + \int^t_0(1+t-\tau)^{-m}\|\partial_1\partial_2(g(\varrho)\nabla \mathop{\mathrm{div}}\nolimits\mathbf{u})\|_{L^2}\mathrm{d}\tau+ \int^t_0(1+t-\tau)^{-m}\|\partial_1\partial_2(f(\varrho)\nabla\varrho)\|_{L^2}\mathrm{d}\tau\nonumber\\[2mm]
\displaystyle&\makebox[-6pt]{~}~:=M_{1241}+M_{1242}+M_{1243}+M_{1244}.
\end{align}
Thanks to \eqref{two-two-deri-ani}, we get
\begin{align*}
\|\partial_1\partial_2(\mathbf{u}\cdot\nabla \mathbf{u})\|_{L^2}& \lesssim \|\partial_1\partial_2\mathbf{u}\cdot\nabla \mathbf{u}\|_{L^2}+\|\partial_1\mathbf{u}\cdot\nabla\partial_2\mathbf{u}\|_{L^2}+\|\partial_2\mathbf{u}\cdot\nabla\partial_1\mathbf{u}\|_{L^2}+\|\mathbf{u}\cdot\nabla\partial_1\partial_2\mathbf{u}\|_{L^2}
 \nonumber \\[2mm]
 & \lesssim \|\nabla \mathbf{u}\|_{L ^{2}}^{\frac{1}{4}}\|\nabla \partial _{1}\mathbf{u}\|_{L ^{2}}^{\frac{1}{4}} \|\nabla \partial _{2}\mathbf{u}\|_{L ^{2}}^{\frac{1}{4}} \|\nabla \partial _{1}\partial _{2}\mathbf{u}\|_{L ^{2}}^{\frac{1}{4}} \|\partial _{1}\partial _{2}\mathbf{u}\|_{L ^{2}}^{\frac{1}{2}} \|\partial _{1}\partial _{2}\partial _{3}\mathbf{u}\|_{L ^{2}}^{\frac{1}{2}}
  \nonumber \\[2mm]
  & \displaystyle \quad+ \|\partial _{1} \mathbf{u}\|_{L ^{2}}^{\frac{1}{4}}\| \partial _{1}^{2}\mathbf{u}\|_{L ^{2}}^{\frac{1}{4}} \| \partial _{1} \partial _{2}\mathbf{u}\|_{L ^{2}}^{\frac{1}{4}} \| \partial _{1}^{2}\partial _{2}\mathbf{u}\|_{L ^{2}}^{\frac{1}{4}} \|\nabla\partial _{2}\mathbf{u}\|_{L ^{2}}^{\frac{1}{2}} \|\nabla\partial _{2}\partial _{3}\mathbf{u}\|_{L ^{2}}^{\frac{1}{2}}
   \nonumber \\[2mm]
   & \displaystyle \quad+ \|\partial _{2} \mathbf{u}\|_{L ^{2}}^{\frac{1}{4}}\| \partial _{1}\partial _{2}\mathbf{u}\|_{L ^{2}}^{\frac{1}{4}} \| \partial _{2}^{2}\mathbf{u}\|_{L ^{2}}^{\frac{1}{4}} \| \partial _{1}\partial _{2}^{2}\mathbf{u}\|_{L ^{2}}^{\frac{1}{4}} \|\nabla\partial _{1}\mathbf{u}\|_{L ^{2}}^{\frac{1}{2}} \|\nabla\partial _{1}\partial _{3}\mathbf{u}\|_{L ^{2}}^{\frac{1}{2}}
    \nonumber \\[2mm]
    & \displaystyle \quad+ \| \mathbf{u}\|_{L ^{2}}^{\frac{1}{4}}\| \partial _{1}\mathbf{u}\|_{L ^{2}}^{\frac{1}{4}} \| \partial _{2}\mathbf{u}\|_{L ^{2}}^{\frac{1}{4}} \| \partial _{1}\partial _{2}\mathbf{u}\|_{L ^{2}}^{\frac{1}{4}} \|\nabla\partial _{1}\partial _{2}\mathbf{u}\|_{L ^{2}}^{\frac{1}{2}} \|\nabla\partial _{1}\partial _{2}\partial _{3}\mathbf{u}\|_{L ^{2}}^{\frac{1}{2}} .
\end{align*}
 Therefore we can bound $ M _{1241} $ as
\begin{align}\label{M1241}
M_{1241}& \lesssim\sup_{0\leq\tau\leq t}\Big[(1+\tau)^{\frac{1}{8}}\|\nabla \mathbf{u}\|^{\frac{1}{4}}_{L^2}(1+\tau)^{\frac{1}{4}}\|\nabla\partial_1 \mathbf{u}\|^{\frac{1}{4}}_{L^2}(1+\tau)^{\frac{3}{8}}\|\nabla\partial_2 \mathbf{u}\|^{\frac{3}{8}}_{L^2}\|\nabla \partial^2_1\partial_2\mathbf{u}\|^{\frac{1}{8}}_{L^2}\Big.\nonumber\\[2mm]
&\displaystyle ~\qquad \qquad\Big.\cdot(1+\tau)^{\frac{4}{5}}\|\partial_1\partial_2\mathbf{u}\|^{\frac{3}{4}}_{L^2}\|\partial_1\partial_2\partial^2_3\mathbf{u}\|^{\frac{1}{4}}_{L^2}\Big]
\int^t_0(1+t-\tau)^{-m}(1+\tau)^{-\frac{31}{20}}\mathrm{d}\tau\nonumber\\[2mm]
&\displaystyle \quad+\sup_{0\leq\tau\leq t}\Big[(1+\tau)^{\frac{1}{4}}\|\partial_1\mathbf{u}\|^{\frac{1}{4}}_{L^2}(1+\tau)^{\frac{4}{15}}\|\partial^2_1\mathbf{u}\|^{\frac{1}{4}}_{L^2}(1+\tau)^{\frac{2}{5}}\|\partial_2\partial_1\mathbf{u}\|^{\frac{3}{8}}_{L^2}\|\partial^3_1\partial_2\mathbf{u}\|^{\frac{1}{8}}_{L^2}\Big.\nonumber\\[2mm]
&\displaystyle ~\qquad \qquad \quad\Big.\cdot(1+\tau)^{\frac{3}{4}}\|\nabla\partial_2\mathbf{ u}\|^{\frac{3}{4}}_{L^2}\|\nabla \partial_2\partial^2_3\mathbf{u}\|^{\frac{1}{4}}_{L^2}\Big]\int^t_0(1+t-\tau)^{-m}(1+\tau)^{-\frac{5}{3}}\mathrm{d}\tau\nonumber\\[2mm]
&\displaystyle \quad+\sup_{0\leq\tau\leq t}\Big[(1+\tau)^{\frac{1}{4}}\|\partial_2\mathbf{u}\|^{\frac{1}{4}}_{L^2}(1+\tau)^{\frac{4}{15}}\|\partial_1\partial_2\mathbf{u}\|^{\frac{1}{4}}_{L^2}(1+\tau)^{\frac{2}{5}}\|\partial^2_2\mathbf{u}\|^{\frac{3}{8}}_{L^2}\|\partial^2_1\partial^2_2\mathbf{u}\|^{\frac{1}{8}}_{L^2}\Big.\nonumber\\[2mm]
&\displaystyle ~ \qquad \qquad \quad \Big.\cdot(1+\tau)^{\frac{3}{4}}\|\nabla \partial_1\mathbf{u}\|^{\frac{3}{4}}_{L^2}\|\nabla \partial_1\partial^2_3\mathbf{u}\|^{\frac{1}{4}}_{L^2}\Big]\int^t_0(1+t-\tau)^{-m}(1+\tau)^{-\frac{5}{3}}\mathrm{d}\tau\nonumber\\
&\displaystyle \quad+\sup_{0\leq\tau\leq t}\Big[(1+\tau)^{\frac{1}{8}}\|\mathbf{u}\|^{\frac{1}{4}}_{L^2}(1+\tau)^{\frac{1}{4}}\|\partial_1\mathbf{u}\|^{\frac{1}{4}}_{L^2}(1+\tau)^{\frac{1}{4}}\|\partial_2\mathbf{u}\|^{\frac{1}{4}}_{L^2}(1+\tau)^{\frac{8}{15}}\|\partial_1\partial_2\mathbf{u}\|^{\frac{1}{2}}_{L^2}\Big.\nonumber\\[2mm]
&\displaystyle ~\qquad \qquad \quad\cdot\|\nabla^2\partial_1\partial_2\mathbf{u}\|^{\frac{1}{4}}_{L^2}\|\nabla\partial_1\partial_2\partial_3\mathbf{u}\|^{\frac{1}{2}}_{L^2}\Big]\int^t_0(1+t-\tau)^{-m}(1+\tau)^{-\frac{139}{120}}\mathrm{d}\tau\nonumber\\[2mm]
&\displaystyle \lesssim\mathcal{E}_{2}^{\frac{13}{16}}(t)\mathcal{E}_{0}^{\frac{3}{16}}(t)\big((1+t)^{-\frac{31}{20}}  + (1+t)^{-\frac{5}{3} } \big)
+ \mathcal{E}_{2}^{\frac{5}{8}}(t)\mathcal{E}_{0}^{\frac{3}{8}}(t)(1+t)^{-\frac{139}{120}  }
 \nonumber \\[2mm]
&\displaystyle \lesssim\mathcal{E}(t)(1+t)^{-\frac{139}{120}}.
\end{align}
Similarly, by \eqref{two-two-deri-ani} and Lemma \ref{decaylemma} again, By the same method as in the proof of estimates for $T_{42}$ and $T_{43}$, we have
\begin{align}
\displaystyle M _{1242} \lesssim  \mathcal{E}(t)(1+t)^{-\frac{16}{15}},\ \ \ M _{1243} \lesssim  \mathcal{E}(t)(1+t)^{-\frac{16}{15}}.
\end{align}
In what follows, we shall estimate $ M _{1244} $. By the Sobolev inequality \eqref{soboleve-ineq}, we get
\begin{align*}
\displaystyle\|\partial_1\partial_2(f(\varrho)\nabla\varrho)\|_{L^2}&\lesssim \|\partial_1\partial_2\varrho\nabla\varrho\|_{L^2}+\|\partial_1\varrho\nabla\partial_2\varrho\|_{L^2}+\|\partial_2\varrho\nabla\partial_1\varrho\|_{L^2}+\|\varrho\nabla\partial_1\partial_2\varrho\|_{L^2}\nonumber\\[2mm]
&\displaystyle \lesssim \|\nabla^2\varrho\|^{\frac{1}{2}}_{L^2}\|\nabla^3\varrho\|^{\frac{1}{2}}_{L^2}\|\partial_1\partial_2\varrho\|_{L^2}+\|\nabla\partial_1\varrho\|^{\frac{1}{2}}_{L^2}\|\nabla^2\partial_1\varrho\|^{\frac{1}{2}}_{L^2}\|\nabla\partial_2\varrho\|_{L^2}\nonumber\\[2mm]
&\displaystyle \quad+\|\nabla\partial_2\varrho\|^{\frac{1}{2}}_{L^2}\|\nabla^2\partial_1\varrho\|^{\frac{1}{2}}_{L^2}\|\nabla\partial_1\varrho\|_{L^2}+\|\nabla\varrho\|^{\frac{1}{2}}_{L^2}\|\nabla^2\varrho\|^{\frac{1}{2}}_{L^2}\|\nabla\partial_1\partial_2\varrho\|_{L^2}.
\end{align*}
This along with Lemma \ref{decaylemma} further yields that
\begin{align}\label{M1244}
\displaystyle M_{1244}&\lesssim\sup_{0\leq\tau\leq t}\Big[(1+\tau)^{\frac{3}{4}}\|\nabla^2\varrho\|^{\frac{3}{4}}_{L^2}\|\nabla^4\varrho\|^{\frac{1}{4}}_{L^2}(1+\tau)\|\partial_1\partial_2\varrho\|_{L^2}\Big]
 \nonumber \\[2mm]
 & \displaystyle \qquad \cdot\int^t_0(1+t-\tau)^{-m}(1+\tau)^{-\frac{7}{4}}\mathrm{d}\tau\nonumber\\[2mm]
&\displaystyle \quad+\sup_{0\leq\tau\leq t}\Big[(1+\tau)^{\frac{3}{4}}\|\nabla\partial_1\varrho\|^{\frac{3}{4}}_{L^2}\|\nabla^3\partial_1\varrho\|^{\frac{1}{4}}_{L^2}(1+\tau)\|\nabla\partial_2\varrho\|_{L^2}\Big]
 \nonumber \\[2mm]
 & \displaystyle \qquad \cdot\int^t_0(1+t-\tau)^{-m}(1+\tau)^{-\frac{7}{4}}\mathrm{d}\tau\nonumber\\[2mm]
&\displaystyle \quad+\sup_{0\leq\tau\leq t}\Big[(1+\tau)^{\frac{3}{4}}\|\nabla\partial_2\varrho\|^{\frac{3}{4}}_{L^2}\|\nabla^3\partial_2\varrho\|^{\frac{1}{4}}_{L^2}(1+\tau)\|\nabla\partial_1\varrho\|_{L^2}\Big]
 \nonumber \\[2mm]
 &\displaystyle \qquad \cdot\int^t_0(1+t-\tau)^{-m}(1+\tau)^{-\frac{7}{4}}\mathrm{d}\tau\nonumber\\[2mm]
&\displaystyle \quad+\sup_{0\leq\tau\leq t}\Big[(1+\tau)^{\frac{1}{2}}\|\nabla\varrho\|^{\frac{1}{2}}_{L^2}(1+\tau)^{\frac{1}{2}}\|\nabla^2\varrho\|^{\frac{1}{2}}_{L^2}(1+\tau)^{\frac{1}{2}}\|\nabla\partial_1\varrho\|^{\frac{1}{2}}_{L^2}\|\nabla\partial_1\partial^2_2\varrho\|^{\frac{1}{2}}_{L^2}\Big]\nonumber\\[2mm]
&\displaystyle \qquad\cdot\int^t_0(1+t-\tau)^{-m}(1+\tau)^{-\frac{3}{2}}\mathrm{d}\tau\nonumber\\[2mm]
&\displaystyle \lesssim
  \mathcal{E}_{2}^{\frac{7}{8}}(t)\mathcal{E}_{0}^{\frac{1}{8}}(t)  (1+t)^{- \frac{7}{4} }
+ \mathcal{E}_{2}^{\frac{3}{4}}(t)\mathcal{E}_{0}^{\frac{1}{4}}(t) (1+t)^{-\frac{3}{2}}
 \nonumber \\[2mm]
 &\displaystyle \lesssim\mathcal{E}(t)(1+t)^{-\frac{3}{2}}.
\end{align}
Thus, substituting \eqref{M1241}--\eqref{M1244} into \eqref{M1240}, we have
\begin{equation}\label{M124}
M_{124}=\int^t_0(1+t-\tau)^{-m}\|\partial_1\partial_2S_2\|_{L^2}\mathrm{d}\tau \lesssim\mathcal{E}(t)(1+t)^{-\frac{16}{15}}.
\end{equation}
Collecting the estimates \eqref{varrhobiaoshi212}, \eqref{M121}, \eqref{M122}, \eqref{M123} and \eqref{M124}, we conclude that
\begin{align}\label{eq_rho_12}
(1+t)\|\partial_1\partial_2\varrho\| _{L ^{2}} \lesssim\mathcal{E}(t)+\|(\varrho_0,\mathbf{u}_0)\|_{L^1}+\|(\partial_1\partial_2\varrho_0,\partial_1\partial_2\mathbf{u}_0)\|_{L^2}.
\end{align}
Now let us establish the estimate of $\partial_1\partial_2\mathbf{u}$. Using the Duhamel principle and \eqref{hori-u-lem-con}, we get
\begin{align}\label{ubiaoshi212}
\displaystyle\makebox[-8pt]{~}\|\partial_1\partial_2\mathbf{u}(t)\|_{L^2}& \lesssim (1+t)^{-\frac{7}{4}}(\|\varrho_0\|_{L^1}+\|\partial_1\partial_2\varrho_0\|_{L^2})+(1+t)^{-\frac{3}{2}}(\|\mathbf{u}_0\|_{L^2_{x_3}L^1_{x_1x_2}}+\|\partial_1\partial_2\mathbf{u}_0\|_{L^2})\nonumber\\[2mm]
&\displaystyle \quad+\int^t_0(1+t-\tau)^{-\frac{7}{4}}\|S_1\|_{L^1}\mathrm{d}\tau+\int^t_0(1+t-\tau)^{-m}\|\partial_1\partial_2S_1\|_{L^2}\mathrm{d}\tau\nonumber\\[2mm]
&\displaystyle \quad+\int^t_0(1+t-\tau)^{-\frac{3}{2}}\|S_2\|_{L^2_{x_3}L^1_{x_1x_2}}\mathrm{d}\tau+\int^t_0(1+t-\tau)^{-m}\|\partial_1\partial_2S_2\|_{L^2}\mathrm{d}\tau\nonumber\\[2mm]
&\displaystyle \lesssim(1+t)^{-\frac{7}{4}}(\|\varrho_0\|_{L^1}+\|\partial_1\partial_2\varrho_0\|_{L^2})+(1+t)^{-\frac{3}{2}}(\|\mathbf{u}_0\|_{L^2_{x_3}L^1_{x_1x_2}}+\|\partial_1\partial_2\mathbf{u}_0\|_{L^2})\nonumber\\[2mm]
&\displaystyle \quad+H_{121}+H_{122}+H_{123}+H_{124}.
\end{align}
By similar arguments as in \eqref{M121}, \eqref{M122} and \eqref{M124}, we can bound $H_{121}$, $H_{122}$ and $H_{124}$ as follows:
\begin{align}\label{H12-1-4}
\displaystyle \begin{cases}
    \displaystyle H_{121}\lesssim\mathcal{E}(t)(1+t)^{-\frac{3}{2}}, \ \ \ H_{122} \lesssim\mathcal{E}(t)(1+t)^{-\frac{137}{120}},\\[2mm]
    \displaystyle H_{124} \lesssim\mathcal{E}(t)(1+t)^{-\frac{16}{15}}.
\end{cases}
\end{align}
Next, we establish the estimate of $H_{123}$. Like the proof of \eqref{H03} and the process of \eqref{H031}--\eqref{H034}, we can estimate $H_{123}$ as follows:
\begin{align*}
\displaystyle H_{123}&\lesssim \int^t_0(1+t-\tau)^{-\frac{3}{2}}\Big(\|\mathbf{u}\cdot\nabla \mathbf{u}\|_{L^2_{x_3}L^1_{x_1x_2}}+\|g(\varrho)\Delta_h\mathbf{u}\|_{L^2_{x_3}L^1_{x_1x_2}}\Big.\nonumber\\[2mm]
&\displaystyle \qquad \quad \qquad \qquad \qquad\Big.+\|g(\varrho)\nabla \mathop{\mathrm{div}}\nolimits\mathbf{u}\|_{L^2_{x_3}L^1_{x_1x_2}}+\|f(\varrho)\nabla\varrho\|_{L^2_{x_3}L^1_{x_1x_2}}\Big)\mathrm{d}\tau\nonumber\\[2mm]
&\displaystyle \lesssim(\mathcal{E}_2(t)+\mathcal{E}^{\frac{3}{4}}_2(t)\mathcal{E}^{\frac{1}{4}}_1(t))\int^t_0(1+t-\tau)^{- \frac{3}{2}}\Big((1+\tau)^{-\frac{3}{2}}+(1+\tau)^{-\frac{5}{4}}\Big)\mathrm{d}\tau\nonumber\\[2mm]
&\displaystyle \quad+ \mathcal{E}_2(t)\int^t_0(1+t-\tau)^{-\frac{3}{2}}(1+\tau)^{-\frac{233}{120}}\log^{\frac{1}{2}}(1+\tau)\mathrm{d}\tau\nonumber\\[2mm]
&\displaystyle \quad+\mathcal{E}^{\frac{1}{2}}_2(t)\mathcal{E}^{\frac{1}{2}}_1(t)\int^t_0(1+t-\tau)^{-\frac{3}{2}}(1+\tau)^{-\frac{15}{8}}\log^{\frac{1}{2}}(1+\tau)\mathrm{d}\tau\nonumber\\
&\displaystyle \quad+\mathcal{E}_2(t)\int^t_0(1+t-\tau)^{-\frac{3}{2}}(1+\tau)^{-\frac{15}{8}}\log^{\frac{1}{2}}(1+\tau)\mathrm{d}\tau\nonumber\\[2mm]
&\displaystyle \lesssim\mathcal{E}(t)(1+t)^{-\frac{5}{4}}.
\end{align*}
This combined with \eqref{ubiaoshi212} and \eqref{H12-1-4} implies that
\begin{equation*}
(1+t)^{\frac{16}{15}}\|\partial_1\partial_2\mathbf{u}\|_{L^2} \lesssim\mathcal{E}(t)+\|\varrho_0\|_{L^1}+\|\mathbf{u}_0\|_{L^2_{x_3}L^1_{x_1x_2}}+\|(\partial_1\partial_2\varrho_0,\partial_1\partial_2\mathbf{u}_0)\|_{L^2}.
\end{equation*}
Therefore we conclude for the case $ (i,j)=(1,2) $ that
\begin{align}
&\displaystyle  (1+t)\|\partial_1\partial_2\varrho\| _{L ^{2}} +(1+t)^{\frac{16}{15}}\|\partial_1\partial_2\mathbf{u}\|_{L^2}
 \nonumber \\[2mm]
 &~\displaystyle\lesssim\mathcal{E}(t)+\|(\varrho_0,\mathbf{u}_0)\|_{L^1}+\|\mathbf{u}_0\|_{L^2_{x_3}L^1_{x_1x_2}}+\|(\partial_1\partial_2\varrho_0,\partial_1\partial_2\mathbf{u}_0)\|_{L^2}.
\end{align}
Similarly, we have for $ (i,j)=(1,1) $ and $ i=2, j=1,2 $ that
\begin{align*}
&\displaystyle  (1+t)\|\partial_i\partial_j\varrho\| _{L ^{2}} +(1+t)^{\frac{16}{15}}\|\partial_i\partial_j\mathbf{u}\|_{L^2}
 \nonumber \\[2mm]
 &~\displaystyle\lesssim\mathcal{E}(t)+\|(\varrho_0,\mathbf{u}_0)\|_{L^1}+\|\mathbf{u}_0\|_{L^2_{x_3}L^1_{x_1x_2}}+\|(\partial_1\partial_2\varrho_0,\partial_1\partial_2\mathbf{u}_0)\|_{L^2}.
\end{align*}

Now let us turn to {\textbf{the case $ i=1,2 $, $ j=3 $}}.
It suffices to investigate the case $i=2,j=3$ since the case $i=1,j=3$ can be dealt with similarly. By \eqref{vrho-con-line-lem} and the Duhamel principle, we get, for any $ m>2 $
\begin{align}\label{varrhobiaoshi223}
\|\partial_2\partial_3\varrho\|_{L^2}
&\displaystyle \lesssim (1+t)^{-\frac{7}{4}}(\|\varrho_0\|_{L^1}+\|\mathbf{u}_0\|_{L^1}+\|\partial_2\partial_3\varrho_0\|_{L^2}+\|\partial_2\partial_3\mathbf{u}_0\|_{L^2})\nonumber\\[2mm]
&\displaystyle \quad+\int^t_0(1+t-\tau)^{-\frac{7}{4}}\|S_1\|_{L^1}\mathrm{d}\tau+\int^t_0(1+t-\tau)^{-m}\|\partial_2\partial_3S_1\|_{L^2}\mathrm{d}\tau\nonumber\\[2mm]
&\displaystyle \quad+\int^t_0(1+t-\tau)^{-\frac{7}{4}}\|S_2\|_{L^1}\mathrm{d}\tau+\int^t_0(1+t-\tau)^{-m}\|\partial_2\partial_3S_2\|_{L^2}\mathrm{d}\tau\nonumber\\[2mm]
&\displaystyle \lesssim(1+t)^{-\frac{7}{4}}(\|\varrho_0\|_{L^1}+\|\mathbf{u}_0\|_{L^1}+\|\partial_2\partial_3\varrho_0\|_{L^2}+\|\partial_2\partial_3\mathbf{u}_0\|_{L^2})\nonumber\\[2mm]
&\displaystyle \quad+M_{231}+M_{232}+M_{233}+M_{234}.
\end{align}
By similar arguments as in $\eqref{M121}$ and \eqref{M123}, we readily derive that
\begin{align}\label{M231}
M_{231} \lesssim\mathcal{E}(t)(1+t)^{-\frac{3}{2}},\ \ \
M_{233} \lesssim \mathcal{E}(t)(1+t)^{-1}.
\end{align}
For the term $M_{232}$, we have
\begin{align}\label{M230}
\displaystyle
M_{232}&\lesssim\int^t_0(1+t-\tau)^{-m}\|\partial_2\partial_3(\varrho \mathop{\mathrm{div}}\nolimits\mathbf{u})\|_{L^2}\mathrm{d}\tau+ \int^t_0(1+t-\tau)^{-m}\|\partial_2\partial_3(\mathbf{u}\cdot\nabla\varrho)\|_{L^2}\mathrm{d}\tau\nonumber\\[2mm]
&:=M_{2321}+M_{2322}.
\end{align}
 Notice that
\begin{align*}
&\displaystyle\|\partial_2\partial_3(\varrho\mathop{\mathrm{div}}\nolimits\mathbf{u})\|_{L^2}
 \nonumber \\[2mm]
 &~ \lesssim \|\partial_2\partial_3\varrho \mathop{\mathrm{div}}\nolimits\mathbf{u}\|_{L^2}+\|\partial_2\varrho\partial_3\mathop{\mathrm{div}}\nolimits\mathbf{u}\|_{L^2}+\|\partial_3\varrho\partial_2\mathop{\mathrm{div}}\nolimits\mathbf{u}\|_{L^2}+\|\varrho\partial_2\partial_3\mathop{\mathrm{div}}\nolimits\mathbf{u}\|_{L^2}\nonumber\\[2mm]
&~\displaystyle \lesssim \|\partial_2\partial_3\varrho\|^{\frac{1}{4}}_{L^2}\|\nabla^2\partial_2\partial_3\varrho\|^{\frac{3}{4}}_{L^2}\|\mathop{\mathrm{div}}\nolimits\mathbf{u}\|_{L^2}+\|\partial_3\mathop{\mathrm{div}}\nolimits\mathbf{u}\|^{\frac{1}{4}}_{L^2}\|\nabla^2\partial_3\mathop{\mathrm{div}}\nolimits\mathbf{u}\|^{\frac{3}{4}}_{L^2}\|\partial_2\varrho\|_{L^2}\nonumber\\[2mm]
&~\displaystyle \quad+\|\partial_2\mathop{\mathrm{div}}\nolimits\mathbf{u}\|^{\frac{1}{4}}_{L^2}\|\nabla^2\partial_2\mathop{\mathrm{div}}\nolimits\mathbf{u}\|^{\frac{3}{4}}_{L^2}\|\partial_3\varrho\|_{L^2}+\|\nabla\varrho\|^{\frac{1}{2}}_{L^2}\|\nabla^2\varrho\|^{\frac{1}{2}}_{L^2}\|\partial_2\partial_3\mathop{\mathrm{div}}\nolimits\mathbf{u}\|_{L^2},
\end{align*}
where we have used \eqref{soboleve-ineq}. Then, by similar arguments as in $M_{1221}$, we get, for $m>2$,
\begin{align}\label{M2321}
M_{2321}& \lesssim\sup_{0\leq\tau\leq t}\Big[(1+\tau)^{\frac{1}{4}}\|\partial_2\partial_3\varrho\|^{\frac{1}{4}}_{L^2}\|\nabla^2\partial_2\partial_3\varrho\|^{\frac{3}{4}}_{L^2}(1+\tau)\|\mathop{\mathrm{div}}\nolimits\mathbf{u}\|_{L^2}\Big]\nonumber\\[2mm]
&\displaystyle \qquad\cdot\int^t_0(1+t-\tau)^{-m}(1+\tau)^{-\frac{5}{4}}\mathrm{d}\tau\nonumber\\[2mm]
&\displaystyle \quad+\sup_{0\leq\tau\leq t}\Big[(1+\tau)^{\frac{1}{4}}\|\partial_3\mathop{\mathrm{div}}\nolimits\mathbf{u}\|^{\frac{1}{4}}_{L^2}\|\nabla^2\partial_3\mathop{\mathrm{div}}\nolimits\mathbf{u}\|^{\frac{3}{4}}_{L^2}(1+\tau)\|\partial_2\varrho\|_{L^2}\Big]\nonumber\\[2mm]
&\displaystyle \qquad\cdot\int^t_0(1+t-\tau)^{-m}(1+\tau)^{-\frac{5}{4}}\mathrm{d}\tau\nonumber\\[2mm]
&\displaystyle \quad+\sup_{0\leq\tau\leq t}\Big[(1+\tau)^{\frac{1}{4}}\|\partial_2\mathop{\mathrm{div}}\nolimits\mathbf{u}\|^{\frac{1}{4}}_{L^2}\|\nabla^2\partial_2\mathop{\mathrm{div}}\nolimits\mathbf{u}\|^{\frac{3}{4}}_{L^2}(1+\tau)\|\partial_3\varrho\|_{L^2}\Big]\nonumber\\[2mm]
&\displaystyle \qquad\cdot\int^t_0(1+t-\tau)^{-m}(1+\tau)^{-\frac{5}{4}}\mathrm{d}\tau\nonumber\\[2mm]
&\displaystyle \quad+\sup_{0\leq\tau\leq t}\Big[(1+\tau)^{\frac{1}{2}}\|\nabla\varrho\|^{\frac{1}{2}}_{L^2}(1+\tau)^{\frac{1}{2}}\ \|\nabla^2\varrho\|^{\frac{1}{2}}_{L^2}(1+\tau)^{\frac{1}{2}}\|\partial_2\mathop{\mathrm{div}}\nolimits\mathbf{u}\|^{\frac{1}{2}}_{L^2}\|\partial_2\partial^2_3\mathop{\mathrm{div}}\nolimits\mathbf{u}\|^{\frac{1}{2}}_{L^2}\Big]\nonumber\\[2mm]
&\displaystyle \qquad\cdot\int^t_0(1+t-\tau)^{-m}(1+\tau)^{-\frac{3}{2}}\mathrm{d}\tau\nonumber\\[2mm]
&\displaystyle \lesssim \Big(\mathcal{E}^{\frac{1}{8}}_2(t)\mathcal{E}^{\frac{3}{8}}_0(t)\mathcal{E}^{\frac{1}{2}}_1(t)+\mathcal{E}^{\frac{1}{8}}_1(t)\mathcal{E}^{\frac{3}{8}}_0(t)\mathcal{E}^{\frac{1}{2}}_2(t)\Big)(1+t)^{-\frac{5}{4}} + + \mathcal{E}_{2}^{\frac{1}{2}}(t)\mathcal{E}^{\frac{1}{4}}_1(t)\mathcal{E}_{0}^{\frac{1}{4}}(t) (1+t)^{-\frac{3}{2}}\nonumber\\[2mm]
&\displaystyle \lesssim\mathcal{E}(t)(1+t)^{-\frac{5}{4}},`
\end{align}
where we have used Lemma \ref{decaylemma}. Thanks to \eqref{two-two-deri-ani} and \eqref{soboleve-ineq}, we have
\begin{align*}
\displaystyle\makebox[-6pt]{~}\|\partial_2\partial_3(\mathbf{u}\cdot\nabla\varrho)\|_{L^2}& \lesssim\|\partial_2\partial_3\mathbf{u}\cdot\nabla\varrho\|_{L^2}+\|\partial_2\mathbf{u}\cdot\nabla\partial_3\varrho\|_{L^2}+\|\partial_3\mathbf{u}\cdot\nabla\partial_2\varrho\|_{L^2}+\|\mathbf{u}\cdot\nabla\partial_2\partial_3\varrho\|_{L^2}\nonumber\\[2mm]
&\lesssim\|\partial_2\partial_3\mathbf{u}\|^{\frac{1}{4}}_{L^2}\|\nabla^2\partial_2\partial_3\mathbf{u}\|^{\frac{3}{4}}_{L^2}\|\nabla\varrho\|_{L^2}+\|\nabla\partial_3\varrho\|^{\frac{1}{4}}_{L^2}\|\nabla^2\partial_3\varrho\|^{\frac{3}{4}}_{L^2}\|\partial_2\mathbf{u}\|_{L^2}\nonumber\\[2mm]
&\displaystyle \quad+\|\mathbf{u}\|^{\frac{1}{4}}_{L^2}\|\partial_1\mathbf{u}\|^{\frac{1}{4}}_{L^2}\|\partial_2\mathbf{u}\|^{\frac{1}{4}}_{L^2}\|\partial_1\partial_2\mathbf{u}\|^{\frac{1}{2}}_{L^2}\|\nabla\partial_2\partial_3\varrho\|^{\frac{1}{2}}_{L^2}\|\nabla\partial_2\partial^2_3\varrho\|^{\frac{1}{2}}_{L^2}
 \nonumber \\[2mm]
 & \displaystyle \quad+\|\partial_3\mathbf{u}\|^{\frac{1}{4}}_{L^2}\|\nabla^2\partial_3\mathbf{u}\|^{\frac{3}{4}}_{L^2}\|\nabla\partial_2\varrho\|_{L^2}.
\end{align*}
This alongside Lemma \ref{decaylemma} leads to
\begin{align}\label{M2322}
\displaystyle
M_{2322}& \lesssim\sup_{0\leq\tau\leq t}(1+\tau)^{\frac{1}{4}}\|\partial_2\partial_3\mathbf{u}\|^{\frac{1}{4}}_{L^2}\|\nabla^2\partial_2\partial_3\mathbf{u}\|^{\frac{3}{4}}_{L^2}(1+\tau)\|\nabla\varrho\|_{L^2}\nonumber\\[2mm]
&\displaystyle \quad \quad\cdot\int^t_0(1+t-\tau)^{-m}(1+\tau)^{-\frac{5}{4}}\mathrm{d}\tau\nonumber\\[2mm]
&\displaystyle \quad+\sup_{0\leq\tau\leq t}\Big[(1+\tau)^{\frac{5}{8}}\|\nabla\partial_3\varrho\|^{\frac{5}{8}}_{L^2}\|\nabla^3\partial_3\varrho\|^{\frac{3}{8}}_{L^2}(1+\tau)\|\partial_2\mathbf{u}\|_{L^2}\Big]
 \nonumber \\[2mm]
 & \displaystyle \qquad \quad \cdot\int^t_0(1+t-\tau)^{-m}(1+\tau)^{-\frac{13}{8}}\mathrm{d}\tau\nonumber\\[2mm]
&\displaystyle \quad+\sup_{0\leq\tau\leq t}\Big[ (1+\tau)^{\frac{1}{8}}\|\mathbf{u}\|^{\frac{1}{4}}_{L^2}(1+\tau)^{\frac{1}{4}}\|\partial_1\mathbf{u}\|^{\frac{1}{4}}_{L^2}(1+\tau)^{\frac{1}{4}}\|\partial_2\mathbf{u}\|^{\frac{1}{4}}_{L^2}(1+\tau)^{\frac{4}{15}}\|\partial_1\partial_2\mathbf{u}\|^{\frac{1}{4}}_{L^2}\Big.\nonumber\\[2mm]
&\displaystyle ~\qquad \qquad \Big.\cdot(1+\tau)^{\frac{1}{4}}
\|\nabla\partial_2\varrho\|^{\frac{1}{4}}_{L^2}\|\nabla\partial_2\partial^2_3\varrho\|^{\frac{3}{4}}_{L^2}\Big]\int^t_0(1+t-\tau)^{-m}(1+\tau)^{-\frac{137}{120}}\mathrm{d}\tau\nonumber\\[2mm]
&\displaystyle \quad+\sup_{0\leq\tau\leq t}\Big[(1+\tau)^{\frac{1}{8}}\|\partial_3\mathbf{u}\|^{\frac{1}{4}}_{L^2}(1+\tau)^{\frac{3}{16}}(\log(1+\tau))^{-\frac{3}{8}}\|\nabla\partial_3\mathbf{u}\|^{\frac{3}{8}}_{L^2}\|\nabla^3\partial_3\mathbf{u}\|^{\frac{3}{8}}_{L^2}(1+\tau)\|\nabla\partial_2\varrho\|_{L^2}\Big]\nonumber\\[2mm]
&\displaystyle \qquad \quad\cdot\int^t_0(1+t-\tau)^{-m}(1+\tau)^{-\frac{21}{16}}(\log(1+\tau))^{\frac{3}{8}}\mathrm{d}\tau\nonumber\\[2mm]
&\displaystyle \lesssim\mathcal{E}^{\frac{5}{8}}_2(t)\mathcal{E}^{\frac{3}{8}}_0(t)\big((1+t)^{-\frac{5}{4}}+(1+t)^{-\frac{137}{120}}\big)+\mathcal{E}^{\frac{13}{16}}_2(t)\mathcal{E}^{\frac{3}{16}}_0(t)((1+t)^{-\frac{13}{8}}+(1+t)^{-\frac{21}{16}}(\log(1+t))^{\frac{3}{8}})\nonumber\\[2mm]
& \lesssim\mathcal{E}(t)(1+t)^{-\frac{137}{120}}.
\end{align}
Combining \eqref{M2321} and \eqref{M2322}, we then get from \eqref{M230} that
\begin{align}\label{M232}
M_{232} \lesssim\mathcal{E}(t)(1+t)^{-\frac{137}{120}}.
\end{align}
Next, we derive the estimate for $M_{234}$.
Notice that
\begin{align}\label{23s2}
\displaystyle\|\partial_2\partial_3S_2\|_{L^2}& \lesssim\|\partial_2\partial_3(\mathbf{u}\cdot\nabla \mathbf{u})\|_{L^2}+\|\partial_2\partial_3(g(\varrho)\Delta_h\mathbf{u})\|_{L^2}+\|\partial_2\partial_3(g(\varrho)\nabla \mathop{\mathrm{div}}\nolimits\mathbf{u})\|_{L^2}\nonumber\\[2mm]
&\displaystyle \quad+\|\partial_2\partial_3(f(\varrho)\nabla\varrho)\|_{L^2}.
\end{align}
Then, we can split $M_{234}$ into four terms as follows
\begin{align}\label{M2340}
  \displaystyle M_{234}& \lesssim \int^t_0(1+t-\tau)^{-m}\|\partial_2\partial_3(\mathbf{u}\cdot\nabla \mathbf{u}) \|_{L^2} \mathrm{d}\tau
  + \int^t_0(1+t-\tau)^{-m}\|\partial_2\partial_3(\varrho\Delta_h\mathbf{u})\|_{L^2}\mathrm{d}\tau \nonumber\\[2mm]
  &\displaystyle \quad+\int^t_0(1+t-\tau)^{-m}\|\partial_2\partial_3(\varrho\nabla \mathop{\mathrm{div}}\nolimits\mathbf{u})\|_{L^2}\mathrm{d}\tau
  +\int^t_0(1+t-\tau)^{-m}\|\partial_2\partial_3(\varrho\nabla\varrho)\|_{L^2}\mathrm{d}\tau\nonumber\\[2mm]
  &\displaystyle:=M_{2341}+M_{2342}+M_{2343}+M_{2344}.
\end{align}
Using \eqref{soboleve-ineq} and the H\"older inequality, we have
\begin{align*}
&\displaystyle\|\partial_2\partial_3(\mathbf{u}\cdot\nabla \mathbf{u})\|_{L^2}
 \nonumber \\[2mm]
 &\displaystyle ~ \lesssim \|\partial_2\partial_3\mathbf{u}\cdot\nabla \mathbf{u}\|_{L^2}+\|\partial_2\mathbf{u}\cdot\nabla\partial_3\mathbf{u}\|_{L^2}+\|\partial_3\mathbf{u}\cdot\nabla\partial_2\mathbf{u}\|_{L^2}+\|\mathbf{u}\cdot\nabla\partial_2\partial_3\mathbf{u}\|_{L^2}\\[2mm]
&~\displaystyle \lesssim \|\nabla \mathbf{u}\|_{L^{\infty}}\|\nabla\partial_2\mathbf{u}\|_{L^2}+\|\partial_2\mathbf{u}\|_{L^4}\|\nabla\partial_3\mathbf{u}\|_{L^4}+\|\mathbf{u}\cdot\nabla\partial_2\partial_3\mathbf{u}\|_{L^2}\\[2mm]
& ~\displaystyle\lesssim \|\nabla^2\mathbf{u}\|^{\frac{1}{2}}_{L^2}\|\nabla^3\mathbf{u}\|^{\frac{1}{2}}_{L^2}\|\nabla\partial_2\mathbf{u}\|_{L^2}+\|\partial_2\mathbf{u}\|^{\frac{1}{4}}_{L^2}\|\nabla\partial_2 \mathbf{u}\|^{\frac{3}{4}}_{L^2}\|\nabla\partial_3\mathbf{u}\|^{\frac{1}{4}}_{L^2}\|\nabla^2\partial_3\mathbf{u}\|^{\frac{3}{4}}_{L^2}\\[2mm]
&~\displaystyle \quad+\|\mathbf{u}\|^{\frac{1}{4}}_{L^2}\|\partial_1\mathbf{u}\|^{\frac{1}{4}}_{L^2}\|\partial_2\mathbf{u}\|^{\frac{1}{4}}_{L^2}\|\partial_1\partial_2\mathbf{u}\|^{\frac{1}{4}}_{L^2}\|\nabla\partial_2\partial_3\mathbf{u}\|^{\frac{1}{2}}_{L^2}\|\nabla\partial_2\partial^2_3\mathbf{u}\|^{\frac{1}{2}}_{L^2},
\end{align*}
where we have used the anisotropic inequality \eqref{two-two-deri-ani} for $\|\mathbf{u}\cdot\nabla\partial_2\partial_3\mathbf{u}\|_{L^2}$. Thus, it holds that
\begin{align}\label{M23410}
&\displaystyle M_{2341}
 \nonumber \\
 &~\displaystyle\lesssim \int^t_0(1+t-\tau)^{-m}\Big(\|\nabla^2\mathbf{u}\|^{\frac{1}{2}}_{L^2}\|\nabla^3\mathbf{u}\|^{\frac{1}{2}}_{L^2}\|\nabla\partial_2\mathbf{u}\|_{L^2} \Big.
  \nonumber\\[2mm]
  &~~ \displaystyle \quad\qquad \qquad \qquad\qquad\Big.+\|\partial_2\mathbf{u}\|^{\frac{1}{4}}_{L^2}\|\nabla\partial_2 \mathbf{u}\|^{\frac{3}{4}}_{L^2}\|\nabla \partial_3\mathbf{u}\|^{\frac{1}{4}}_{L^2}\|\nabla^2\partial_3\mathbf{u}\|^{\frac{3}{4}}_{L^2}\Big)\mathrm{d}\tau\nonumber\\[2mm]
&~\displaystyle \quad+\int^t_0(1+t-\tau)^{-m}\|\mathbf{u}\|^{\frac{1}{4}}_{L^2}\|\partial_1\mathbf{u}\|^{\frac{1}{4}}_{L^2}\|\partial_2\mathbf{u}\|^{\frac{1}{4}}_{L^2}\|\partial_1\partial_2\mathbf{u}\|^{\frac{1}{4}}_{L^2}\|\nabla\partial_2\partial_3\mathbf{u}\|^{\frac{1}{2}}_{L^2}\|\nabla\partial_2\partial_3 ^{2}\mathbf{u}\|^{\frac{1}{2}}_{L^2}\mathrm{d}\tau\nonumber\\[2mm]
&~\displaystyle:=M_{{2341}_1}+M_{{2341}_2},
\end{align}
with
\begin{align}\label{M23411}
&\displaystyle M_{{2341}_1}
 \nonumber \\
 &~\displaystyle\lesssim\sup_{0\leq \tau\leq t}\Big[(1+\tau)\|\nabla\partial_2\mathbf{u}\|_{L^2}(1+\tau)^{\frac{3}{8}}(\log(1+\tau))^{\frac{3}{4}}\|\nabla^2\mathbf{u}\|^{\frac{3}{4}}_{L^2}\|\nabla^4\mathbf{u}\|^{\frac{1}{4}}_{L^2}\Big]\nonumber\\[2mm]
&~\displaystyle \qquad\cdot\int^t_0(1+t-\tau)^{-m}(1+\tau)^{-\frac{11}{8}}(\log(1+\tau))^{\frac{3}{4}}\mathrm{d}\tau\nonumber\\[2mm]
&~\displaystyle \quad+\sup_{0\leq \tau\leq t}\Big[(1+\tau)^{\frac{1}{4}}\|\partial_2\mathbf{u}\|^{\frac{1}{4}}_{L^2}(1+\tau)^{\frac{3}{4}}\|\nabla\partial_2 \mathbf{u}\|^{\frac{3}{4}}_{L^2}(1+\tau)^{\frac{5}{16}}(\log(1+\tau))^{-\frac{5}{8}}\|\nabla\partial_3\mathbf{u}\|^{\frac{5}{8}}_{L^2}\Big.\nonumber\\[2mm]
&~\displaystyle \quad\qquad\qquad \Big. \cdot\|\nabla^3\partial_3\mathbf{u}\|^{\frac{3}{8}}_{L^2}\Big]\int^t_0(1+t-\tau)^{-m}(1+\tau)^{-\frac{21}{16}}(\log(1+\tau))^{\frac{5}{8}}\mathrm{d}\tau\nonumber\\[2mm]
&~\displaystyle \lesssim \mathcal{E}^{\frac{7}{8}}_{2}(t)\mathcal{E}^{\frac{1}{8}}_0(t)(1+t)^{-\frac{11}{8}}(\log(1+t))^{\frac{3}{4}}+\mathcal{E}^{\frac{13}{16}}_{2}(t)\mathcal{E}^{\frac{3}{16}}_0(t)(1+t)^{-\frac{21}{16}}(\log(1+t))^{\frac{5}{8}}\nonumber\\[2mm]
&~\displaystyle \lesssim\mathcal{E}(t)(1+t)^{-\frac{21}{16}}(\log(1+t))^{\frac{5}{8}}
\end{align}
provided $ m>\frac{11}{8}$, and
\begin{align}\label{M23412}
 \displaystyle M_{{2341}_2}& \lesssim\sup_{0\leq \tau\leq t}\Big[(1+\tau)^{\frac{1}{8}}\|\mathbf{u}\|^{\frac{1}{4}}_{L^2}(1+\tau)^{\frac{1}{4}}\|\partial_1\mathbf{u}\|^{\frac{1}{4}}_{L^2}(1+\tau)^{\frac{1}{4}}\|\partial_2\mathbf{u}\|^{\frac{1}{4}}_{L^2}(1+\tau)^{\frac{4}{15}}\|\partial_1\partial_2\mathbf{u}\|^{\frac{1}{4}}_{L^2}\Big.\nonumber\\[2mm]
&\displaystyle \qquad \qquad \Big.\cdot (1+\tau)^{\frac{1}{4}} \|\partial_2\partial_3\mathbf{u}\|^{\frac{1}{4}}_{L^2}\|\nabla\partial_2\partial^2_3\mathbf{u}\|^{\frac{3}{4}}_{L^2}\Big]\int^t_0(1+t-\tau)^{-m}(1+\tau)^{-\frac{137}{120}}\mathrm{d}\tau\nonumber\\[2mm]
&\lesssim  \mathcal{E}^{\frac{5}{8}}_2(t)\mathcal{E}^{\frac{3}{8}}_0(t)(1+t)^{-\frac{137}{120}}
  \lesssim\mathcal{E}(t)(1+t)^{-\frac{137}{120}}
\end{align}
provided $m>\frac{137}{120}$. Here we have used Lemma \ref{decaylemma} and \eqref{dissip-1}. Therefore, substituting \eqref{M23411} and \eqref{M23412} into \eqref{M23410}, we have
\begin{equation}\label{M2341}
M_{2341} \lesssim \int^t_0(1+t-\tau)^{-m}\|\partial_2\partial_3(\mathbf{u}\cdot\nabla \mathbf{u})\|_{L^2}\mathrm{d}\tau \lesssim\mathcal{E}(t)(1+t)^{-\frac{137}{120}}.
\end{equation}
For the term $M_{2342}$, we utilize \eqref{soboleve-ineq} and the H\"older inequality to get
\begin{align*}
&\displaystyle\|\partial_2\partial_3(\varrho\Delta_h\mathbf{u})\|_{L^2}
 \nonumber \\[2mm]
 &~\displaystyle \lesssim\|\partial_2\partial_3\varrho\Delta_h\mathbf{u}\|_{L^2}+\|\partial_2\varrho\partial_3\Delta_h\mathbf{u}\|_{L^2}
+\|\partial_3\varrho\partial_2\Delta_h\mathbf{u}\|_{L^2}+\|\varrho\partial_2\partial_3\Delta_h\mathbf{u}\|_{L^2}\nonumber\\[2mm]
&~\displaystyle \lesssim\|\Delta_h\mathbf{u}\|^{\frac{1}{4}}_{L^2}\|\nabla^2\Delta_h\mathbf{u}\|^{\frac{3}{4}}_{L^2}\|\partial_2\partial_3\varrho\|_{L^2}+\|\partial_2\varrho\|^{\frac{1}{4}}_{L^2}\|\nabla\partial_2\varrho\|^{\frac{3}{4}}_{L^2}\|\partial_3\Delta_h\mathbf{u}\|^{\frac{1}{4}}_{L^2}\|\nabla\partial_3\Delta_h\mathbf{u}\|^{\frac{3}{4}}_{L^2}\nonumber\\[2mm]
&~\displaystyle \quad+\|\partial_3\varrho\|^{\frac{1}{4}}_{L^2}\|\nabla\partial_3\varrho\|^{\frac{3}{4}}_{L^2}\|\partial_2\Delta_h\mathbf{u}\|^{\frac{1}{4}}_{L^2}\|\nabla\partial_2\Delta_h\mathbf{u}\|^{\frac{3}{4}}_{L^2}+\|\nabla\varrho\|^{\frac{1}{2}}_{L^2}\|\nabla^2\varrho\|^{\frac{1}{2}}_{L^2}\|\Delta_h\partial_2\partial_3\mathbf{u}\|_{L^2}.
\end{align*}
This along with Lemma \ref{decaylemma} further leads to
\begin{align}\label{M2342}
\makebox[-6pt]{~}M_{2342}&\leq\sup_{0\leq\tau\leq t}\Big[(1+\tau)^{\frac{4}{15}}\|\Delta_h\mathbf{u}\|^{\frac{1}{4}}_{L^2}\|\nabla^2\Delta_h\mathbf{u}\|^{\frac{3}{4}}_{L^2}(1+\tau)\|\partial_2\partial_3\varrho\|_{L^2}\Big]
 \nonumber \\[2mm]
 & \displaystyle \quad \qquad \cdot\int^t_0(1+t-\tau)^{-m}(1+\tau)^{-\frac{19}{15}}\mathrm{d}\tau\nonumber\\[2mm]
&\displaystyle \quad+\sup_{0\leq \tau\leq t}\Big[(1+\tau)^{\frac{1}{4}}\|\partial_2\varrho\|^{\frac{1}{4}}_{L^2}(1+\tau)^{\frac{3}{4}}\|\nabla\partial_2\varrho\|^{\frac{3}{4}}_{L^2}(1+\tau)^{\frac{2}{15}}\|\Delta_h\mathbf{u}\|^{\frac{1}{8}}_{L^2}\|\nabla\partial_3\Delta_h\mathbf{u}\|^{\frac{7}{8}}_{L^2}\Big]\nonumber\\[2mm]
&\displaystyle \quad \qquad\cdot\int^t_0(1+t-\tau)^{-m}(1+\tau)^{-\frac{17}{15}}\mathrm{d}\tau\nonumber\\[2mm]
&\displaystyle \quad+\sup_{0\leq \tau\leq t}\Big[(1+\tau)^{\frac{1}{4}}\|\partial_3\varrho\|^{\frac{1}{4}}_{L^2}(1+\tau)^{\frac{3}{4}}\|\nabla\partial_3\varrho\|^{\frac{3}{4}}_{L^2}(1+\tau)^{\frac{2}{15}}\|\Delta_h\mathbf{u}\|^{\frac{1}{8}}_{L^2}\|\nabla\partial_2\Delta_h\mathbf{u}\|^{\frac{7}{8}}_{L^2}\Big]\nonumber\\[2mm]
&\displaystyle \quad \qquad\cdot\int^t_0(1+t-\tau)^{-m}(1+\tau)^{-\frac{17}{15}}\mathrm{d}\tau\nonumber\\[2mm]
&\displaystyle \quad+\sup_{0\leq \tau\leq t}(1+\tau)^{\frac{1}{2}}\|\nabla\varrho\|^{\frac{1}{2}}_{L^2}(1+\tau)^{\frac{1}{2}}\|\nabla^2\varrho\|^{\frac{1}{2}}_{L^2}\nonumber\\[2mm]
&\displaystyle \quad \qquad\cdot\Big(\int^t_0(1+t-\tau)^{-2m}(1+\tau)^{-\frac{32}{15}}\mathrm{d}\tau\Big)^{\frac{1}{2}}\Big(\int^t_0(1+\tau)^{\frac{2}{15}}\|\Delta_h\partial_2\partial_3\mathbf{u}\|^2_{L^2}\mathrm{d}\tau\Big)^{\frac{1}{2}}\nonumber\\[2mm]
&\displaystyle \lesssim\mathcal{E}_{2}^{\frac{5}{8}}(t)\mathcal{E}_{0}^{\frac{3}{8}}(t)(1+t)^{-\frac{19}{15}}
+  \mathcal{E}_{2}^{\frac{9}{16}}(t)\mathcal{E}_{0}^{\frac{7}{16}}(t)(1+t)^{-\frac{17}{15}}
+  \mathcal{E}_{2}^{\frac{1}{2}}(t)\mathcal{E}_{1}^{\frac{1}{2}}(t)(1+t)^{-\frac{16}{15}}  \notag\\[2mm]
&\displaystyle \lesssim\mathcal{E}(t)(1+t)^{-\frac{16}{15}}
\end{align}
for $m>2$. Similarly, we have
\begin{align}\label{M2343}
M_{2343}=\int^t_0(1+t-\tau)^{-m}\|\partial_2\partial_3(g(\varrho)\nabla\mathop{\mathrm{div}}\nolimits\mathbf{u})\|_{L ^{2}}\mathrm{d}\tau \lesssim\mathcal{E}(t)(1+t)^{-\frac{16}{15}}.
\end{align}
For the last term $M_{2344}$, using \eqref{soboleve-ineq} again, we have
\begin{align*}
\|\partial_2\partial_3(\varrho\nabla\varrho)\|_{L^2}& \lesssim\|\partial_2\partial_3\varrho\nabla\varrho\|_{L^2}+\|\partial_2\varrho\nabla\partial_3\varrho\|_{L^2}+\|\partial_3\varrho\nabla\partial_2\varrho\|_{L^2}+\|\varrho\nabla\partial_2\partial_3\varrho\|_{L^2}\nonumber\\[2mm]
&\displaystyle \lesssim \|\nabla\varrho\|^{\frac{1}{4}}_{L^2}\|\nabla^3\varrho\|^{\frac{3}{4}}_{L^2}\|\partial_2\partial_3\varrho\|_{L^2}+\|\partial_2\varrho\|^{\frac{1}{4}}_{L^2}\|\nabla^2\partial_2\varrho\|^{\frac{3}{4}}_{L^2}\|\nabla\partial_3\varrho\|_{L^2}\nonumber\\[2mm]
&\displaystyle \quad+\|\partial_3\varrho\|^{\frac{1}{4}}_{L^2}\|\nabla^2\partial_3\varrho\|^{\frac{3}{4}}_{L^2}\|\nabla\partial_2\varrho\|_{L^2}+\|\nabla\varrho\|^{\frac{1}{2}}_{L^2}\|\nabla\varrho\|^{\frac{1}{2}}_{L^2}\|\nabla\partial_2\partial_3\varrho\|_{L^2}.
\end{align*}
Therefore we derive for $ M _{2344} $ that
\begin{align}\label{M2344}
\displaystyle M_{2344}& \lesssim\sup_{0\leq\tau\leq t}\Big[(1+\tau)^{\frac{1}{4}}\|\nabla\varrho\|^{\frac{1}{4}}_{L^2}(1+\tau)^{\frac{3}{8}}\|\nabla^2\varrho\|^{\frac{3}{8}}_{L^2}\|\nabla^4\varrho\|^{\frac{3}{8}}_{L^2}(1+\tau)\|\partial_2\partial_3\varrho\|_{L^2}\Big]\nonumber\\
&\displaystyle \quad \qquad\cdot\int^t_0(1+t-\tau)^{-m}(1+\tau)^{-\frac{13}{8}}\mathrm{d}\tau\nonumber\\
&\displaystyle \quad+\sup_{0\leq\tau\leq t}\Big[(1+\tau)^{\frac{1}{4}}\|\partial_2\varrho\|^{\frac{1}{4}}_{L^2}(1+\tau)^{\frac{3}{8}}\|\nabla\partial_2\varrho\|^{\frac{3}{8}}_{L^2}\|\nabla^3\partial_2\varrho\|^{\frac{3}{8}}_{L^2}(1+\tau)\|\nabla\partial_3\varrho\|_{L^2}\Big]
 \nonumber \\[2mm]
 & \displaystyle \quad\quad \qquad \cdot\int^t_0(1+t-\tau)^{-m}(1+\tau)^{-\frac{13}{8}}\mathrm{d}\tau\nonumber\\[2mm]
&\displaystyle \quad+\sup_{0\leq\tau\leq t}(1+\tau)^{\frac{1}{4}}\|\partial_3\varrho\|^{\frac{1}{4}}_{L^2}(1+\tau)^{\frac{3}{8}}\|\nabla\partial_3\varrho\|^{\frac{3}{8}}_{L^2}\|\nabla^3\partial_3\varrho\|^{\frac{3}{8}}_{L^2}(1+\tau)\|\nabla\partial_2\varrho\|_{L^2}\nonumber\\
&\displaystyle \quad \quad \qquad\cdot\int^t_0(1+t-\tau)^{-m}(1+\tau)^{-\frac{13}{8}}\mathrm{d}\tau\nonumber\\
&\displaystyle \quad+\sup_{0\leq\tau\leq t}(1+\tau)^{\frac{1}{2}}\|\nabla\varrho\|^{\frac{1}{2}}_{L^2}(1+\tau)^{\frac{1}{2}}\|\nabla^2\varrho\|^{\frac{1}{2}}_{L^2}(1+\tau)^{\frac{1}{2}}\|\partial_2\partial_3\varrho\|^{\frac{1}{2}}_{L^2}\|\nabla^2\partial_2\partial_3\varrho\|^{\frac{1}{2}}_{L^2}\nonumber\\
&\displaystyle \quad \quad \qquad\cdot\int^t_0(1+t-\tau)^{-m}(1+\tau)^{-\frac{3}{2}}\mathrm{d}\tau\nonumber\\[2mm]
& \displaystyle \lesssim\mathcal{E}_{2}^{\frac{13}{16}}(t)\mathcal{E}_{0}^{\frac{3}{16}}(t)(1+t)^{-\frac{13}{8}}
+ \mathcal{E}_{2}^{\frac{3}{4}}(t)\mathcal{E}_{0}^{\frac{1}{4}}(t)(1+t)^{-\frac{3}{2}}  \notag\\[2mm]
&\displaystyle \lesssim\mathcal{E}(t)(1+t)^{-\frac{3}{2}},
\end{align}
provided $ m>2 $, where we have used Lemma \ref{decaylemma}. Gathering the estimates \eqref{M2341}, \eqref{M2342}, \eqref{M2343} and \eqref{M2344}, we thus arrive at
\begin{align}\label{M234}
M_{234}=\int^t_0(1+t-\tau)^{-m}\|\partial_2\partial_3S_2\|_{L^2}\mathrm{d}\tau \lesssim\mathcal{E}(t)(1+t)^{-\frac{16}{15}}.
\end{align}
Inserting the estimates \eqref{M231}, \eqref{M232} and \eqref{M234} into \eqref{varrhobiaoshi223}, we then obtain that
\begin{align}\label{23rho}
(1+t)\|\partial_2\partial_3\varrho\|_{L^2} \lesssim\mathcal{E}(t)+\|(\rho_0,\mathbf{u}_0)\|_{L^1}+\|(\partial_{2}\partial_3\varrho_0,\partial_2\partial_3\mathbf{u}_0)\|_{L^2}.
\end{align}
Now let us turn to the estimate of $\|\partial_2\partial_3\mathbf{u}\|_{L^2}$. Using \eqref{anisou-lem-con}, we get
\begin{align}\label{ubiaoshi223}
&\displaystyle\|\partial_2\partial_3\mathbf{u}\|_{L^2}
 \nonumber \\[2mm]
 &~\displaystyle \lesssim(1+t)^{-\frac{7}{4}}(\|\varrho_0\|_{L^1}+\|\partial_2\partial_3\varrho_0\|_{L^2})+(1+t)^{-1}(\|\partial_3\mathbf{u}_0\|_{L^2_{x_3}L^1_{x_1x_2}}+\|\partial_2\partial_3\mathbf{u}_0\|_{L^2})\nonumber\\[2mm]
&~\displaystyle \quad+\int^t_0(1+t-\tau)^{-\frac{7}{4}}\|S_1\|_{L^1}\mathrm{d}\tau+\int^t_0(1+t-\tau)^{-m}\|\partial_2\partial_3S_1\|_{L^2}\mathrm{d}\tau\nonumber\\[2mm]
&~\displaystyle \quad+\int^t_0(1+t-\tau)^{-1}\|\partial_3S_2\|_{L^2_{x_3}L^1_{x_1x_2}}\mathrm{d}\tau+\int^t_0(1+t-\tau)^{-m}\|\partial_2\partial_3S_2\|_{L^2}\mathrm{d}\tau\nonumber\\[2mm]
&~\displaystyle \lesssim(1+t)^{-\frac{7}{4}}(\|\varrho_0\|_{L^1}+\|\partial_2\partial_3\varrho_0\|_{L^2})+(1+t)^{-1}(\|\partial_3\mathbf{u}_0\|_{L^2_{x_3}L^1_{x_1x_2}}+\|\partial_2\partial_3\mathbf{u}_0\|_{L^2})\nonumber\\[2mm]
&~\displaystyle \quad+H_{231}+H_{232}+H_{233}+H_{234}.
\end{align}
By similar arguments as in the derivation of \eqref{M231}, \eqref{M232} and \eqref{M234}, respectively, we can bound $H_{231}$, $H_{232}$ and $H_{234}$ as follows
\begin{gather*}
\displaystyle
    \displaystyle H_{231} \lesssim\mathcal{E}(t)(1+t)^{-\frac{3}{2}},\ \  H_{232} \lesssim\mathcal{E}(t)(1+t)^{-\frac{137}{120}}\ \ \mbox{and\ }\ H_{234}\lesssim\mathcal{E}(t)(1+t)^{-\frac{16}{15}}.
\end{gather*}
As for estimate $H_{233}$, by similar arguments as in \eqref{p3d}--\eqref{H33_4}, we easily get
\begin{equation*}
H_{233}\lesssim\mathcal{E}_2(t)(1+t)^{-1}.
\end{equation*}
Therefore we have from \eqref{ubiaoshi223} that
\begin{equation*}
(1+t)\|\partial_2\partial_3\mathbf{u}\|_{L^2} \lesssim\mathcal{E}(t)+\|\varrho_0\|_{L^1}+\|\partial_3\mathbf{u}_0\|_{L^2_{x_3}L^1_{x_1x_2}}+\|(\partial_2\partial_3\mathbf{u}_0,\partial_2\partial_3\varrho_0)\|_{L^2}.
\end{equation*}
Finally, let us consider {\bf the case $i=j=3.$} We shall establish estimates for $ \|\partial _{3}^{2}\varrho\|_{L ^{2}} $ and $ \|\partial _{3}^{2}\mathbf{u}\|_{L ^{2}} $.
As before, let us begin with the estimate in terms of $ \varrho $. This can be done by similar arguments as in the proof of \eqref{23rho}. Precisely, thanks to \eqref{vrho-con-line-lem}, \eqref{varrhobiaoshi212} and the H\"older inequality, we get
\begin{align}\label{varrhobiaoshi233}
\displaystyle \|\partial^2_3\varrho\|_{L^2}
& \lesssim (1+t)^{-\frac{7}{4}}(\|(\rho_0,\mathbf{u}_0)\|_{L^1}+\|(\partial^2_3\varrho_0,\partial^2_3\mathbf{u}_0)\|_{L^2})\nonumber\\[2mm]
&\displaystyle \quad+\int^t_0(1+t-\tau)^{-\frac{7}{4}}\|S_1\|_{L^1}\mathrm{d}\tau+\int^t_0(1+t-\tau)^{-m}\|\partial^2_3S_1\|_{L^2}\mathrm{d}\tau\nonumber\\[2mm]
&\displaystyle \quad+\int^t_0(1+t-\tau)^{-\frac{7}{4}}\|S_2\|_{L^1}\mathrm{d}\tau+\int^t_0(1+t-\tau)^{-m}\|\partial^2_3S_2\|_{L^2}\mathrm{d}\tau\nonumber\\[2mm]
&\displaystyle \lesssim (1+t)^{-\frac{7}{4}}(\|(\rho_0,\mathbf{u}_0)\|_{L^1}+\|(\partial^2_3\varrho_0,\partial^2_3\mathbf{u}_0)\|_{L^2})+M_{331}+M_{332}+M_{333}+M_{334},
\end{align}
where, by similar arguments as in the derivation of $\eqref{M121}$, \eqref{M122}, \eqref{M123} and \eqref{M234}, we can control $M_{33i}~(1 \leq i \leq 4)$ as follows
\begin{align*}
\begin{cases}
 \displaystyle   M_{331} \lesssim\mathcal{E}(t)(1+t)^{-\frac{3}{2}},\ \ \ M_{333} \lesssim\mathcal{E}(t)(1+t)^{-1},\\[2mm]
    \displaystyle M_{332}\lesssim\mathcal{E}(t)(1+t)^{-\frac{137}{120}},\ \ \ M_{334} \lesssim\mathcal{E}(t)(1+t)^{-\frac{16}{15}}.
\end{cases}
\end{align*}
Therefore, we have from \eqref{varrhobiaoshi233} that
\begin{align*}
\displaystyle (1+t)\|\partial^2_3\varrho\|_{L^2}\lesssim\mathcal{E}(t)+\|(\rho_0,\mathbf{u}_0)\|_{L^1}+\|(\partial^2_3\varrho_0,\partial^2_3\mathbf{u}_0)\|_{L^2}.
\end{align*}
Now it remains to estimate $\|\partial^2_3\mathbf{u}\|_{L ^{2}}$. Using \eqref{verti-u-conl-lem}, we get
\begin{align}\label{ubiaoshi233}
 \displaystyle\makebox[-8pt]{~}\|\partial^2_3\mathbf{u}\|_{L^2}
&\lesssim(1+t)^{-\frac{7}{4}}\Big(\|\varrho_0\|_{L^1}+\|\partial^2_3\varrho_0\|_{L^2}\Big)+(1+t)^{-\frac{1}{2}}\Big(\|\partial^2_3\mathbf{u}_0\|_{L^2_{x_3}L^1_{x_1x_2}}+\|\partial^2_3\mathbf{u}_0\|_{L^2}\Big)\nonumber\\[2mm]
&\displaystyle \quad+\int^t_0(1+t-\tau)^{-\frac{7}{4}}\|S_1\|_{L^1}\mathrm{d}\tau+\int^t_0(1+t-\tau)^{-m}\|\partial^2_3S_1\|_{L^2}\mathrm{d}\tau\nonumber\\[2mm]
&\displaystyle \quad+\int^t_0(1+t-\tau)^{-\frac{1}{2}}\|\partial^2_3S_2\|_{L^2_{x_3}L^1_{x_1x_2}}\mathrm{d}\tau+\int^t_0(1+t-\tau)^{-m}\|\partial^2_3S_2\|_{L^2}\mathrm{d}\tau\nonumber\\[2mm]
&\displaystyle \lesssim(1+t)^{-\frac{7}{4}}\Big(\|\varrho_0\|_{L^1}+\|\partial^2_3\varrho_0\|_{L^2}\Big)+(1+t)^{-\frac{1}{2}}\Big(\|\partial^2_3\mathbf{u}_0\|_{L^2_{x_3}L^1_{x_1x_2}}+\|\partial^2_3\mathbf{u}_0\|_{L^2}\Big)\nonumber\\[2mm]
&\displaystyle \displaystyle \quad+H_{331}+H_{332}+H_{333}+H_{334},
\end{align}
where in virtue of similar arguments as in the derivation of \eqref{M121}, \eqref{M122} and \eqref{M234} respectively, $H_{331}$, $H_{332}$ and $H_{334}$ can be estimated as
\begin{align}\label{H331}
H_{331} \lesssim\mathcal{E}(t)(1+t)^{-\frac{3}{2}},\ \ \ H_{332}\lesssim\mathcal{E}(t)(1+t)^{-\frac{137}{120}},\ \mbox{ and }\ H_{334}\lesssim\mathcal{E}(t)(1+t)^{-\frac{16}{15}}.
\end{align}
For $H_{333}$, recalling the definition of $ S _{2} $ in \eqref{nonlinear}, we have
\begin{align}\label{h333-split}
\displaystyle H_{333}&\lesssim\int^t_0(1+t-\tau)^{-\frac{1}{2}}\Big(\|\partial^2_3(\mathbf{u}\cdot\nabla\mathbf{u})\|_{L^2_{x_3}L^1_{x_1x_2}}+\|\partial^2_3(g(\varrho)\Delta_h\mathbf{u})\|_{L^2_{x_3}L^1_{x_1x_2}}\Big.\nonumber\\[2mm]
&\displaystyle \quad \Big.+\|\partial^2_3(g(\varrho)\nabla\mathop{\mathrm{div}}\nolimits\mathbf{u})\|_{L^2_{x_3}L^1_{x_1x_2}}+\|\partial^2_3(f(\varrho)\nabla\varrho)\|_{L^2_{x_3}L^1_{x_1x_2}}\Big)\mathrm{d}\tau
 \nonumber \\[2mm]
&\displaystyle=:H_{3331}+H_{3332}+H_{3333}+H_{3334}.
\end{align}
We next consider the term $H_{3331}$. It follows from Lemma \ref{lemma2.1} that
\begin{align*}
&\displaystyle\|\partial^2_3(\mathbf{u}\cdot\nabla \mathbf{u})\|_{L^2_{x_3}L^1_{x_1x_2}}=\|\partial^2_{3}u_j\partial_j\mathbf{u}+2\partial_3u_j\partial_j\partial_3\mathbf{u}+u_j\partial_j\partial^2_3\mathbf{u}\|_{L^2_{x_3}L^1_{x_1x_2}}\nonumber\\[2mm]
&~\displaystyle \lesssim\|\partial_3\mathbf{u}\|^{\frac{1}{2}}_{L^2}\|\partial^2_3\mathbf{u}\|^{\frac{1}{2}}_{L^2}\|\partial^2_3u_3\|_{L^2}+\|\nabla_h\mathbf{u}\|^{\frac{1}{2}}_{L^2}\|\nabla_h\partial_3\mathbf{u}\|^{\frac{1}{2}}_{L^2}\|\partial^2_3\mathbf{u}_h\|_{L^2}\nonumber\\[2mm]
&~\displaystyle \quad+\|\partial_3u_3\|^{\frac{1}{2}}_{L^2}\|\partial^2_3u_3\|^{\frac{1}{2}}_{L^2}\|\partial^2_3\mathbf{u}\|_{L^2}+\|\partial_3\mathbf{u}_h\|^{\frac{1}{2}}_{L^2}\|\partial^2_3\mathbf{u}_h\|^{\frac{1}{2}}_{L^2}\|\nabla_h\partial_3\mathbf{u}\|_{L^2}\nonumber\\[2mm]
&~\displaystyle \quad+\|u_3\|^{\frac{1}{2}}_{L^2}\|\partial_3u_3\|^{\frac{1}{2}}_{L^2}\|\partial^3_3\mathbf{u}\|_{L^2}+\|\mathbf{u}_h\|^{\frac{1}{2}}_{L^2}\|\partial_3\mathbf{u}_h\|^{\frac{1}{2}}_{L^2}\|\nabla_h\partial^2_3\mathbf{u}\|_{L^2}\nonumber\\[2mm]
&~\displaystyle \lesssim\|\partial_3\mathbf{u}\|^{\frac{1}{2}}_{L^2}\|\partial^2_3\mathbf{u}\|^{\frac{1}{2}}_{L^2}(\|\nabla _{h}\partial _{3}\mathbf{u}\|_{L^2}+\|\partial _{3}\mathop{\mathrm{div}}\nolimits\mathbf{u}\|_{L ^{2}})
 \nonumber \\[2mm]
 &~\displaystyle \quad+\|\nabla_h\mathbf{u}\|^{\frac{1}{2}}_{L^2}(\|\nabla_h\partial_3\mathbf{u}\|^{\frac{1}{2}}_{L^2}+\|\partial_3\mathop{\mathrm{div}}\nolimits\mathbf{u}\|^{\frac{1}{2}}_{L^2})\|\partial^2_3\mathbf{u}\|_{L^2}\nonumber\\[2mm]
&~\displaystyle \quad+\|\mathop{\mathrm{div}}\nolimits\mathbf{u}\|^{\frac{1}{2}}_{L^2}(\|\partial_3\mathop{\mathrm{div}}\nolimits\mathbf{u}\|^{\frac{1}{2}}_{L^2}+\|\nabla_h\partial_3\mathbf{u}\| _{L ^{2}}^{\frac{1}{2}})\|\partial^2_3\mathbf{u}\|_{L ^{2}}+\|u_3\|^{\frac{1}{2}}_{L^2}\|\partial_3u_3\|^{\frac{1}{2}}_{L^2}\|\partial^3_3\mathbf{u}\|_{L^2}\nonumber\\[2mm]
&~\displaystyle \quad+\|\mathbf{u}_h\|^{\frac{1}{2}}_{L^2}\|\partial_3\mathbf{u}_h\|^{\frac{1}{2}}_{L^2}\|\nabla_h\partial^2_3\mathbf{u}\|_{L^2}).
\end{align*}
This implies that
\begin{gather}\label{h3331-sp}
\displaystyle H _{3331}=  \int^t_0(1+t-\tau)^{-\frac{1}{2}}\|\partial^2_3(\mathbf{u}\cdot\nabla\mathbf{u})\|_{L^2_{x_3}L^1_{x_1x_2}}\mathrm{d}\tau  \lesssim \tilde{H}_{1}+\tilde{H}_{2},
\end{gather}
with
\begin{align}\label{H1-TILD}
\displaystyle \tilde{H}_{1}&= \int^t_0(1+t-\tau)^{-\frac{1}{2}}\Big(\|\partial_3\mathbf{u}\|^{\frac{1}{2}}_{L^2}\|\partial^2_3\mathbf{u}\|^{\frac{1}{2}}_{L^2}(\|\nabla _{h}\partial _{3}\mathbf{u}\|_{L^2}+\|\partial _{3}\mathop{\mathrm{div}}\nolimits\mathbf{u}\|_{L ^{2}})
 \nonumber \\[2mm]
 & \displaystyle \qquad\Big.
  +\|\nabla_h\mathbf{u}\|^{\frac{1}{2}}_{L^2}(\|\nabla_h\partial_3\mathbf{u}\|^{\frac{1}{2}}_{L^2}+\|\partial _{3}\mathop{\mathrm{div}}\nolimits \mathbf{u}\|_{L ^{2}}^{\frac{1}{2}})\|\partial^2_3\mathbf{u}\|_{L^2}+\|u_3\|^{\frac{1}{2}}_{L^2}\|\partial_3u_3\|^{\frac{1}{2}}_{L^2}\|\partial^3_3\mathbf{u}\|_{L^2}
   \nonumber \\[2mm]
   &\displaystyle \qquad \Big.+\|\mathop{\mathrm{div}}\nolimits\mathbf{u}\|^{\frac{1}{2}}_{L^2}(\|\partial_3\mathop{\mathrm{div}}\nolimits\mathbf{u}\|^{\frac{1}{2}}_{L^2}+\|\nabla_h\partial_3\mathbf{u}\|^{\frac{1}{2}}_{L^2})\|\partial^2_3\mathbf{u}\|_{L^2}\Big)\mathrm{d}\tau \nonumber\\[2mm]
  &\displaystyle \lesssim \Big(\mathcal{E}_2(t)+\mathcal{E}_{2}^{\frac{1}{2}}(t)\mathcal{E}^{\frac{1}{2}}_1(t)\Big)\int^t_0(1+t-\tau)^{-\frac{1}{2}}(1+\tau)^{-\frac{3}{2}}(\log(1+\tau))^{\frac{1}{2}}\mathrm{d}\tau \nonumber\\[2mm]
 &\displaystyle \quad
+\left(\mathcal{E}_2(t)+\mathcal{E}_{2}^{\frac{3}{4}}\mathcal{E}^{\frac{1}{4}}_1(t)\right)\int^t_0(1+t-\tau)^{-\frac{1}{2}}(1+\tau)^{-\frac{3}{2}}\log(1+\tau)\mathrm{d}\tau
 \nonumber\\[2mm]
  & \displaystyle \quad
+\Big(\mathcal{E}^{\frac{1}{4}}_2(t)\mathcal{E}^{\frac{1}{4}}_1(t)+\mathcal{E}_{2}^{\frac{1}{2}}(t)\Big)\mathcal{E}^{\frac{1}{2}}_0(t)\int^t_0(1+t-\tau)^{-\frac{1}{2}}(1+\tau)^{-1}\log(1+\tau)\mathrm{d}\tau \nonumber\\[2mm]
&\displaystyle \quad
+\Big(\mathcal{E}^{\frac{1}{2}}_1(t)\mathcal{E}^{\frac{1}{2}}_2(t)+\mathcal{E}^{\frac{1}{4}}_1(t)\mathcal{E}^{\frac{3}{4}}_2(t)\Big)\int^t_0(1+t-\tau)^{-\frac{1}{2}}(1+\tau)^{-\frac{3}{2}}\log(1+\tau)\mathrm{d}\tau
 \nonumber \\[2mm]
&\displaystyle \lesssim\mathcal{E}(t)\Big((1+t)^{-\frac{1}{2}}+(1+t)^{-\frac{1}{2}}\log(1+\tau)\Big) \nonumber\\[2mm]
&\displaystyle \lesssim\mathcal{E}(t)(1+t)^{-\frac{1}{2}}\log(1+t)
\end{align}
and
\begin{align*}
\displaystyle \tilde{H}_{2}&= \int^t_0(1+t-\tau)^{-\frac{1}{2}}\|\mathbf{u}_h\|^{\frac{1}{2}}_{L^2}\|\partial_3\mathbf{u}_h\|^{\frac{1}{2}}_{L^2}\|\nabla_h\partial_3\mathbf{u}\|^{\frac{1}{2}}_{L^2}\|\nabla_h\partial^3_3\mathbf{u}\|^{\frac{1}{2}}_{L^2}\mathrm{d}\tau
 \nonumber \\
 &\displaystyle \lesssim\mathcal{E}^{\frac{3}{4}}_2(t)\mathcal{E}^{\frac{1}{4}}_0(t)\int^t_0(1+t-\tau)^{-\frac{1}{2}}(1+\tau)^{-1}\mathrm{d}\tau
  \nonumber \\
 &\displaystyle \lesssim\mathcal{E}^{\frac{1}{2}}_2(t)\mathcal{E}^{\frac{1}{2}}_0(t)(1+t)^{-\frac{1}{2}}\log(1+t)\nonumber\\
 &\displaystyle \lesssim\mathcal{E}(t)(1+t)^{-\frac{1}{2}}\log(1+t).
\end{align*}
Here we have used Lemma \ref{decaylemma} and \eqref{dissip-1}. Therefore we have from \eqref{h3331-sp} that
\begin{align}\label{H3331}
\displaystyle
H_{3331} \lesssim\mathcal{E}(t)(1+t)^{-\frac{1}{2}}\log(1+t).
\end{align}
For the term $H_{3332}$ on the right hand side of \eqref{h333-split}, using \eqref{two}, we have
\begin{align*}
&\displaystyle\|\partial^2_3(g(\varrho)\Delta_h\mathbf{u})\|_{L^2_{x_3}L^1_{x_1x_2}}
 \nonumber \\[2mm]
 &~\displaystyle\lesssim \|\partial^2_3\varrho\Delta_h\mathbf{u}\|_{L^2_{x_3}L^1_{x_1x_2}}+2\|\partial_3\varrho\partial_3\Delta_h\mathbf{u}\|_{L^2_{x_3}L^1_{x_1x_2}}+\|\varrho\partial^2_3\Delta_h\mathbf{u}\|_{L^2_{x_3}L^1_{x_1x_2}}\nonumber\\[2mm]
&~\displaystyle \lesssim\|\partial^2_3\varrho\|^{\frac{1}{2}}_{L^2}\|\partial^3_3\varrho\|^{\frac{1}{2}}_{L^2}\|\Delta_h\mathbf{u}\|_{L^2}+\|\partial_3\varrho\|^{\frac{1}{2}}_{L^2}\|\partial^2_3\varrho\|^{\frac{1}{2}}_{L^2}\|\partial_3\Delta_h\mathbf{u}\|_{L^2}+\|\varrho\|^{\frac{1}{2}}_{L^2}\|\partial_3\varrho\|^{\frac{1}{2}}_{L^2}\|\partial^2_3\Delta_h\mathbf{u}\|_{L^2}.
\end{align*}
Hence, we have
\begin{align}\label{H3332}
\displaystyle\makebox[-4pt]{~} H_{3332}&=\int^t_0(1+t-\tau)^{-\frac{1}{2}}\|\partial^2_3(g(\varrho)\Delta_h\mathbf{u})\|_{L^2_{x_3}L^1_{x_1x_2}}\mathrm{d}\tau\nonumber\\
&\displaystyle \lesssim\sup_{0\leq\tau\leq t}\Big[(1+\tau)^{\frac{3}{4}}\|\partial^2_3\varrho\|^{\frac{3}{4}}_{L^2}\|\partial^4_3\varrho\|^{\frac{1}{4}}_{L^2}(1+\tau)^{\frac{16}{15}}\|\Delta_h\mathbf{u}\|_{L^2}\Big]\int^t_0(1+t-\tau)^{-\frac{1}{2}}(1+\tau)^{-\frac{109}{60}}\mathrm{d}\tau\nonumber\\[2mm]
&\displaystyle \quad+\sup_{0\leq\tau\leq t}\Big[(1+\tau)^{\frac{1}{2}}\|\partial_3\varrho\|^{\frac{1}{2}}_{L^2}(1+\tau)^{\frac{1}{2}}\|\partial^2_3\varrho\|^{\frac{1}{2}}_{L^2}(1+\tau)^{\frac{8}{15}}\|\Delta_h\mathbf{u}\|^{\frac{1}{2}}_{L^2}\|\partial^2_3\Delta_h\mathbf{u}\|^{\frac{1}{2}}_{L^2}\Big]
\nonumber \\[2mm]
 & \displaystyle \quad \qquad \cdot\int^t_0(1+t-\tau)^{-\frac{1}{2}}(1+\tau)^{-\frac{23}{15}}\mathrm{d}\tau\nonumber\\
&\displaystyle \quad+\sup_{0\leq\tau\leq t}\Big[(1+\tau)^{\frac{3}{8}}(\log(1+\tau))^{-\frac{1}{2}}\|\varrho\|^{\frac{1}{2}}_{L^2}(1+\tau)^{\frac{1}{2}}\|\partial_3\varrho\|^{\frac{1}{2}}_{L^2}\Big]
 \nonumber \\[2mm]
&\displaystyle \quad \qquad \cdot \int^t_0(1+t-\tau)^{-\frac{1}{2}}(1+\tau)^{-\frac{113}{120}}(\log(1+\tau))^{\frac{1}{2}}(1+\tau)^{\frac{1}{15}}\|\partial^2_3\Delta_h\mathbf{u}\|_{L^2}\mathrm{d}\tau\nonumber\\
&\displaystyle \lesssim(\mathcal{E}_{2}^{\frac{7}{8}}(t)\mathcal{E}_{0}^{\frac{1}{8}}(t)+\mathcal{E}_{2}^{\frac{3}{4}}(t)\mathcal{E}_{0}^{\frac{1}{4}}(t))(1+t)^{-\frac{1}{2}}
+\mathcal{E}_{2}^{\frac{1}{2}}(t)\Big(\int^t_0(1+t-\tau)^{-1}(1+\tau)^{-\frac{113}{60}}\log(1+\tau)\mathrm{d}\tau\Big)^{\frac{1}{2}}\nonumber\\
&\quad\quad\cdot\Big(\int^t_0(1+\tau)^{\frac{2}{15}}\|\partial^2_3\Delta_h\mathbf{u}\|^2_{L^2}\mathrm{d}\tau\Big)^{\frac{1}{2}}   \notag\\[2mm]
&\displaystyle \lesssim(\mathcal{E}_{2}^{\frac{7}{8}}(t)\mathcal{E}_{0}^{\frac{1}{8}}(t)+\mathcal{E}_{2}^{\frac{3}{4}}(t)\mathcal{E}_{0}^{\frac{1}{4}}(t)+\mathcal{E}^{\frac{1}{2}}_2(t)\mathcal{E}^{\frac{1}{2}}_1(t))(1+t)^{-\frac{1}{2}}\nonumber\\
&\displaystyle \lesssim\mathcal{E}(t)(1+t)^{-\frac{1}{2}}.
\end{align}
Similarly, we have
\begin{align}\label{H3333}
H_{3333}&=\int^t_0(1+t-\tau)^{-\frac{1}{2}}\|\partial^2_3(g(\varrho)\nabla \mathop{\mathrm{div}}\nolimits\mathbf{u})\|_{L^2_{x_3}L^1_{x_1x_2}}\mathrm{d}\tau \lesssim\mathcal{E}(t)(1+t)^{-\frac{1}{2}}.
\end{align}
Moreover, using \eqref{two} again, we have
\begin{align*}
&\displaystyle\|\partial^2_3(f(\varrho)\nabla\varrho)\|_{L^2_{x_3}L^1_{x_1x_2}}
 \nonumber \\[2mm]
 &~\displaystyle\lesssim \|\partial^2_3f(\varrho)\nabla\varrho\|_{L^2_{x_3}L^1_{x_1x_2}}+2\|\partial_3f(\varrho)\nabla\partial_3\varrho\|_{L^2_{x_3}L^1_{x_1x_2}}+\|\varrho\nabla\partial^2_3\varrho\|_{L^2_{x_3}L^1_{x_1x_2}}\nonumber\\[2mm]
&~\displaystyle \lesssim\|\partial^2_3\varrho\|^{\frac{1}{2}}_{L^2}\|\partial^3_3\varrho\|^{\frac{1}{2}}_{L^2}\|\nabla\varrho\|_{L^2}+\|\partial_3\varrho\|^{\frac{1}{2}}_{L^2}\|\partial^2_3\varrho\|^{\frac{1}{2}}_{L^2}\|\nabla\partial_3\varrho\|_{L^2}+\|\varrho\|^{\frac{1}{2}}_{L^2}\|\partial_3\varrho\|^{\frac{1}{2}}_{L^2}\|\nabla\partial^2_3\varrho\|_{L^2}.
\end{align*}
This along with Lemma \ref{decaylemma} further yields that
\begin{align}\label{H3334}
H_{3334}&=\int^t_0(1+t-\tau)^{-\frac{1}{2}}\|\partial^2_3(f(\varrho)\nabla\varrho)\|_{L^2_{x_3}L^1_{x_1x_2}}\mathrm{d}\tau\nonumber\\[2mm]
&\displaystyle \lesssim\sup_{0\leq \tau\leq t}\Big[(1+\tau)^{\frac{3}{4}}\|\partial^2_3\varrho\|^{\frac{3}{4}}_{L^2}\|\partial^4_3\varrho\|^{\frac{1}{4}}_{L^2}(1+\tau)\|\nabla\varrho\|_{L^2}\Big]
 \nonumber \\
 &\displaystyle \quad \qquad \cdot
\int^t_0(1+t-\tau)^{-\frac{1}{2}}(1+\tau)^{-\frac{7}{4}}\mathrm{d}\tau\nonumber\\[2mm]
&\displaystyle \quad+\sup_{0\leq \tau\leq t}\Big[(1+\tau)^{\frac{1}{2}}\|\partial_3\varrho\|^{\frac{1}{2}}_{L^2}(1+\tau)^{\frac{1}{2}}\|\partial^2_3\varrho\|^{\frac{1}{2}}_{L^2}(1+\tau)\|\nabla\partial_3\varrho\|_{L^2}\Big]\nonumber\\[2mm]
&\displaystyle \quad \qquad\cdot\int^t_0(1+t-\tau)^{-\frac{1}{2}}(1+\tau)^{-2}\mathrm{d}\tau\nonumber\\[2mm]
&\displaystyle \quad+\sup_{0\leq \tau\leq t}\Big[(1+\tau)^{\frac{3}{8}}(\log(1+\tau))^{-\frac{1}{2}}\|\varrho\|^{\frac{1}{2}}_{L^2}(1+\tau)^{\frac{1}{2}}\|\partial_3\varrho\|^{\frac{1}{2}}_{L^2}(1+\tau)^{\frac{1}{2}}\|\partial^2_3\varrho\|^{\frac{1}{2}}_{L^2}\|\nabla^2\partial^2_3\varrho\|^{\frac{1}{2}}_{L^2}\Big]
\nonumber \\
 &\displaystyle \quad \qquad \cdot\int^t_0(1+t-\tau)^{-\frac{1}{2}}(1+\tau)^{-\frac{11}{8}}(\log(1+\tau))^{\frac{1}{2}}\mathrm{d}\tau\nonumber\\[2mm]
&\displaystyle \lesssim\mathcal{E}_{2}^{\frac{7}{8}}(t)\mathcal{E}_{0}^{\frac{1}{8}}(t)(1+t)^{-\frac{1}{2}}
+\mathcal{E}_{2}(t) (1+t)^{-\frac{1}{2}}
+ \mathcal{E}^{\frac{3}{4}}_{2}(t)\mathcal{E}^{\frac{1}{4}}_0(t)(1+t)^{-\frac{1}{2}}  \notag\\[2mm]
&\displaystyle \lesssim\mathcal{E}(t)(1+t)^{-\frac{1}{2}}.
\end{align}
Summing up the estimates \eqref{H3331}--\eqref{H3334}, we obtain that
\begin{align*}
H_{333}=\int^t_0(1+t-\tau)^{-\frac{1}{2}}\|\partial^2_3S_2\|_{L^2_{x_3}L^1_{x_1x_2}}\mathrm{d}\tau \lesssim\mathcal{E}(t)(1+t)^{-\frac{1}{2}}\log(1+\tau).
\end{align*}
This combined with \eqref{ubiaoshi233} and \eqref{H331} gives
\begin{align*}
\displaystyle(1+t)^{\frac{1}{2}}(\log(1+t))^{-1}\|\partial^2_3\mathbf{u}\|_{L^2}\lesssim\mathcal{E}(t)+\|\varrho_0\|_{L^1}+\|\partial^2_3\mathbf{u}_0\|_{L^2_{x_3}L^1_{x_1x_2}}+\|(\partial^2_3\mathbf{u}_0,\partial^2_3\varrho_0)\|_{L^2},
\end{align*}
and thus ends the proof of Lemma \ref{lemma4.8}.
\end{proof}
Gathering the estimates of $\mathcal{E}_0$, $\mathcal{E}_1$ and $\mathcal{E}_2$ in Proposition \ref{prop4.1} and Lemmas \ref{lemma5.5}--\ref{lemma4.8}, we ultimately get an estimate of the form \eqref{goal-esti} which can be precisely stated as
\begin{align}\label{E}
\mathcal{E}(t) \lesssim F(\varrho_0,\mathbf{u}_0)+ \mathcal{E}^{\frac{3}{2}}(t)+\mathcal{E}^2(t),
\end{align}
with
\begin{align*}
F(\varrho_0,\mathbf{u}_0)=\|(\varrho_0,\mathbf{u}_0)\|^2_{H^4}+\|(\varrho_0,\mathbf{u}_0)\|^2_{L^1}+\|(\mathbf{u}_0,\partial_3\mathbf{u}_0,\partial^2_3\mathbf{u}_0)\|^2_{L^2_{x_3}L^1_{x_1x_2}}.
\end{align*}

\vspace{4mm}

\subsection{Proof of Theorem \ref{thm2}}
With the global existence result in Section \ref{FIR-PROOF} as well as the estimate \eqref{E} at hand, to finish the proof of Theorem \ref{thm2}, it suffices to varify the \emph{a priori} assumptions in \eqref{appri-assum-Nonlin}. Indeed, once the \emph{a priori }assumptions are varified, one can immediately utilize the bootstrap arguments to obtain the estimates in Theorem \ref{thm2}. To proceed, let us recall \eqref{appri-assum-Nonlin} again and set $ \Lambda $ to be small enough such that \eqref{E} can be rewritten as
\begin{align*}
\displaystyle \mathcal{E}(t) &\lesssim F(\varrho_0,\mathbf{u}_0)+ \mathcal{E}^{\frac{3}{2}}(t)+\mathcal{E}^2(t)
 \nonumber \\
  & \displaystyle \leq C _{1}F(\varrho_0,\mathbf{u}_0)+ \frac{1}{2}\mathcal{E}(t)
\end{align*}
for some constant $ C _{1}>0 $. Then it follows that
\begin{gather*}
\displaystyle \mathcal{E}(t) \leq 2 C _{1} F(\varrho_0,\mathbf{u}_0).
\end{gather*}
Therefore if we take $ \Lambda =4C _{1}F(\varrho_0,\mathbf{u}_0)$ and set $ F(\varrho_0,\mathbf{u}_0) $ to be small enough, then it holds that
\begin{gather*}
 \displaystyle  \mathcal{E}(t) \leq 2 C _{1} F(\varrho_0,\mathbf{u}_0)<\Lambda.
 \end{gather*}
 This varifies the \emph{a priori} assumptions in \eqref{appri-assum-Nonlin}, and thus ends the proof of Theorem \ref{thm2}. \hfill $ \square $

\vspace*{2cm}

\section*{Acknowledgement}
The research of Z.F. Feng was supported by the National Natural Science Foundation of China (Grant No. 12101095), the Research Project of Chongqing Education Commission (Grant No. CXQT21014), the Talent Program of Chongqing Normal University (Grant No. BWQB2023008) and the grant of Chongqing Young Experts' Workshop. The research of G.Y. Hong was supported by the National Natural Science Foundation of China (Grant No. 12522111), the Guangdong Provincial Pearl River Talents Program (Grant No. 2023QN10X436), the Guangdong Basic and Applied Basic Research Foundation (Grant No. 2024A1515012306), and the Fundamental Research Funds for the
Central Universities (Grant No. 2025ZYGXZR040). The research of Y.H. Wang was partially supported by the National Natural Science Foundation of
China (Grant No. 12401274) and the Natural Science Foundation of Hunan Province (Grant No. 2024JJ6302). J. Wu was partially supported by the National Science Foundation of USA (Grant Nos. DMS 2104682 and DMS 2309748).


\end{document}